\documentclass[11pt,reqno]{amsart}

\usepackage[latin9]{luainputenc}
\usepackage{amsfonts}
\usepackage{amsmath}
\usepackage{amstext}
\usepackage{amsthm}
\usepackage{amssymb}
\usepackage{braket}
\usepackage{bm}
\usepackage{cancel}
\usepackage{caption}
\usepackage{color}
\usepackage{enumerate}
\usepackage{enumitem}
\usepackage{esint}
\usepackage{extarrows}
\usepackage[a4paper,margin=1in,footskip=0.25in]{geometry}
\usepackage{graphicx}
\usepackage{longtable}
\usepackage{mathrsfs}
\usepackage{mathtools}
\usepackage{dsfont}
\usepackage{mdframed}
\usepackage{stmaryrd}
\usepackage{soul}
\usepackage{subcaption}
\usepackage{tikz}
\usepackage{tikz-cd}
\usepackage{xcolor}
\usepackage{ulem}
\usepackage{hyperref}

\numberwithin{equation}{section}
\graphicspath{ {./images/} }
\newtheorem{defi}{Definition}[section]
\newtheorem{lem}[defi]{Lemma}
\newtheorem{theo}[defi]{Theorem}
\newtheorem{cor}[defi]{Corollary}
\newtheorem{pro}[defi]{Proposition}

\newtheorem{rem}[defi]{Remark}

\newtheorem{hyp}{Hypothesis}
\newtheorem{assump}{Assumption}

\DeclareMathOperator{\divop}{div}

\DeclareMathOperator*{\esssup}{ess\,sup}
\DeclareMathOperator{\N}{\mathbb{N}}
\DeclareMathOperator{\R}{\mathbb{R}}

\DeclareMathOperator{\bj}{\mathbf{j}}
\DeclareMathOperator{\hbj}{\hat{\mathbf{j}}}
\DeclareMathOperator{\bxi}{\boldsymbol{\xi}}
\DeclareMathOperator{\hbxi}{\hat{\boldsymbol{\xi}}}
\DeclareMathOperator{\bx}{\mathbf{x}}
\DeclareMathOperator{\bz}{\mathbf{z}}

\definecolor{darkmagenta}{rgb}{0.55, 0.0, 0.55}

\definecolor{greenn}{rgb}{0.0, 0.5, 0.0}

\newcommand{\corr}[1]{{\color{black}{#1}}}

\allowdisplaybreaks 

\title[Multi-agent systems over hypergraphs with unbounded rank]{Mean-field limit of non-exchangeable multi-agent systems over hypergraphs with unbounded rank}

\author[Nathalie Ayi]{Nathalie Ayi}
\address{Sorbonne Universit\'e, Universit\'e Paris Cit\'e, CNRS, Inria, Laboratoire Jacques-Louis Lions (LJLL), Institut Universitaire de France, F-75005 Paris, France}
\email{nathalie.ayi@sorbonne-universite.fr}

\author[Nastassia Pouradier Duteil]{Nastassia Pouradier Duteil}
\address{Sorbonne Universit\'e, Inria,  Universit\'e Paris Cit\'e, CNRS, Laboratoire Jacques-Louis Lions (LJLL),EPC MUSCLEES, F-75005 Paris, France}
\email{nastassia.pouradier\_duteil@sorbonne-universite.fr}

\author[David Poyato]{David Poyato}
\address{Departamento de Matem\'atica Aplicada and Research Unit ``Modeling Nature'' (MNat), Facultad de Ciencias, Universidad de Granada, 18071 Granada, Spain}
\email{davidpoyato@ugr.es}

\begin{document}

\date{\today}



\thanks{\textbf{Acknowledgment.} NA has received funding from the Emergence grant of Sorbonne University and partially from the French National Research Agency for the project CONVIVIALITY (ANR-23-CE40-0003). NPD has received funding from the Emergence grant of Sorbonne University. DP has received funding from the European Union's Horizon Europe research and innovation program under the Marie Sk{\l}odowska-Curie grant agreement No 101064402, and partially from Grant C-EXP-265-UGR23 funded by Consejer\'ia de Universidad, Investigaci\'on e Innovaci\'on \& ERDF/EU Andalusia Program, from Grant PID2022-137228OB-I00 funded by the Spanish Ministerio de Ciencia, Innovaci\'on y Universidades, MICIU/AEI/10.13039/501100011033 \& ``ERDF/EU A way of making Europe'', and from Modeling Nature Research Unit, project QUAL21-011. We also thank the creators of the HGX Python library \cite{L-23-HGX}, which we have used in some of our plots.}

\begin{abstract}
Interacting particle systems are known for their ability to generate large-scale self-organized structures from simple local interaction rules between each agent and its neighbors. In addition to studying their emergent behavior, a main focus of the mathematical community has been concentrated on deriving their large-population limit. In particular, the mean-field limit consists of describing the limit system by its population density in the product space of positions and labels. The strategy to derive such limits is often based on a careful combination of methods ranging from analysis of PDEs and stochastic analysis, to kinetic equations and graph theory. In this article, we focus on a generalization of multi-agent systems that includes higher-order interactions, which has largely captured the attention of the applied community in the last years. In such models, interactions between individuals are no longer assumed to be binary ({\it i.e.} between a pair of particles). Instead, individuals are allowed to interact by groups so that a full group jointly generates a non-linear force on any given individual. The underlying graph of connections is then replaced by a hypergraph, which we assume to be dense, but possibly non-uniform and of unbounded rank. For the first time in the literature, we show that when the interaction kernels are regular enough, then the mean-field limit is determined by a limiting Vlasov-type equation, where the hypergraph limit is encoded by a so-called UR-hypergraphon (unbounded-rank hypergraphon), and where  the resulting mean-field force admits infinitely-many orders of interactions.
\end{abstract}

\maketitle

\tableofcontents

\section{Introduction}\label{sec:introduction}

The study of collective behavior has attracted much attention in the last twenty years, including within the mathematical community. Its numerous applications include among many others the study of animal group coordination, opinion formation and cell organization in tissues. Both theoretical and experimental communities have tackled the question of self-organization, in the aim of elucidating how global collective patterns emerge in a group driven by only local interactions.
To address this question, mathematical models have been developed, relying mainly on two important hypotheses: 
\begin{hyp}[Exchangeability]
The particles are assumed to be indistinguishable (or \textbf{exchangeable}), that is that exchanging any two particles does not modify the overall dynamics of the group. In other words, the population is composed of identical particles. \label{HE}
\end{hyp}
\begin{hyp}[Pairwise interaction]
The particles are assumed to interact \textbf{pairwise}. More precisely, this implies that the influence of the group on each given particle is viewed as the superposition of the pairwise interactions between this particle and all remaining ones. \label{HP}
\end{hyp}

Hypotheses \ref{HE} and \ref{HP} can be used to write a simple general model giving the evolution of each particle's position (or opinion, or velocity, etc.) $X_i^N\in\R^d$ as the result of pairwise interactions with all remaining particles:
\begin{equation}
\frac{dX_i^N(t)}{dt} = \frac{1}{N} \sum_{j=1}^N K(X_i^N(t),X_j^N(t)), \qquad i\in\{1,\ldots,N\}.
\label{eq:micro_exch}
\end{equation}

Here, because of Hypothesis \ref{HE}, the interaction weights are all equal (normalized to $\frac{1}{N}$), and due to Hypothesis \ref{HP}, the total effect of the group on the evolution of $X_i^N$ is modeled by the sum of the pairwise interactions. 

Although this model provides a good description of many multi-agent systems with indistinguishable particles, in other applications, Hypotheses \ref{HE} and \ref{HP} need to be revised. 
One way to model the non-exchangeable nature of the particles ({\it i.e.} to remove Hypothesis \ref{HE}) is to describe their pairwise interaction via an underlying weighted graph. This leads to the system 
\begin{equation}
\frac{dX_i^N(t)}{dt} = \sum_{j=1}^N w_{ij}^N K(X_i^N(t),X_j^N(t)), \qquad i\in\{1,\ldots,N\},
\label{eq:micro_nonexch}
\end{equation}
in which the interaction between two particles $i$ and $j$ no longer depends only on their positions $X_i^N$ and $X_j^N$ in the state space (via the interaction kernel $K$) but also on their individual labels $i$ and $j$, via the weight $w_{ij}^N$. Agents' labels play the role of the vertices $V_N$ of the underlying graph $G_N = (V_N, E_N, W_N)$, by setting $E_N:=\{1,\ldots,N\}$, while directed edges $V_N := \{(i,j)\in V_N^2:\, w_{ij}^N\neq 0\}$ are entirely determined by their adjacency matrix $W_N := (w_{ij}^N)_{1\leq i,j\leq N}$.

A growing body of works has subsequently introduced a generalization of this second model in order to remove Hypothesis \ref{HP}, relying on the theory of \textit{hypergraphs} \cite{Battiston-20,Battiston22,B-21}. These approaches are based on the idea that the basic interaction unit of many dynamical systems includes more than two nodes \cite{LCB-20,Sahasrabuddhe21,MTB-20,Skardal20,WXZ-21,XS-21,XWS-20,ZLB-23}. 
Applications are numerous and involve the modeling of opinion dynamics \cite{Boccaletti2023,CGL-24,Neuhauser22,SchaweHernandez22}, contagion propagation \cite{Battiston-20,MajhiPercMatjaz22}, synchronization of oscillators \cite{XWS-20,XS-21,WXZ-21,LCB-20,ZLB-23}, animal communication \cite{Iacopini24} and evolutionary game dynamics \cite{Battiston-20,MajhiPercMatjaz22}.
This leads to the following multi-agent system with high-order interactions, which will be the focus of the present article:
\begin{equation}\label{eq:multi-agent-system}
\begin{cases}
\displaystyle \frac{dX_i^N(t)}{dt}=\sum_{\ell=1}^{N-1}\sum_{j_1,\ldots, j_\ell=1}^N w^{\ell,N}_{ij_1\cdots j_\ell}\,K_\ell(X_i^N(t),X_{j_1}^N(t),\ldots,X_{j_\ell}^N(t)), \\
X_i^N(0)=X_{i,0}^N, \qquad i\in\{1,\ldots,N\}.
\end{cases}
\end{equation}
The above system \eqref{eq:multi-agent-system} describes the evolution of $N$ agents with states $X_i^N=X_i^N(t)\in \mathbb{R}^d$. Interactions among agents are given as the superposition of all possible $(\ell+1)$-body interactions. The functions $K_\ell=K_\ell(x,x_1,\ldots,x_\ell)$ represent the $(\ell+1)$-body interaction kernels, and each weight $w^{\ell,N}_{ij_1\cdots j_\ell}$ describes the underlying $(\ell+1)$-body couplings or connections among agents. More precisely, the tensor $(w^{\ell,N}_{ij_1\cdots j_\ell})_{1\leq i,j_1\ldots j_\ell\leq N}$ can be seen as the adjacency tensor of a $(\ell+1)$-uniform hypergraph describing the underlying communication architecture among agents. In particular, we note that the above sum moves over all possible types of multi-body interactions ranging from $\ell=1$ ({\it i.e.}, pairwise interactions as in Equation \eqref{eq:micro_nonexch}) to $\ell=N-1$ ({\it i.e.}, full group interactions). For modeling purposes, we neglect summands with $\ell\geq N$ which would necessarily lead to undesired self-interaction. For similar reasons, we shall always assume that $w^{\ell,N}_{ij_1\cdots j_\ell}=0$ whenever $i\in \{j_1,\ldots,j_\ell\}$.

The goal of the present article is to derive the mean-field limit of the multi-agent system \eqref{eq:multi-agent-system} as $N$ tends to infinity. Our result studies a form of ``propagation of chaos'' for this highly non-exchangeable multi-agent system and it provides quantitative convergence rates as $N$ tends to infinity towards the associated Vlasov equation:
\begin{align}\label{eq:vlasov-equation}
\begin{cases}
\partial_t \mu_t^\xi+\divop_x(F_w[\mu_t](\cdot,\xi)\,\mu_t^\xi)=0,\quad t\geq 0,\,x\in \mathbb{R}^d,\,\xi\in [0,1],\\
\mu_{t=0}^\xi=\mu_0^\xi.
\end{cases}
\end{align}
Here, $(\mu_t^\xi)_{\xi\in [0,1]}\subset \mathcal{P}(\mathbb{R}^d)$ is a measurable family of probability measures parametrized by a continuous label $\xi\in [0,1]$ for any $t\geq 0$. Note that all existing individuals contribute to jointly generate a mean-field force onto individuals with label $\xi$ of the form:
\begin{equation}\label{eq:vlasov-equation-force}
\begin{aligned}
F_w[\mu_t](x,\xi):=&\sum_{\ell=1}^\infty \int_{[0,1]^\ell} w_\ell(\xi,\xi_1,\ldots,\xi_\ell)\\
&\quad\times \left(\int_{\mathbb{R}^{d\ell}}K_\ell(x,x_1,\ldots,x_\ell)\,d\mu_t^{\xi_1}(x_1)\,\cdots \,d\mu_t^{\xi_\ell}(x_\ell)\right)d\xi_1,\ldots\,d\xi_\ell.
\end{aligned}
\end{equation}
Additionally, the above object $w=(w_\ell)_{\ell \in \mathbb{N}}$ consists in a sequence of uniformly bounded functions $w_\ell\in L^\infty([0,1]^{\ell+1})$. It describes, at the macroscopic scale, the connection among agents via groups of fixed size $\ell+1$ for every $\ell\in \mathbb{N}$.

Inspired by the classical limit theory of dense graphs ({\it i.e.}, graphons \cite{LS-06}), and their subsequent extensions to sparse graphs ({\it i.e.}, $L^p$ graphons \cite{BCCZ-18,BCCZ-19}, graphops \cite{BS-20} and s-graphons\cite{KLS-19}), some recent works have also proposed limit theories that are valid for the broader class of dense hypergraphs \cite{ES-12,RS-24,Z-15}. On the one hand, the case of hypergraphs with uniform cardinality $k\in \mathbb{N}$ was treated in \cite{ES-12,Z-15} using the notion of $k$-uniform hypergraphons, but a new doubling of variables was introduced compared to the treatment of graphons, which complicates the study of mean-field limits operating over these structures. On the other hand, the limit of non-uniform hypergraphs with bounded ranks was recently considered in \cite{Z-23-arxiv}, where a multilinear extension of graphops \cite{BS-20} was obtained. Lastly, in the more general setting of non-uniform hypergraphs with unbounded rank, a limit theory of dense simplicial complexes was obtained in \cite{RS-24}, which exploits the concept of complexons. Simplicial complexes are a subclass of hypergraphs whose hyperedges are closed under restrictions, that is, all subgroups of nodes forming a hyperedge also form a smaller hyperedge. Their limiting objects were named complexon and consist of a sequence $w=(w_\ell)_{\ell \in \mathbb{N}}$ of equivalence classes (under a suitable quotient map) of uniformly bounded functions $w_\ell\in L^\infty([0,1]^{\ell+1})$ with $\ell\in \mathbb{N}$.  

In this paper, we shall not restrict ourselves to coupling weights determined by complexons, which are defined as the above-mentioned quotient objects, but we will rather exploit the primitive non-quotient objects endowed with a suitable topology, which, for distinction with the terminology in \cite{RS-24}, we will call \textit{UR-hypergraphons}, that is, hypergraphons of unbounded rank ({\it cf.} Section \ref{subsec:UR-hypergraphons}).
This framework allows considering hypergraphs and hypergraphons with no \textit{a priori} bound on the maximum hyperedge size. 

Importantly, several studies have reported the frequent occurrence of large-size hyperedges in hypergraphs constructed from real-world datasets. For instance, in \cite{Davison16}, the authors studied the hypergraph structure of fMRI data acquired among 77 subjects
performing cognitively demanding tasks. It was found that the majority of subjects have a hypergraph composed of one large hyperedge and many small hyperedges, with a rough power law for the distribution of the smaller sizes.
More recently, a detailed analysis of hypergraphs coming from six different fields (such as high-school contacts, email recipients, drug substances, or publication co-authors) have revealed that hyperedge sizes follow a heavy-tail (power law) distribution \cite{Ko2022,Kook2020}. These studies of real-world data highlight the importance of proposing and analyzing models for hypergraphs of unbounded rank, {\it i.e.} hypergraphs whose largest hyperedge can potentially contain all the hypergraph nodes. 

 While in most steps of this article (but not all), the coupling weights encoded by the hypergraphs can be assumed to be quite general (supposing of course some natural summability conditions), some more particular assumptions will be considered in the statement of our main result.
For readability, we thus present a list of sufficient assumptions for our main result to hold, and expose Theorem \ref{theo:main} in this setting. As will be recalled throughout the article, not all assumptions are in fact necessary, and the specific requirements for each result will be specified in the corresponding statements. The following assumptions concern the interaction kernels, the coupling weights and the initial data of the multi-agent system \eqref{eq:multi-agent-system}.

\begin{assump}[On the interaction kernels]\label{hyp:main-kernels}
~
\begin{enumerate}[label=(\roman*)]
\item {\bf (Bounded-Lipschitz interaction kernels)}:\\
Assume that $K_\ell\in {\rm BL}(\mathbb{R}^{d(\ell+1)})$, {\it i.e.}, there exist $B_\ell,\,L_\ell>0$ such that
\begin{align}
|K_\ell(x,x_1,\cdots,x_\ell)|&\leq B_\ell.\label{eq:hypothesis-kernels-bounded}\\
|K_\ell(x,x_1,\cdots,x_\ell)-K_\ell(y,y_1,\ldots,y_\ell)|&\leq L_\ell\left(|x-y|+\sum_{k=1}^\ell |x_k-y_k|\right),\label{eq:hypothesis-kernels-lipschitz}
\end{align}
and denote the bounded-Lipschitz norm of $K_\ell$ by
\begin{equation}\label{eq:kernels-bounded-lipschitz-constant}
BL_\ell:=\max\{B_\ell,L_\ell\},\quad \ell\in \mathbb{N}.
\end{equation}
Additionally, assume that there exists $\eta\geq 0$ such that
\begin{equation}\label{eq:hypothesis-kernels-scaling}
\sum_{\ell=1}^\infty\frac{\sqrt{\ell!}}{\eta^\ell}\,B_\ell<\infty,\qquad \sum_{\ell=1}^\infty \ell L_\ell<\infty.
\end{equation}
\item {\bf (Higher regularity of interaction kernels)}:\\
Assume that $K_\ell\in H^{\frac{d(\ell+1)}{2}+\varepsilon}$ for some $\varepsilon>0$, and additionally
\begin{equation}\label{eq:hypothesis-kernels-regularity-scaling}
\sum_{\ell=1}^\infty \frac{4^\ell\pi^{\frac{d\ell}{4}}}{\sqrt{\Gamma(\frac{d\ell}{2})}}\Vert K_\ell\Vert_{H^{\frac{d(\ell+1)}{2}+\varepsilon}}<\infty,
\end{equation}
where $\Gamma$ represents the Gamma function $\Gamma(a):=\int_0^\infty s^{a-1}e^{-s}\,ds$ for $a>0$.
\end{enumerate}
\end{assump}

\begin{assump}[On the coupling weights]\label{hyp:main-weights}
~
\begin{enumerate}[label=(\roman*)]
\setcounter{enumi}{2}
\item {\bf (Absence of loops)}:\\
Assume that $w_{ij_1\cdots j_\ell}^{\ell,N}$ verify
\begin{equation}\label{eq:hypothesis-absence-loops}
w_{ij_1\cdots j_\ell}^{\ell,N}=0,\quad \mbox{ if }i\in \{j_1,\ldots,j_\ell\}.
\end{equation}
\item {\bf (Scaling of the coupling weights)}:\\
Assume that there exists $W>0$ such that
\begin{equation}\label{eq:hypothesis-weights-uniform-bound}
\sup_{N\in \mathbb{N}}\,\max_{1\leq \ell\leq N-1}\,\max_{1\leq i,j_1,\cdots,j_\ell\leq N} N^\ell w^{\ell,N}_{ij_1\cdots j_\ell}\leq W,
\end{equation}
\end{enumerate}
\end{assump}

\begin{assump}[On the symmetry of kernels and weights]\label{hyp:main-symmetry}
~
\begin{enumerate}[label=(\roman*)]
\setcounter{enumi}{4}
\item {\bf (Symmetry of coupling weights)}:\\
Assume that $w_{ij_1\ldots,j_\ell}^{\ell,N}$ verify
\begin{equation}\label{eq:hypothesis-weights-full-symmetry}
w^{\ell,N}_{ij_1\cdots j_\ell}=w^{\ell,N}_{\sigma(i)\sigma(j_1)\cdots \sigma(j_\ell)},\quad \forall\,\sigma\in \mathcal{S}_{\ell+1},
\end{equation}
which in particular implies
\begin{equation}\label{eq:hypothesis-weights-symmetry}
w^{\ell,N}_{ij_1\cdots j_\ell}=w^{\ell,N}_{i\sigma(j_1)\cdots \sigma(j_\ell)},\quad \forall\,\sigma\in \mathcal{S}_{\ell}.
\end{equation}
\item {\bf (Symmetry of interaction kernels)}:\\
Assume that $K_\ell$ verify
\begin{equation}\label{eq:hypothesis-kernels-symmetry}
K_\ell(x,x_1,\cdots,x_\ell)=K_\ell(x,x_{\sigma(1)},\ldots,x_{\sigma(\ell)}),\quad \forall\,\sigma\in \mathcal{S}_{\ell}.
\end{equation}
\end{enumerate}
\end{assump}

\begin{assump}[On the initial data]\label{hyp:main-initial-data}
~
\begin{enumerate}[label=(\roman*)]
\setcounter{enumi}{6}
\item {\bf (Scaling of the initial data)}:\\
Assume that $X_{i,0}^N$ satisfies
\begin{equation}\label{eq:hypothesis-initial-data}
\corr{\sup_{N\in \mathbb{N}}\max_{1\leq i\leq N}\mathbb{E}|X_{i,0}^N|<\infty.}
\end{equation}
\end{enumerate}
\end{assump}

In Assumption 3 $(v)$, $\mathcal{S}_{\ell}$ denotes the set of permutations of $\llbracket 1,\ell\rrbracket$ (see Notations). Given this list of assumptions, the main contribution of the article can be stated in an informal way as follows:

\begin{theo}\label{theo:main}
Assume that the kernels $K_\ell$ and the weights $w_{ij_1\ldots j_\ell}^{\ell,N}$ satisfy Assumptions  \ref{hyp:main-kernels}, \ref{hyp:main-weights} and \ref{hyp:main-symmetry}. For any $(X_{1,0}^N,\ldots,X_{N,0}^N)$ with {\it i.i.d.} $X_{i,0}^N$ (but $N$ dependent law) satisfying Assumption \ref{hyp:main-initial-data}, consider the unique solutions $(X_1^N,\ldots,X_N^N)$ to \eqref{eq:multi-agent-system}. Then, there is a subsequence $N_k\to \infty$ such that the mean-field limit of the multi-agent system \eqref{eq:multi-agent-system} is characterized in a suitable sense by a solution to the Vlasov-type equation \eqref{eq:vlasov-equation}-\eqref{eq:vlasov-equation-force} for some $(\mu^\xi_t)_{\xi\in [0,1]}\subset \mathcal{P}(\mathbb{R}^d)$ and some $w=(w_\ell)_{\ell\in \mathbb{N}}$ such that $\sup_{\ell\in \mathbb{N}}\Vert w_\ell\Vert_{L^\infty}\leq W$.
\end{theo}

A more rigorous statement is given in Theorem \ref{theo:main-rigorous-formulation}. We refer to Remarks \ref{rem:stability-estimate-why-regularity-assumption} and \ref{rem:stability-estimate-why-Fourier-transform} for some relaxation of the regularity assumption \eqref{eq:hypothesis-kernels-regularity-scaling} under more particular shapes on the interaction kernels $K_\ell$, and to Remark \ref{rem:main-why-iid-initial-data} for some comments about the need for the {\it i.i.d.} assumption on the initial data. Finally, we also refer to Remark \ref{rem:main-hierarchy-observables} for an alternative reformulation of the Vlasov equation \eqref{eq:vlasov-equation}-\eqref{eq:vlasov-equation-force} in terms of a hierarchy of observables indexed by directed hypertrees (which extends the hierachy indexed by trees and derived in \cite{JPS-25} in the binary case).

The paper is organized as follows. Section \ref{sec:preliminaries} is devoted to recalling the definitions of various types of hypergraphs, as well as of their various limit objects, which will allow us to situate our work in the vast literature of graph and hypergraph limit theories. We also provide a few examples inspired from the existing literature to illustrate the notions of hypergraphs and hypergraphons, establishing rigorously the derivation from the former to the latter in some specific cases.
As remarked in \cite{JPS-25}, the classical notion of propagation of chaos cannot be expected to hold for non-exchangeable particle systems, which are by definition not identically distributed. Instead, we show the weaker notion of propagation of \textit{independence}, generalizing the approach of \cite{JPS-25} to hypergraphs. Section \ref{sec:propagation-independence} is devoted to showing that propagation of independence holds for an auxiliary particle system which approximates the original particle system \eqref{eq:multi-agent-system} as $N$ tends to infinity.
In Section \ref{sec:well-posedness}, we prove the well-posedness of the limit Vlasov equation \eqref{eq:vlasov-equation}-\eqref{eq:vlasov-equation-force} by a fixed-point argument relying on the Lipschitz regularity of the force $F$ in \eqref{eq:vlasov-equation-force} with respect to the measure $\mu_t$, for an $L^p$-Bounded Lipschitz fibered distance.
Section \ref{sec:stability-estimate} is dedicated to proving the stability of the Vlasov equation with respect to both the initial datum $\mu_0$ and the UR-hypergraphon $(w_\ell)$.
In Section \ref{sec:proof-main-result}, we state our main result in full detail and prove convergence of the solution to the particle system \eqref{eq:multi-agent-system} toward the solution to the Vlasov equation  \eqref{eq:vlasov-equation}-\eqref{eq:vlasov-equation-force}. Lastly, in Section \ref{sec:numerics} we illustrate the convergence result of Theorem \ref{theo:main} via numerical simulations based on some examples of hypergraphs and multi-agent dynamics presented in Section \ref{subsec:examples}. We end with some conclusions and future perspectives in Section \ref{sec:conclusions}.

\medskip

{\bf Notation:}\\
Along the paper, we will alleviate the notation by introducing the following shorthand. For any $N\in \mathbb{N}$ we denote $\llbracket1,N\rrbracket=\{1,\ldots,N\}$. Given $\ell\in \mathbb{N}$ and $1\leq k\leq \ell$, we define the multi-indices
$$\bj_\ell=(j_1,\ldots,j_\ell)\in \llbracket1,N\rrbracket^\ell,\quad \hbj_{\ell,k}=(j_1,\ldots,j_{k-1},j_{k+1},\ldots, j_\ell)\in \llbracket1,N\rrbracket^{\ell-1}.$$
Additionally, given $\bj_\ell\in\llbracket1,N\rrbracket^\ell$ and $\sigma\in \mathcal{S}_\ell$ (a permutation of the finite set $\llbracket1,\ell\rrbracket$), we define the new multi-index $\sigma(\bj_\ell)$ by rearranging their components according to $\sigma$, that is, 
$$\sigma(\bj_\ell)=(j_{\sigma(1)},\ldots,j_{\sigma(\ell)})\in  \llbracket1,N\rrbracket^\ell.$$
Similarly, we shall denote the multi-variables
$$\bx_\ell=(x_1,\ldots,x_\ell)\in \mathbb{R}^{d\ell},\quad \bxi_\ell=(\xi_1,\ldots,\xi_\ell)\in [0,1]^\ell,\quad \hbxi_{\ell,k}=(\xi_1,\ldots,\xi_{k-1},\xi_{k+1},\ldots,\xi_\ell)\in [0,1]^{\ell-1}.$$

Let us introduce here the different spaces which will appear through the article. We denote:
\begin{itemize}
    \item for $p \in [1,\infty]$, $L^p_+(X)$ the space of functions belonging to $L^p(X)$ and taking nonnegative values,
    \item  $\mathcal{M}(X)$ the space of finite Borel measures on $X$ and  $\mathcal{M}_+(X)$ the space of finite Borel nonnegative measures on $X$, 
    \item  $\mathcal{P}(X)$ the space of probability measure on $X$, 
    \item $\text{BL}(X)$ the space of bounded-Lipschitz functions  $\phi$ such that the bounded-Lipschitz norm is finite {\it i.e.}
    $$
\| \phi \|_{\text{BL}} := \max \{ \| \phi\|_{L^\infty}, [\phi]_{\rm Lip}\} < + \infty
$$
 where $ [\phi]_{\rm Lip}$ stands for the Lipschitz constant of $\phi$.  $\text{BL}(X)$ is equipped with the following distance: for $\mu_1, \mu_2 \in \text{BL}(X)$, 
 $$d_{\rm BL}(\mu_1,\mu_2)
=\sup_{\Vert \phi\Vert_{\rm BL}\leq 1}\int_{X}\phi(x)(d\mu_1(x)-d\mu_2(x)),$$
 \item $\text{BL}_1(X)$ the subspace of bounded-Lipschitz functions $\phi \in \text{BL}(X)$ such that $\|\phi\|_{\text{BL}} \leq 1$. 
\end{itemize}

\section{Preliminaries}\label{sec:preliminaries}

\subsection{Graphs and hypergraph limit theories}\label{subsec:graphons-hypergraphons}

In this section, we briefly introduce some useful notation and the necessary terminology to study hypergraphs, and their graph limit theories as proposed in previous literature \cite{ES-12,RS-24,Z-15,Z-23-arxiv}.

\subsubsection{Basics on hypergraphs}\label{subsubsubsec:basics-on-hypergraphs}

We first recall notation for (finite) hypergraphs.

\begin{defi}[Hypergraph]
A weighted hypergraph (or simply hypergraph) consists in a triple $H=(V,E,W)$, where $V$ is a finite set, $E\subset 2^{V}$ is a finite collection of subsets of $V$, and $W:2^{V}\longrightarrow \mathbb{R}_+$ satisfies $W(e)>0$ if, and only if, $e\in E$. Elements in $V$ are called nodes (or vertices), and without loss of generality we shall assume that $V=\llbracket1,N\rrbracket$ for some $N\in \mathbb{N}$, elements in $E$ are called hyperedges, and $W$ is called the weight function. We also define:
\begin{enumerate}[label=(\roman*)]
    \item (Cardinality) The cardinality of a hyperedge edge $e\in E$ is defined by $\#e$.
    \item (Rank) The rank of the hypergraph $H$ is the largest cardinality of any hyperedge.
    \item (Adjacency tensors) For each $\ell\in \llbracket1,N-1\rrbracket$, we define the $(\ell+1)$-order adjacency tensor
    $$(w_{ij_1\cdots j_\ell}^{\ell,N})_{1\leq i,j_1,\ldots,j_\ell\leq N}:\quad w_{ij_1\cdots j_\ell}^{\ell,N}:=W(i,j_1,\ldots,j_\ell),\quad i,j_1,\cdots,j_\ell\in \llbracket1,N\rrbracket.$$
\end{enumerate}
\label{defi:hypergraph}
\end{defi}

We remark that hypergraphs are generalizations of graphs. More specifically, hypergrahs whose hyperedges have cardinality $2$ amount to usual graphs.

\begin{rem}[Directed hypergraphs]
By definition the above hypergraphs have been considered undirected. Specifically, hyperedges are regarded as non-ordered subsets of nodes and therefore weights are assumed to be symmetric under permutations, that is, $w_{ij_1\cdots j_\ell}^{\ell,N}=w_{\sigma(i)\sigma(j_1)\cdots\sigma(j_\ell)}^{\ell,N}$ for all hyperedge $e=\{i,j_1,\ldots,j_\ell\}$, all $\sigma\in\mathcal{S}_{\ell+1}$ and all $\ell\in \llbracket1,N-1\rrbracket$.

However, as for graphs, directed weighted hypergraphs (or simply directed hypergraphs) could have also been considered by replacing hyperedges by directed hyperedges. A directed hyperedge (or hyperarch) is a pair $e=\{i;h\}$ consisting of a tail $i\in V$ and a head $h\in 2^V$ with $i\notin h$. Since the head is a non-ordered subset of nodes, the corresponding weights are assumed to be symmetric under permutations on the nodes of the head, but not the tail, that is, $w_{ij_1\cdots j_\ell}^{\ell,N}=w_{i\sigma(j_1)\cdots \sigma(j_\ell)}^{\ell,N}$ for all directed hyperedge $e=\{i;j_1,\ldots,j_\ell\}$, all $\sigma\in \mathcal{S}_\ell$ and all $\ell\in \llbracket1,N-1\rrbracket$.

See Figure \ref{fig:undirected-directed-hyperedge} for a graphical representation of the undirected edge $e=\{1,2,3,4\}$ and the directed edge $e=\{1;2,3,4\}$ with same nodes $1,2,3,4$.
\end{rem}

\begin{figure}[t]
\begin{center}
\includegraphics[width=0.45\textwidth]{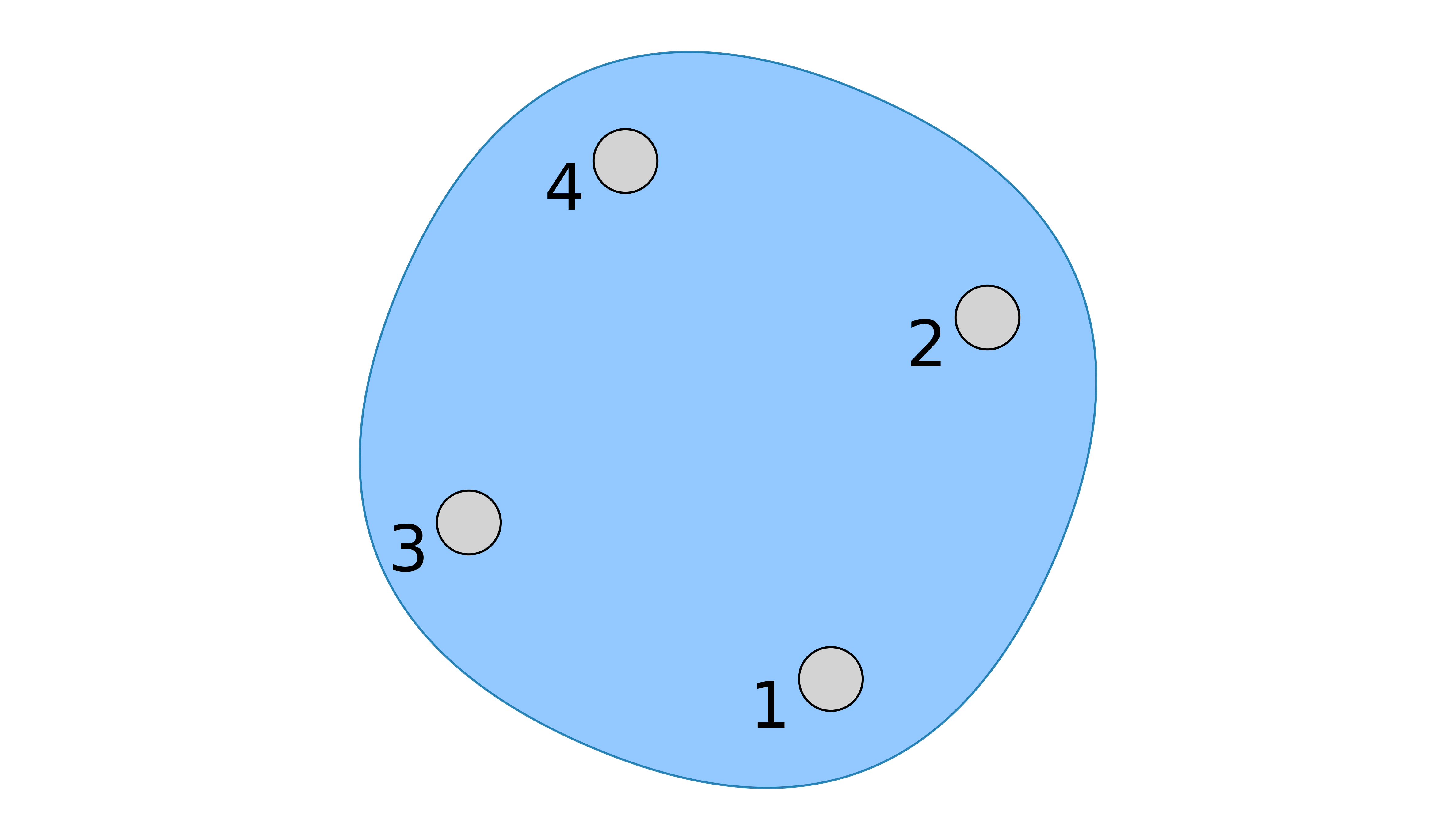}
\includegraphics[width=0.45\textwidth]{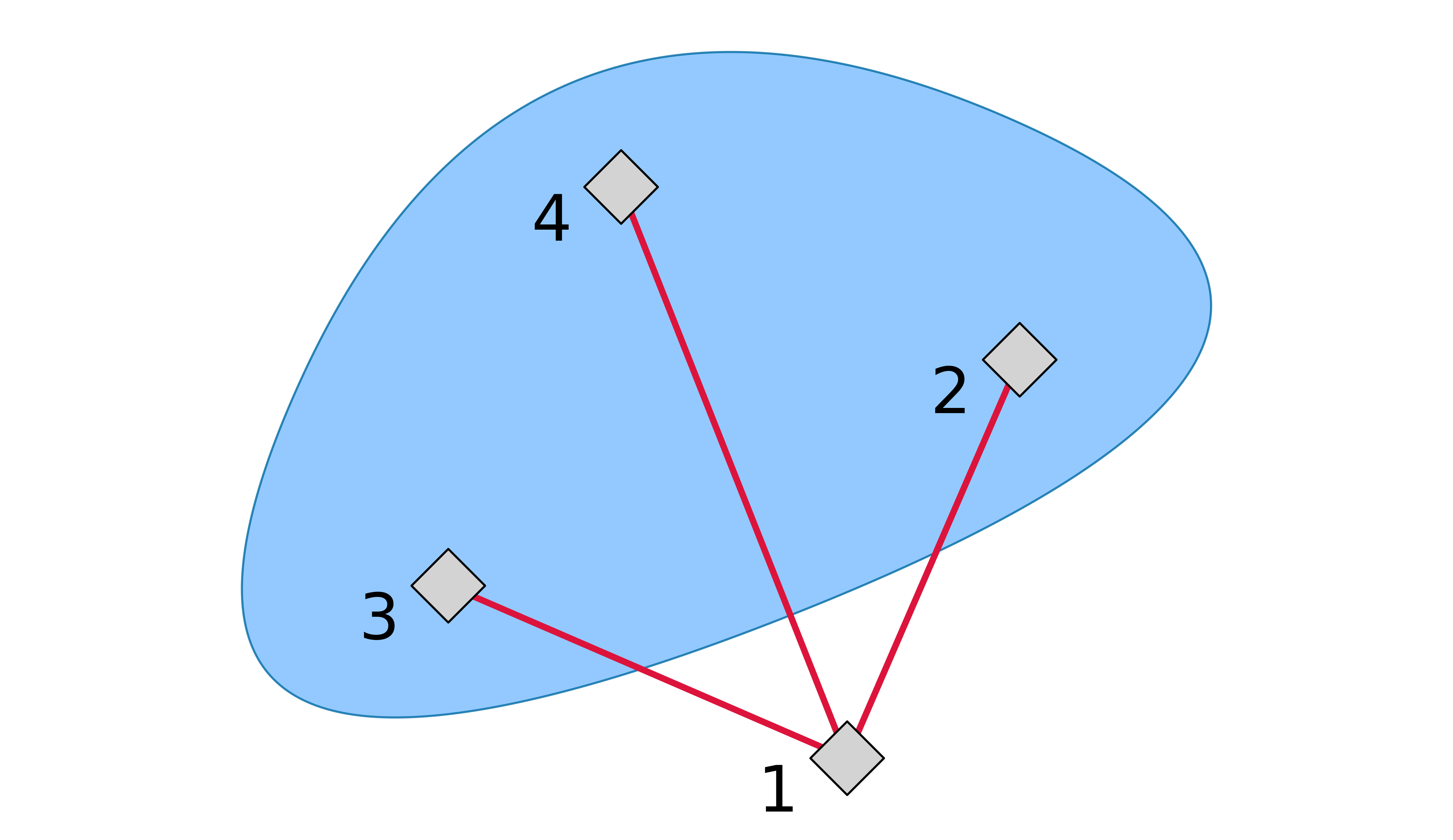}
\caption{Undirected hyperedge (left) and directed hyperedge (right).}
\label{fig:undirected-directed-hyperedge}
\end{center}
\end{figure}

So defined, directed hypergraphs are well suited to formulate dynamical multi-agent system under the form \eqref{eq:multi-agent-system-mckean-sde}, since each term $w_{ij_1\cdots j_\ell}^{\ell,N} K_\ell(X_i^N,X_{j_1}^N,\ldots, X_{j_\ell}^N)$ represents the strength of the force generated jointly by the agents $X_{j_1}^N,\ldots ,X_{j_\ell}^N$ (acting as the head) and felt by the agent $X_i^N$ (acting as the tail). However, more general hypergraphs could have been considered by admiting that tails of the directed hyperedges may contain more than one element, or further considering directed hyperedges simply as ordered lists of nodes. The more general the directed hyperedges are, the less symmetry we assume on the weights.

\medskip

\begin{defi}[Types of hypergraphs]
~
\begin{enumerate}[label=(\roman*)]
    \item (Unweighted hypergraph) An unweighted hypergraph is a weighted hypergraph with weights valued in $\{0,1\}$. In such case, given $i,j_1,\ldots,j_\ell\in \llbracket1,N\rrbracket$, the set $e=\{i,j_1,\ldots,j_\ell\}$ is a hyperedge if, and only if, $w_{ij_1\cdots j_\ell}^{\ell,N}=1$.
    \item (Uniform hypergraphs) A $k$-uniform hypergraph is a weighted hypergraph whose hyperedges all have the same cardinality $k$. In other words, $w_{ij_1\cdots j_\ell}^{\ell,N}=0$ when $\ell\neq k-1$.
    \item (Simplicial complexes) A simplicial complex is an unweighted hypergraph whose hyperedges are closed under taking subsets, that is, if $e'\subset e\subset 2^V$ and $e\in E$ then $e'\in E$. 
\end{enumerate}
\label{defi:hypergraphtypes}
\end{defi}

\subsubsection{Graph limit theories}\label{subsubsubsec:graph-limit-theories}

Before introducing the recent theories on hypergraph limits, we briefly review some of the main results in the theory of graph limits, and its use to derive mean-field limits for multi-agent systems with binary interactions like \eqref{eq:micro_nonexch}.

\begin{defi}[Graphons]\label{defi:graphons}
Given $W>0$, we define the set of graphons (graph functions) as
$$\mathcal{G}_W:=\{w\in L^\infty_+([0,1]^2):\,\Vert w\Vert_{L^\infty}\leq W,\,\mbox{and }w\mbox{ is symmetric}\}.$$
For any two graphons $w,\bar w\in \mathcal{G}_W$, we define the labeled cut distance
$$
d_\square(w,\bar w):=\sup_{S,T\subset [0,1]}\left\vert\iint_{S\times T} (w(\xi,\zeta)-\bar w(\xi,\zeta))\,d\xi\,d\zeta\right\vert,
$$
as well as the (unlabeled) cut distance
$$
\delta_\square(w,\bar w)=\inf_{\Phi} d_\square(w,\bar w^\Phi),
$$
where $\Phi$ ranges over all bijective measure-preserving maps $\Phi:[0,1]\longrightarrow [0,1]$, and $\bar w^\Phi\in \mathcal{G}_W$ represents the rearranged graphon $\bar w^\Phi(\xi,\zeta)=\bar w(\Phi(\xi),\Phi(\zeta))$.
\end{defi}

The space of graphons $(\mathcal{G}_W,\delta_\square)$ was first introduced in \cite{LS-06} as one of the first limit theories for a sequence of graphs. Specifically, the authors proved that $\delta_\square$ is a pseudodistance on the space of graphons $\mathcal{G}_W$ (as it is invariant under rearrangements of graphons) which, induced on the quotient space $\mathcal{G}_W/\sim$ that identifies graphons identical  modulo rearrangements, becomes a compact metric space. Note that for any sequence of graphs $(G_N)_{N\in \mathbb{N}}$ with an increasing number of nodes $N$ we can associate a graphon $w^{G_N}\in \mathcal{G}_1$ via the piecewise function
$$w^{G_N}(\xi,\zeta)=N\sum_{i,j=1}^N w_{ij}^N\mathds{1}_{I_i^N\times I_j^N}(\xi,\zeta),\quad \xi,\zeta\in [0,1],$$
where $(w_{ij}^N)_{1\leq i,j\leq N}$ is the adjacency matrix of $G_N$ and $I_i^N :=[\frac{i-1}{N},\frac{i}{N})$, $I_j^N :=[\frac{j-1}{N},\frac{j}{N})$ for all $1\leq i,j\leq N$. In view of the scaling $w_{ij}^N\leq \frac{W}{N}$ natural for mean-field systems (see also \eqref{eq:hypothesis-weights-uniform-bound} for hypergraphs), we have intentionally rescaled the coupling weights $w_{ij}^N$ by $N$ in the above piecewise definition. By compactness of the space $(\mathcal{G}_W,\delta_\square)$ such sequence must have a converging subsequence to a limiting graphon $w\in \mathcal{G}_W$ with respect to the cut distance $\delta_\square$. Additionally, in \cite{LS-06} it was also proved that any graphon $w\in \mathcal{G}_W$ can be approximated in the cut distance by a sequence of finite graphs as above.

We remark that the labeled cut distance is strictly weaker than the $L^1$ norm by triangle inequality, that is,
$$d_\square(w,\bar w)\leq \Vert w-\bar w\Vert_{L^1},$$
for all $w,\bar w\in \mathcal{G}_W$. Approximating any bounded Borel-measurable function by characteristic functions of Borel sets, the labeled cut distance $d_\square$ can be reformulated as follows 
$$d_\square(w,\bar w)=\sup_{\phi,\psi:[0,1]\to [0,1]}\left\vert\iint_{[0,1]^2}\phi(\xi)\,\psi(\zeta)\,(w(\xi,\zeta)-\bar w(\xi,\zeta))\,d\xi\,d\zeta\right\vert,$$
and therefore we obtain the following useful equivalence with operator norms
$$d_\square(w,\bar w)\leq \Vert T^w-T^{\bar w}\Vert_{\infty\to 1}\leq 4d_\square(w,\bar w),$$
where above $T^w:L^\infty([0,1])\longrightarrow L^1([0,1])$ stands for the adjacency operator of $w$, which consists in the bounded linear operator defined by
$$T^w[\psi](\xi):=\int_0^1 w(\xi,\zeta)\,\psi(\zeta)\,d\zeta,\quad \xi\in [0,1],$$
for each $\psi\in L^\infty([0,1])$, and similarly we define $T^{\bar w}$.

An additional important feature of the cut distance is that the associated notion of convergence is characterized by an object intimately related to graph theory, namely, the homomorphism density. Specifically, $\delta_\square(w_n,w)\to 0$ if, and only if, $\tau(F,w_n)\to \tau (F,w)$ for all finite graphs $F$. Here, $\tau(F,w)$ is the graphon extension of the homomorphism density of a finite graph $F$ into a second finite graph $G$, which counts the proportion of graph homomorphisms from $V(F)$ to $V(G)$ among all  maps from $V(F)$ to $V(G)$, that is,
$$\tau(F,w):=\int_{[0,1]^{\# V(F)}}\prod_{(i,j)\in E(F)}\,w(\xi_i,\xi_j)\prod_{i\in V(F)}\,d\xi_i.$$
We refer to the monograph \cite{L-12} for further details.

For general sequences of unweighted graphs $(G_N)_{N\in \mathbb{N}}$ the obtained limit $w\in \mathcal{G}_1$ is the trivial graphon $w\equiv 0$ unless the sequence of graphs is dense, meaning that the number of edges is quadratic on the number of verticies, {\it i.e.}, 
$$\# E(G_N)\approx\# V(G_N)^2.$$
The above suggests that the space $(\mathcal{G}_W,\delta_\square)$ is too small to capture the limit of  sequences of sparse graphs, and a more general theory is needed. This is the origin of other recent graph limit theories operating over sequences of graphs with intermediate density/sparseness:
$$\# V(G_N)\lesssim \# E(G_N)\lesssim \# V(G_N)^2,$$
which we briefly review below.

\medskip

$\diamond$ {\bf (Graphings)} In \cite{BS-01}, a first graph limit theory appeared to study limits of very sparse graph sequences $(G_N)_{N\in \mathbb{N}}$ with a number of edges growing linearly with respect to the number of vertices, that is, $\# E(G_N)\approx \# V(G_N)$. More specifically, this limit theory is known as the local or Benjamini-Schram convergence, it applies to sequences of graphs with bounded degree and limiting objects are probability measures over the set of rooted, connected (possibly infinite) graphs with bounded degree. The theory was later extended in \cite{BR-11,HLS-14}, where a more general notion of convergence of a sequence of graphs with bounded degree was proposed and was called the local-global convergence in contrast to the local or Benjamini-Schram convergence. In the local-global convergence, limiting objects are called graphings and are defined as (possibly infinite) unweighted graphs $G=(V(G),E(G))$ with bounded degree, where nodes $V(G)=\mathcal{X}$ form a Polish space, edges $E(G)\subset \mathcal{X}\times \mathcal{X}$ are Borel-measurable, and there is a probability measure $\nu\in \mathcal{P}(\mathcal{X})$ such that
$$\int_A e(\xi,B)\,d\nu(\xi)=\int_B e(\xi,A)\,d\nu(\xi),$$
for all $A,B\subset \mathcal{X}$ Borel-measurable, where $e(\xi,A)$ is the number of edges joining the node $\xi\in \mathcal{X}$ to any node in $A\subset \mathcal{X}$.

\medskip

$\diamond$ {\bf ($L^p$ graphons)} The theory of $L^p$ graphons was proposed in \cite{BCCZ-19,BCCZ-18} as an extension of the theory of $L^\infty$ graphons and proved to be useful for sequences of graphs $(G_N)_{N\in \mathbb{N}}$ with intermediate density/sparsity. More particularly, $L^p$ graphons are defined as
$$\mathcal{G}_W^p:=\{w\in L^p_+([0,1]^2):\,\Vert w\Vert_{L^p}\leq W,\,\mbox{and }w\mbox{ is symmetric}\},$$
for any $p\in (1,\infty]$. The novelty of this setting is that, as proved in \cite[Theorem 2.8]{BCCZ-19}, given a sparse (unweighted) graph sequence $(G_N)_{N\in \mathbb{N}}$ which are uniformly $L^p$ upper regular (see \cite{BCCZ-19} for a precise definition), one can define the renormalized $L^p$ graphons $w^{G_N}$ given by
$$ w^{G_N}(\xi,\zeta)=\frac{1}{D(G_N)}\sum_{i,j=1}^N w_{ij}^N\mathds{1}_{I_i^N\times I_j^N}(\xi,\zeta),\quad \xi,\zeta\in [0,1],$$
where $(w_{ij}^N)_{1\leq i,j\leq N}$ is the adjacency matrix of $G_N$, and $D(G_N)=\frac{\# E(G_N)}{\# V(G_N)^2}$ is the edge density of the graph (which tends to zero for sparse graphs), and we have that $w^{G_N}$ has a converging subsequence to a limiting $L^p$ graphon $w\in \mathcal{G}_W^p$ with respect to the cut distance $\delta_\square$. The normalization by the edge density ensures that limiting $L^p$ graphons of sparse graph sequences are not necessarily trivial, contrarily to what happened in the theory of dense graph limits.

\medskip

$\diamond$ {\bf (Graphops)} Graphops where introduced in \cite{BS-20} as an attempt to derive a unifying theory of graphons, graphings and $L^p$ graphons. In this setting, objects are regarded as linear operators
$${\tilde{\mathcal{G}}}_{p,q,W}:=\left\{T:\,L^p(\Omega)\longrightarrow L^q(\Omega):\,\begin{array}{l}T\mbox{ is linear, bounded, positivity-preserving,}\\
\mbox{self-adjoint, and verifies }\Vert T\Vert_{p\to q}\leq W,\\
\mbox{and }(\Omega,\Sigma,\mathbb{P})\mbox{ is some probability space}
\end{array}\right\},
$$
where $p\in [1,\infty)$, $q\in [1,\infty)$ and $\Vert \cdot\Vert_{p\to q}$ denotes the operator norm from $L^p$ to $L^q$. The underlying probability space is not fixed in purpose, as it varies in the limit. In \cite[Theorem 2.10]{BS-20} the authors proved that a suitable quotient of ${\tilde{\mathcal{G}}}_{p,q,W}$ endowed with the action convergence distance $d_M$ (see \cite{BS-20} for further details) is compact. 

We remark that graphons $w\in \mathcal{G}_W$ can be realized as graphops through their adjacency operator $T^w:L^p([0,1])\longrightarrow L^q([0,1])$ as above, and also graphings can be realized as graphops. Similarly, finite graphs $G_N$ can also be realized as graphops by defining the (discrete) adjacency operator $T^{G_N}:L^p(\llbracket 1,N\rrbracket)\longrightarrow L^q(\llbracket 1,N\rrbracket)$. Therefore, suitably renormalizing $T^{G_N}$ so that their operator norms $\Vert \cdot\Vert_{p\to q}$ stay uniformly bounded, we have the existence of a limiting graphop. In \cite{BS-20} it was also observed that action convergence is equivalent to convergence in cut distance when restricted to graphons, and local-global convergence when restricted to graphings.

Graphops are intimately related to s-graphons, which were poposed in \cite{KLS-19} as an alternative theory of sparse graphs where the limiting objects consist in symmetric probability measures $\nu\in \mathcal{P}([0,1]\times [0,1])$. More precisely, \cite[Theorem 6.3]{BS-20} shows that any graphop $T\in \tilde{\mathcal{G}}_{p,q,W}$ on the probability space $(\Omega,\Sigma,\mathbb{P})$ admits a unique representation as a symmetric finite measure $\nu_T\in \mathcal{M}_+(\Omega\times \Omega)$ with marginals absolutely continuous with respect to $\mathbb{P}$, namely,
$$\nu_T(A\times B)=\int_{\Omega} T[\mathds{1}_A](\xi)\,\mathds{1}_B(\xi)\,d\mathbb{P}(\xi),\quad A,B\in \Sigma.$$

\subsubsection{Mean-field limits on graphs}\label{subsubsubsec:mean-field-limits-graphs}

The above graph limit theories have proven to be extremely useful in the last years to derive rigorous mean-field limits of multi-agent systems with binary interactions such as \eqref{eq:micro_nonexch} toward a suitable Vlasov-type equation. The fundamental obstruction to apply standard methods of mean-field limits \cite{Braun77,Do-79,Golse16} and propagation of chaos \cite{S-91} is the fact that \eqref{eq:micro_nonexch} is a non-exchangeable system. Specifically, exchanging two agents modifies the overall dynamics of the groups as interactions are mediated by weights, which introduce a distinguishable feature on agents. Depending on the degree of density or sparseness of the sequence of the maps of connections, graphons, graphops and other extensions have been used.

The first attempts to derive the mean-field limit of interacting particle systems in the case of dense graph sequences were developed in \cite{CM-19} (for Lipschitz graphons) and \cite{KM-18} (for non-Lipschitz graphons), extending Neuzert's method for exchangeable multi-agent systems \cite{NeunzertWick74}. The method of proof relies on the stability of the limiting Vlasov equation with respect to the initial datum and the involved graphon in the following metric:
$$d_1(\mu,\bar\mu)=\int_0^1d_{\rm BL}(\mu^\xi,\bar\mu^\xi)\,d\xi,$$
(see Definition \ref{defi:fibered-probability-measures-Lpdbl-distance} for an $L^p$ extension). This estimate involves a continuous dependence on the graphon in the $L^1$-norm, which is a stronger requirement than the cut distance. Since the compactness of $(\mathcal{G}_W,\delta_\square)$ could not be used, the authors prescribed an {\it ad hoc} discretization of the graphon $w$ by finite graphs $G_N$ to guarantee the stronger convergence $\Vert w^{G_N}-w\Vert_{L^1}\to 0$. In particular, general dense graph sequences are not supported by this argument.

This strategy was later extended in \cite{GK-22} to a sparse setting, using graphops. In this case, the stability estimate of the limiting Vlasov equation was quantified in the analogous metric, but required stronger convergence of graphops than action convergence, namely, the narrow convergence $\nu_{T_n}^\xi\to \nu_T^\xi$ of the disintegrations (see Theorem \ref{theo:disintegration}) for $\mathbb{P}$-a.e. $\xi\in \Omega$. Therefore, an {\it ad hoc} discretization is needed to guarantee the stronger convergence, which the authors achieved when the probability space $(\Omega,\Sigma,\mathbb{P})$ is a compact Abelian group with its Haar measure.

A further application of this strategy in a sparse setting was recently proposed in \cite{KX-22} to deal with a new class of graph limits called digraph measures, consisting in bounded maps $\xi\in \Omega\longmapsto \nu^\xi\in \mathcal{M}_+(\Omega)$. Whilst intimately related to graphops \cite{BS-20} and s-graphons \cite{KLS-19}, digraph measures do not arise from graph limit theory, and in particular they do not form a compact space. The authors obtained a similar estimate with respect to the distance
$$d_\infty(\mu,\bar \mu)=\sup_{\xi\in \Omega} d_{\rm BL}(\mu^\xi,\bar\mu^\xi),$$
which required the convergence of digraph measures $\sup_{\xi\in \Omega} d_{\rm BL}(\nu_n^\xi,\nu^\xi)\to 0$. In  \cite{KX-22} the authors obtained suitable discretizations of digraph measures by finite graphs under the additional assumption that $\Omega$ is a compact subset of an Euclidean space, and the digraph measure $\xi\in\Omega\longmapsto \nu^\xi$ is continuous in the bounded-Lipschitz distance.

There are only a handful of results where the graph limit theories in Section \ref{subsubsubsec:graph-limit-theories} have been used in full, that is, the natural compact topologies on the graph limit spaces have proved enough in the mean-field limit and no preparation on the discretized weights is needed. To the best of our knowledge some examples are \cite{BCN-24,JPS-25,JZ-23-arxiv}.

In \cite{BCN-24}, the authors derived the mean-field limit for a multi-agent system in a 1-dimensional periodic domain over a sequence of dense random graphs. In particular, a stability estimate of the limiting Vlasov equation was derived with respect to the pseudodistance
$$\bar d_1(\mu,\bar\mu)=d_{\rm BL}\left(\int_0^1 \mu^\xi\,d\xi,\int_0^1 \bar\mu^\xi\,d\xi\right),$$
which operates over marginals, and therefore the structure variable $\xi\in [0,1]$ is disregarded and averaged. Interestingly, the stability estimate depended continuously on the cut distance between the involved graphons, and therefore the compactess of $(\mathcal{G}_W,\delta_\square)$ was enough, and weights did not need to be well prepared.

In \cite{JPS-25,JZ-23-arxiv}, the authors defined the space of extended graphons as a new space of graph limits consisting of measure-valued bounded maps
$$\mathcal{EG}_W:=\left\{\nu\in L^\infty_\xi([0,1],\mathcal{M}_\zeta([0,1]))\cap L^\infty_\zeta ([0,1],\mathcal{M}_\xi([0,1])):\,\Vert \nu\Vert_{L^\infty_\xi\mathcal{M}_\zeta},\,\Vert \nu\Vert_{L^\infty_\zeta\mathcal{M}_\xi}\leq W\right\}.$$
It was proved that $\mathcal{G}_W\subset \mathcal{EG}_W\subset \tilde{\mathcal{G}}_{p,p,W}$ for any $p\in [1,\infty]$, and the space was endowed with a notion of convergence characterized by the convergence  $\tau(T,\nu_n)\to \tau(T,\nu)$ of the homomorphism density for any tree $T$. The resulting space was proved compact, which gave access to a notion of convergence of sparse graphs under the simple mild assumptions that
$$\sup_{N\in \mathbb{N}}\max_{1\leq i\leq N}\sum_{j=1}^N w_{ij}^N+\max_{1\leq j\leq N}\sum_{i=1}^N w_{ij}^N\leq W.$$
Additionally, the authors obtained some stability estimates for the limiting Vlasov equation depending continuously on the extended graphons with respect to such a topology, thus deriving the mean-field limit for a large class of sparse graphs, and not necessarily well prepared.

We refer to \cite{BCG-24,BCJ-24,CDG-20,DGL-16,JZ-23-arxiv,LP-24,L-20,PT-22-arxiv} for further recent results on continuum and mean-field limits for non-exchangeable multi-agent systems on graphs, and also to \cite{AP-24} for a comprehensive review of the recent literature.

\subsubsection{Hypergraph limit theories}\label{subsubsubsec:hypergraph-limit-theories} The literature on hypergraph limits is scarce and more recent. To the best of our knowledge, the results available thus far are mostly concentrated in the papers \cite{ES-12,RS-24,Z-15,Z-23-arxiv}, where three types of objects have been studied, namely, uniform hypergraphons, hypergraphops and complexons. We briefly review the main outcomes in this subsection.

\medskip

$\diamond$ {\bf (Uniform hypergraphons)} In \cite{ES-12} (later reformulated also in \cite{Z-15}), a limiting theory of dense uniform hypergraphs (see Figure \ref{fig:uniform-hypergraphs}) was proposed, leading the so-called uniform hypergraphons. Specifically, the sequence of dense hypergraphs was assumed to have a fixed and common cardinality $k\in \mathbb{N}$ on all hyperedges, and the limiting objects were defined by

$$k\mbox{-}\mathcal{H}_W:=\left\{w\in L^\infty_+([0,1]^{2^k-2}):\,\Vert w\Vert_{L^\infty}\leq W,\mbox{ and }w\mbox{ is symmetric}\right\}.$$
Elements in $L^\infty_+([0,1]^{2^k-2})$ have $2^k-2$ variables indexed by the non-empty and proper subsets of $\llbracket1,k\rrbracket$, and symmetric means that whenever a permutation $\sigma\in \mathcal{S}_k$ is set, and whenever we rearrange the variables $\xi_1,\xi_2,\ldots,\xi_k,\xi_{\{1,2\}},\,\xi_{\{1,3\}}$, etc by $\xi_{\sigma(1)},\xi_{\sigma(2)},\ldots,\xi_{\sigma(k)},\xi_{\{\sigma(1),\sigma(2)\}},\,\xi_{\{\sigma(1),\sigma(3)\}}$, etc, the value of $w$ does not change. The need for additional coordinates is related to the need for suitable regularity partitions for hypergraphs \cite{G-07}.

Similarly to graphons ({\it i.e.}, $k=2$), this space was endowed with the topology induced by the convergence of the homomorphism densities $\tau(F,w_n)\to \tau(F,w)$, that is,
$$\tau(F,w)=\int \prod_{e\in E(F)}w(\xi_e)\,d\xi,$$
for any $k$-uniform hypergraph $F$. Above we denote $\xi=(\xi_s:\,\emptyset\neq s\subset V(F),\,\# s\leq k-1)$ the collection of variables indexed by all nonempty subsets of $V(F)$ with size smaller that $k$, and $\xi_e=(\xi_s:\,\emptyset\neq s\subsetneq e)$ the subcollection of variables indexed by all nonempty and proper subsets of $e$. As for the binary case, the resulting space is compact, see \cite[Corollary 1.8]{Z-15}.

\begin{figure}[t]
\begin{subfigure}{\textwidth}
\begin{center}
\includegraphics[width=0.32\textwidth]{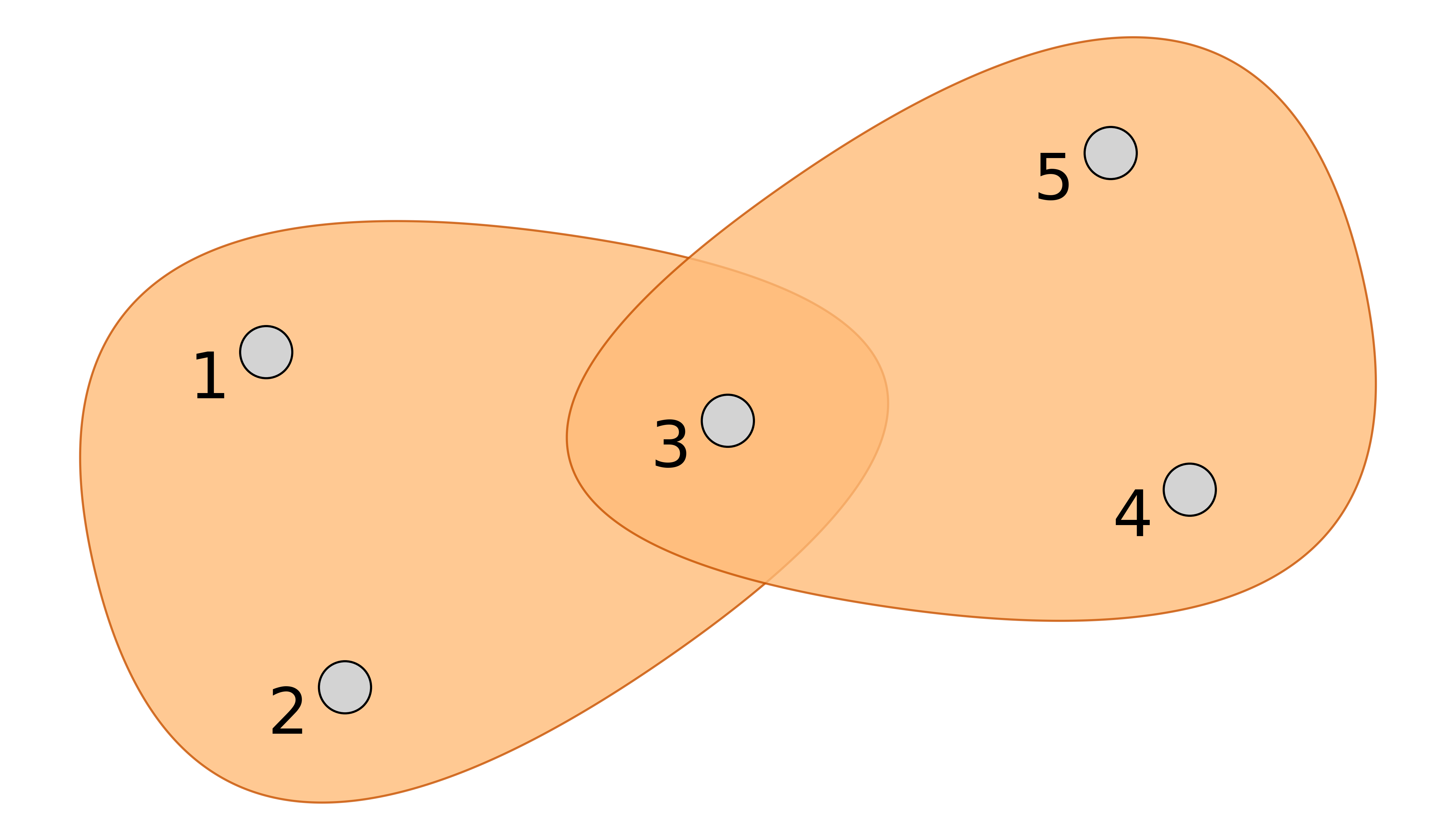}
\includegraphics[width=0.32\textwidth]{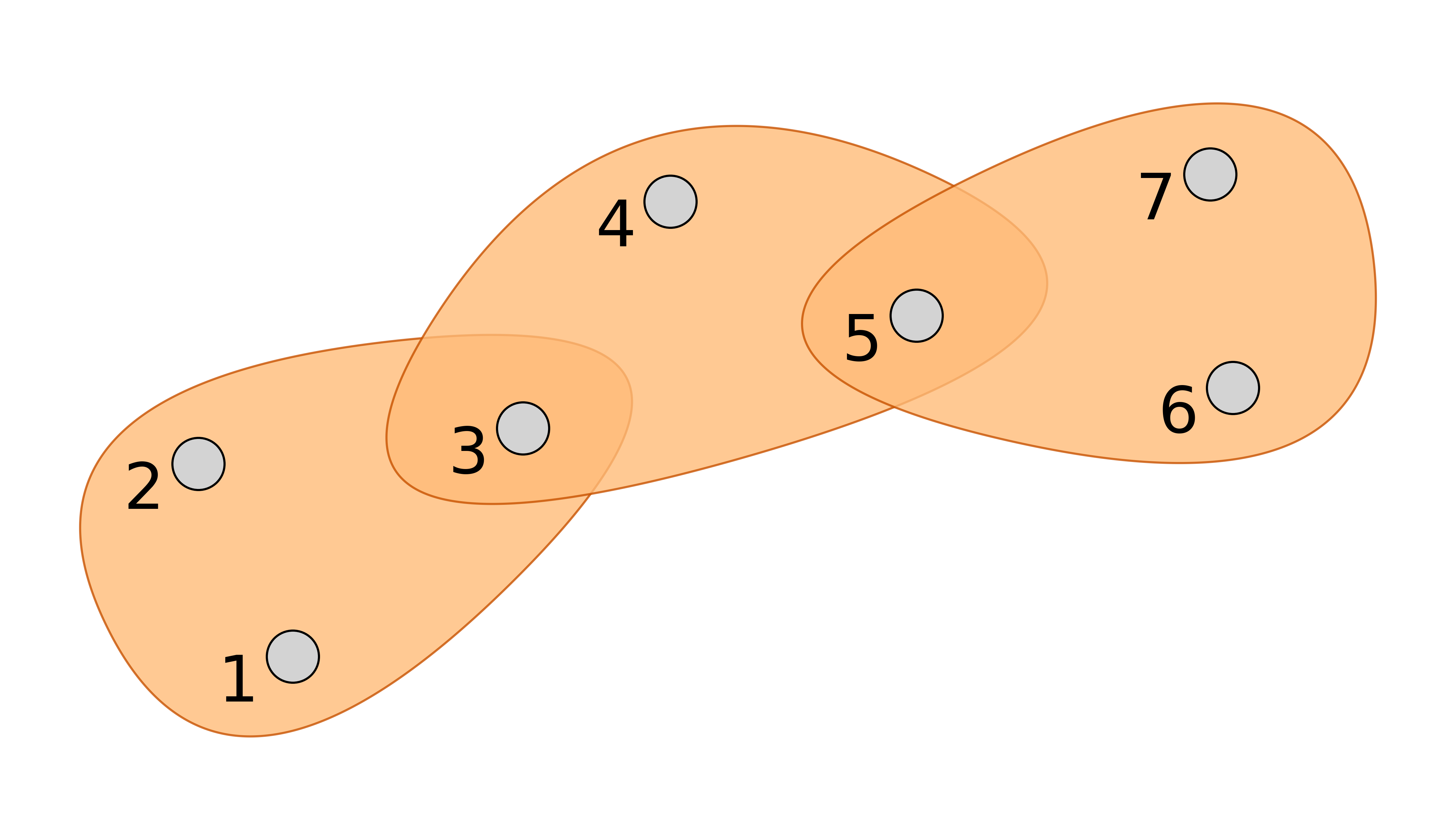}
\includegraphics[width=0.32\textwidth]{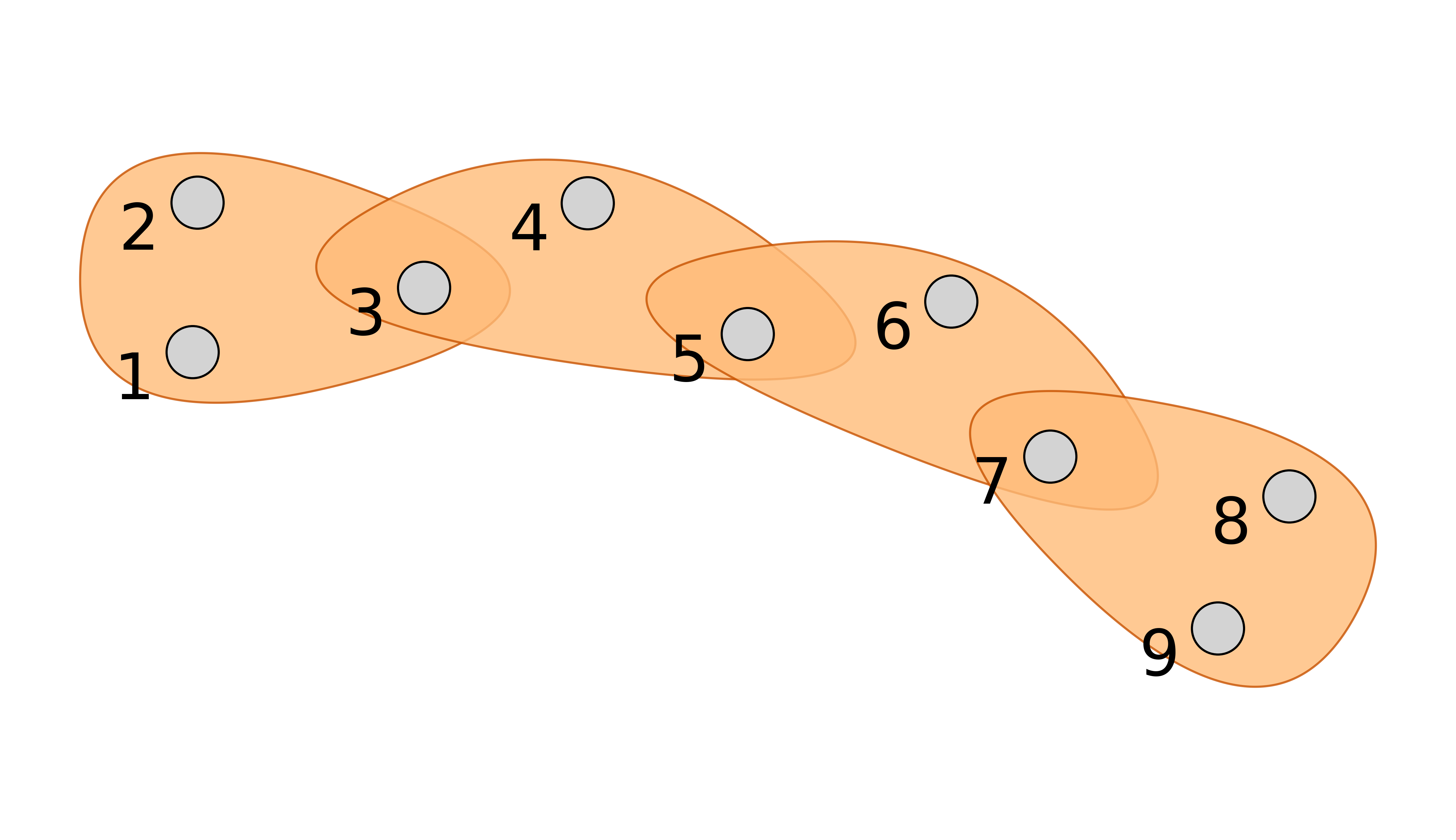}
\end{center}
\caption{Sequence of $3$-uniform hypergraphs}
\label{fig:uniform-hypergraphs}
\end{subfigure}
\begin{subfigure}{\textwidth}
\begin{center}
\includegraphics[width=0.32\textwidth]{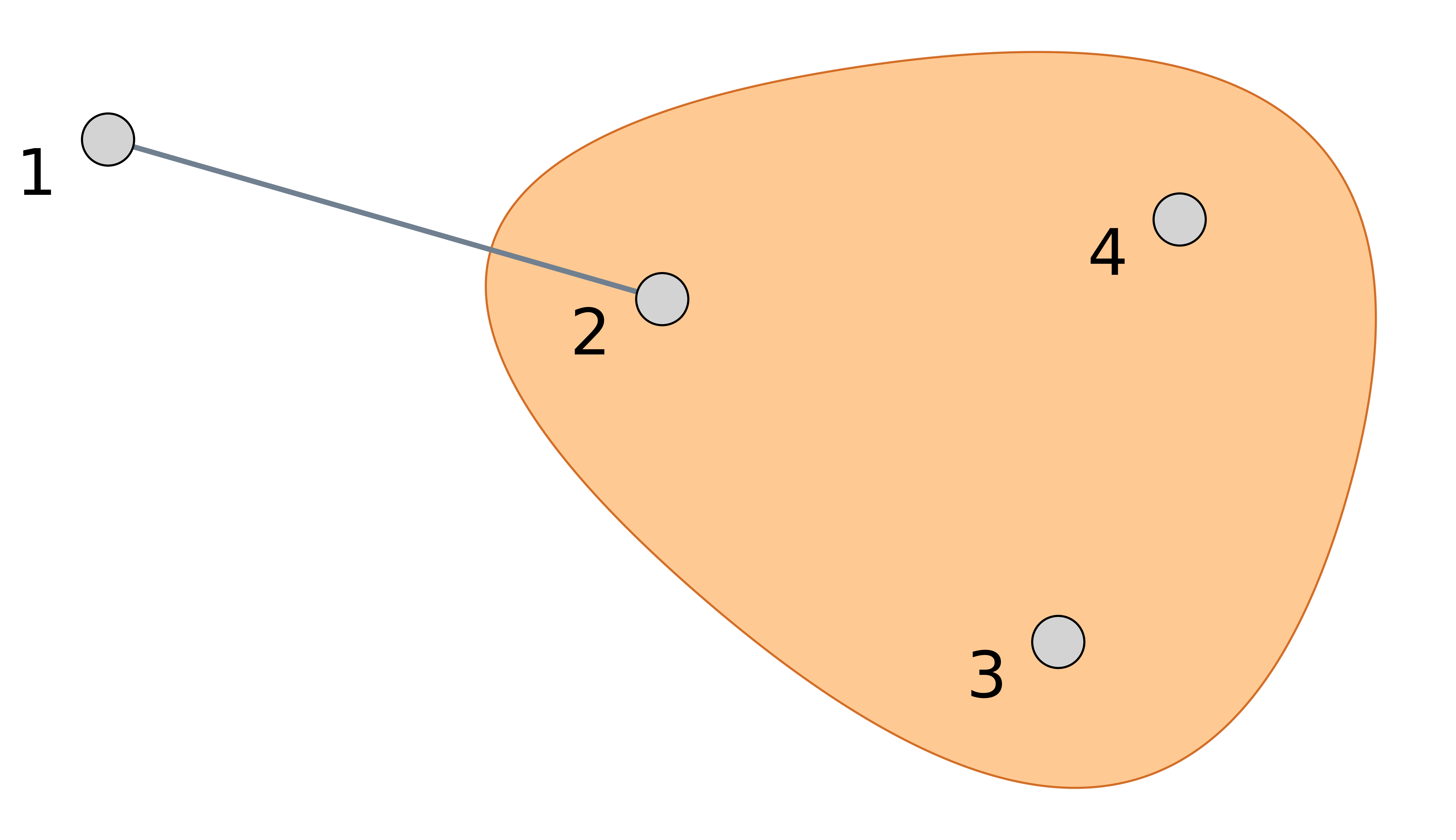}
\includegraphics[width=0.32\textwidth]{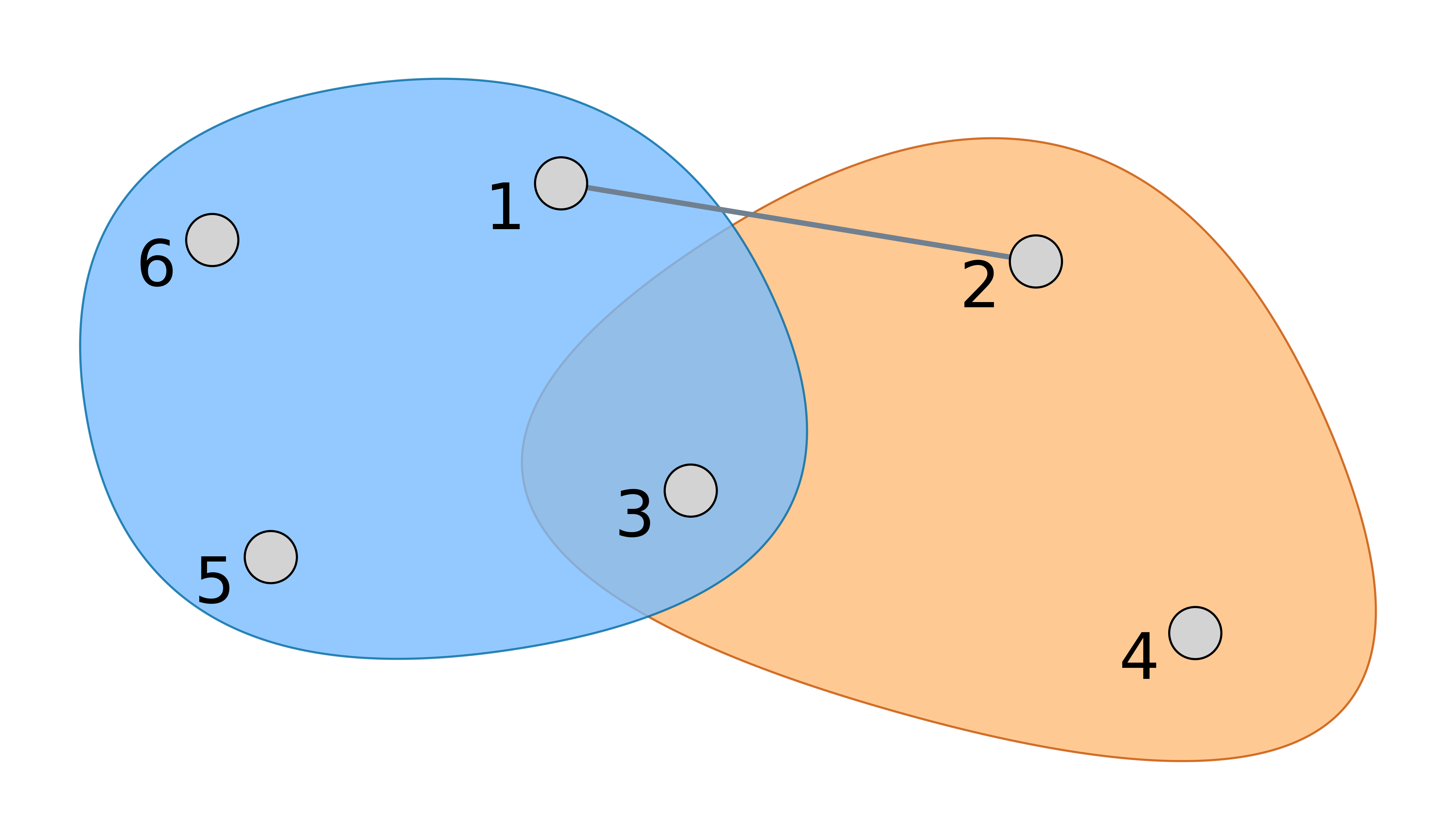}
\includegraphics[width=0.32\textwidth]{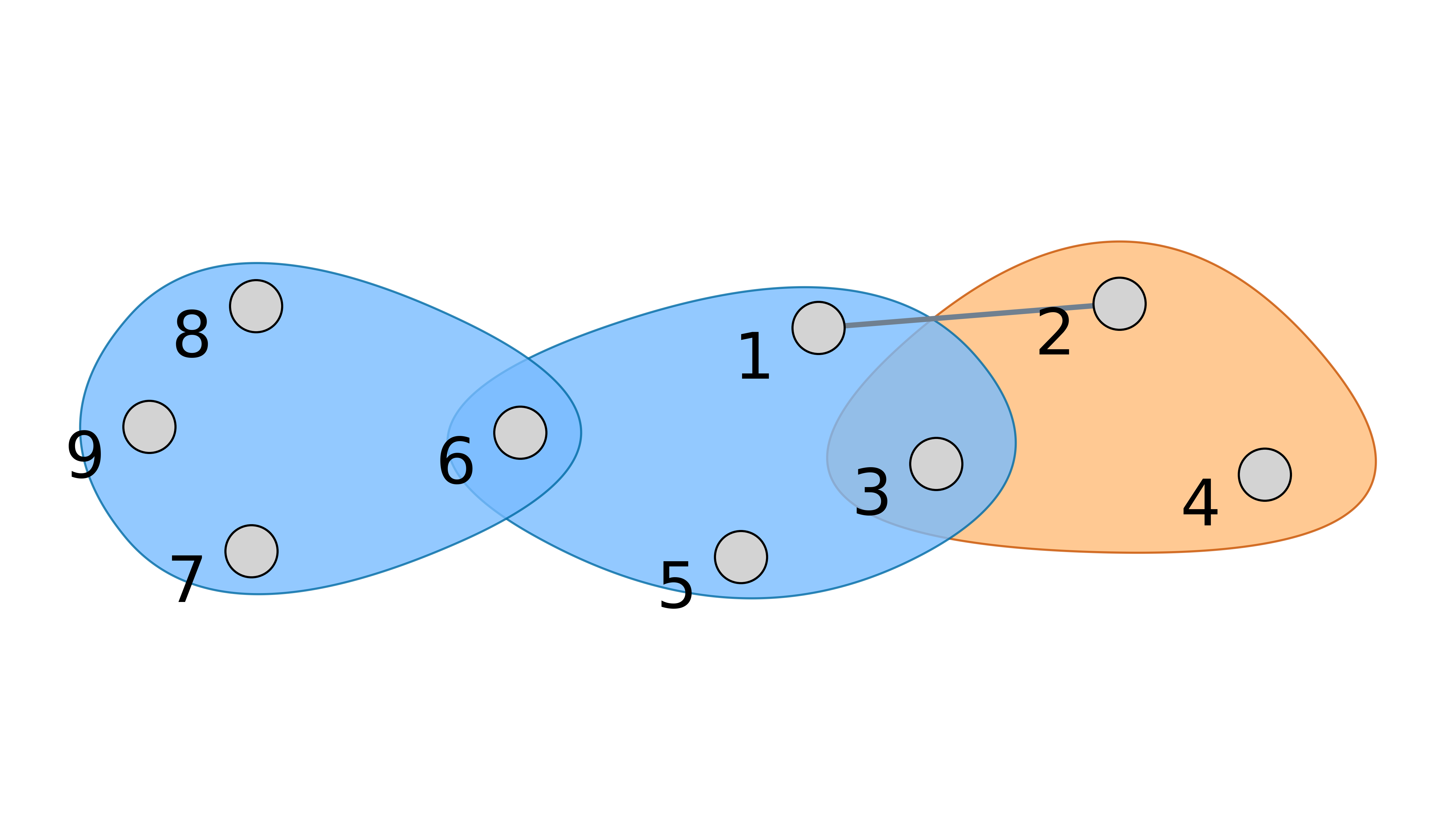}
\end{center}
\caption{Sequence of non-uniform hypergraphs}
\label{fig:nonuniform-hypergraphs}
\end{subfigure}
\caption{Sequences of uniform (top) and non-uniform (bottom) hypergraphs}
\label{fig:sequence-hypergraphs}
\end{figure}

\medskip

$\diamond$ {\bf (Hypergraphops)} As an extension of graphops \cite{BS-20}, hypergraphops were proposed in \cite{Z-23-arxiv} as a limit theory for non-uniform (see Figure \ref{fig:nonuniform-hypergraphs}) sparse hypergraphs with bounded rank $r\in \mathbb{N}$:
$$r\mbox{-}{\tilde{\mathcal{H}}}_{p,q,W}:=\left\{T:\,L^p(\Omega)^{r-1}\longrightarrow L^q(\Omega):\,\begin{array}{l}T\mbox{ is multilinear, bounded, positivity-preserving,}\\
\mbox{symmetric, and verifies }\Vert T\Vert_{(L^p)^{r-1}\to L^q}\leq W,\\
\mbox{and }(\Omega,\Sigma,\mathbb{P})\mbox{ is some probability space}
\end{array}\right\},
$$
where $p\in [1,\infty)$, $q\in [1,\infty)$ and $\Vert \cdot\Vert_{(L^p)^{r-1}\to L^q}$ denotes the multilinear operator norm from $(L^p)^{r-1}$ to $L^q$. Again, a suitable quotient of $k\mbox{-}{\tilde{\mathcal{H}}}_{p,q,W}$ endowed with the action convergence distance $d_M$ (see \cite{Z-23-arxiv} for further details) is compact. As for graphs, finite hypergraphs $(H_N)_{N\in \mathbb{N}}$ with bounded rank $r$ can be realized as hypergraphops. In \cite{Z-23-arxiv} several options were proposed through the definition of the $s$-action for any $s=1,\ldots,r-1$. Each way of embedding hypergraphs into graphops seems to provide a notion of convergence for hypergraphs that is better adapted to dense, sparse, uniform and non-uniform settings.

\medskip

$\diamond$ {\bf (Complexons)} In \cite{RS-24}, a limit theory of dense simplicial complexes of unbounded rank was derived. A simplicial complex is a hypergraph so that any subset of a hyperedge is again a hyperedge. By definition, simplicial complexes are non-uniform and therefore their limit cannot be tackled using $k$-uniform hypergraphons. Since the rank is also unbounded, hypergraphops are not an option either. The alternative approach in \cite{RS-24} was to define the set of complexons as the set of equivalence classes
$$\mathcal{C}_1:=\left\{w^\circ=(w_\ell^\circ)_{\ell\in \mathbb{N}}:\,\begin{array}{l} w=(w_\ell)_{\ell\in \mathbb{N}},\,w_\ell\in L^\infty_+([0,1]^{\ell+1}),\,\Vert w_\ell\Vert_{L^\infty}\leq 1,\\
\mbox{and }w_\ell\mbox{ is symmetric for all }\ell\in \mathbb{N}\end{array}\right\}$$
where $w^\circ$ is called the faceted version of $w$ and is defined by
$$w_\ell^\circ(\xi,\xi_1,\ldots,\xi_\ell):=\prod_{k=1}^{\ell}\prod_{\{i_0,\ldots,i_k\}\subset \llbracket0,k\rrbracket}w_\ell(\xi_{i_0},\ldots,\xi_{i_k}),\quad \xi,\xi_1,\ldots,\xi_\ell\in [0,1],$$
where we denote $\xi_0:=\xi$. The property $w_\ell^\circ\geq w_{\ell+1}^\circ$ for all $\ell\in \mathbb{N}$ is reminiscent of the closure of simplicial complexes under restrictions. $\mathcal{C}_1$ was endowed with the topology induced by the convergence of the homomorphism densities $\tau(F^\circ,w_n^\circ)\to \tau(F^\circ,w^\circ)$, that is,
$$\tau(F^\circ,w^\circ)=\int_{[0,1]^{\# V(F)}}\prod_{k\in \mathbb{N}} \prod_{\{i_0,\ldots,i_k\}\in E(F^0)} w^\circ(\xi_{i_0},\ldots,\xi_{i_k})\prod_{i\in V(F^\circ)}d\xi_0\ldots\,d\xi_{\# V(F)-1},$$
where for any simplicial complex $F$, the notation $F^\circ$ stands for the hypergraph with same set of nodes $V(F^\circ)=V(F)$ but only the maximal hyperedges $E(F^\circ)\subset E(F)$ which are not strict subsets of larger hyperedges. So defined, $\mathcal{C}_1$ is compact as proven in \cite[Theorem 13]{RS-24}.

The proof is based on a characterization of the topology induced by the homomorphism density, as the topology induced by a suitable higher-order extension of the cut distance of graphons. This cut distance is actually well defined for general non-faceted $w$ and, on this more general space, it induces a compact topology \cite[Lemma 33]{RS-24}. Since we do not need to characterize the topology by homomorphism densities, instead of working on the quotient space consisting of complexons, we shall rather work on the primitive space, which we define below.

\subsection{Hypergraphons of unbounded rank}\label{subsec:UR-hypergraphons}

In this section we introduce the limit theory for hypergraphs of unbounded rank which we shall use in this paper, and which is inspired by the treatment in \cite{RS-24}. In particular, we introduce their underlying topology, a basic compactness result reminiscent of the classical result for dense graph limits in \cite{LS-06}, and some example of convergence of sequences of hypergraphs toward its associated limiting object.

\begin{defi}[Hypergraphons of unbounded rank]\label{defi:UR-hypergraphons}
Given $W>0$, we define the set of hypergraphons of unbounded rank (or UR-hypergraphons) as follows
$$\mathcal{H}_W:=\left\{w=(w_\ell)_{\ell\in \mathbb{N}}:\begin{array}{l} w_\ell\in L^\infty_+([0,1]^{\ell+1}),\,\Vert w_\ell\Vert_{L^\infty}\leq W,\\
\mbox{and }w_\ell\mbox{ is symmetric for all }\ell\in \mathbb{N}\end{array}\right\}.$$
For any two UR-hypergraphons $w,\bar w\in \mathcal{H}_W$, we define the $\ell$-th order labeled cut distance
$$d_{\square,\ell}(w_\ell,\bar w_\ell):=\sup_{S,S_1,\ldots,S_\ell\subset [0,1]}\left\vert \int_{S\times S_1\times \cdots\times S_\ell} (w_\ell-\bar w_\ell)\,d\xi\,d\xi_1\ldots\,d\xi_\ell\right\vert,$$
for every $\ell\in \mathbb{N}$. To comprehend all possible orders, we define the labeled cut distance
$$d_\square(w,\bar w;(\alpha_\ell)_{\ell\in \mathbb{N}}):=\sum_{\ell=1}^\infty \alpha_\ell\, d_{\square,\ell}(w_\ell,\bar w_\ell),$$
where $(\alpha_\ell)_{\ell\in \mathbb{N}}$ is a strictly positive summable sequence, as well as the (unlabeled) cut distance
$$
\delta_\square(w,\bar w;(\alpha_\ell)_{\ell\in \mathbb{N}})=\inf_{\Phi} d_\square(w,\bar w^\Phi;(\alpha_\ell)_{\ell\in \mathbb{N}}),
$$
where $\Phi$ ranges over all bijective measure-preserving maps $\Phi:[0,1]\longrightarrow [0,1]$, and $\bar w^\Phi\in \mathcal{H}_W$ represents the rearranged UR-hypergraphon $\bar w^\Phi_\ell(\xi,\xi_1,\ldots,\xi_\ell)=\bar w_\ell(\Phi(\xi),\Phi(\xi_1),\ldots,\Phi(\xi_\ell)).$
\end{defi}

As for the binary case, the labeled cut distance on UR-hypergraphs admits an alternative representation as an operator norm. Specifically, the $\ell$-th order labeled cut distance can be reformulated in terms of test functions, namely,
$$d_{\square,\ell}(w_\ell,\bar w_\ell)=\sup_{\phi,\psi_1,\ldots,\psi_\ell:[0,1]\to [0,1]}\left\vert\iint_{[0,1]^{\ell+1}}\phi(\xi)\,\psi_1(\xi_1)\cdots \psi_\ell(\xi_\ell)\,(w_\ell-\bar w_\ell)\,d\xi\,d\xi_1\ldots\,d\xi_\ell\right\vert,$$
and therefore we obtain the following useful equivalence with multilinear operator norms
$$d_{\square,\ell}(w_\ell,\bar w_\ell)\leq \Vert T^{w_\ell}-T^{\bar w_\ell}\Vert_{(L^\infty)^\ell\to L^1}\leq 2^\ell d_{\square,\ell}(w_\ell,\bar w_\ell),$$
where above the multilinear operator $T^{w_\ell}:L^\infty([0,1])^\ell\longrightarrow L^1([0,1])$ stands for the $\ell$-th order adjacency operator of $w_\ell$, which consists in the bounded multilinear operator defined by
\begin{equation}\label{eq:adjacency-multilinear-operator}
T^{w_\ell}[\psi_1,\ldots,\psi_\ell](\xi):=\int_{[0,1]^\ell} w_\ell(\xi,\xi_1,\ldots,\xi_\ell)\,\psi_1(\xi_1)\cdots\,\psi_\ell(\xi_\ell)\,d\xi_1\ldots\,d\xi_\ell,\quad \xi\in [0,1],
\end{equation}
for each $\psi_1,\ldots,\psi_\ell\in L^\infty([0,1])$. The operator $T^{\bar w_\ell}$ is defined similarly. All the above implies the following equivalence between the labeled cut-distance and the multilinear operator norms.

\begin{pro}\label{pro:cut-distance-operator-norm}
For any strictly positive and summable sequence $(\alpha_\ell)_{\ell \in\mathbb{N}}$ we have
$$d_\square (w,\bar w;(\alpha_\ell)_{\ell\in \mathbb{N}})\leq \sum_{\ell=1}^\infty \alpha_\ell\, \Vert T^{w_\ell}-T^{\bar w_\ell}\Vert_{(L^\infty)^\ell\to L^1}\leq d_{\square} (w,\bar w;(2^\ell \alpha_\ell)_{\ell\in \mathbb{N}}),$$
for all $w,\bar w\in \mathcal{H}_W$, where the $\ell$-adjacency operators $T^{w_\ell}$ and $T^{\bar w_\ell}$ are defined by \eqref{eq:adjacency-multilinear-operator}.
\end{pro}

As mentioned above, the interest in the space of UR-hypergraphons endowed with the cut distance is the following compactness result from \cite[Lemma 33]{RS-24}.

\begin{pro}[Compactness of UR-hypergraphons]\label{pro:compactness-UR-hypergraphons}
The unlabeled cut distance $\delta_\square$ induces topologically equivalent pseudometrics on the space of UR-hypergraphons $\mathcal{H}_W$ which do not depend on the summable sequence $(\alpha_\ell)_{\ell\in \mathbb{N}}$. When induced on the quotient $\mathcal{H}_W/\sim$ that identifies UR-hypergraphons identical modulo rearrangement, the space becomes a compact metric space.
\end{pro}

Similarly to the case of sequences of dense graphs, one can characterize the convergence of a sequence of hypergraphs via the unlabeled cut-distance.
For any sequence of hypergraphs $(H_N)_{N\in \mathbb{N}}$ with an increasing number of nodes $N$, we can associate a UR-hypergraphon $w=(w_\ell^{H_N})_{\ell\in\N }$ via the family of piecewise-constant functions
\begin{equation}
w_\ell^{H_N}(\xi,\xi_1,\ldots,\xi_\ell)=N^\ell \sum_{i,j_1,\ldots, j_\ell=1}^N w_{ij_1\cdots j_\ell}^{\ell,N}
\mathds{1}_{I_i^N\times I_{j_1}^N\times\cdots \times I_{j_\ell}^N}(\xi,\xi_1,\cdots,\xi_\ell),
\label{eq:hypergraph_to_hypergraphon}
\end{equation}
for all  $(\xi,\xi_1,\cdots,\xi_\ell)\in [0,1]^{\ell+1}$, where $(w_{ij_1\cdots j_\ell}^{\ell,N})_{1\leq i,j_1,\ldots , j_\ell\leq N}$ is the adjacency tensor of order $\ell+1$ of $H_N$, and $I_i^N :=[\frac{i-1}{N},\frac{i}{N})$ for all $1\leq i\leq N$. In view of the scaling condition \eqref{eq:hypothesis-weights-uniform-bound} which we assume on our hypergraphs to derive the mean-field limit, we have intentionally rescaled the coupling weights $w_{ij_1,\ldots,j_\ell}^{\ell,N}$ by $N^\ell$ in the above piecewise definition. Note that it ensures that $w^{H_N}\in \mathcal{H}_W$ for the same $W$ given in Assumption \eqref{eq:hypothesis-weights-uniform-bound}. This suggests the following notion of convergence of hypergraphs (possibly with unbounded rank) toward a UR-hypergraphon.

\begin{defi}[Convergence of hypergraphs]
A sequence of hypergraphs $(H_N)_{N\in\N}$ is said to converge to a UR-hypergraphon $w\in \mathcal{H}_W$ when $\lim_{N\to\infty}\delta_\square(w, w^{H_N};(\alpha_\ell)_{\ell\in \mathbb{N}})=0$ for some (and then all) positive and summable sequence $(\alpha_\ell)_{\ell\in\N}$.
\end{defi}

The above Proposition \ref{pro:compactness-UR-hypergraphons} ensures that any such sequence of hypergraphs $(H_N)_{N\in \mathbb{N}}$ satisfying the scaling condition  \eqref{eq:hypothesis-weights-uniform-bound}, must converge toward some limiting UR-hypergraphon up to a subsequence. In fact, as it happens in the binary case, the family of finite hypergraphs is dense in UR-hypergraphons. In Propositions \ref{prop:conv_hypergraph_to_graphon1} and \ref{prop:conv_hypergraph_to_graphon2} below we provide two different constructive methods to show that UR-hypergraphons can be suitably approximated in the cut distance by sequences of finite hypergraphs: $L^1$-approximations, and pointwise approximations.

\begin{pro}[$L^1$-approximations of UR-hypergraphons]\label{prop:conv_hypergraph_to_graphon1}
Let $w = (w_\ell)_{\ell\in\N}\in \mathcal{H}_W$ be a UR-hypergraphon. 
 Let $(H_N)_{N\in\N}$ be the sequence of hypergraphs whose adjacency tensors $(w^{\ell,N})_{\ell\in \mathbb{N}}$ are defined by $L^1$-approximations of the UR-hypergraphon $w$, as follows:
\[
w^{\ell,N}_{ij_1\cdots j_\ell} = \left\{\begin{array}{ll}
 \displaystyle N \int_{I^N_i\times I^N_{j_1}\times \cdots\times I^N_{j_\ell}}w_\ell(\xi,\xi_1,\cdots,\xi_\ell)\,d\xi \,d\xi_1 \cdots d\xi_\ell, , & \mbox{if }\ell\in \llbracket1,N-1\rrbracket,\\
 0, & \mbox{otherwise}\end{array}\right.
\]
for all $ (i,j_1\cdots, j_\ell) \in \llbracket 1, N\rrbracket^{\ell+1}$.
Then, the sequence of hypergraphs $(H_N)_{N\in \mathbb{N}}$ converges as $N$ tends to infinity to the UR-hypergraphon $w$ in the labeled cut-distance $d_\square$ (and hence also in the unlabeled cut-distance $\delta_\square$).
\end{pro}

\corr{
\begin{proof}
Let $w = (w_\ell)_{\ell\in\N}\in \mathcal{H}_W$ and let $(w_\ell^{H_N})_{\ell\in\N}\in \mathcal{H}_W$ denote the piecewise-constant hypergraphon obtained from the hypergraph $H_N$, as defined in \eqref{eq:hypergraph_to_hypergraphon}. We aim to prove that for any positive summable sequence $(\alpha_\ell)_{\ell\in\N}$, $\lim_{N\to\infty} d_\square(w, w^{H_N};(\alpha_\ell)_{\ell\in \mathbb{N}})=0$.
We actually show the convergence of the sequence of hypergraphons in the stronger $L^1$-norm, recalling that for any $\ell\in\N$, the labeled and unlabeled cut-distances of order $\ell$ satisfy: 
\[
\delta_{\square,\ell}(w_\ell,w^{H_N}_\ell) \leq d_{\square,\ell}(w_\ell,w^{H_N}_\ell) \leq \|w_\ell-w^{H_N}_\ell\|_{L^1(I^{\ell+1})}.
\]
Without loss of generality, by density of Lipschitz functions in $L^1([0,1]^{\ell+1})$, consider $w_\ell\in {\rm Lip}([0,1]^{\ell+1})$.
We can then check that
\begin{equation*}
\|w_\ell-w^{H_N}_\ell\|_{L^1(I^{\ell+1})} \leq \frac{\sqrt{\ell+1}}{N} \left[w_\ell\right]_\mathrm{Lip}
\end{equation*}
for each $\ell\in\llbracket 1,N-1\rrbracket$, where $\left[w_\ell\right]_\mathrm{Lip}$ denotes the Lipschitz semi-norm of $w_\ell$. 
Now, note that for any summable sequence $(\alpha_\ell)_{\ell\in \mathbb{N}}$, it holds
$\alpha_\ell d_{\square,\ell}(w_\ell,w_\ell^{H_N})\leq 2W\alpha_\ell,$
for all $\ell\in \mathbb{N}$ and all $N\in\mathbb{N}$. 
Then, by the dominated convergence theorem, \[\lim_{N\to \infty}d_\square(w,w^{H_N};(\alpha_\ell)_{\ell\in\N})=0,\]
which concludes the proof.
\end{proof}
}

\begin{pro}[Pointwise approximation of UR-hypergraphons]\label{prop:conv_hypergraph_to_graphon2}
Let $w = (w_\ell)_{\ell\in\N}\in \mathcal{H}_W$ be a UR-hypergraphon, and suppose that $w_\ell$ is continuous for all $\ell\in\N$. 
For all $N\in\N$, let $(\bar\xi^N_i)_{i\in\llbracket 1,N\rrbracket}$ be a sequence of points satisfying $\bar\xi^N_i\in I_i^N$ for all $i\in\llbracket 1,N\rrbracket$, where $I_i^N = \left[ \frac{i-1}{N},\frac{i}{N}\right)$. Let $(\bar H_N)_{N\in\N}$ be the sequence of hypergraphs whose adjacency tensors $(\bar w^{\ell,N})_{\ell\in\mathbb{N}}$ are defined by evaluating the UR-hypergraphon $w$ on the grid points $(\bar\xi_i^N)^{\ell+1}_{i\in\llbracket 1,N\rrbracket}$, as follows:
\[
\bar w^{\ell,N}_{ij_1\cdots j_\ell} = \left\{\begin{array}{ll}
 \displaystyle\frac{1}{N^\ell} w_\ell(\bar\xi^N_{i},\bar\xi^N_{j_1},\cdots,\bar\xi^N_{j_\ell}), & \mbox{if }\ell\in \llbracket 1,N-1\rrbracket,\\
 0, & \mbox{otherwise}
 \end{array}\right.
 \]
for all $ (i,j_1\cdots, j_\ell) \in \llbracket 1, N\rrbracket^{\ell+1}$.
Then, the sequence of hypergraphs $(\bar H_N)_{N\in \mathbb{N}}$ converges as $N$ tends to infinity to the UR-hypergraphon $w$ in the labeled cut-distance $d_\square$ (and hence also in the unlabeled cut-distance $\delta_\square$).
\end{pro}

\begin{proof}
As in the proof of Proposition \ref{prop:conv_hypergraph_to_graphon1}, we will prove the stronger convergence result in $L^1$-norm.
Since $w_\ell$ is continuous and $I^{\ell+1}$ is compact, then $w_\ell$ is uniformly continuous for all $\ell\in \mathbb{N}$. Denoting by $\varpi_\ell:\mathbb{R}_+\longrightarrow \mathbb{R}_+$ the modulus of continuity of $w_\ell$, we can show
\corr{\begin{equation*}
\|w_\ell- w^{\bar H_N}_\ell\|_{L^1(I^{\ell+1})}
\leq \varpi_\ell\left(\frac{\sqrt{\ell+1}}{N}\right),
\end{equation*}
The proof is concluded in the same way as in the proof of Proposition \ref{prop:conv_hypergraph_to_graphon1}.
}
\end{proof}

In this paper we study the mean-field limit of the multi-agent system \eqref{eq:multi-agent-system} over dense non-uniform hypergraphs of unbounded rank (note that $r=N$). We rely on the above hypergraph limit theory of UR-hypergraphons, which we show is compatible with the mean-field limit without any  {\it ad hoc} preparations of the weights. Of course our approach also works when weights are prepared according to Propositions \ref{prop:conv_hypergraph_to_graphon1} or \ref{prop:conv_hypergraph_to_graphon2}.

To the best of our knowledge, none of the above hypergraph limit theories has been exploited previously in the community of mean-field limits. There is only one result available in the literature where a mean-field limit of a multi-agent system with higher-order interactions has been studied \cite{KX-25}. Inspired by the previous work \cite{KX-22} by the same authors on digraph measures, a new class of limits of hypergraphs with bounded rank $r\in \mathbb{N}$ was proposed in terms of directed hypergraph measures, which consist of bounded maps $\xi\in \Omega\longmapsto \nu_\ell^\xi\in \mathcal{M}_+(\Omega^\ell)$ for $\ell\in 1,\ldots,r-1$. Directed hypergraph measures are related to hypergraphops \cite{Z-23-arxiv}, but do not arise from hypergraph limit theory. As in the binary case, the authors' strategy was to design a method to well-prepare the hypergraphs so as to approximate any given directed hypergraph measure in a strong enough topology, compatible with their stability estimate.

\subsection{Some examples of multi-agent systems over hypergraphs}\label{subsec:examples}

To illustrate the definitions of hypergraph and hypergraphon ({\it cf.} Definitions \ref{defi:hypergraph}, \ref{defi:hypergraphtypes} and \ref{defi:UR-hypergraphons}), we provide a few examples inspired from the existing literature.

\subsubsection{Examples of hypergraphs}\label{subsubsec:examples-hypergraphs}

\,~

\medskip $\diamond$ {\bf (Uniform hypergraphs)}: The simplest example of hypergraphs is the uniform hypergraph, in which all hyperedges have the same cardinality.  
When the cardinality is $r = 2$, all hyperedges are in fact binary edges, so that the hypergraph is a graph. 

\medskip $\diamond$ {\bf (Simplicial complex from real-world networks)}: In \cite{Skardal20}, the authors propose a model for oscillator dynamics on two different simplicial complexes, which are designed from real datasets. The Macaque brain dataset consists of 242 interconnected regions of the brain, while the UK power grid network consists of 120 nodes and 165 transmission lines.
From the simple graphs $(w^{2,N}_{k_1k_2})_{1\leq k_1,k_2\leq N}$ given by these datasets, hyperedges are constructed using the following simple criterion. For any $\ell\in\{1,\cdots,N-1\}$, 
$$w^{\ell,N}_{i j_1\cdots j_\ell} = \begin{cases}1 \quad \text{ if }\quad w^{2,N}_{k_1k_2} = 1\text{ for all }k_1,k_2\in\{i,j_1,\ldots,j_\ell\},\\
0 \quad \text{ otherwise.}\end{cases}$$
that is a hyperedge exists in the hypergraph when all possible pairs of nodes in the hyperedge are edges of the graph. The resulting hypergraph is of rank $r=N$ if and only if the network is fully coupled. Otherwise, the hypergraph is of bounded rank $r<N$.

\medskip $\diamond$ {\bf (Hypergraph for homogeneous groups)}: For $\theta\in (0,1]$, the $\theta-$fixed radius near neighbor graph is defined by drawing an edge between each pair of agents whose labels $(i,j)$ satisfy $|i-j|\leq \theta N$. It is an example of an undirected and unweighted graph.
We can extend this notion to hypergraphs by drawing a hyperedge between any group of agents whose diameter is bounded by $\theta N$, defining
\begin{equation}\label{eq:hypergraph_homogeneous}
w^{\ell,N}_{ij_1\cdots j_\ell} = \begin{cases}
\displaystyle  \frac{1}{N^\ell}, & \text{if } \quad \max_{k_1,k_2\in \{i,j_1,\ldots, j_\ell\}} |k_1-k_2| \leq \theta N,\\
0, & \text{otherwise},
\end{cases}
\end{equation}
for each $\ell\in\{1,\cdots,N-1\}$. As mentioned above, the scaling by $N^\ell$ in our weighted hypergraphs is natural in view of the scaling condition \eqref{eq:hypothesis-weights-uniform-bound} which ensures that the associated sequence of UR-hypergraphons $(w^{H_N})_{N\in \mathbb{N}}$ defined by \eqref{eq:hypergraph_to_hypergraphon} all lie in $\mathcal{H}_W$ (indeed with $W=1$). A pixel representation of this hypergraph for the hyperedges of dimension $\ell=1$ and $\ell=2$ is given in Figure \ref{fig:Hypergraph-nearestneighbor}. If the labels represent agents' identities, this hypergraph models the idea that agents interact if and only if they have similar identities, {\it i.e.} if they form a homogeneous group. 
Notice that if $\theta=1$, the hypergraph is fully connected, which represents all-to-all coupling.

Note that although we use the terminology ``fixed radius near-neighbor'' and ``group diameter'', the distance considered to build the hypergraph are in the space of labels. Confusion should not be made with so-called ``bounded confidence'' models for \textit{exchangeable} particle systems \cite{H02}, in which the interaction kernel has bounded support, which implies that agents interact if and only if their distance \textit{in the state space} is small enough.

\begin{figure}[h!]
\centering
\includegraphics[width = 0.3\textwidth, trim = 0cm -3cm 0cm 0cm, clip=true]{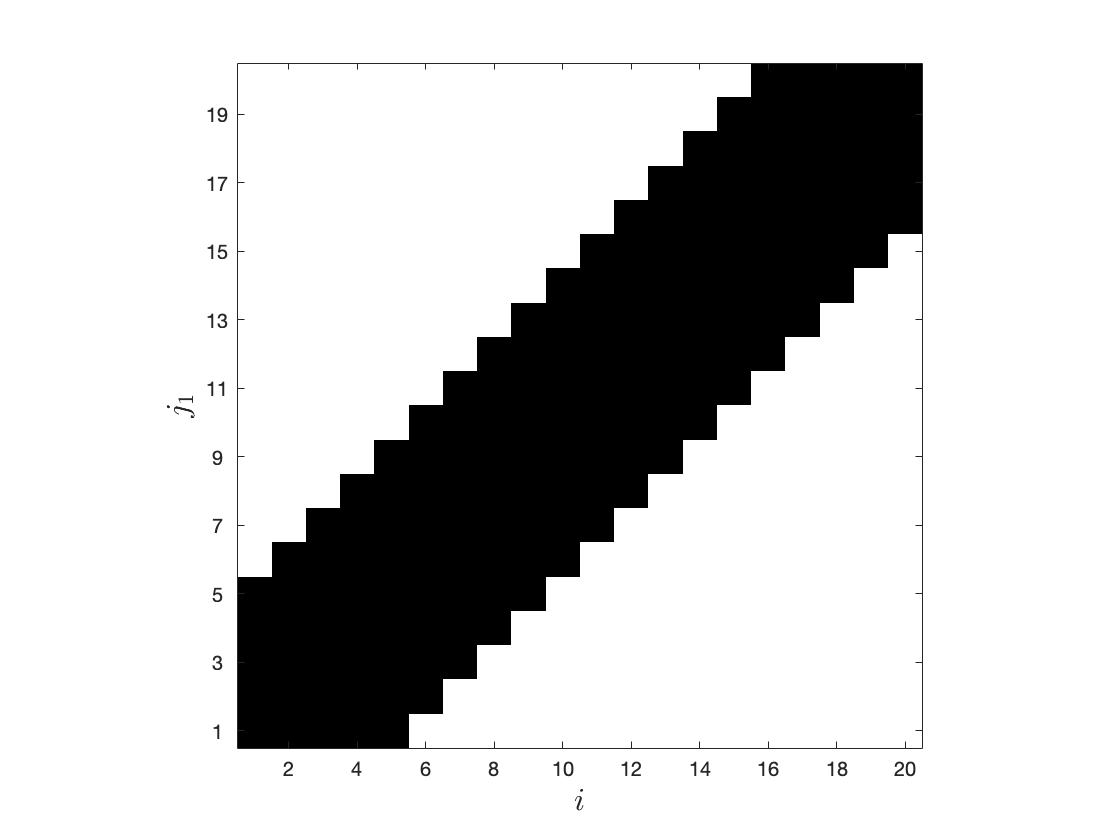}
\includegraphics[width = 0.4\textwidth]{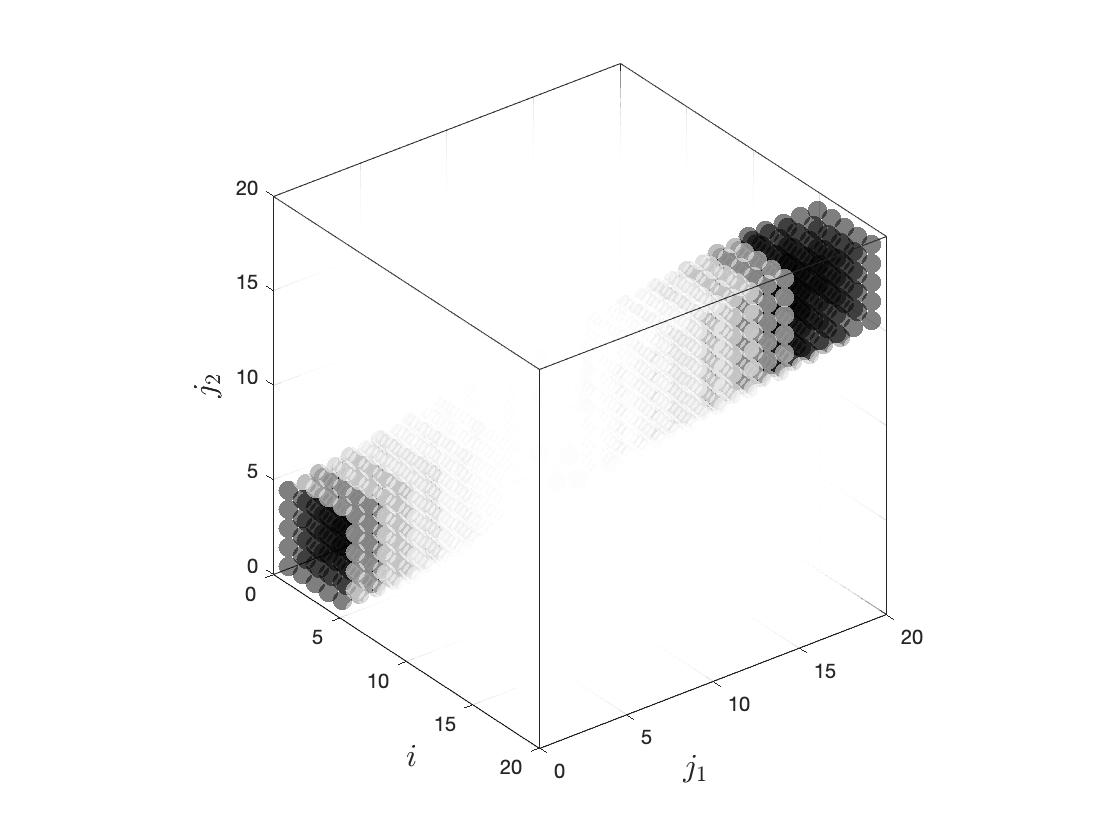}
\caption{Pixel representation of the matrix and 3-tensor respectively corresponding to $w^{1,N}$ and $w^{2,N}$ given by \eqref{eq:hypergraph_homogeneous}, with $\theta=0.3$ and $N=20$. As this is an unweighted hypergraph, the presence of a hyperedge is represented in black. Transparency was used to indicate depth in the 3-dimensional representation.}
\label{fig:Hypergraph-nearestneighbor}
\end{figure}

\begin{rem}\label{rem:conv-homogeneous}
The sequence of hypergraphs for homogeneous groups defined in Equation \eqref{eq:hypergraph_homogeneous} can be shown to be obtained by evaluating the UR-hypergraphon $w=(w_\ell)_{\ell\in \mathbb{N}}$, defined by
\begin{equation}\label{eq:hypergraphon_homogeneous}
w_{\ell}(\xi_0,\cdots,\xi_\ell) = \begin{cases}
\displaystyle  1 \quad \text{ if } \quad \max_{i,j\in \{0,\cdots,\ell\}} |\xi_{i}-\xi_{j}| \leq \theta,\\
0 \quad \text{otherwise}
\end{cases}
\end{equation}
on the grid $(\bar\xi_i)^{\ell+1}_{i\in\llbracket 1,N\rrbracket}$, where $\bar\xi_i := \frac{i-1}{N}$ for all $i\in\llbracket 1,N\rrbracket$. 
Note that Proposition \ref{prop:conv_hypergraph_to_graphon2} cannot be  used to prove the convergence of the hypergraphs \eqref{eq:hypergraph_homogeneous} toward the UR-hypergraphon \eqref{eq:hypergraphon_homogeneous} since, so defined, $w_\ell$ is discontinuous for all $\ell\in\N$. Whilst the continuity assumption of Proposition \ref{prop:conv_hypergraph_to_graphon2} could seem restrictive, it is not always necessary and can be relaxed in many cases. Indeed, convergence can still be proven in this case, as shown in Proposition \ref{prop:convergence-discontinuous} (see Appendix \ref{app:proof_convergence_hypergraphs_homogeneous_groups}). 
\end{rem}

\medskip $\diamond$ {\bf (Weighted hypergraph for balanced groups)}: In this example, the hypergraph is also of unbounded rank ($r = N$), and each hyperedge's weight is assumed to be a decreasing function of the distance between the average value of its nodes' labels and the average value of \textit{all} nodes' labels ({\it i.e.} $\frac{N+1}{2}$). More precisely, given a decreasing function $f: [0,\frac{1}{2}]\rightarrow\R_+$, we define
\begin{equation}\label{eq:hypergraph_balanced}
w^{\ell,N}_{j_0 j_1\cdots j_\ell} = \frac{1}{N^\ell} f\left(\frac{1}{N}\left|\frac{1}{\ell+1}\sum_{k=0}^\ell j_k-\frac{N+1}{2}\right|\right),
\end{equation}
for all $\ell \in \N$. A pixel representation of such a hypergraph's hyperedges of dimension $\ell=1$ and $\ell=2$ is given in Figure \ref{fig:Hypergraph-balanced}. This hypergraph models the idea that the more balanced a group's identities are ({\it i.e.} the average label of the group is close to $\frac{N+1}{2}$), the larger the group interaction. 
\begin{figure}[h!]
\centering
\includegraphics[width = 0.3\textwidth, trim = 0cm -3cm 0cm 0cm, clip=true]{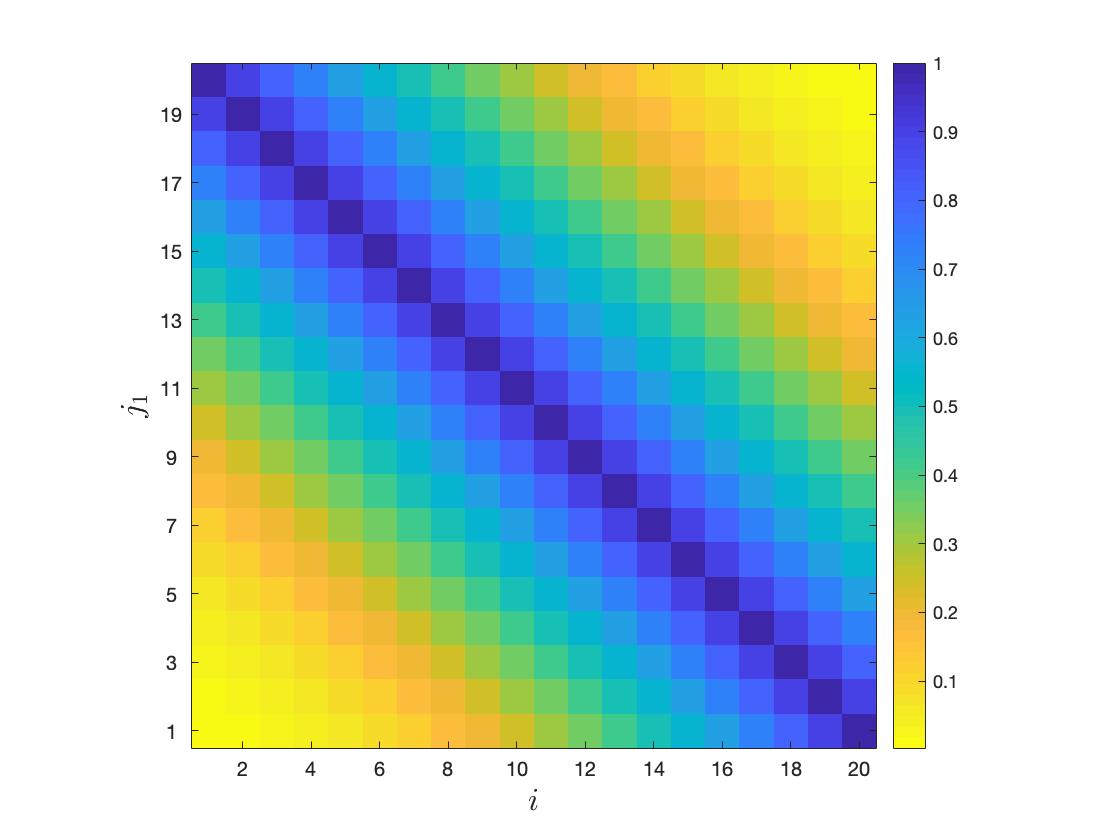}
\includegraphics[width = 0.4\textwidth]{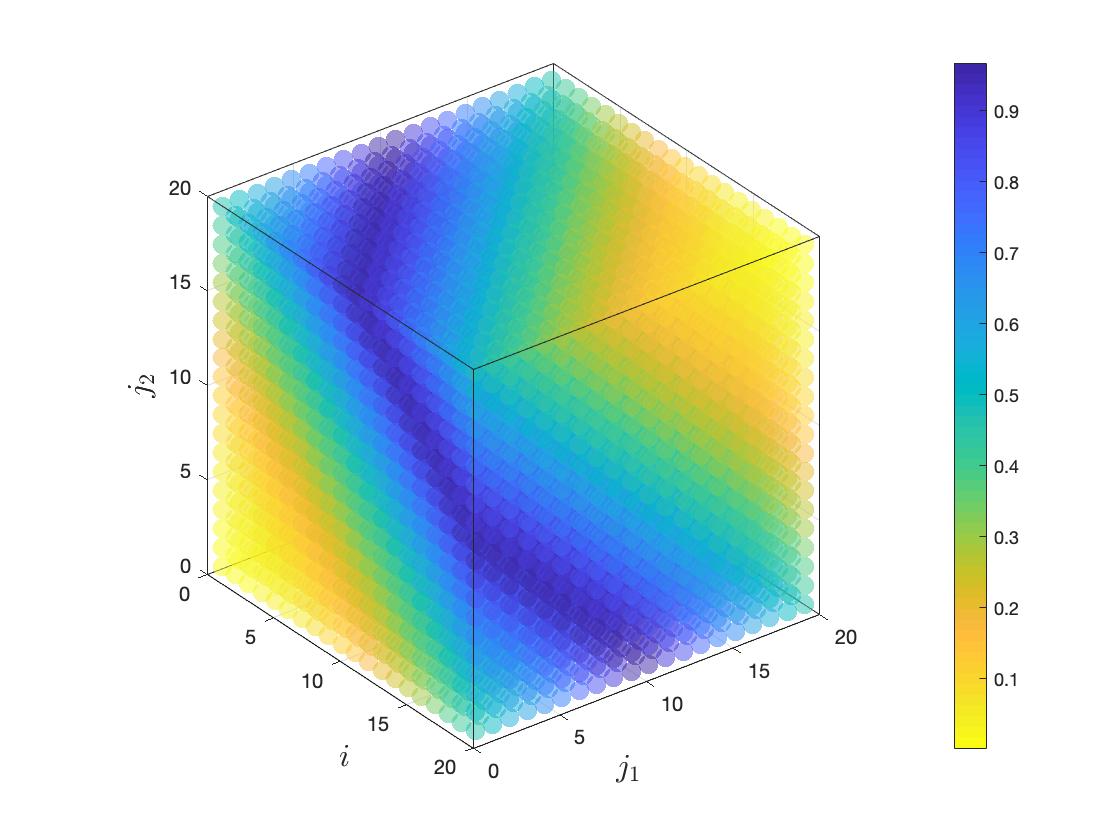}
\caption{Pixel representation of the matrix and 3-tensor respectively corresponding to $w^{1,N}$ and $w^{2,N}$ given by \eqref{eq:hypergraph_balanced}, with $f:x\mapsto4(x-\frac{1}{2})^2$ and $N=20$. For this weighted hypergraph, the value of each hyperedge's weight is represented on a color scale.}
\label{fig:Hypergraph-balanced}
\end{figure}

\begin{rem}\label{rem:conv-balanced}
The sequence of hypergraphs for balanced groups defined in Equation \eqref{eq:hypergraph_balanced} can be shown to be obtained by evaluating the UR-hypergraphon $w=(w_\ell)_{\ell\in \mathbb{N}}$, defined by 
\begin{equation}\label{eq:hypergraphon_balanced}
w_\ell(\xi_0,\cdots,\xi_\ell) = f\left(\left|\frac{1}{\ell+1}\sum_{i=0}^\ell \xi_i-\frac{1}{2}\right|\right),
\end{equation}
on the grid  $(\bar\xi_i)^{\ell+1}_{i\in\llbracket 1,N\rrbracket}$, where $\bar\xi_i := \frac{i-\frac12}{N}$ for all $i\in\llbracket 1,N\rrbracket$.
Then, whenever the function $f$ is continuous, convergence of convergence of the hypergraphs \eqref{eq:hypergraph_balanced} toward the UR-hypergraphon \eqref{eq:hypergraphon_balanced} is a direct application of Proposition~\ref{prop:conv_hypergraph_to_graphon2}.
\end{rem}

\subsubsection{Models of multi-agent dynamics on hypergraphs}\label{subsubsec:multi-agenty-systems-on-hypergraphs}

Various examples of higher-order multi-agent dynamics have recently been proposed in the literature. Applications range from opinion dynamics to contagion propagation, synchronization of oscillators, animal communication, and evolutionary game dynamics. We present some examples, and refer to \cite{Battiston-20,Battiston22,CGL-24,Iacopini24,LCB-20,Sahasrabuddhe21,MTB-20,Neuhauser22,Skardal20,WXZ-21,XS-21,XWS-20,ZLB-23} for further reading.

\medskip 
$\diamond$ {\bf (Higher-order synchronization models)}: Several works studied generalizations of the Kuramoto model of coupled oscillators with higher-order interactions \cite{Battiston-20,LCB-20,XS-21,XWS-20,ZLB-23}.
Denoting by $\theta_i^N$ the phase and $\Omega_i^N$ the natural frequency of the $i$-th oscillator, its evolution is prescribed by an equation of the form: 
\begin{equation*}
\frac{d\theta^N_i(t)}{dt} = \Omega_i^N+\sum_{\ell=1}^{N-1} \sum_{j_1,\cdots,j_\ell=1}^N w^{\ell,N}_{ij_1\cdots j_\ell} \sin\left(\sum_{k=1}^\ell \theta_{j_k}^N(t) -\ell\,\theta_i^N(t)\right).
\end{equation*}
We remark that the interaction kernels $K_\ell(x,x_1,\ldots,x_\ell)=\sin(x_1+\cdots+x_\ell-\ell\,x)$ satisfy the symmetry assumption \eqref{eq:hypothesis-kernels-symmetry}. Notice that the second-order term (for $\ell=1$) is exactly that encountered in the classical Kuramoto model \cite{K-75}:
\[
\sum_{j=1}^N w^{1,N}_{ij} \sin\left(\theta_{j}^N(t) - \theta_i^N(t)\right).
\]

In \cite{Skardal20}, the authors studied another extension of the Kuramoto model for coupled oscillators with higher order interactions arising from phase-reductions of limit-cycle oscillators. More precisely, the evolution of the $N$ oscillators is given by 
\begin{equation*}
\begin{split}
\frac{d\theta^N_i(t)}{dt} = \Omega_i^N & \sum_{j_1=1}^N w^{1,N}_{ij_1} \sin(\theta_{j_1}^N(t)-\theta_i^N(t))+ \sum_{j_1=1}^N\sum_{j_2=1}^N w^{2,N}_{ij_1j_2} \sin(2 \theta_{j_1}^N(t)-\theta_{j_2}^N(t)-\theta_i^N(t))\\
& + \sum_{j_1=1}^N\sum_{j_2=1}^N\sum_{j_3=1}^N w^{3,N}_{ij_1j_2j_3} \sin(\theta_{j_1}^N(t)+\theta_{j_2}^N(t)-\theta_{j_3}^N(t)-\theta_i^N(t)).
\end{split}
\end{equation*}
Note that the interaction kernels $K_2(x,x_1,x_2)=\sin(2 x_{1}-x_{2}-x)$ and $K_3(x,x_1,x_2,x_3)=\sin(x_{1}+x_{2}-x_{3}-x)$ do not satisfy the symmetry assumption \eqref{eq:hypothesis-kernels-symmetry}.
The underlying hypergraph $(w^{\ell,N})_{\ell\in\{1,\ldots,3\}}$ is of rank 3, and is taken to be a simplicial complex constructed from real network datasets, as explained above.

\medskip $\diamond$ {\bf (Higher-order opinion dynamics models)}: In \cite{Neuhauser22}, a model for higher-order opinion dynamics is proposed on a uniform hypergraph of rank $2$. Denoting $x_i^N(t)$ the opinion of the $i$-th oscillator, its evolution is given by the following dynamics:  
\begin{equation*}
\frac{d x_i^N(t)}{dt} = \sum_{j_1=1}^N\sum_{j_2=1}^N w^{2,N}_{ij_1j_2} e^{\lambda|x_{j_1}^N(t)-x_{j_2}^N(t)|}\left(\frac{x_{j_1}^N(t)+x_{j_2}^N(t)}{2}-x_i^N(t)\right).
\end{equation*}
These dynamics model the idea that each agent $i$ is attracted towards the average opinion of the joint pair $(j_1,j_2)$, with an intensity that depends on the distance between the opinions of $j_1$ and $j_2$. More precisely, if $\lambda<0$, then the pair of agents $(j_1,j_2)$ will have a larger influence if their opinions $x_{j_1}$ and $x_{j_2}$ are similar ({\it i.e.} $|x_{j_1}-x_{j_2}|$ is small). On the opposite, if $\lambda>0$, then the pair of agents $(j_1,j_2)$ will have a larger influence if their opinions $x_{j_1}$ and $x_{j_2}$ are dissimilar ({\it i.e.}  $|x_{j_1}-x_{j_2}|$ is large). Notice that the interaction kernel $K_2(x,x_1,x_2)= e^{\lambda|x_{1}-x_{2}|}\left(\frac{x_{1}+x_{2}}{2}-x\right)$ does satisfy the symmetry assumption \eqref{eq:hypothesis-kernels-symmetry}.

One could further generalize this model to include interactions of groups of all sizes, with the following dynamics: 
\begin{equation*}
\frac{d x_i^N(t)}{dt} = \sum_{\ell=1}^{N-1} \sum_{j_1,\cdots,j_\ell=1}^N w^{\ell,N}_{ij_1\cdots j_\ell} e^{\lambda \mathrm{diam}(x_{j_1}^N(t),\ldots,x_{j\ell}^N(t))}\left(\frac{1}{\ell}\sum_{k=1}^\ell x_{j_k}^N(t) -x_i^N(t)\right),
\end{equation*}
where $\mathrm{diam}(x_{j_1},\ldots,x_{j\ell}) := \max_{k_1,k_2\in\{j_1,\cdots j_\ell\}}|x_{k_1}-x_{k_2}|.$
Here, each agent $i$ is attracted towards the average opinion of the group $\{j_1,\ldots,j_\ell\}$, with an intensity that depends on the diameter of the group.

\subsection{Functional setting}\label{subsec:functional-setting}

In this section, we set the functional setting which we will use all along the paper. As mentioned above, solutions to the Vlasov equation \eqref{eq:vlasov-equation}-\eqref{eq:vlasov-equation-force} will consist of parametrized families $(\mu^\xi)_{\xi\in [0,1]}\subset \mathcal{P}(\mathbb{R}^d)$ of probability densities. We shall then study existence and uniqueness of distributional solutions living in suitable measure-valued spaces. This suggests the following definition.

\begin{defi}[Borel family of probability measures]\label{defi:Borel-family-probability-measures}
Consider any $\nu\in \mathcal{P}([0,1])$, and let $(\mu^\xi)_{\xi\in [0,1]}\subset \mathcal{P}(\mathbb{R}^d)$ be a parametrized family of probability measures defined for $\nu$-a.e. $\xi\in [0,1]$. We say that $(\mu^\xi)_{\xi\in [0,1]}$ is a Borel family if the map $\xi\in [0,1]\longmapsto \mu^\xi(B)$ is Borel-measurable for every Borel set $B\subset \mathbb{R}^d$.
\end{defi}

As explained in \cite[Proposition 2.14]{PP-23}, the following three conditions are equivalent for a parametrized family $(\mu^\xi)_{\xi\in [0,1]}\subset \mathcal{P}(\mathbb{R}^d)$ of probability measures:

\begin{enumerate}[label=(\roman*)]
\item {\it (Borel family I)} $(\mu^\xi)_{\xi\in [0,1]}$ is a Borel family as in Definition \ref{defi:Borel-family-probability-measures}.
\item {\it (Borel family II)} The following scalar function
$$\xi\in [0,1]\longmapsto \int_{\mathbb{R}^d}\phi(x)\,d\mu^\xi(x).$$
is Borel-measurable for every bounded and Borel-measurable $\phi:\mathbb{R}^d\longrightarrow \mathbb{R}$.
\item {\it (Random probability measure)} The following measured-valued map
$$\xi\in [0,1]\longmapsto \mu^\xi\in \mathcal{P}(\mathbb{R}^d),$$
is Borel-measurable when $\mathcal{P}(\mathbb{R}^d)$ is endowed with its narrow topology.
\end{enumerate}

Therefore, Borel families of probability measures can be alternatively regarded as various different objects, among them: {\it Markov transition kernels} or also {\it random probability measures} ({\it i.e.}, random variables with values in $\mathcal{P}(\mathbb{R}^d)$), also called {\it Young measures}. As we shall see below, an alternative reformulation, which will prove useful in our approach is the following.

\begin{defi}[Fibered probability measures]\label{defi:fibered-probability-measures}
Consider any $\nu\in \mathcal{P}([0,1])$. We define the space of fibered probability measures by
$$\mathcal{P}_\nu (\mathbb{R}^d\times [0,1]):=\{\mu\in \mathcal{P}(\mathbb{R}^d\times [0,1]):\,\pi_{\xi\#}\mu=\nu\},$$
where $\pi_\xi(x,\xi)=\xi$ is the projection on the second component, and therefore $\pi_{\xi\#}\mu$ stands for the marginal of $\mu$ in the second component.
\end{defi}

So defined, fibered probability measures seem to be slightly different objects than the above Borel families of probability measures. However, it turns out that Definition \ref{defi:fibered-probability-measures} is also an alternative representation of the Borel families of probability measures in Definition \ref{defi:Borel-family-probability-measures} as clarified by the following classical result.

\begin{theo}[Disintegration theorem]\label{theo:disintegration}
Consider any $\nu \in\mathcal{P}([0,1])$ and $\mu\in \mathcal{P}_\nu(\mathbb{R}^d\times [0,1])$, then there exists a $\nu$-a.e. uniquely defined Borel family $(\mu^\xi)_{\xi\in [0,1]}\subset \mathcal{P}(\mathbb{R}^d)$ so that
$$\iint_{\mathbb{R}^d\times [0,1]}\varphi(x,\xi)\,d\mu(x,\xi)=\int_0^1\left(\int_{\mathbb{R}^d}\varphi(x,\xi)\,d\mu^\xi(x)\right)\,d\nu(\xi),$$
for every bounded Borel-measurable map $\varphi:\mathbb{R}^d\times [0,1]\longrightarrow \mathbb{R}$. Conversely, given any Borel family $(\mu^\xi)_{\xi\in [0,1]}$, then we can associate a unique fibered probability measure $\mu\in \mathcal{P}_\nu(\mathbb{R}^d\times [0,1])$ so that the above formula holds true, and for simplicity, we shall write $\mu(x,\xi)=\mu^\xi(x)\otimes \nu(\xi)$.
\end{theo}

For this reason, without loss of generality we will often identify measures $\mu\in \mathcal{P}_\nu(\mathbb{R}^d\times [0,1])$ with their associated Borel families of ($\nu$-a.e. defined) measures $(\mu^\xi)_{\xi \in [0,1]}\subset \mathcal{P}(\mathbb{R}^d)$.

We would like to endow the space of fibered probability measures (or Borel families of probability measures, or Markov kernels, or random measures, or Young measures, etc) with a proper metric structure. Note that the set of fibered probability measures $\mathcal{P}_\nu(\mathbb{R}^d\times [0,1])$ in Definition \ref{defi:fibered-probability-measures} is a closed subspace of $\mathcal{P}(\mathbb{R}^d\times [0,1])$ endowed with the narrow topology (which is metrizable). Hence, we may consider inducing such a (metrizable) topology on $\mathcal{P}_\nu(\mathbb{R}^d\times [0,1])$ so that we have $\mu_n\to \mu$ narrowly if, and only if,
$$\iint_{\mathbb{R}^d\times [0,1]}\varphi(x,\xi)\,d\mu_n(x,\xi)\to \iint_{\mathbb{R}^d\times [0,1]}\varphi(x,\xi)\,d\mu(x,\xi),$$
for all $\varphi\in C_b(\mathbb{R}^d\times [0,1])$. 

\begin{rem}[Narrow topology is not appropriate]\label{rem:issues-narrow-topology}
The use of narrow topology on $\mathcal{P}_\nu(\mathbb{R}^d\times [0,1])$ was discussed in \cite{PP-23}, and it was observed that it coincides with the canonical {\it stable topology} \cite{Ba-99,CRV-04,HL-15,V-89} when the marginal in the second component is fixed, see \cite[Corollary 2.9]{JM-81}, \cite[Theorem 2.1.1(D)]{CRV-04} or \cite[Lemma 2.1]{BL-18-arxiv}. Unfortunately, a strong drawback of choosing the narrow topology is that it is not stable under disintegration, meaning that if $\mu_n\to\mu$ narrowly, then it is not necessary true that $\mu_n^\xi\to \mu^\xi$ narrowly for $\nu$-a.e. $\xi\in [0,1]$, see \cite[Remark 2.16]{PP-23}.
\end{rem}

Instead of using narrow topology, we shall consider a stronger topology as follows.

\begin{defi}\label{defi:fibered-probability-measures-Lpdbl-distance}
Consider any $\nu\in \mathcal{P}([0,1])$ and any $p\in [1,\infty]$. Then, we define
$$d_{p,\nu}(\mu_1,\mu_2):=\left(\int_0^1d _{\rm BL}^p(\mu_1^\xi,\mu_2^\xi)\,d\nu(\xi)\right)^{1/p},\quad \mu_1,\mu_2\in \corr{\mathcal{P}_\nu(\mathbb{R}^d\times [0,1])}.$$
\end{defi}

\corr{Note that for any $\mu_1,\mu_2\in \mathcal{P}_\nu(\mathbb{R}^d\times [0,1])$ the map $\xi\in [0,1]\mapsto d_{\rm BL}(\mu_1^\xi,\mu_2^\xi)$ is Borel-measurable and bounded (because $d_{\rm BL}$ is bounded by $2$). Then, $d_{p,\nu}(\mu_1,\mu_2)$ is always well defined.} We note that the space in Definition \ref{defi:fibered-probability-measures-Lpdbl-distance} encodes a metric-valued $L^p$ space, and as such, it is a new metric space. Specifically, we have the following result, whose proof we omit since it be found in \cite[Appendix A]{PP-23} for $p=2$, and its adaptation to general $p$ is straightforward.

\begin{pro}\label{pro:fibered-probability-measures-Lpdbl-distance}
For any $\nu\in\mathcal{P}([0,1])$ and $p\in [1,\infty]$ the space $(\corr{\mathcal{P}_\nu(\mathbb{R}^d\times [0,1])},d_{p,\nu})$ in Definition \ref{defi:fibered-probability-measures-Lpdbl-distance} is a complete metric space. In addition, it is separable when $p\in [1,\infty)$.
\end{pro}

A related fibered distance was proposed in \cite{PP-23} by the third author for $p=2$, but the bounded-Lipschitz distance $d_{\rm BL}$ on $\mathcal{P}(\mathbb{R}^d)$ was replaced by the Wasserstein distance $W_2$ on $\mathcal{P}_2(\mathbb{R}^d)$. The resulting space was proven to have a weakly Riemannian structure reminiscent of the classical one on the Wasserstein space $(\mathcal{P}_2(\mathbb{R}^d),W_2)$ proposed by {\sc Otto} \cite{O-01}, and an abstract theory of gradient flows for functionals on  such a space was derived. 

In this paper, we stick to the choice of fibered bounded-Lipschitz distances $d_{p,\nu}$ and general $p\in [1,\infty)$. We emphasize though that the ultimate goal of considering general $p\in [1,\infty)$ is not simply a generalization of the well-posedness results of \eqref{eq:vlasov-equation}-\eqref{eq:vlasov-equation-force} to such functional setting, but actually, the common choice $p=\infty$ in previous literature (see {\it e.g.}, \cite{KX-22, KX-25}) will prove insufficient to obtain suitable stability estimates in Section \ref{sec:stability-estimate} involving a continuous dependence on the underlying UR-hypergraphon with respect to the cut distance. For this reason the case $p=\infty$ is discarded in the statement of our main result Theorem \ref{theo:main}. 

\corr{Indeed, the case of the metric $d_{\infty,\nu}$ is quite special and it departs from all the other metrics $d_{p,\nu}$ in a substancial way as we formulate in the next result.

\begin{pro}
For any $\nu\in \mathcal{P}([0,1])$, the following results hold true.
\begin{enumerate}[label=(\roman*)]
\item {\bf (Equivalent distances)} Consider $1\leq p_1<p_2<\infty$. Then, we have
$$d_{p_1,\nu}(\mu_1,\mu_2)\leq d_{p_2,\nu}(\mu_1,\mu_2)\leq  2^{1-\frac{1}{p_2}}d_{p_1,\nu}(\mu_1,\mu_2)^{\frac{p_1}{p_2}},$$
for every $\mu_1,\mu_2\in \mathcal{P}_\nu(\mathbb{R}^d\times [0,1])$.
\medskip
\item {\bf (Convergence in probability)} Consider any $1\leq p<\infty$, and any $(\mu_n)_{n\in \mathbb{N}}$ and $\mu$ in $\mathcal{P}_\nu(\mathbb{R}^d\times [0,1])$. Then, the following statements are equivalent
\begin{itemize}
\item $\displaystyle\lim_{n\to \infty}d_{p,\nu}(\mu_n,\mu)=0$.
\item $\displaystyle\lim_{n\to\infty}\nu(\{\xi\in [0,1]:\,d_{\rm BL}(\mu_n^\xi,\mu^\xi)\geq\varepsilon\})=0$ for all $\varepsilon>0$, that is, $\mu_n$ converge to $\mu$ in probability when regarded as random elements in $\mathcal{P}(\mathbb{R}^d)$.
\end{itemize}
\end{enumerate}
\end{pro}

\begin{proof}
First, we remark that since $d_{\rm BL}$ is bounded by $2$, then $(d_{\rm BL}(\mu_n^\cdot,\mu^\cdot))_{n\in \mathbb{N}}$ is a uniformly bounded sequence of functions.

On the one hand, part {\it (i)} is a straightforward exercice that follows from Jensen's inequality (for the first bound) and H{\"o}lder's inequality (for the second bound), where the above uniform bound (by $2$) of the sequence $(d_{\rm BL}(\mu_n^\cdot,\mu^\cdot))_{n\in \mathbb{N}}$ is used.

On the other hand, for part {\it (ii)} note that the above uniform bound of $(d_{BL}(\mu_n^\cdot,\mu^\cdot))_{n\in \mathbb{N}}$ implies that the sequence is uniformly $\nu$-integrable. Then, the equivalence of the convergence in $d_{p,\nu}$ (or equivalently $d_{1,\nu}$) and convergence in probability is a direct consequence of Vitalli's convergence theorem.
\end{proof}

The relation of this notion of convergence with the weaker notion of narrow convergence in $\mathcal{P}(\mathbb{R}^d\times [0,1])$ is summarized in the following result.
}

\begin{pro}\label{pro:fibered-probability-measures-Lpdbl-distance-relation-to-others}
For any $\nu\in \mathcal{P}([0,1])$ and $p\in [1,\infty]$ we have
$$d_{\rm BL}(\pi_{x\#}\mu_1,\pi_{x\#}\mu_2)\leq d_{\rm BL}(\mu_1,\mu_2)\leq d_{p,\nu}(\mu_1,\mu_2),$$
for every $\mu_1,\mu_2\in \corr{\mathcal{P}_\nu(\mathbb{R}^d\times [0,1])}$. Thereby, the $d_{p,\nu}$ topology is finer than the narrow topology.
\end{pro}
\begin{proof}
$\diamond$ {\sc Step 1}: $d_{\rm BL}(\pi_{x\#}\mu_1,\pi_{x\#}\mu_2)\leq d_{\rm BL}(\mu_1,\mu_2)$.\\
Consider any bounded and Lipschitz test function $\phi:\mathbb{R}^d\longrightarrow \mathbb{R}$ with $\Vert\phi\Vert_{\rm BL}\leq 1$. By the definition of the pushfoward and bounded-Lipschitz distance on $\mathcal{P}(\mathbb{R}^d\times [0,1])$, we have
\begin{align*}
&\int_{\mathbb{R}^d}\phi(x)\,(d\pi_{x\#}\mu_1(x)-d\pi_{x\#}\mu_2(x))=\int_{\mathbb{R}^d\times [0,1]}\phi(x)(d\mu_1(x,\xi)-d\mu_2(x,\xi))\\
&\qquad\leq \Vert\phi\Vert_{\rm BL} \,d_{\rm BL}(\mu_1,\mu_2)\leq d_{\rm BL}(\mu_1,\mu_2),
\end{align*}
and then we end by taking supremum over all $\phi$

\medskip

$\diamond$ {\sc Step 2}: $d_{\rm BL}(\mu_1,\mu_2)\leq d_{p,\nu}(\mu_1,\mu_2)$.\\
Consider any bounded and Lipschitz test function $\varphi:\mathbb{R}^d\times [0,1]\longrightarrow\mathbb{R}$ with $\Vert\varphi\Vert_{\rm BL}\leq 1$. Using the disintegration Theorem \ref{theo:disintegration} and H{\" o}lder's inequality with $\frac{1}{p}+\frac{1}{q}=1$, it holds
\begin{align*}
&\iint_{\mathbb{R}^d\times [0,1]}\varphi\,d(\mu_1-\mu_2)=\int_0^1\int_{\mathbb{R}^d}\varphi(x,\xi)\,(d\mu_1^\xi(x)-d\mu_2^\xi(x))\,d\nu(\xi)\\
&\qquad\leq \int_0^1\Vert \varphi(\cdot,\xi)\Vert_{\rm BL}\,d_{\rm BL}(\mu_1^\xi,\mu_2^\xi)\,d\nu(\xi)\leq \left(\int_0^1\Vert \varphi(\cdot,\xi)\Vert_{\rm BL}^q\,d\nu(\xi)\right)^{1/q}d_{p,\nu}(\mu_1,\mu_2).
\end{align*}
Since $\Vert \varphi(\cdot,\xi)\Vert_{\rm BL}\leq \Vert \varphi\Vert_{\rm BL}\leq 1$, we conclude by arbitrariness of the test function $\varphi$.
\end{proof}

In most cases, we will not have an isolated measure $\mu\in \mathcal{P}_\nu(\mathbb{R}^d\times [0,1])$, but a full curve $t\in [0,T]\mapsto\mu_t\in \mathcal{P}_\nu(\mathbb{R}^d\times [0,1])$. Given $t\in [0,T]$ we could disintegrate each $\mu_t$ using the above Theorem \ref{theo:disintegration}, but potentially we may have that the resulting Borel family $(\mu_t^\xi)_{\xi\in [0,1]}$ is not well defined on a $\nu$-negligible set depending on $t$. Additionally, this construction does not ensure whether for a.e. $\xi\in [0,1]$ the family $(\mu_t^\xi)_{t\in [0,T]}$ parametrized by time must be a Borel family, and that will be important later. The following strategy can be used to circumvent both issues.

\begin{theo}[Time-dependent disintegrations]\label{theo:disintegration-time-dependent}
Consider any $\nu\in \mathcal{P}([0,1])$ and any Borel family of probability measures $(\mu_t)_{t\in [0,T]}\subset \mathcal{P}_\nu(\mathbb{R}^d\times [0,1])$. Then, there exists a Borel family $(\mu_t^\xi)_{(t,\xi)\in [0,T]\times [0,1]}\subset \mathcal{P}(\mathbb{R}^d)$ defined for $dt\otimes\nu$-a.e. $(t,\xi)\in[0,T]\times[0,1]$ such that
\begin{equation}\label{eq:disintegration-time-dependent}
\int_0^T\left(\iint_{\mathbb{R}^d\times [0,1]}\psi(t,x,\xi)\,d\mu_t(x,\xi)\right)\,dt=\iint_{[0,T]\times [0,1]}\left(\int_{\mathbb{R}^d}\psi(t,x,\xi)\,d\mu_t^\xi(x)\right)\,dt\,d\nu(\xi),
\end{equation}
for every bounded Borel-measurable map $\psi:[0,T]\times \mathbb{R}^d\times [0,1]\longrightarrow \mathbb{R}$. In particular,
\begin{enumerate}[label=(\roman*)]
\item For a.e. $t\in [0,T]$, the slice $(\mu_t^\xi)_{\xi\in [0,1]}$ is a Borel family defined for $\nu$-a.e. $\xi\in [0,1]$, and it corresponds to a possible disintegration of $\mu_t$ in the sense of Theorem \ref{theo:disintegration}.
\item For $\nu$-a.e. $\xi\in [0,1]$, the slice $(\mu_t^\xi)_{t\in [0,T]}$ is a Borel family defined for a.e. $t\in [0,T]$.
\end{enumerate}
\end{theo}

\begin{proof}
Since $(\mu_t)_{t\in [0,T]}\subset\mathcal{P}_\nu(\mathbb{R}^d\times [0,1])$ is a Borel family, Theorem \ref{theo:disintegration} allows finding a measure $\hat\mu:=\frac{1}{T}dt_{\lfloor [0,T]}\otimes \mu_t(x,\xi)\in \mathcal{P}([0,T]\times \mathbb{R}^d\times [0,1])$ such that
\begin{equation}\label{eq:disintegration-time-dependent-step-1}
\iiint_{[0,T]\times \mathbb{R}^d\times [0,1]}\psi(t,x,\xi)\,d\hat\mu(t,x,\xi)=\int_0^T\left(\iint_{\mathbb{R}^d\times [0,1]}\psi(t,x,\xi)\,d\mu_t(x,\xi)\right)\frac{dt}{T},
\end{equation}
for every bounded Borel-measurable map $\psi:[0,T]\times \mathbb{R}^d\times [0,1]\longrightarrow \mathbb{R}$. Consider the marginal 
$$\hat\nu:=\pi_{(t,\xi)_\#}\hat\mu=\frac{1}{T}dt_{\lfloor [0,T]}\otimes \nu(\xi)\in \mathcal{P}([0,T]\times [0,1]),$$
and disintegrate $\hat\mu$ with respect $(t,\xi)$ as in Theorem \ref{theo:disintegration} to obtain a Borel $\hat\nu$-a.e. defined family $(\mu_t^\xi)_{(t,\xi)\in [0,T]\times [0,1]}$, which is Borel-measurable jointly in the two indices $(t,\xi)$, and such that
\begin{equation}\label{eq:disintegration-time-dependent-step-2}
\iiint_{[0,T]\times \mathbb{R}^d\times [0,1]}\psi(t,x,\xi)\,d\hat\mu(t,x,\xi)=\iint_{[0,T]\times[0,1]}\left(\int_{\mathbb{R}^d}\psi(t,x,\xi)\,d\mu_t^\xi(x)\right)\,\frac{dt}{T}\otimes \nu(\xi),
\end{equation}
for every bounded Borel-measurable map $\psi:[0,T]\times \mathbb{R}^d\times [0,1]\longrightarrow \mathbb{R}$. Equating \eqref{eq:disintegration-time-dependent-step-1} and \eqref{eq:disintegration-time-dependent-step-2} and multyplying by $T$ implies \eqref{eq:disintegration-time-dependent}.

Let us consider the $dt\otimes\nu$-negligible set $\mathcal{N}\subset [0,T]\times [0,1]$ consisting of the points where the Borel family $(\mu_t^\xi)_{(t,\xi)\in [0,T]\times [0,1]}$ is not defined and note that $\mathcal{N}$ can be reformulated as
$$\mathcal{N}=\{(t,\xi)\in [0,T]\times [0,1]:\,t\in \mathcal{N}_\xi^1\}=\{(t,\xi)\in [0,T]\times [0,1]:\,\xi\in \mathcal{N}_t^2\},$$
where $\mathcal{N}_\xi^1$ and $\mathcal{N}_t^2$ are the slice in the first and second components:
\begin{align*}
\mathcal{N}_\xi^1&:=\{t\in [0,T]:\,(t,\xi)\in \mathcal{N}\},\quad \xi\in [0,1],\\
\mathcal{N}_t^2&:=\{\xi\in [0,1]:\,(t,\xi)\in \mathcal{N}\},\quad t\in [0,T].
\end{align*}
By Fubini theorem, we have that
$$
0=\hat\nu(\mathcal{N})=\int_{[0,T]\times [0,1]}\mathds{1}_\mathcal{N}(t,\xi)\,d\hat\nu(t,\xi)=\frac{1}{T}\int_0^T\nu(\mathcal{N}_t^2)\,dt=\frac{1}{T}\int_0^1|\mathcal{N}_\xi^1|\,d\nu(\xi).
$$
Hence, $\nu(\mathcal{N}_t^2)=0$ for a.e. $t\in [0,T]$ and $|\mathcal{N}_\xi^1|=0$ for $\nu$-a.e. $\xi\in [0,1]$. Thefore, for a.e. $t\in [0,T]$ the slice $(\mu_t^\xi)_{\xi\in [0,1]}$ is defined except on a $t$-dependent $\nu$-negligible set, and for $\nu$-a.e. $\xi\in [0,1]$ the slice $(\mu_t^\xi)_{t\in [0,T]}$ is defined except on a $\xi$-dependent Lebesgue negligible set. The Borel-measurability of each slice follows from the joint measurability in $(t,\xi)$.

Finally, note that $(\mu_t^\xi)_{\xi\in [0,1]}$ must be a disintegration of $\mu_t$ for a.e. $t\in [0,T]$. To this end, we consider any dense sequence $\{\varphi_n\}_{n\in \mathbb{N}}\subset C_c(\mathbb{R}^d\times [0,1])$, which exists by separability, and consider \eqref{eq:disintegration-time-dependent} for the special test function $\psi(t,x,\xi):=\eta(t)\varphi_n(x,\xi)$ with $\eta\in C([0,T])$. Then,
$$\int_0^T\eta(t)\left(\iint_{\mathbb{R}^d\times [0,1]}\varphi_n(x,\xi)\,d\mu_t(x,\xi)\right)\,dt=\int_0^T\eta(t)\left(\int_0^1\int_{\mathbb{R}^d}\varphi_n(x,\xi)\,d\mu_t^\xi(x)\,d\nu(\xi)\right)\,dt,$$
for all $\psi\in C([0,T])$ and all $n\in \mathbb{N}$. Since $\eta$ is arbitrary and the time integrand is integrable, then the fundamental calculus of variations shows that for every $n\in \mathbb{N}$ there exists a Lebesgue-negligible set $\mathcal{N}_n\subset [0,T]$ so that
$$\int_{\mathbb{R}^d\times [0,1]}\varphi_n(x,\xi)\,d\mu_t(x,\xi)=\int_0^1\int_{\mathbb{R}^d}\varphi_n(x,\xi)\,d\mu_t^\xi(x)\,d\nu(\xi),$$
for all $t\in [0,T]\setminus\mathcal{N}_n$ and all $n\in \mathcal{N}$. Defining the new set $\mathcal{N}_\infty:=\cup_{n\in \mathbb{N}}$, which is Lebesgue-negligible again and by density of the sequence $\{\phi_n\}_{n\in \mathbb{N}}$, we have that for every $t\in [0,T]\setminus \mathcal{N}_\infty$ the Borel family $(\mu_t^\xi)_{\xi\in [0,1]}$ is a disintegration of $\mu_t$ in the sense of Theorem \ref{theo:disintegration}.
\end{proof}

In the above result we may be tempted to further claim that if $\mu\in C([0,T],\mathcal{P}_\nu(\mathbb{R}^d\times [0,1])\mbox{-narrow})$, then we also have $\mu^\xi\in C([0,T],\mathcal{P}(\mathbb{R}^d)\mbox{-narrow})$ for $\nu$-a.e. fiber $\xi\in [0,1]$. However, this is not true in general as noted in Remark \ref{rem:issues-narrow-topology} because the narrow topology is not disintegrable. 

\begin{rem}[On the reference measure $\nu$]\label{rem:reference-nu}
In this paper we shall restrict to $\nu=d\xi_{\lfloor [0,1]}$, that is, the Lebegue measure on $[0,1]$. The fundamental reason is that, as for the classical theory of dense graph limits \cite{LS-06}, the limit theory of dense simplicial complex in \cite{RS-24} focuses on hypergraphs with weighted hyperedges, but non-weighted vertices. Alternatively, we may think that vertices are weighted uniformly, which explains the uniform density $\nu=d\xi_{\lfloor [0,1]}$. However, we could have added heterogeneous weights (mass) to vertices, thus changing the underlying measure $\nu$. This was treated in previous literature \cite{KX-22,KX-25} for Vlasov-type equations like \eqref{eq:vlasov-equation}-\eqref{eq:vlasov-equation-force}  restricted to binary, or higher-order interactions of bounded order.
\end{rem}

In order to make the language of Borel families of probability measures $(\mu_t^\xi)_{\xi\in [0,1]}\subset \mathcal{P}(\mathbb{R}^d)$ in Definition \ref{defi:Borel-family-probability-measures} compatible with the language of fibered probability measures $\mu_t\in \mathcal{P}_\nu(\mathbb{R}^d\times [0,1])$ in Definition \ref{defi:fibered-probability-measures} at the level of the Vlasov equation, we reformulate \eqref{eq:vlasov-equation}-\eqref{eq:vlasov-equation-force} as follows:
\begin{equation}\label{eq:vlasov-equation-joint}
\begin{aligned}
&\partial_t \mu_t+\divop_x(F_w[\mu_t]\,\mu_t)=0,\quad t\geq 0,\,x\in \mathbb{R}^d,\,\xi\in [0,1],\\
&\mu_{t=0}=\mu_0,
\end{aligned}
\end{equation}
with $\mu_0\in \mathcal{P}_\nu(\mathbb{R}^d\times [0,1])$, where the mean-field force reads
\begin{equation}\label{eq:vlasov-equation-force-joint}
\begin{aligned}
F_w[\mu_t](x,\xi):=&\sum_{\ell=1}^\infty \int_{[0,1]^\ell} w_\ell(\xi,\xi_1,\ldots,\xi_\ell)\\
&\quad\times \left(\int_{\mathbb{R}^{d\ell}}K_\ell(x,x_1,\ldots,x_\ell)\,d\mu_t^{\xi_1}(x_1)\,\cdots \,d\mu_t^{\xi_\ell}(x_\ell)\right)d\xi_1,\ldots\,d\xi_\ell,
\end{aligned}
\end{equation}

We note that measure-valued solutions to \eqref{eq:vlasov-equation-joint}-\eqref{eq:vlasov-equation-force-joint} are indeed expected to take values in the subspace of probability measures $\mathcal{P}_\nu(\mathbb{R}^d\times [0,1])$. Specifically, since the divergence in $\eqref{eq:vlasov-equation-joint}_1$ only affects the variable $x$, then the dynamics takes place in a fiberwise, that is, measure-valued solutions should satisfy $\pi_{\xi\#}\mu_t=\pi_{\xi\#}\mu_0=\nu$ for all $t\geq 0$. This further justifies our choice of proposing a theory of measure-valued solutions living in fibered spaces $\mathcal{P}_\nu(\mathbb{R}^d\times [0,1])$. We refer to Section \ref{subsec:notion-distributional-solution} for precise definitions of distributional solutions on the above space, where in particular the equivalence of \eqref{eq:vlasov-equation}-\eqref{eq:vlasov-equation-force} and \eqref{eq:vlasov-equation-joint}-\eqref{eq:vlasov-equation-force-joint} is shown.

\section{Propagation of independence over hypergraphs}\label{sec:propagation-independence}

Inspired by the approach developed in \cite{JPS-25}, we propose the following auxiliary multi-agent system, which we expect to approximate the original multi-agent system \eqref{eq:multi-agent-system} as $N\to \infty$ (under suitable scaling assumptions on the coupling weights):
\begin{equation}\label{eq:multi-agent-system-mckean-sde}
\begin{cases}
\displaystyle \frac{d\bar X_i^N}{dt}=\sum_{\ell=1}^{N-1}\sum_{j_1,\ldots, j_\ell=1}^N w^{\ell,N}_{ij_1\cdots j_\ell}\,\mathbb{E}_i^N K_\ell(\bar X_i^N,\bar X_{j_1}^N,\ldots,\bar X_{j_\ell}^N),\\
\bar X_i^N(0)=X_{i,0}^N.
\end{cases}
\end{equation}
Above, $\mathbb{E}_i^N=\mathbb{E}[\,\cdot\,\vert \bar{\mathcal{F}}_i^N]$ denotes the expectation conditioned to the natural filtration $\bar{\mathcal{F}}_i^N$ of $\bar X_i^N$, that is, the family of time-dependent sub-$\sigma$-algebras defined by 
\begin{equation}\label{eq:filtration-barFi}
\bar{\mathcal{F}}_i^N(t)=\sigma(\{\bar X_i^N(s):\,0\leq s\leq t\}).
\end{equation}
An analogue of the intermediary system \eqref{eq:multi-agent-system-mckean-sde} was first introduced in \cite{JPS-25} for multi-agent systems with binary interactions (that is $\ell=1$). It consists in the non-exchangeable version of the McKean process proposed in \cite{S-91} to derive propagation of chaos of exchangeable multi-agent systems with bounded-Lipschitz binary interaction kernels via the trajectorial approach. Here, we extend it to operate over non-exchangeable multi-agent systems evolving under higher-order interactions characterized by hypergraphs of unbounded order.

First, we study the well-posedness of the auxiliary multi-agent system \eqref{eq:multi-agent-system-mckean-sde}.

\begin{lem}[Well-posedness of the intermediary system]\label{lem:multi-agent-system-mckean-sde-well-posedness}
Assume that the kernels $K_\ell$ verify the assumptions \eqref{eq:hypothesis-kernels-bounded} and \eqref{eq:hypothesis-kernels-lipschitz}. Then, for any $(X_{1,0}^N,\ldots,X_{N,0}^N)$ such that $\mathbb{E}|X_{i,0}^N|<\infty$, and for any $(w^{\ell,N}_{ij_1\ldots j\ell})_{1\leq i,j_1,\ldots,j_\ell\leq N}$ with $\ell=1,\cdots,N-1$, there is a global-in-time solution $(\bar X_1^N,\ldots,\bar X_N^N)$ to \eqref{eq:multi-agent-system-mckean-sde} which is unique pathwise and in law.
\end{lem}

We omit the proof since it is an elementary extension of the classical case with binary interactions and constant weights, see \cite[Proposition 1.3.]{JM-98} and \cite[Theorem 1.1]{S-91}.

We now find a quantitative error estimate about the deviation between the original system \eqref{eq:multi-agent-system} and the approximate one \eqref{eq:multi-agent-system-mckean-sde}. Contrarily to other approaches ({\it e.g.}, \cite{JPS-25,S-91}), where only $\ell_\infty$-based estimates were obtained, we find general $\ell_p$-based error estimates.

\corr{
\begin{lem}[Error estimate]\label{lem:multi-agent-system-mckean-sde-error-estimate}
Assume that the kernels $K_\ell$ satisfy the assumptions \eqref{eq:hypothesis-kernels-bounded} and \eqref{eq:hypothesis-kernels-lipschitz}, and the weights $(w^{\ell,N}_{ij_1\ldots j\ell})_{1\leq i,j_1,\ldots,j_\ell\leq N}$ with $\ell=1,\cdots,N-1$ satisfy \eqref{eq:hypothesis-absence-loops}. Suppose additionally that the symmetry hypotheses \eqref{eq:hypothesis-weights-symmetry} and \eqref{eq:hypothesis-kernels-symmetry} are satisfied. For any $(X_{1,0}^N,\ldots,X_{N,0}^N)$ with independent $X_{i,0}^N$ such that $\mathbb{E}|X_{i,0}^N|<\infty$, consider the unique solutions $(X_1^N,\ldots,X_N^N)$ to \eqref{eq:multi-agent-system} and $(\bar X_1^N,\ldots,\bar X_N^N)$ to \eqref{eq:multi-agent-system-mckean-sde} as in Lemma \ref{lem:multi-agent-system-mckean-sde-well-posedness}. Then,
\begin{equation}\label{eq:multi-agent-system-mckean-sde-error-estimate}
\left(\frac{1}{N}\sum_{i=1}^N\left(\mathbb{E}|X_i^N(t)-\bar X_i^N(t)|\right)^p\right)^{1/p}\leq e^{(\tilde C_\infty^N+C_p^N)t}\varepsilon_p^N,
\end{equation}
for all $t\geq 0$, where the constants $\tilde C_\infty^N$, $C_p^N$ and $\varepsilon_p^N$ are defined by
\begin{align}
\tilde C_\infty^N&:=\max_{1\leq i\leq N} \sum_{\ell=1}^{N-1}L_\ell\sum_{\bj_\ell\in \llbracket1,N\rrbracket^\ell} w^{\ell,N}_{i\bj_\ell},\label{eq:multi-agent-system-mckean-sde-error-estimate-constant-tildeC-infty}\\
C_p^N&:=\left(\sum_{i=1}^N\left(\sum_{\ell=1}^{N-1}L_\ell\sum_{k=1}^\ell \sum_{\hbj_{\ell,k}\in \llbracket1,N\rrbracket^{\ell-1}}\left(\sum_{j_k=1}^N (w^{\ell,N}_{i\bj_\ell})^q\right)^{1/q}\right)^p\right)^{1/p},\label{eq:multi-agent-system-mckean-sde-error-estimate-constant-C-p}\\
\varepsilon_p^N&:=2\left(\frac{1}{N}\sum_{i=1}^N\left(\sum_{\ell=1}^{N-1}\sqrt{\ell!}\,B_\ell\left(\sum_{j_1,\ldots,j_\ell=1}^N (w^{\ell,N}_{ij_1\cdots j_\ell})^2\right)^{1/2}\right)^p\right)^{1/p},\label{eq:multi-agent-system-mckean-sde-error-estimate-constant-epsilon}
\end{align}
and $p,q\in [1,\infty]$ are related by $\frac{1}{p}+\frac{1}{q}=1$.
\end{lem}

\begin{proof}
We give a proof only for $1<p<\infty$, but the same argument goes through for $p=1$ and $p=\infty$ by suitably replacing sums by $\max$ over $i$ when needed, and then we omit the details. We start by comparing the original $X_i^N$ and auxiliary multi-agent system $\bar X_i^N$. Specifically, we integrate \eqref{eq:multi-agent-system} and  \eqref{eq:multi-agent-system-mckean-sde} over the time interval $[0,t]$, take their difference,  add and subtract the intermediary term $K_\ell(\bar X_i^N,\bar X_{j_1}^N,\ldots,\bar X_{j_\ell}^N)$, and take global expectations to obtain
\begin{equation*}
\mathbb{E}|X_i^N(t)-\bar X_i^N(t)|\leq \int_0^t \mathbb{E}A_i^N(s)\,ds+\int_0^t \mathbb{E}B_i^N(s)\,ds,
\end{equation*}
where each term takes the form
\begin{align}
A_i^N&:=\sum_{\ell=1}^{N-1}\sum_{\bj_\ell\in \llbracket1,N\rrbracket^\ell} w^{\ell,N}_{i\bj_\ell}\left\vert K_\ell(X_i^N,X_{j_1}^N,\cdots, X_{j_\ell}^N)-K_\ell(\bar X_i^N,\bar X_{j_1}^N,\cdots, \bar X_{j_\ell}^N)\right\vert,\label{eq:multi-agent-system-mckean-sde-error-estimate-step-A}\\
B_i^N&:=\sum_{\ell=1}^{N-1}\left\vert\sum_{\bj_\ell\in \llbracket1,N\rrbracket^\ell} w^{\ell,N}_{i\bj_\ell}(K_\ell(\bar X_i^N,\bar X_{j_1}^N,\cdots, \bar X_{j_\ell}^N)-\mathbb{E}_i^N K_\ell(\bar X_i^N,\bar X_{j_1}^N,\cdots, \bar X_{j_\ell}^N))\right\vert.\label{eq:multi-agent-system-mckean-sde-error-estimate-step-B}
\end{align}
Taking $\ell_p$ norms with respect to the index $i$, and by Minkowski's inequality, it holds
\begin{equation}\label{eq:multi-agent-system-mckean-sde-error-estimate-step-1}
\begin{aligned}
&\left(\frac{1}{N}\sum_{i=1}^N(\mathbb{E}|X_i^N(t)-\bar X_i^N(t)|)^p\right)^{1/p}\\
&\qquad\leq \int_0^t \left(\frac{1}{N}\sum_{i=1}^N(\mathbb{E}A_i^N(s))^p\right)^{1/p}\,ds+\int_0^t \left(\frac{1}{N}\sum_{i=1}^N(\mathbb{E}B_i^N(s))^p\right)^{1/p}\,ds.
\end{aligned}
\end{equation}
In the sequel we focus on finding suitable bounds for $\mathbb{E}|A_i^N|^p$ and $\mathbb{E}|B_i^N|^p$.

\medskip

$\diamond$ {\sc Step 1}: Control of $\mathbb{E}A_i^N$.\\
On the one hand, using the Lipschitz-continuity assumption \eqref{eq:hypothesis-kernels-lipschitz} of $K_\ell$ on \eqref{eq:multi-agent-system-mckean-sde-error-estimate-step-A} and taking global expectations we have
\begin{align*}
\mathbb{E}A_i^N&\leq \sum_{\ell=1}^{N-1}L_\ell\sum_{\bj_\ell\in \llbracket1,N\rrbracket^\ell} w^{\ell,N}_{i\bj_\ell}\left( \mathbb{E}|X_i^N-\bar X_i^N|+\sum_{k=1}^\ell\mathbb{E}|X_{j_k}^N-\bar X_{j_k}^N|\right)\leq I_i^N+II_i^N.
\end{align*}
Taking $\ell_p$ norms with respect to $i$ in the first term $I_i^N$ yields
$$\left(\frac{1}{N}\sum_{i=1}^N (I_i^N)^p\right)^{1/p}\leq \tilde C_\infty^N\left(\frac{1}{N}\sum_{i=1}^N(\mathbb{E}|X_i^N-\bar X_i^N|)^p\right)^{1/p},$$
where the constant $\tilde C_\infty^N$ is explicitly given by \eqref{eq:multi-agent-system-mckean-sde-error-estimate-constant-tildeC-infty}. The second term $II_i^N$ can be controlled by
\begin{align*}
II_i^N&=\sum_{\ell=1}^{N-1}L_\ell\sum_{k=1}^\ell \sum_{\hbj_{\ell,k}\in \llbracket1,N\rrbracket^{\ell-1}}\sum_{j_k=1}^N w^{\ell,N}_{i\bj_\ell}\mathbb{E}|X_{j_k}^N-\bar X_{j_k}^N|\\
&\leq \sum_{\ell=1}^{N-1}L_\ell\sum_{k=1}^\ell \sum_{\hbj_{\ell,k}\in \llbracket1,N\rrbracket^{\ell-1}}\left(\sum_{j_k=1}^N (w^{\ell,N}_{i\bj_\ell})^q\right)^{1/q}\left(\sum_{j=1}^N (\mathbb{E}|X_j^N-\bar X_j^N|)^p\right)^{1/p},
\end{align*}
where in the last inequality we have used H{\"o}lder's inequality on the sum. Then, taking $\ell_p$ norms with respect to $i$ yields
$$\left(\frac{1}{N}\sum_{i=1}^N (II_i^N)^p\right)^{1/p}\leq C_p^N \left(\frac{1}{N}\sum_{i=1}^N(\mathbb{E}|X_i^N-\bar X_i^N|)^p\right)^{1/p},$$
where the constant $C_p^N$ is also explicitly given in \eqref{eq:multi-agent-system-mckean-sde-error-estimate-constant-C-p}. Therefore, we readily obtain a bound on the first term of the right hand side of \eqref{eq:multi-agent-system-mckean-sde-error-estimate-step-1}:
\begin{equation}\label{eq:multi-agent-system-mckean-sde-error-estimate-step-2}
\left(\frac{1}{N}\sum_{i=1}^N(\mathbb{E}A_i^N)^p\right)^{1/p}\leq (\tilde C_\infty^N+C_p^N)\left(\frac{1}{N}\sum_{i=1}^N(\mathbb{E}|X_i^N-\bar X_i^N|)^p\right)^{1/p}.
\end{equation}

\medskip
$\diamond$ {\sc Step 2}: Control of $\mathbb{E}B_i^N$.\\
In this second step, we extend the approach proposed in \cite{S-91} for exchangeable multi-agent system and also in \cite{JPS-25} for non-exchangeable multi-agent systems with binary interactions to our non-exchangeable multi-agent system governed by on higher-order interactions. Additionally, we extend the $\ell_\infty$-based estimates on the decay of correlations by general $\ell_p$-based estimates. Taking expectations in \eqref{eq:multi-agent-system-mckean-sde-error-estimate-step-B} and using Jensen's inequality on each of the terms of the sum over $\ell$ to control first order moments by the square root of second order moments we have
$$\mathbb{E}B_i^N\leq \sum_{\ell=1}^{N-1}R_i^{\ell,N},$$
where the errors $R_i^{\ell,N}$ are defined by
$$R_i^{\ell,N}=\left(\mathbb{E}\left\vert\sum_{\bj_\ell\in \llbracket1,N\rrbracket^\ell} w^{\ell,N}_{i\bj_\ell}(K_\ell(\bar X_i^N,\bar X_{j_1}^N,\cdots, \bar X_{j_\ell}^N)-\mathbb{E}_i^N K_\ell(\bar X_i^N,\bar X_{j_1}^N,\cdots, \bar X_{j_\ell}^N))\right\vert^2\right)^{1/2}.$$
Squaring $R_i^{\ell,N}$ and expanding the involved square we obtain
\begin{align}
&(R_i^{\ell,N})^2=\sum_{\bj_\ell\in \llbracket1,N\rrbracket^\ell}\sum_{\bj_\ell'\in \llbracket1,N\rrbracket^\ell}w^{\ell,N}_{i\bj_\ell}w^{\ell,N}_{i\bj_\ell'}\mathbb{E}\Big[(K_\ell(\bar X_i^N,\bar X_{j_1}^N,\cdots, \bar X_{j_\ell}^N)-\mathbb{E}_i^N K_\ell(\bar X_i^N,\bar X_{j_1}^N,\cdots, \bar X_{j_\ell}^N))\nonumber\\
&\qquad\cdot (K_\ell(\bar X_i^N,\bar X_{j_1'}^N,\cdots, \bar X_{j_\ell'}^N)-\mathbb{E}_i^N K_\ell(\bar X_i^N,\bar X_{j_1'}^N,\cdots, \bar X_{j_\ell'}^N))\Big].\label{eq:multi-agent-system-mckean-sde-error-estimate-step-R}
\end{align}
We now distinguish two different cases for the multi-indices $\bj_\ell$ and $\bj_\ell'$ depending on whether we have $\{j_1,\ldots,j_\ell\}=\{j_1',\ldots,j_\ell'\}$ or  $\{j_1,\ldots,j_\ell\}\neq\{j_1',\ldots,j_\ell'\}$.

\medskip

$\triangleright$ {\sc Case 1}: $\{j_1,\ldots,j_\ell\}=\{j_1',\ldots,j_\ell'\}$.\\
In this case there exists a permutation $\sigma\in \mathcal{S}_\ell$ such that $\sigma(\bj_\ell')=\bj_\ell$. By the symmetry assumptions \eqref{eq:hypothesis-weights-symmetry} on the coupling weights $(w^{\ell,N}_{ij_1\cdots j_\ell})_{1\leq i,j_1,\ldots,j_\ell\leq N}$ we have
$$w^{\ell,N}_{i\bj_\ell'}=w^{\ell,N}_{i\sigma(\bj_\ell')}=w^{\ell,N}_{i\bj_\ell}.$$
Similarly, by the symmetry \eqref{eq:hypothesis-kernels-symmetry} for the kernels $K_\ell$ we also have
\begin{align*}
&K_\ell(\bar X_i^N,\bar X_{j_1'}^N,\cdots, \bar X_{j_\ell'}^N)-\mathbb{E}_i^N K_\ell(\bar X_i^N,\bar X_{j_1'}^N,\cdots, \bar X_{j_\ell'}^N)\\
&\qquad=K_\ell(\bar X_i^N,\bar X_{\sigma(j_1')}^N,\cdots, \bar X_{\sigma(j_\ell')}^N)-\mathbb{E}_i^N K_\ell(\bar X_i^N,\bar X_{\sigma(j_1')}^N,\cdots, \bar X_{\sigma(j_\ell')}^N)\\
&\qquad=K_\ell(\bar X_i^N,\bar X_{j_1}^N,\cdots, \bar X_{j_\ell}^N)-\mathbb{E}_i^N K_\ell(\bar X_i^N,\bar X_{j_1}^N,\cdots, \bar X_{j_\ell}^N).
\end{align*}
Therefore, the corresponding diagonal term in the sum \eqref{eq:multi-agent-system-mckean-sde-error-estimate-step-B} simplifies into
$$(w^{\ell,N}_{i,\bj_\ell})^2\mathbb{E}\left\vert K_\ell(\bar X_i^N,\bar X_{j_1}^N,\cdots, \bar X_{j_\ell}^N)-\mathbb{E}_i^N K_\ell(\bar X_i^N,\bar X_{j_1}^N,\cdots, \bar X_{j_\ell}^N)\right\vert^2\leq 4B_\ell^2 (w^{\ell,N}_{i,\bj_\ell})^2.$$

\medskip

$\triangleright$ {\sc Case 2}: $\{j_1,\ldots,j_\ell\}\neq\{j_1',\ldots,j_\ell'\}$.\\
In this case, there is $1\leq m<\ell$ and two permutations $\sigma,\sigma'\in \mathcal{S}_\ell$ so that $\sigma(\bj_\ell)$ and $\sigma'(\bj_\ell')$ share exactly the first $m$ components, and the last $\ell-m$ components are disjoint, that is,  
\begin{equation}\label{eq:multi-agent-system-mckean-sde-error-estimate-disjoint}
j_{\sigma(1)}=j_{\sigma'(1)}',\ldots,j_{\sigma(m)}=j_{\sigma'(m)}',\quad \mbox{and}\quad \{j_{\sigma(m+1)},\ldots,j_{\sigma(\ell)}\}\cap \{j_{\sigma'(m+1)}',\ldots,j_{\sigma'(\ell)}'\}=\emptyset.
\end{equation}
Similarly to $\mathbb{E}_i^N$, we define the new expectation $\mathbb{E}_{i,m}^N=\mathbb{E}[\,\cdot\,\vert \bar{\mathcal{F}}_{i,m}^N]$ which is conditioned to the natural filtration $\bar{\mathcal{F}}_{i,m}^N$ of the vector $(\bar X_i^N,\bar X_{j_{\sigma(1)}}^N,\ldots,\bar X_{j_{\sigma(m)}}^N)$, that is
$$\bar{\mathcal{F}}_{i,m}^N(t)=\sigma(\{\bar X_i^N(s),\bar X_{j_{\sigma(1)}}^N(s),\ldots,\bar X_{j_{\sigma(m)}}^N(s):\,0\leq s\leq t\}).$$

The law of iterated expectations $\mathbb{E}=\mathbb{E}\,\mathbb{E}_{i,m}^N$ along with the above symmetry assumption \eqref{eq:hypothesis-kernels-symmetry} on the kernels imply
\begin{align*}
&\mathbb{E}\Big[(K_\ell(\bar X_i^N,\bar X_{j_1}^N,\cdots, \bar X_{j_\ell}^N)-\mathbb{E}_i^N K_\ell(\bar X_i^N,\bar X_{j_1}^N,\cdots, \bar X_{j_\ell}^N))\\
&\qquad\cdot (K_\ell(\bar X_i^N,\bar X_{j_1'}^N,\cdots, \bar X_{j_\ell'}^N)-\mathbb{E}_i^N K_\ell(\bar X_i^N,\bar X_{j_1'}^N,\cdots, \bar X_{j_\ell'}^N))\Big]\\
&=\mathbb{E}\Big\{\mathbb{E}_{i,m}^N\Big[ (K_\ell(\bar X_i^N,\bar X_{j_{\sigma(1)}}^N,\ldots, \bar X_{j_{\sigma(m)}}^N,\bar X_{j_{\sigma(m+1)}}^N,\ldots,\bar X_{j_{\sigma(\ell)}}^N)\\
&\qquad\qquad\qquad -\mathbb{E}_i^N K_\ell(\bar X_i^N,\bar X_{j_{\sigma(1)}}^N,\ldots, \bar X_{j_{\sigma(m)}}^N,\bar X_{j_{\sigma(m+1)}}^N,\ldots,\bar X_{j_{\sigma(\ell)}}^N))\\
&\qquad\cdot(K_\ell(\bar X_i^N,\bar X_{j_{\sigma(1)}}^N,\ldots, \bar X_{j_{\sigma(m)}}^N,\bar X_{j_{\sigma'(m+1)}'}^N,\ldots,\bar X_{j_{\sigma'(\ell)}'}^N)\\
&\qquad\qquad\qquad -\mathbb{E}_i^N K_\ell(\bar X_i^N,\bar X_{j_{\sigma(1)}}^N,\ldots, \bar X_{j_{\sigma(m)}}^N,\bar X_{j_{\sigma'(m+1)}'}^N,\ldots,\bar X_{j_{\sigma'(\ell)}'}^N))\Big]\Big\}.
\end{align*}
By \eqref{eq:multi-agent-system-mckean-sde-error-estimate-disjoint} and the independence of all $\bar X_i^N$ we observe that the random vectors $(\bar X_{j_{\sigma(m+1)}}^N,\ldots,\bar X_{j_{\sigma(\ell)}}^N)$ and $(\bar X_{j_{\sigma'(m+1)}'}^N,\ldots,\bar X_{j_{\sigma'(\ell)}'}^N)$ are independent and therefore, the random variable
\begin{align*}
&K_\ell(\bar X_i^N,\bar X_{j_{\sigma(1)}}^N,\ldots, \bar X_{j_{\sigma(m)}}^N,\bar X_{j_{\sigma(m+1)}}^N,\ldots,\bar X_{j_{\sigma(\ell)}}^N)\\
&\qquad\qquad\qquad -\mathbb{E}_i^N K_\ell(\bar X_i^N,\bar X_{j_{\sigma(1)}}^N,\ldots, \bar X_{j_{\sigma(m)}}^N,\bar X_{j_{\sigma(m+1)}}^N,\ldots,\bar X_{j_{\sigma(\ell)}}^N)
\end{align*}
and the random variable
\begin{align*}
&K_\ell(\bar X_i^N,\bar X_{j_{\sigma(1)}}^N,\ldots, \bar X_{j_{\sigma(m)}}^N,\bar X_{j_{\sigma'(m+1)}'}^N,\ldots,\bar X_{j_{\sigma'(\ell)}'}^N)\\
&\qquad\qquad\qquad -\mathbb{E}_i^N K_\ell(\bar X_i^N,\bar X_{j_{\sigma(1)}}^N,\ldots, \bar X_{j_{\sigma(m)}}^N,\bar X_{j_{\sigma'(m+1)}'}^N,\ldots,\bar X_{j_{\sigma'(\ell)}'}^N)
\end{align*}
are independent conditioned to the sub-$\sigma$-algebra $\bar{\mathcal{F}}_{i,m}^N$. Hence, the conditional expectation $\mathbb{E}_{i,m}^N$ of their product is the product of their conditional expectations. The conditional expectations of those random variables separately take the value
\begin{align*}
&\mathbb{E}_{i,m}^N\Big[K_\ell(\bar X_i^N,\bar X_{j_{\sigma(1)}}^N,\ldots, \bar X_{j_{\sigma(m)}}^N,\bar X_{j_{\sigma(m+1)}}^N,\ldots,\bar X_{j_{\sigma(\ell)}}^N)\\
&\qquad\qquad\qquad -\mathbb{E}_i^N K_\ell(\bar X_i^N,\bar X_{j_{\sigma(1)}}^N,\ldots, \bar X_{j_{\sigma(m)}}^N,\bar X_{j_{\sigma(m+1)}}^N,\ldots,\bar X_{j_{\sigma(\ell)}}^N)\Big]\\
&\mathbb{E}_{i,m}^N K_\ell(\bar X_i^N,\bar X_{j_{\sigma(1)}}^N,\ldots, \bar X_{j_{\sigma(m)}}^N,\bar X_{j_{\sigma(m+1)}}^N,\ldots,\bar X_{j_{\sigma(\ell)}}^N)\\
&\qquad\qquad\qquad -\mathbb{E}_{i,m}^N\mathbb{E}_i^N K_\ell(\bar X_i^N,\bar X_{j_{\sigma(1)}}^N,\ldots, \bar X_{j_{\sigma(m)}}^N,\bar X_{j_{\sigma(m+1)}}^N,\ldots,\bar X_{j_{\sigma(\ell)}}^N)=0,
\end{align*}
where in the last step we have used that since $\bar{\mathcal{F}}_{i,m}^N\subset \bar{\mathcal{F}}_i^N$ then the law of iterated expactations yields $\mathbb{E}_{i,m}^N\mathbb{E}_i^N=\mathbb{E}_{i,m}^N$, and similarly for the second random variable.

\medskip

Altogether shows that if $\bj_\ell$ and $\bj_\ell'$ lie in {\sc Case 2}, the corresponding non-diagonal term in \eqref{eq:multi-agent-system-mckean-sde-error-estimate-step-R} necessarily vanishes. Therefore only multi-indices as in {\sc Case 1} leading to diagonal terms can persist. Since given $\bj_\ell$ there are $\ell!$ posibilities for $\bj_\ell'$ as in {\sc Case 1} (that is, all possible permutations of the $\ell$ components), then we obtain the following bound for $R_i^{\ell,N}$ in \eqref{eq:multi-agent-system-mckean-sde-error-estimate-step-R}:
$$
R_i^{\ell,N}\leq 2B_\ell \sqrt{\ell!}\left(\sum_{j_1,\ldots,j_\ell=1}^N (w^{\ell,N}_{ij_1\cdots j_\ell})^2\right)^{1/2}.
$$
Hence, we have the following control of $\mathbb{E}B_i^N$ in \eqref{eq:multi-agent-system-mckean-sde-error-estimate-step-B}:
$$\mathbb{E}B_i^N\leq 2\sum_{\ell=1}^{N-1}B_\ell \sqrt{\ell!}\left(\sum_{j_1,\ldots,j_\ell=1}^N (w^{\ell,N}_{ij_1\cdots j_\ell})^2\right)^{1/2}.$$
Taking $\ell_p$ norms with respect to $i$ yields
\begin{equation}\label{eq:multi-agent-system-mckean-sde-error-estimate-step-3}
\left(\frac{1}{N}\sum_{i=1}^N(\mathbb{E}B_i^N)^p\right)^{1/p}\leq \varepsilon_p^N,
\end{equation}
where $\varepsilon_p^N$ is given in \eqref{eq:multi-agent-system-mckean-sde-error-estimate-constant-epsilon}. Plugging \eqref{eq:multi-agent-system-mckean-sde-error-estimate-step-2} and \eqref{eq:multi-agent-system-mckean-sde-error-estimate-step-3} into \eqref{eq:multi-agent-system-mckean-sde-error-estimate-step-1} implies the integral inequality:
\begin{align*}
&\left(\frac{1}{N}\sum_{i=1}^N(\mathbb{E}|X_i^N(t)-\bar X_i^N(t)|)^p\right)^{1/p}\\
&\qquad\leq (\tilde C_\infty^N+C_p^N)\int_0^t \left(\frac{1}{N}\sum_{i=1}^N\mathbb{E}(|X_i^N(s)-\bar X_i^N(s)|)^p\right)^{1/p}\,ds +\varepsilon_p^Nt,
\end{align*}
for all $t\geq 0$. Thereby, since $X_i^N(0)=\bar X_i^N(0)$, Gr{\"o}nwall's lemma ends the proof.
\end{proof}

\begin{rem}[Stronger Error Estimate]
The error estimate \eqref{eq:multi-agent-system-mckean-sde-error-estimate} in Lemma \ref{lem:multi-agent-system-mckean-sde-error-estimate} can be regarded as a fibered $W_1$ estimate, which leads to our estimates in terms of the fibered $d_{\rm BL}$ metrics in Definition \ref{defi:fibered-probability-measures-Lpdbl-distance}. However, this error estimate could also be proven in the stronger form
\begin{equation*}
\left(\frac{1}{N}\sum_{i=1}^N\mathbb{E}\left(|X_i^N(t)-\bar X_i^N(t)|^p\right)\right)^{1/p}\leq e^{(\tilde C_\infty^N+C_p^N)t}\varepsilon_p^N,
\end{equation*}
under the additional assumption that $p\in [1,2]$. Note that this corresponds to a stronger fibered $W_p$ estimate since the power $p$ has come inside the expectation. The restriction on $p\in [1,2]$ stems from {\sc Step 2} of the proof, in order to interpolate the $L^p$ norm $(\mathbb{E}|B_i^N|^p)^{1/p}$ by the $L^2$ norm $(\mathbb{E}|B_i^N|^2)^{1/2}$ through Jensen's inequality. By doing so, the above would again be controlled by a squared expectation error term as we do in the proof of Lemma \ref{lem:multi-agent-system-mckean-sde-error-estimate}, and also as in the seminal work \cite{S-91}, where an analogous cancellation property is obtained thanks to the independence of the random variables $\bar X_i^N$. We remark that the power $p=2$ is not the only special case where this cancellation takes place, as noticed in \cite[Proposition 2.3]{JM-98}, where the authors also studied the case $p=4$. However, since the stronger form of the sability estimate is not needed for the upcoming results, we present its weaker form \eqref{eq:multi-agent-system-mckean-sde-error-estimate}, without any restriction on $p$.
\end{rem}
}

\begin{rem}[Propagation of independence]\label{rem:propagation-independence}
We observe that Lemma \ref{lem:multi-agent-system-mckean-sde-error-estimate} applies to general hypergraphs satisfying the symmetry condition \eqref{eq:hypothesis-weights-symmetry}. In order for the right-hand side of \eqref{eq:multi-agent-system-mckean-sde-error-estimate} to decay as $N\to \infty$ one needs to assume additionally the following (general) conditions
$$\sup_{N\in \mathbb{N}}(\tilde C_\infty^N+C_p^N)<\infty,\quad \lim_{N\to \infty}\varepsilon_p^N=0.$$

A particular case where the above two conditions are met is when the scaling assumption \eqref{eq:hypothesis-weights-uniform-bound} of weights and the scaling assuption \eqref{eq:hypothesis-kernels-scaling} of kernels hold true. Specifically, note that the former guarantees the following control of the above constants
$$\tilde C_\infty^N\leq W\sum_{\ell=1}^\infty L_\ell,\quad C_p^N\leq W\sum_{\ell=1}^\infty\ell L_\ell,\quad \varepsilon_p^N\leq 2W\sum_{\ell=1}^\infty \frac{\sqrt{\ell!}\,B_\ell}{N^{\ell/2}}.$$

For binary interactions, the first term $\ell=1$ in the above control of $\varepsilon_p^N$ coincides with the classical decay $N^{-1/2}$ in propagation of chaos \cite{S-91}. Interestingly, when multi-body interactions are considered, we find that the higher $\ell$, the quicker the decay of the corresponding term $N^{-\ell/2}$. When all possible multi-body interactions are superposed, the scaling assumption $\eqref{eq:hypothesis-kernels-scaling}_2$ guarangees the uniform-in-$N$ bound of $\tilde C_\infty^N$ and $C_p^N$, and also $\eqref{eq:hypothesis-kernels-scaling}_1$ implies 
$$\varepsilon_p^N\leq \left(2W\sum_{\ell=1}^\infty\frac{\sqrt{\ell!}\,B_\ell}{\eta^\ell}\right)\frac{\eta^{1/2}}{N^{1/2}},$$
for all $N\geq \eta$. Hence, we are able to recover the global decay $N^{-1/2}$, which coincides with the weakest of the decays among all possible multi-body interactions.
\end{rem}

In the following, we show that for independent initial data $X_{i,0}^N$, and therefore independent $\bar X_i^N$ solving \eqref{eq:multi-agent-system-mckean-sde}, their laws fulfill a closed system of coupled PDEs.

\begin{lem}[Coupled PDE system]\label{lem:coupled-pde-system}
Assume that the kernels $K_\ell$ verify the assumptions \eqref{eq:hypothesis-kernels-bounded} and \eqref{eq:hypothesis-kernels-lipschitz}, and the weights $(w^{\ell,N}_{ij_1\ldots j\ell})_{1\leq i,j_1,\ldots,j_\ell\leq N}$ with $\ell=1,\cdots,N-1$ satisfy \eqref{eq:hypothesis-absence-loops}. For any $(X_{1,0}^N,\ldots,X_{N,0}^N)$ with independent $X_{i,0}^N$ such that $\mathbb{E}|X_{i,0}^N|<\infty$, consider the unique solution $(\bar X_1^N,\ldots,\bar X_N^N)$ to \eqref{eq:multi-agent-system-mckean-sde} as in Lemma \ref{lem:multi-agent-system-mckean-sde-well-posedness}, and define their associated laws
\begin{equation}\label{eq:laws-of-bar-Xi}
\bar\lambda_t^{N,i}:={\rm Law}(\bar X_i^N(t)),\quad t\geq 0,\,1\leq i\leq N.
\end{equation}
Then, $(\bar \lambda^{N,i})_{1\leq i\leq N}$ is a solution in distributional sense to the following coupled PDE system
\begin{equation}\label{eq:coupled-pde-system}
\begin{aligned}
&\partial_t\bar\lambda_t^{N,i}+\divop_x(F_i^N[\bar\lambda_t^{N,1},\cdots,\bar\lambda_t^{N,N}]\,\bar\lambda_t^{N,i})=0,\quad t\geq 0,\,x\in \mathbb{R}^d,\,1\leq i\leq N,\\
&\bar \lambda_0^{N,i}={\rm Law}(X_{i,0}^N),
\end{aligned}
\end{equation}
where the forces $F_i^N$ take the form
\begin{equation}\label{eq:coupled-pde-system-force}
F_i^N[\bar\lambda_t^{N,1},\cdots,\bar\lambda_t^{N,N}](x)=\sum_{\ell=1}^{N-1}\sum_{j_1,\cdots,j_\ell=1}^N w^{\ell,N}_{ij_1\cdots j_\ell} \int_{\mathbb{R}^{d\ell}}K_\ell(x,x_1,\ldots,x_\ell)\,d\bar\lambda_t^{N,1}(x_1)\cdots \,d\bar\lambda_t^{N,N}(x_N).
\end{equation}
\end{lem}
\begin{proof}
Assume that $i\notin\{j_1,\ldots,j_\ell\}$ and note that all $\bar X_j^N(t)$ can be regarded as measurable functions $\bar X_j^N(t):\Omega_N\longrightarrow \mathbb{R}^d$ on a common probability space $(\Omega_N,\Sigma_N,\mathbb{P}_N)$. Then, by independence of $\bar X_i^N$ and $(\bar X_{j_1}^N,\ldots,\bar X_{j_\ell}^N)$ and the definition of conditional expectation we have
\begin{align*}
&\mathbb{E}_i^N K_\ell(\bar X_i^N,\bar X_{j_1}^N,\ldots,\bar X_{j_\ell}^N)(\omega)\\
&\qquad=\mathbb{E}\left[K_\ell(\bar X_i^N,\bar X_{j_1}^N,\ldots,\bar X_{j_\ell}^N)|\bar X_i^N=\bar X_i^N(\omega)\right]=\mathbb{E}K_\ell(\bar X_i^N(\omega),\bar X_{j_1}^N,\ldots,\bar X_{j_\ell}^N)\\
&\qquad=\int_{\mathbb{R}^{d\ell}} K_\ell(\bar X_i^N(\omega),x_1,\ldots,x_\ell)\,d\bar\lambda_t^{N,j_1}(x_1)\cdots\,d\bar\lambda_t^{N,j_\ell}(x_\ell),
\end{align*}
for all $\omega\in \Omega_N$, where in the last step we have used the ``law of the unconscious statistician''. Hence, we obtain the following identity
$$w^{\ell,N}_{ij_1\cdots j_\ell}\mathbb{E}_i^N K_\ell (\bar X_i^N,\bar X_{j_1}^N,\ldots, \bar X_{j_\ell}^N)=w^{\ell,N}_{ij_1\cdots j_\ell}\int_{\mathbb{R}^d} K_\ell(\bar X_i^N,x_1,\ldots,x_\ell)\,d\bar\lambda_t^{N,j_1}(x_1)\cdots\,d\bar\lambda_t^{N,j_\ell}(x_\ell),$$
for all $i,j_1,\ldots,j_\ell\in \llbracket1,N\rrbracket$ because we are assuming that $w^{\ell,N}_{ij_1\cdots j_\ell}=0$ for all $i\in \{j_1,\ldots,j_\ell\}$ by \eqref{eq:hypothesis-absence-loops}, and therefore the multi-agent system \eqref{eq:multi-agent-system-mckean-sde} can be reformulated in an alternative way closer to the formulation of the McKean process introduced in \cite{S-91} for exchangeable systems:
\begin{equation}\label{eq:multi-agent-system-mckean-sde-2}
\begin{aligned}
&\frac{d\bar X_i^N}{dt}=\sum_{\ell=1}^{N-1}\sum_{j_1,\ldots, j_\ell=1}^N w^{\ell,N}_{ij_1\cdots j_\ell}\int_{\mathbb{R}^d} K_\ell(\bar X_i^N,x_1,\ldots,x_\ell)\,d\bar\lambda_t^{N,j_1}(x_1)\cdots\,d\bar\lambda_t^{N,j_\ell}(x_\ell),\\
&\bar X_i^N(0)=X_{i,0}^N.
\end{aligned}
\end{equation}
Then, by a standard argument based on the law of unconcious statistician again we obtain that the laws $\bar\lambda_t^{N,i}$ must satisfy the coupled PDE system \eqref{eq:coupled-pde-system}.
\end{proof}

Note that according to Remark \ref{rem:propagation-independence}, Lemma \ref{lem:multi-agent-system-mckean-sde-error-estimate} allows approximating the multi-agent system \eqref{eq:multi-agent-system}, consisting of non-exchangeable and dependent random variables $X_i^N(t)$, by an intermediary multi-agent system \eqref{eq:multi-agent-system-mckean-sde}, consisting of still non-exchangeable but independent random variables $\bar X_i^N(t)$ under general enough conditions on the hypergraphs and the interaction kernels ({\it e.g.}, \eqref{eq:hypothesis-kernels-scaling} and \eqref{eq:hypothesis-weights-uniform-bound}). The independence can also be observed in Lemma \ref{lem:coupled-pde-system}, where the individual laws $\bar\lambda_t^{N,i}$ of each approximate agent $\bar X_i^N(t)$ solve a closed system of $N$ coupled PDEs \eqref{eq:coupled-pde-system}-\eqref{eq:coupled-pde-system-force}. Whilst the degree of complexity has been reduced enormously once correlations have been destroyed, this coupled system of PDEs is still highly complex since the number of PDEs grows with $N$ and one would like to have a simpler representation, in particular as $N\to \infty$. This was already explored in \cite{JPS-25} by using a graphon reformulation of \eqref{eq:coupled-pde-system}-\eqref{eq:coupled-pde-system-force}.

\begin{defi}[Graphon reformulation]\label{defi:graphon-reformulation}
For every $N\in \mathbb{N}$, and $t\in \mathbb{R}_+$ we define
$$
\mu^{N}_t\in \mathcal{P}_\nu(\mathbb{R}^d\times [0,1]),\quad \bar\mu^{N}_t\in  \mathcal{P}_\nu(\mathbb{R}^d\times [0,1]),\quad w^N=(w^N_\ell)_{\ell\in \mathbb{N}},
$$
where each element is given as follows
\begin{align}
\mu^{N,\xi}_t&:=\sum_{i=1}^N\mathds{1}_{I_i^N}(\xi)\delta_{X_i^N(t)},\quad \xi\in [0,1],\label{eq:graphon-reformulation-muN}\\
\bar\mu^{N,\xi}_t&:=\sum_{i=1}^{N}\mathds{1}_{I_i^N}(\xi)\,\bar \lambda^{N,i}_t,\quad \xi\in [0,1],\label{eq:graphon-reformulation-barmuN}\\
w_\ell^{N}(\xi,\xi_1,\ldots,\xi_\ell)&:=\sum_{i,j_1,\ldots,j_\ell=1}^N \mathds{1}_{I_i^N\times I_{j_1}^N\times\cdots\times I_{j_\ell}^N}(\xi,\xi_1,\ldots,\xi_\ell)\,N^\ell w^{\ell,N}_{ij_1\cdots j_\ell},\quad \xi,\xi_1,\cdots,\xi_\ell\in [0,1],\label{eq:graphon-reformulation-wN}
\end{align}
for all $1\leq \ell\leq N-1$, $w_\ell^N\equiv 0$ for all $\ell\geq N$, and $I_i^N=[\frac{i-1}{N},\frac{i}{N})$ for all $1\leq i\leq N$.
\end{defi}

On the one hand, note that $\bar \mu^{N}_t\in\mathcal{P}_\nu(\mathbb{R}^d\times [0,1])$, and also $\mu^{N}_t\in \mathcal{P}_\nu(\mathbb{R}^d\times [0,1])$ for each realization. On the other hand, $w_\ell^N\in L^\infty([0,1]^{\ell+1})$ for each $\ell\in \mathbb{N}$ and their $L^\infty$ norms are uniformly bounded with respect to $\ell\in \mathbb{N}$. If we further assume the additional symmetry hypothesis \eqref{eq:hypothesis-weights-full-symmetry} on the weights, then $w^N=(w_\ell^N)_{\ell\in \mathbb{N}}$ are UR-hypergraphons ({\it cf.} Definition \ref{defi:UR-hypergraphons}) for all $N\in \mathbb{N}$, though $L^\infty$ norms may blow up as $N\to\infty$ unless we additionally assume the fundamental hypothesis \eqref{eq:hypothesis-weights-uniform-bound} on the coupling weights, which will become important in Section \ref{sec:proof-main-result}. Then, we have the enough measurability properties to check whether the graphon reformulation $(\bar \mu^{N},w^N)$ associated with the intermediary multi-agent system \eqref{eq:multi-agent-system-mckean-sde} is a genuine distributional solution to the Vlasov-equation \eqref{eq:vlasov-equation-joint}-\eqref{eq:vlasov-equation-force-joint}.

\begin{lem}\label{lem:graphon-reformulation-solves-Vlasov}
Assume that the kernels $K_\ell$ verify the assumptions \eqref{eq:hypothesis-kernels-bounded} and \eqref{eq:hypothesis-kernels-lipschitz}, and the weights $(w^{\ell,N}_{ij_1\ldots j\ell})_{1\leq i,j_1,\ldots,j_\ell\leq N}$ with $\ell=1,\cdots,N-1$ satisfy \eqref{eq:hypothesis-absence-loops}. For any $(X_{1,0}^N,\ldots,X_{N,0}^N)$ with independent $X_{i,0}^N$ such that $\mathbb{E}|X_{i,0}^N|<\infty$, consider the unique solution $(\bar X_1^N,\ldots,\bar X_N^N)$ to \eqref{eq:multi-agent-system-mckean-sde} as in Lemma \ref{lem:multi-agent-system-mckean-sde-well-posedness}, their associated laws $(\bar \lambda^{N,i})_{1\leq i\leq N}$ as in \eqref{eq:laws-of-bar-Xi} and the graphon reformulation $(\bar \mu^{N},w^N)$ in Definition \ref{defi:graphon-reformulation}. Then, $\bar \mu^{N}$ is a distributional solution to the Vlasov equation \eqref{eq:vlasov-equation-joint}-\eqref{eq:vlasov-equation-force-joint} with hypergraphon $w^N=(w_\ell^N)_{\ell\in \mathbb{N}}$ and initial datum $\bar \mu_{t=0}^{N,\xi}=\sum_{i=1}^N\mathds{1}_{I_i^N}(\xi)\,{\rm Law}(X_{i,0}).$
\end{lem}

\begin{proof}
Consider any $\xi\in [0,1]$, fix $i\in \llbracket1,N\rrbracket$ such that $\xi\in I_i^N$ and note that $\bar\mu^{N,\xi}_t=\bar\lambda_t^{N,i}$. Then, anything we have to prove is that the following identity holds true
$$F_i^N[\bar \lambda_t^{N,1},\ldots,\bar\lambda_t^{N,N}]=F_{w^N}[\bar\mu_t^N](\cdot,\xi).$$
Indeed, bearing in mind the scaling $N^\ell$ present in the definition of piecewise constant functions $w_\ell^N$ and the fact that $|I_{j_1}\times \cdots\times I_{j_\ell}|=N^\ell$, then one recognises that the discrete sums below correspond to integrals of piecewise constant functions and therefore
\begin{align*}
&F_i^N[\bar \lambda_t^{N,1},\cdots,\bar \lambda_t^{N,N}](x)\\
&\quad=\sum_{\ell=1}^{N-1}\sum_{j_1,\cdots,j_\ell=1}^N w^{\ell,N}_{ij_1\cdots j_\ell} \int_{\mathbb{R}^{d\ell}}K_\ell(x,x_1,\ldots,x_\ell)\,d\bar \lambda_t^{N,1}(x_1)\cdots \,d\bar\lambda_t^{N,N}(x_N)\\
&\quad=\sum_{\ell=1}^{N-1}\frac{1}{N^\ell}\sum_{j_1,\cdots,j_\ell=1}^N N^\ell w^{\ell,N}_{ij_1\cdots j_\ell} \int_{\mathbb{R}^{d\ell}}K_\ell(x,x_1,\ldots,x_\ell)\,d\bar \lambda_t^{N,1}(x_1)\cdots \,d\bar\lambda_t^{N,N}(x_N)\\
&\quad=\sum_{\ell=1}^{N-1}\int_{[0,1]^{\ell}} w_\ell^N(\xi,\xi_1,\cdots,\xi_\ell) \int_{\mathbb{R}^{d\ell}}K_\ell(x,x_1,\ldots,x_\ell)\,d\bar \mu_t^{N,\xi_1}(x_1)\cdots \,d\bar\mu_t^{N,\xi_N}(x_N)\,d\xi_1\cdots\,d\xi_\ell,
\end{align*}
and then one finds $F_{w^N}[\bar\mu_t^N](x,\xi)$ in the right-hand side, thus ending the proof.
\end{proof}

Whilst the precise notion of distributional solution to \eqref{eq:vlasov-equation-joint}-\eqref{eq:vlasov-equation-force-joint} and also \eqref{eq:coupled-pde-system}-\eqref{eq:coupled-pde-system-force} is standard, in Section \ref{sec:well-posedness} we formulate the definition more precisely and we show that the Vlasov equation admits unique solutions living in appropriate measure-valued spaces.

\section{Well-posedness of the Vlasov equation over UR-hypergraphons}\label{sec:well-posedness}

In this section we study the well-posedness of the Vlasov equation \eqref{eq:vlasov-equation-joint}-\eqref{eq:vlasov-equation-force-joint}. We will not restrict to a well-posedness theory of \eqref{eq:vlasov-equation-joint}-\eqref{eq:vlasov-equation-force-joint} over UR-hypergraphons $w=(w_\ell)_{\ell\in \mathbb{N}}$ only, but actually the theory will be aplicable to a broader class of $w$ where $L^\infty$ norms are not available, thus opening the way to a sparse setting.

\subsection{Properties of the force and well-posedness of characteristics}\label{subsec:properties-force-wellposedness-characteristics}

In this section we introduce some properties of the force $F_w[\mu_t]$, which will prove useful to solve the kinetic equation \eqref{eq:vlasov-equation-joint}-\eqref{eq:vlasov-equation-force-joint} by the method of characteristics based on the Cauchy-Lipschitz theorem.

\begin{pro}[Properties of the force I]\label{pro:properties-force-I}
Assume that the kernels $K_\ell$ satisfy the assumptions \eqref{eq:hypothesis-kernels-bounded} and \eqref{eq:hypothesis-kernels-lipschitz}, consider any Borel family of probability measures $(\mu_t)_{t\in \mathbb{R}_+}\subset \mathcal{P}_\nu(\mathbb{R}^d\times [0,1])$, any $w=(w_\ell)_{\ell\in \mathbb{N}}$, and suppose that
\begin{equation}\label{eq:bounded-lipschitz-force-constants}
B_F:=\left\Vert \sum_{\ell=1}^\infty B_\ell \Vert w_\ell\Vert_{L^1_{\bxi_\ell}}\right\Vert_{L^\infty_\xi}<\infty,\quad L_F:=\left\Vert\sum_{\ell=1}^\infty L_\ell\Vert w_\ell\Vert_{L^1_{\bxi_\ell}}\right\Vert_{L^\infty_\xi}<\infty.
\end{equation}
Then, the force $F_w[\mu_t](x,\xi)$ in \eqref{eq:vlasov-equation-force} is well-defined for all $t\in \mathbb{R}_+$, all $x\in \mathbb{R}^d$ and a.e. $\xi\in [0,1]$, and it satisfies the following three properties:
\begin{enumerate}[label=(\roman*)]
    \item (Boundedness) 
    $$|F_w[\mu_t](x,\xi)|\leq B_F,$$
    for all $t\in\mathbb{R}_+$, all $x\in \mathbb{R}^d$ and a.e. $\xi\in [0,1]$.
    \item (Lipschitz-continuity with respect to $x$)
    $$|F_w[\mu_t](x,\xi)-F_w[\mu_t](\tilde x,\xi)|\leq L_F|x-\tilde x|,$$
    for all $t\in\mathbb{R}_+$, all $x,\tilde x\in \mathbb{R}^d$ and a.e. $\xi\in [0,1]$.
    \item (Measurability)
    $$(t,\xi)\in \mathbb{R}_+\times [0,1]\mapsto F_w[\mu_t](x,\xi),$$
    is Borel-measurable for all $x\in \mathbb{R}^d$.
\end{enumerate}
\end{pro}

\begin{proof}
We first show the boundedness property, which in particular ensures that $F_w[\mu_t](x,\xi)$ is well defined for all $t\in \mathbb{R}_+$, all $x\in \mathbb{R}^d$ and a.e. $\xi\in [0,1]$. Specifically, take any $t\in \mathbb{R}_+$, any $x\in \mathbb{R}^d$ and any $\xi$ for which all $w_\ell(\xi,\cdot)$ are defined (therefore a.e. $\xi\in [0,1]$), and note that
\begin{align*}
|F_w[\mu_t](x,\xi)|&\leq \sum_{\ell=1}^\infty\int_{[0,1]^\ell}|w_\ell(\xi,\xi_1,\ldots,\xi_\ell)|\left(\int_{\mathbb{R}^d}B_\ell\,d\mu_t^{\xi_1}(x_1)\cdots\,d\mu_t^{\xi_\ell}(x_\ell)\right)\,d\xi_1\ldots\,d\xi_\ell\\
&\qquad =\sum_{\ell=1}^\infty B_\ell\int_{[0,1]^\ell}|w_\ell(\xi,\xi_1,\ldots,\xi_\ell)|\,d\xi_1\ldots,\,d\xi_\ell\leq B_F,
\end{align*}
where in the first step we have used the boundedness assumption \eqref{eq:hypothesis-kernels-bounded} of the interaction kernels $K_\ell$, in the second step we have use that $\mu^{\xi_1}_t,\ldots,\mu^{\xi_\ell}_t\in \mathcal{P}(\mathbb{R}^d)$, and in the last step we have taken essential supremum with respect to $\xi\in [0,1]$.

Similarly, for all $t\geq 0$, $x,\tilde x\in \mathbb{R}^d$ and a.e. $\xi\in [0,1]$, the Lipschitz-continuity \eqref{eq:hypothesis-kernels-lipschitz} implies
\begin{align*}
&|F_w[\mu_t](x,\xi)-F_w[\mu_t](\tilde x,\xi)|\\
&\qquad\leq \sum_{\ell=1}^\infty \int_{[0,1]^\ell}|w_\ell(\xi,\xi_1,\ldots,\xi_\ell)| \left(\int_{\mathbb{R}^{d\ell}}L_\ell |x-\tilde x|\,d\mu_t^{\xi_1}(x_1)\cdots \,d\mu_t^{\xi_\ell}(x_\ell)\right)\,d\xi_1\ldots\,d\xi_\ell\\
&\qquad=\left(\sum_{\ell=1}^\infty L_\ell\int_{[0,1]^\ell}|w_\ell(\xi,\xi_1,\ldots,\xi_\ell)| \,d\xi_1\ldots\,d\xi_\ell\right)|x-\tilde x| \leq L_F |x-\tilde x|,
\end{align*}

Finally, since $(\mu_t)_{t\in \mathbb{R}_+}$ is a Borel family, then so is $(\mu_t^{\otimes^{\ell}})_{t\in \mathbb{R}_+}$ (we give a proof below). Note that each $w_\ell(\xi,\cdot)\,K_\ell(x,\cdot)$ is bounded and Borel measurable. Then each
$$(t,\xi)\in\mathbb{R}_+\times [0,1]\mapsto \int_{(\mathbb{R}^{d}\times [0,1])^\ell} w_\ell(\xi,\cdot)\,K_\ell(x,\cdot)\,d\mu_t^{\otimes^{\ell}},$$
is Borel-measurable for all $\ell\in \mathbb{N}$, then so is $(t,\xi)\in \mathbb{R}_+\times [0,1]\mapsto F_w[\mu_t](x,\xi)$ since it can be obtained as the sum over $\ell$ of all the above Borel-measurable functions.

To conclude, we recall why the tensor product of Borel families of probability measures is a new Borel family, which has been used above, and it is based on a standard application of Dynkin's $\pi-\lambda$ theorem \cite[Theorem 3.2]{B-95}. For any $\ell\in \mathbb{N}$, we define the sets
\begin{align*}
\mathcal{P}_\ell&:=\{B_1\otimes \cdots\otimes B_\ell\subset (\mathbb{R}^d)^\ell:\,B_\ell\subset \mathbb{R}^d\mbox{ is a Borel set}\},\\
\Lambda_\ell&:=\{E_\ell\subset (\mathbb{R}^d)^\ell \mbox{ Borel sets}: t\in \mathbb{R}_+\to \mu_t^{\otimes^{\ell}}(E_\ell)\mbox{ is Borel-measurable}\}.
\end{align*}
On the one hand, it is clear that $\mathcal{P}_\ell\subset \Lambda_\ell$ because for any collection of Borel sets $B_1,\ldots,B_\ell\subset\mathbb{R}^d$ we have that $t\in \mathbb{R}_+\mapsto \mu_t^{\otimes^{\ell}}(B_1\otimes\cdots\otimes B_\ell)=\mu_t(B_1)\cdots\mu_t(B_\ell)$ is Borel-measurable (as it is the product of $\ell$ measurable functions by hypothesis). Additionally, it is clear that $\mathcal{P}_\ell$ is a $\pi$-set, {\it i.e.}, it is closed under finite intersections, and $\Lambda_\ell$ is a $\lambda$-system. Specifically, we have that $(\mathbb{R}^d)^{\ell}\in \Lambda_\ell$, if $E_\ell^1\subset E_\ell^2\in\Lambda_\ell$ then $E_\ell^2\setminus E_\ell^1\in \Lambda_\ell$, and also if $\{E_\ell^n\}_{n\in \mathbb{N}}\subset \Lambda_\ell$ is an increasing sequence, then $\cup_{n\in \mathbb{N}}E_\ell^n\in \Lambda_\ell$. Therefore, the $\pi-\lambda$ theorem ensures that $\sigma(\mathcal{P}_\ell)\subseteq \Lambda_\ell$. Since $\sigma(\mathcal{P}_\ell)$ is the Borel $\sigma$-algebra of $(\mathbb{R}^d)^{\ell}$, we conclude by definition of $\Lambda_\ell$.
\end{proof}

\begin{pro}[Well-posed characteristics]\label{pro:well-posed-characteristics}
Assume that the kernels $K_\ell$ satisfy the assumptions \eqref{eq:hypothesis-kernels-bounded} and \eqref{eq:hypothesis-kernels-lipschitz}, consider any Borel family of probability measures $(\mu_t)_{t\in \mathbb{R}_+}\subset \mathcal{P}_\nu(\mathbb{R}^d\times [0,1])$, any $w=(w_\ell)_{\ell\in \mathbb{N}}$, and suppose that \eqref{eq:bounded-lipschitz-force-constants} holds. Then, there is a unique global-in-time Caratheodory solution $X_w[\mu](t,x,\xi)$ to the characteristic system
\begin{equation}\label{eq:characteristic-system}
\begin{aligned}
\frac{d}{dt}X_w[\mu](t,x,\xi)&=F_w[\mu_t](X_w[\mu](t,x,\xi),\xi),\quad t\geq 0,\\
X_w[\mu](0,x,\xi)&=x,
\end{aligned}
\end{equation}
for all $x\in \mathbb{R^d}$ and a.e. $\xi\in [0,1]$. For simplicity of notation, let us denote
\begin{equation}\label{eq:characteristics-flow-map}
\mathcal{T}_t^\xi[w,\mu](x):=X_w[\mu](t,x,\xi),\quad t\geq 0,\,x\in \mathbb{R}^d,\,\mbox{a.e. }\xi\in [0,1].
\end{equation}
Then, the following properties hold true:
\begin{enumerate}[label=(\roman*)]
    \item For every $(t,x)\in \mathbb{R}_+\times \mathbb{R}^d$, the map $\xi\in [0,1]\mapsto \mathcal{T}_t^\xi[w,\mu](x)$ is Borel-measurable.
    \item For a.e. $\xi\in [0,1]$, the map $(t,x)\in \mathbb{R}_+\times \mathbb{R}^d\mapsto \mathcal{T}_t^\xi[w,\mu](x)$ is continuous.
    \item For every $t\in \mathbb{R}_+$ and a.e. $\xi\in [0,1]$, the map $x\in \mathbb{R}^d\mapsto \mathcal{T}_t^\xi[w,\mu](x)$ is Lipschitz-continuous with Lipschitz constant $e^{L_Ft}$.
\end{enumerate}
\end{pro}

\begin{proof}
By Proposition \ref{pro:properties-force-I} we have that for a.e. $\xi\in [0,1]$, the map $$(t,x)\in \mathbb{R}_+\times \mathbb{R}^d\mapsto F_w[\mu_t](x,\xi),$$
satisfies the Caratheodory conditions. Specifically:
\begin{enumerate}[label=(\alph*)]
\item For every $t\geq 0$, the map $x\in \mathbb{R}^d\mapsto F_w[\mu](x,\xi)$ is Lipschitz-continuous.
\item For every $x\in \mathbb{R}^d$, the map $t\in \mathbb{R}_+\mapsto F_w[\mu](x,\xi)$ is Borel-measurable.
\item There exists some $m\in L^1_{loc}(\mathbb{R}_+)$ so that $|F_w[\mu_t](x,\xi)|\leq m(t)$ for all $t\geq 0$ and all $x\in \mathbb{R}^d$ ({\it e.g.}, take $m(t)=B_F$ for all $t\geq 0$ with $B_F$ in \eqref{eq:bounded-lipschitz-force-constants}).
\end{enumerate}
Therefore, by Caratheodory's existence theorem, for every initial datum $x\in \mathbb{R}^d$ there is a unique absolutely continuous trajectory $t\in \mathbb{R}_+\mapsto X_w[\mu](t,x)$ solving \eqref{eq:characteristic-system} in the sense of Caratheodory, that is, the differential equations holds for a.e. $t\geq 0$. 

Since $F_w$ is uniformly bounded by Proposition \ref{pro:properties-force-I}, and therefore it has sublinear growth, the trajectories are indeed defined globally in time. Additionally, the continuity in $(t,x)\in \mathbb{R}_+\times \mathbb{R}^d$ in item (ii) is clear by the Caratheodory theory, and the Borel-measurability with respect to $\xi\in [0,1]$ in item (i) is also true by the measurability of $F_w[\mu]$ in Proposition \ref{pro:properties-force-I}. Regarding the Lipschitz-continuity with respect to $x\in \mathbb{R}^d$ in item (iii) note that
\begin{align*}
&|X_w[\mu](t,x,\xi) - X_w[\mu](t,\tilde x,\xi)]|\\
&\qquad= \left|\int_0^t  \left(F_w[\mu_s](X_w[\mu](s,x,\xi),\xi)-F_w[\mu_s](X[\mu](s,\tilde x,\xi),\xi)\right) ds + (x - \tilde x) \right|\\
&\qquad\leq \int_0^t  \left|F_w[\mu_s](X[\mu](s,x,\xi),\xi)-F_w[\mu_s](X_w[\mu](s,\tilde x,\xi),\xi)\right| ds + |x - \tilde x| \\
&\qquad\leq \int_0^t  L_F \left|X_w[\mu](s,x,\xi)-X_w[\mu](s,\tilde x,\xi)\right| ds + |x - \tilde x|,
\end{align*}
for all $t\in \mathbb{R}_+$, all $x,\tilde x\in \mathbb{R}^d$ and a.e. $\xi\in [0,1]$, where in the last step we have used the Lipschitz continuity of $F_w[\mu](x,\xi)$ with respect to $x\in \mathbb{R}^d$ in item (ii) of Proposition \ref{pro:properties-force-I}. Thus, by Gr{\" o}nwall's lemma, we obtain 
\begin{equation*}
|\mathcal{T}_t^\xi[w,\mu](x) - \mathcal{T}_t^\xi[w,\mu](\tilde x)| \leq e^{L_Ft}|x - \tilde x|,  
\end{equation*}
which implies that $\mathcal{T}_t^\xi[w,\mu]$ is Lipschitz-continuous for all $t\in \mathbb{R}_+$ and a.e. $\xi\in [0,1]$, and we further deduce that following control on the Lipschitz constant
\begin{equation*}
[\mathcal{T}_t^\xi[w,\mu]]_{\text{Lip}} \leq e^{L_Ft}.  
\end{equation*}
\end{proof}

\begin{rem}
We remark that the scaling conditions on the kernels $K_\ell$
\begin{equation}\label{eq:hypothesis-kernels-scaling-II}
\sum_{\ell=1}^\infty B_\ell<\infty,\quad \sum_{\ell=1}^\infty L_\ell<\infty,
\end{equation}
which is weaker than \eqref{eq:hypothesis-kernels-scaling}, together with the scaling condition 
$$\sup_{\ell\in \mathbb{N}}\,\Vert w_\ell\Vert_{L^\infty}\leq W,$$
({\it e.g.}, $w\in \mathcal{H}_W$ by Definition \ref{defi:UR-hypergraphons}) ensure the above hypothesis \eqref{eq:bounded-lipschitz-force-constants} since
$$B_F\leq W\sum_{\ell=1}^\infty B_\ell,\quad L_F\leq W\sum_{\ell=1}^\infty L_\ell.$$
\end{rem}

\subsection{Notion of distributional solution}\label{subsec:notion-distributional-solution}

\begin{defi}[Distributional solutions of \eqref{eq:vlasov-equation-joint}-\eqref{eq:vlasov-equation-force-joint}]\label{defi:distributional-solution-Vlasov-joint}
Consider any curve of probability measures $\mu\in C([0,T],\mathcal{P}_\nu(\mathbb{R}^d\times [0,1])\mbox{-narrow})$ for some $T>0$, and any $w=(w_\ell)_{\ell\in \mathbb{N}}$. We say that $\mu$ is a distributional solution of the system \eqref{eq:vlasov-equation-joint}-\eqref{eq:vlasov-equation-force-joint} if
\begin{align}
&\int_0^T\iint_{\mathbb{R}^d\times [0,1]}\left(\partial_t\psi(t,x,\xi)+F_w[\mu_t](t,x,\xi)\cdot\nabla_x\psi(t,x,\xi)\right)\,d\mu_t(x,\xi)\,dt\nonumber\\
&\qquad=-\iint_{\mathbb{R}^d\times [0,1]}\psi(0,x,\xi)\,d\mu_0(x,\xi),\label{eq:distributional-solution-joint}
\end{align}
for all $\psi\in C^1_c([0,T)\times \mathbb{R}^d\times [0,1])$.
\end{defi}

We note that, under the assumptions \eqref{eq:bounded-lipschitz-force-constants}, the force $F_w[\mu]$ is bounded and Borel-measurable jointly in $(t,x,\xi)$ by Proposition \ref{pro:properties-force-I} because it is Borel-measurable in $(t,\xi)$ and continuous in the variable $x$. Therefore, so is the integrand $\partial_t\psi+F[\mu]\cdot\nabla_x\psi$ in the above weak formulation in Definition \ref{defi:distributional-solution-Vlasov-joint}, and therefore distributional solutions are well defined. There is an additional subtlety coming from the fact that $F_w[\mu]$ is defined for a.e. $\xi\in [0,1]$, which could seem problematic at first glance when integrated agains $\mu_t$ if different representatives take different values over the atoms of $\nu$. Nevertheless, in this paper we restrict to the Lebesgue measure $\nu=d\xi_{\lfloor [0,1]}$ as justified in Remark \ref{rem:reference-nu}, which eliminates this issue.

Also note that in Definition \ref{defi:distributional-solution-Vlasov-joint} we may have weakened the time-continuity to simply Borel-measurability of the family $(\mu_t)_{t\in [0,T]}\subset \mathcal{P}_\nu(\mathbb{R}^d\times [0,1])$, and the above notion of distributional solutions would still be well defined. However, it is well known that whenever a Borel family $(\mu_t)_{t\in [0,T]}\subset \mathcal{P}_\nu(\mathbb{R}^d\times [0,1])$ solves a continuity equation in distributional sense and
\begin{equation}\label{eq:ambrosio-integrability-velocity-field}
\int_0^T\iint_{\mathbb{R}^d\times [0,1]}|F_w[\mu_t](x,\xi)|\,d\mu_t(x,\xi)\,dt<\infty,
\end{equation}
which in our case holds true by Proposition \ref{pro:properties-force-I}, then there must exist a time representative which is narrowly continuous, see \cite[Theorem 8.2.1]{AGS-08}. For this reason we restrict to solutions living in $C([0,T],\mathcal{P}_\nu(\mathbb{R}^d\times [0,1]))$ without loss of generality.

\begin{pro}[Distributional solutions of \eqref{eq:vlasov-equation-joint}-\eqref{eq:vlasov-equation-force-joint}]\label{pro:distributional-solution-Vlasov-fibered}
Consider any curve of probability measures $\mu\in C([0,T],\mathcal{P}_\nu(\mathbb{R}^d\times [0,1])\mbox{-narrow})$, and its associated Borel family $(\mu_t^\xi)_{(t,\xi)\in [0,T]\times [0,1]}$ as in the disintegration Theorem \ref{theo:disintegration-time-dependent}. Then, the following conditions are equivalent:
\begin{enumerate}[label=(\roman*)]
\item $\mu$ is a distributional solution to \eqref{eq:vlasov-equation-joint}-\eqref{eq:vlasov-equation-force-joint}.
\item For a.e. $\xi\in [0,1]$, $(\mu_t^\xi)_{t\in [0,T]}$ is a distributional solution to \eqref{eq:vlasov-equation}-\eqref{eq:vlasov-equation-force}, {\it i.e.},
\begin{equation}\label{eq:distributional-solution-fibered}
\int_0^T\int_{\mathbb{R}^d}\left(\partial_t\varphi(t,x)+F_w[\mu_t](t,x,\xi)\cdot\nabla_x\varphi(t,x)\right)\,d\mu_t^\xi(x)\,dt=-\int_{\mathbb{R}^d}\varphi(0,x)\,d\mu_0^\xi(x),
\end{equation}
for all $\varphi\in C^1_c([0,T)\times \mathbb{R}^d)$.
\item For a.e. $\xi\in [0,1]$, $(\mu_t^\xi)_{t\in [0,T]}$ is the push forward along the flow map, {\it i.e.},
\begin{equation}\label{eq:distributional-solution-push-forward}
\mu_t^\xi=\mathcal{T}_t^\xi[w,\mu]_{\#}\mu_0^\xi,\quad \mbox{a.e. }t\geq 0,
\end{equation}
where $\mathcal{T}_t^\xi$ is given in \eqref{eq:characteristics-flow-map}.
\end{enumerate}

\begin{proof}
The proof of the equivalence between (i) and (ii) follows from a density argument identical to the one used in the proof of the disintegration Theorem \ref{theo:disintegration-time-dependent} applied to the integrals in \eqref{eq:distributional-solution-joint}, and then we omit it. We then focus on the proof of the equivalence between (ii) and (iii).

Assume that (ii) holds and let us fix a.e. $\xi\in [0,1]$ so that the Borel family $(\mu_t^\xi)_{\in [0,T]}$, the force $F_w[\mu_t](x,\xi)$ in \eqref{eq:vlasov-equation-joint}, and its flow map $\mathcal{T}_t^\xi[w,\mu](x)$ in \eqref{eq:characteristics-flow-map} are all defined. As discussed below Theorem \ref{theo:disintegration-time-dependent}, we do not expect $\mu^\xi\in C([0,T],\mathcal{P}_\nu(\mathbb{R}^d))$. However, note that
$$\int_0^T\int_{\mathbb{R}^d}|F_w[\mu_t](x,\xi)|\,d\mu_t^\xi(x)\,dt<\infty,$$
and therefore \cite[Lemma 8.1.2]{AGS-08} again ensures the existence of a narrowly continuous representative of the Borel family $(\mu_t^\xi)_{t\in [0,T]}$, namely $\tilde\mu^\xi\in C([0,T],\mathcal{P}(\mathbb{R}^d)\mbox{-narrow})$ such that $\mu_t^\xi=\tilde\mu_t^\xi$ for a.e. $t\in [0,T]$. Since the characteristic system associated to $F_w[\mu](x,\xi)$ for such a $\xi$ has unique solutions for all $x\in \mathbb{R}^d$, then a standard argument shows that
$$\tilde\mu_t=\mathcal{T}_t^\xi[w,\mu]_{\#}\tilde\mu_0^\xi,\quad \forall\,t\geq 0,$$
see \cite[Proposition 8.1.8]{AGS-08}. Note that $\tilde\mu^\xi$ clearly solves \eqref{eq:distributional-solution-fibered} for the same initial datum $\mu_0^\xi$. Then, $\tilde\mu_0^\xi=\mu_0^\xi$ and therefore \eqref{eq:distributional-solution-push-forward} holds true for a.e. $t\geq 0$.

Conversely, assume that (iii) holds true and take any $\varphi\in C^1_c([0,T)\times \mathbb{R}^d)$. Then, using \eqref{eq:distributional-solution-push-forward} on the left-hand side of \eqref{eq:distributional-solution-fibered} yields
\begin{align*}
&\int_0^T\int_{\mathbb{R}^d}\left(\partial_t\varphi(t,x)+F_w[\mu_t](t,x,\xi)\cdot\nabla_x\varphi(t,x)\right)\,d\mu_t^\xi(x)\,dt\\
&\qquad =\int_0^T\int_{\mathbb{R}^d}\left(\partial_t\varphi(t,\mathcal{T}_t^\xi[w,\mu](x))+F_w[\mu_t](t,\mathcal{T}_t^\xi[w,\mu](x),\xi)\cdot\nabla_x\varphi(t,\mathcal{T}_t^\xi[w,\mu](x))\right)\,d\mu_0^\xi(x)\,dt\\
&\qquad=\int_0^T\int_{\mathbb{R}^d}\frac{d}{dt}\varphi(t,\mathcal{T}_t^\xi[w,\mu](x))\,d\mu_0^\xi(x)\,dt=\int_0^T\frac{d}{dt}\int_{\mathbb{R}^d}\varphi(t,\mathcal{T}_t^\xi[w,\mu](x))\,d\mu_0^\xi(x)\,dt\\
&\qquad=-\int_{\mathbb{R}^d}\varphi(0,x)\,d\mu_0^\xi(x),
\end{align*}
where we have used the dominated convergence theorem, the fundamental theorem of calculus, and the fact that $\varphi$ vanishes at $t=T$.
\end{proof}
\end{pro}

\subsection{Fixed point argument}\label{subsec:fixed-point-argument}

We show the well-posedness of the system \eqref{eq:vlasov-equation-joint}-\eqref{eq:vlasov-equation-force-joint} (or equivalently, \eqref{eq:vlasov-equation}-\eqref{eq:vlasov-equation-force}) through a fixed point. In view of Definition \ref{defi:distributional-solution-Vlasov-joint} and Proposition \ref{pro:distributional-solution-Vlasov-fibered}, solutions could be built in the more general space \corr{$C([0,T],\mathcal{P}_\nu(\mathbb{R}^d\times [0,1])\mbox{-narrow})$}. However, in Section \ref{sec:stability-estimate} we will restrict to solutions living in either of the subspaces \corr{$C([0,T],\mathcal{P}_\nu(\mathbb{R}^d\times [0,1])\mbox{-}d_{p,\nu})$}. For this reason, we decided to restrict to a well-posedness theory in the later subspaces. Our main concern is that a quantitative stability estimate similar to the one derived in Section \ref{sec:stability-estimate}, but operating over solutions lying in the former larger space, seems to require additional Lipschitz-continuity assumptions $w_\ell$, which we do no want. We remark that this additional assumption was already present in previous literature for binary interactions \cite{CM-19,PT-22-arxiv}, as it stems from the classical Dobrushin stability estimate \cite{Do-79}.

In the fixed point argument, we will employ the following additional Lipschitz-continuity property of the force $F_w[\mu]$ in the variable $\mu$ with respect to the $d_{p,\nu}$ distance.

\begin{pro}[Properties of the force II]\label{pro:properties-force-II}
Assume that the kernels $K_\ell$ satisfy the assumptions \eqref{eq:hypothesis-kernels-bounded} and \eqref{eq:hypothesis-kernels-lipschitz}, consider any Borel family of probability measures $(\mu_t)_{t\in \mathbb{R}_+}\subset \mathcal{P}_\nu(\mathbb{R}^d\times [0,1])$, any $w=(w_\ell)_{\ell\in \mathbb{N}}$, and suppose that
\begin{equation}\label{eq:well-posedness-constant-Cp}
C_p:=\left\Vert\sum_{\ell=1}^\infty BL_\ell\sum_{k=1}^\ell \Vert w_\ell \Vert_{L^q_{\xi_k}L^1_{\hbxi_{\ell,k}}}\right\Vert_{L^p_\xi}<\infty,
\end{equation}
for some $p\in [1,\infty]$ and $q\in [1,\infty]$ so that $\frac{1}{p}+\frac{1}{q}=1$. Then, 
$$|F_w[\mu_t](x,\xi)-F_w[\bar\mu_t](x,\xi)|\leq\left(\sum_{\ell=1}^\infty BL_\ell\sum_{k=1}^\ell \Vert w_\ell(\xi,\cdot)\Vert_{L^q_{\xi_k}L^1_{\hbxi_{\ell,k}}}\right)d_{p,\nu}(\mu_t,\bar\mu_t),$$
for all $t\in\mathbb{R}_+$, all $x\in \mathbb{R}^d$, a.e. $\xi\in [0,1]$.
\end{pro}

\begin{proof}
Note that condition \eqref{eq:well-posedness-constant-Cp} in particular implies that the above coefficient
$$\sum_{\ell=1}^\infty BL_\ell\sum_{k=1}^\ell \Vert w_\ell(\xi,\cdot)\Vert_{L^q_{\xi_k}L^1_{\hbxi_{\ell,k}}},$$
is finite for a.e. $\xi\in [0,1]$. Then, we have
\begin{align*}
&|F_w[\mu_t](x,\xi)-F_w[\bar\mu_t](x,\xi)|\\
&\qquad=\sum_{\ell=1}^\infty\sum_{k=1}^\ell\int_{[0,1]^\ell}|w_\ell(\xi,\bxi_\ell)|\\
&\qquad\qquad\times\int_{\mathbb{R}^{d(\ell-1)}}\left\vert\int_{\mathbb{R}^d}K_\ell(x,\bx_\ell)\,(d\mu_t^{\xi_k}(x_k)-d\bar\mu_t^{\xi_k}(x_k))\right\vert\,\prod_{j<k}\,d\mu_t^{\xi_j}(x_j)\,\prod_{j>k}d\bar\mu_t^{\xi_j}(x_j)\,d\bxi_\ell\\
&\qquad \leq \sum_{\ell=1}^\infty \sum_{k=1}^\ell\int_{[0,1]^\ell}|w_\ell(\xi,\bxi_\ell)|\int_{\mathbb{R}^{d(\ell-1)}} BL_\ell\,d_{\rm BL}(\mu_t^{\xi_k},\bar\mu_t^{\xi_k}) \,\prod_{j<k}\,d\mu_t^{\xi_j}(x_j)\,\prod_{j>k}d\bar\mu_t^{\xi_j}(x_j)\,d\bxi_\ell\\
&\qquad= \sum_{\ell=1}^\infty BL_\ell\sum_{k=1}^\ell\int_0^1\Vert w_\ell(\xi,\cdot)\Vert_{L^1_{\hbxi_{\ell,k}}}\,d_{\rm BL}(\mu_t^{\xi_k},\bar\mu_t^{\xi_k})\,d\xi_k\\
&\qquad\leq\left(\sum_{\ell=1}^\infty BL_\ell\sum_{k=1}^\ell \Vert w_\ell(\xi,\cdot)\Vert_{L^q_{\xi_k}L^1_{\hbxi_{\ell,k}}}\right)d_{p,\nu}(\mu_t,\bar\mu_t),
\end{align*}
for all $t\in [0,T]$ and all $x\in \mathbb{R}^d$, where we recall that $BL_\ell=\max\{B_\ell,L_\ell\}$ is the bounded-Lipschitz kernel of each kernel $K_\ell$ as in \eqref{eq:kernels-bounded-lipschitz-constant}, and where in the last step we have used H{\" o}lder's inequality in the integral with respect to $\xi_k$ for the exponents $p,q\in (1,\infty)$ with $\frac{1}{p}+\frac{1}{q}=1$.
\end{proof}

\begin{theo}[Well-posedness of the Vlasov equation]\label{theo:well-posedness-vlasov}
Assume that the kernels $K_\ell$ satisfy  \eqref{eq:hypothesis-kernels-bounded} and \eqref{eq:hypothesis-kernels-lipschitz}, consider any $w=(w_\ell)_{\ell\in \mathbb{N}}$, and suppose that they satisfy \eqref{eq:bounded-lipschitz-force-constants} and \eqref{eq:well-posedness-constant-Cp} for some $p\in [1,\infty]$. Then, for every \corr{$\mu_0\in \mathcal{P}_\nu(\mathbb{R}^d\times [0,1])$} there exists a unique \corr{$\mu\in C(\mathbb{R}_+,\mathcal{P}_\nu(\mathbb{R}^d\times [0,1])\mbox{-}d_{p,\nu})$} global-in-time distributional solution  to \eqref{eq:vlasov-equation-joint}-\eqref{eq:vlasov-equation-force-joint} issued at $\mu_0$.
\end{theo}

\begin{proof}
We shall restrict our proof to the case $1<p<\infty$, but clear adaptations can be done for the limiting cases $p=1$ and $p=\infty$, and then we will omit those details. 

Fix any $T>0$ and any initial probability measure \corr{$\mu_*\in \mathcal{P}_\nu(\mathbb{R}^d\times [0,1])$}. For any curve \corr{$\mu\in C([0,T],\mathcal{P}_\nu(\mathbb{R}^d\times [0,1])\mbox{-}d_{p,\nu})$} we define \corr{$\mathcal{F}[\mu]\in C([0,T],\mathcal{P}_\nu(\mathbb{R}^d\times [0,1])\mbox{-}d_{p,\nu})$} by
\begin{equation}\label{eq:fixed-point-mapping}
\mathcal{F}[\mu]_t^\xi:=\mathcal{T}_t^\xi[w,\mu]_{\#}\mu_*^\xi,\quad t\in [0,T] ,\,\mbox{a.e. }\xi\in [0,1],
\end{equation}
where $\mathcal{T}_t^\xi$ is the flow-map \eqref{eq:characteristics-flow-map}. We recall that by Proposition \ref{pro:well-posed-characteristics} the characteristic system \eqref{eq:characteristic-system} is globally-in-time well posed for a.e. $\xi\in [0,1]$ under the assumptions \eqref{eq:bounded-lipschitz-force-constants}. Therefore, the family of measures $(\mathcal{F}[\mu]_t^\xi)_{t\in [0,T]\times [0,1]}$ is defined for all $t\in [0,T]$ and a.e. $\xi\in [0,1]$.

\medskip

$\diamond$ {\sc Step 1}: We show that \corr{$\mathcal{F}[\mu]\in C([0,T],\mathcal{P}_\nu(\mathbb{R}^d\times [0,1])\mbox{-}d_{p,\nu})$}.\\
First, by Proposition \ref{pro:well-posed-characteristics} we have that $\mathcal{T}_t^\xi[\mu]$ is continuous with respect to $(t,x)\in \mathbb{R}_+\times \mathbb{R}^d$ and Borel-measurable with respect to $\xi\in [0,1]$. Additionally, $(\mu_*^\xi)_{\xi\in [0,1]}$ is also a Borel-measurable family. Altogether, this implies that $(\mathcal{F}[\mu]_t^\xi)_{t\in [0,T]\times [0,1]}\subset\mathcal{P}(\mathbb{R}^d)$ is a Borel family, and then so is the family $(\mathcal{F}[\mu]_t)_{t\in [0,T]}\subset \mathcal{P}_\nu(\mathbb{R}^d\times [0,1])$. Second, for a.e. $\xi\in [0,1]$ and all $0\leq t_1\leq t_2\leq T$ we have
\begin{align*}
d_{\rm BL}(\mathcal{F}[\mu]_{t_1}^\xi,\mathcal{F}[\mu]_{t_2}^\xi)&=\sup_{\Vert \phi\Vert_{\rm BL}\leq 1}\int_{\mathbb{R}^d}\phi(x)(d\mathcal{F}[\mu]_{t_1}^\xi(x)-d\mathcal{F}[\mu]_{t_2}^\xi(x))\\
&=\sup_{\Vert \phi\Vert_{\rm BL}\leq 1 }\int_{\mathbb{R}^d}(\phi(\mathcal{T}_{t_1}^\xi[w,\mu](x))-\phi(\mathcal{T}_{t_2}^\xi[w,\mu](x)))\,d\mu_*^\xi(x)\\
&\leq \Vert \mathcal{T}_{t_1}^\xi[w,\mu]-\mathcal{T}_{t_2}^\xi[w,\mu]\Vert_{L^\infty}\leq \int_{t_1}^{t_2}\Vert F_w[\mu_t](\mathcal{T}_t^\xi[\mu],\xi)\Vert_{L^\infty}\,dt\leq B_F|t_1-t_2|,
\end{align*}
where above we have used the definion of the bounded-Lipschitz metric, the fact that $\mathcal{T}_t^\xi[w,\mu]$ is the flow map of $F_w[\mu_t](\cdot,\xi)$ and the triangle inequality, along with the uniform bound of $F_w[\mu]$ by $B_F$ in Proposition \ref{pro:properties-force-I}. Taking $L^p$-norms with respect to $\xi\in [0,1]$ yields
$$d_{p,\nu}(\mathcal{F}[\mu]_{t_1},\mathcal{F}[\mu]_{t_2})\leq B_F|t_1-t_2|,$$
for all $t_1,t_2\in [0,T]$. Since \corr{$\mathcal{F}_0[\mu]=\mu_* \in \mathcal{P}_\nu(\mathbb{R}^d\times [0,1])$}, the above along with the triangle inequality implies that \corr{$\mathcal{F}[\mu]_t\in \mathcal{P}_\nu(\mathbb{R}^d\times [0,1])$} for all $t\in [0,T]$ and, additionally, \corr{$\mathcal{F}[\mu]\in C([0,T],\mathcal{P}_\nu(\mathbb{R}^d\times [0,1])\mbox{-}d_{p,\nu})$} (in fact Lipschitz-continuous in time).

\medskip

$\diamond$ {\sc Step 2}: We show that $\mathcal{F}$ is contractive for small $T$.\\
Consider any couple \corr{$\mu,\bar\mu\in C([0,T],\mathcal{P}_\nu(\mathbb{R}^d\times [0,1])\mbox{-}d_{p,\nu})$}. For a.e. $\xi\in [0,1]$ and all $t\in [0,T]$ a similar argument as above yields
\begin{align}
d_{\rm BL}(\mathcal{F}[\mu]_{t}^\xi,\mathcal{F}[\bar\mu]_{t}^\xi)&=\sup_{\Vert \phi\Vert_{\rm BL}\leq 1}\int_{\mathbb{R}^d}\phi(x)(d\mathcal{F}[\mu]_{t}^\xi(x)-d\mathcal{F}[\bar\mu]_{t}^\xi(x))\nonumber\\
&=\sup_{\Vert \phi\Vert_{\rm BL}\leq 1}\int_{\mathbb{R}^d}(\phi(\mathcal{T}_{t}^\xi[w,\mu](x))-\phi(\mathcal{T}_{t}^\xi[w,\bar\mu](x)))\,d\mu_*^\xi(x)\nonumber\\
&\leq \int_{\mathbb{R}^d}|X_w[\mu](t,x,\xi)-X_w[\bar\mu](t,x,\xi)|\,d\mu_*^\xi(x).\label{eq:well-posedness-vlasov-step-0}
\end{align}
We then need a continuous dependence of the flow map with respect to $\mu$. To this end, note that by definition of the characteristic system \eqref{eq:characteristic-system} we have
\begin{equation}\label{eq:well-posedness-vlasov-step-1}
|X_w[\mu](t,x,\xi)-X_w[\bar\mu](t,x,\xi)|\leq I_1+I_2,
\end{equation}
where each factor reads
\begin{align}
I_1&:=\int_0^t |F_w[\mu_s](X_w[\mu](s,x,\xi),\xi)-F_w[\mu_s](X_w[\bar\mu](s,x,\xi),\xi)|\,ds,\label{eq:well-posedness-vlasov-step-2}\\
I_2&:=\int_0^t |F_w[\mu_s](X_w[\bar\mu](s,x,\xi),\xi)-F_w[\bar\mu_s](X_w[\bar\mu](s,x,\xi),\xi)|\,ds.\label{eq:well-posedness-vlasov-step-3}
\end{align}
For the first term $I_1$ we use the Lipschitz-continuity of the force $F_w[\mu]$ with respect to $x$ in Proposition \ref{pro:properties-force-I}, which implies
\begin{equation}\label{eq:well-posedness-vlasov-step-4}
I_1\leq L_F\int_0^t|X_w[\mu](s,x,\xi)-X_w[\bar\mu](s,x,\xi)|\,ds.
\end{equation}
Regarding the second factor, we use Proposition \ref{pro:properties-force-II}, which applied to \eqref{eq:well-posedness-vlasov-step-3} yields
\begin{equation}\label{eq:well-posedness-vlasov-step-5}
I_2\leq \left(\sum_{\ell=1}^\infty BL_\ell\sum_{k=1}^\ell \Vert w_\ell(\xi,\cdot)\Vert_{L^q_{\xi_k}L^1_{\hbxi_{\ell,k}}}\right)\int_0^td_{p,\nu}(\mu_s,\bar\mu_s)\,ds.
\end{equation}
Plugging \eqref{eq:well-posedness-vlasov-step-4}-\eqref{eq:well-posedness-vlasov-step-5} into \eqref{eq:well-posedness-vlasov-step-1} implies
\begin{align*}
|X_w[\mu](t,x,\xi)-X_w[\bar\mu](t,x,\xi)|&\leq L_F\int_0^t |X_w[\mu](s,x,\xi)-X_w[\bar\mu](s,x,\xi)|\,ds\\
&+\left(\sum_{\ell=1}^\infty BL_\ell\sum_{k=1}^\ell \Vert w_\ell(\xi,\cdot)\Vert_{L^q_{\xi_k}L^1_{\hbxi_{\ell,k}}}\right)\int_0^t d_{p,\nu}(\mu_s,\bar\mu_s)\,ds,
\end{align*}
for all $t\in [0,T]$. Since $X_w[\mu](0,x,\xi)=x=X_w[\bar\mu](0,x,\xi)$, then Gr{\" o}nwall's lemma implies
\begin{equation}\label{eq:well-posedness-vlasov-step-6}
|X_w[\mu](t,x,\xi)-X_w[\bar\mu](t,x,\xi)|\leq \left(\sum_{\ell=1}^\infty BL_\ell\sum_{k=1}^\ell \Vert w_\ell(\xi,\cdot)\Vert_{L^q_{\xi_k}L^1_{\hbxi_{\ell,k}}}\right)\int_0^t e^{L_F(t-s)}d_{p,\nu}(\mu_s,\bar\mu_s)\,ds,
\end{equation}
for all $t\in [0,T]$ all $x\in \mathbb{R}^d$ and a.e. $\xi\in [0,1]$. Integrating \eqref{eq:well-posedness-vlasov-step-6} with respect to $\mu_*^\xi(x)$ in the variable $x$, plugging it in \eqref{eq:well-posedness-vlasov-step-0}, and taking $L^p$ norms with respect to $\xi\in [0,1]$ yields
$$d_{p,\nu}(\mathcal{F}[\mu]_t,\mathcal{F}[\bar\mu]_t)\leq C_p\int_0^t e^{L_F(t-s)}d_{p,\nu}(\mu_s,\bar\mu_s)\,ds,$$
where the constant $C_p$ is defined in \eqref{eq:well-posedness-constant-Cp}. Taking the uniform norm with respect to $t$ implies
\begin{equation}\label{eq:well-posedness-vlasov-step-7}
\sup_{t\in [0,T]} d_{p,\nu}(\mathcal{F}[\mu]_t,\mathcal{F}[\bar\mu]_t)\leq \frac{C_p}{L_F} (e^{L_F T}-1)\sup_{t\in [0,T]}d_{p,\nu}(\mu_t,\bar\mu_t).
\end{equation}

\medskip

$\diamond$ {\sc Step 3}: Global-in-time well posedness.\\
Note that the above ensures that $\mathcal{F}:\corr{C([0,T],\mathcal{P}_\nu(\mathbb{R}^d\times [0,1])\mbox{-}d_{p,\nu})}\longrightarrow \corr{C([0,T],\mathcal{P}_\nu(\mathbb{R}^d\times [0,1])\mbox{-}d_{p,\nu})}$ is well defined and contractive as long as 
$$T<\frac{1}{L_F}\log\left(1+\frac{L_F}{C_p}\right).$$
Under this condition, by the Banach contraction principle we have that the operator $\mathcal{F}$ has a unique fixed point in \corr{$C([0,T],\mathcal{P}_\nu(\mathbb{R}^d\times [0,1])\mbox{-}d_{p,\nu})$}. Bearing in mind the particular form of $\mathcal{F}$ in \eqref{eq:fixed-point-mapping} and the characterization of distributional solutions in Proposition \ref{pro:distributional-solution-Vlasov-fibered}, we have that the fixed point corresponds to the unique distributional solution to \eqref{eq:vlasov-equation-joint}-\eqref{eq:vlasov-equation-force-joint} issued at $\mu_*$ defined in the interval $[0,T]$. Iterating the construction on the intervals $[kT,(k+1)T]$ for all $k\in\mathbb{N}$, which can be done because the lifespan $[0,T]$ of the above local-in-time solution does not depends on the initial datum $\mu_*$, one obtains the unique global-in-time distributional solution.
\end{proof}

\begin{rem}
We remark that the scaling conditions on the kernels $K_\ell$
\begin{equation}\label{eq:hypothesis-kernels-scaling-III}
\sum_{\ell=1}^\infty \ell BL_\ell<\infty,
\end{equation}
which is weaker than \eqref{eq:hypothesis-kernels-scaling} but stronger than \eqref{eq:hypothesis-kernels-scaling-II}, together with the scaling condition 
$$\sup_{\ell\in \mathbb{N}}\,\Vert w_\ell\Vert_{L^\infty}\leq W,$$
({\it e.g.}, $w\in \mathcal{H}_W$ by Definition \ref{defi:UR-hypergraphons}) ensure the above hypothesis \eqref{eq:bounded-lipschitz-force-constants} and \eqref{eq:well-posedness-constant-Cp} since we have
$$B_F\leq W\sum_{\ell=1}^\infty B_\ell,\quad L_F\leq W\sum_{\ell=1}^\infty L_\ell,\quad C_p\leq W \sum_{\ell=1}^\infty \ell BL_\ell.$$
\end{rem}

Our remark that Theorem \ref{theo:well-posedness-vlasov} does not only operate over absolutely continuous initial measures, but also over initial measures $\mu_0$ having an atomic part leads to the following corollary for initial measures supported on the graph of a function $X_0:[0,1]\longrightarrow \mathbb{R}^d$.

\begin{cor}[Well-posedness of the continuum-limit equation]\label{cor:well-posedness-continuum-limit-equation}
Under the hypothesis in Theorem \ref{theo:well-posedness-vlasov}, assume that $X_0\in L^p([0,1],\mathbb{R}^d)$ and set $\mu_0^\xi=\delta_{X_0(\xi)}$. Then, the unique solution $\mu$ to the Vlasov equation \eqref{eq:vlasov-equation-joint}-\eqref{eq:vlasov-equation-force-joint} must have the form $\mu_t^\xi=\delta_{X(t,\xi)}$ where $X\in C(\mathbb{R}_+,L^p([0,1],\mathbb{R}^d))$ is the unique solution to the continuum-limit equation
\begin{equation}\label{eq:continuum-limit-equation}
\begin{cases}
\displaystyle\partial_t X(t,\xi)=\sum_{\ell=1}^\infty\int_{[0,1]^{\ell+1}}w_\ell(\xi,\xi_1,\ldots,\xi_\ell)\,K_\ell(X(t,\xi),X(t,\xi_1),\ldots,X(t,\xi_\ell))\,d\xi_1\ldots\,d\xi_\ell,\\
\displaystyle X(0,\cdot)=X_0.
\end{cases}
\end{equation}
\end{cor}

Equation \eqref{eq:continuum-limit-equation} can be regarded as the natural higher-order extension of the version with binary interactions introduced in \cite{M-14} in the context of non-linear heat equations on dense graphs. We also refer to \cite{AP-21,AP-23-arxiv,GKX-23,PT-22-arxiv} for further extensions based on binary interactions. A  higher-order extension similar to \eqref{eq:continuum-limit-equation} was recently introduced in \cite{BBK-23} to study bifurcation and stability properties of twisted stated for the Kuramoto model restricted to ternary interactions.

\section{Stability estimate of the Vlasov equation over UR-hypergraphons}\label{sec:stability-estimate}

In this section we study the stability of the solutions to the Vlasov equation \eqref{eq:vlasov-equation-joint}-\eqref{eq:vlasov-equation-force-joint} with respect to the initial datum \corr{$\mu_0\in \mathcal{P}_\nu(\R^d\times [0,1])$} and the involved UR-hypergraphon $w=(w_\ell)_{\ell\in \mathbb{N}}$. As in Section \ref{sec:well-posedness}, we shall not restrict to UR-hypergraphons only, but actually the stability estimate will be applicable to a broader class of $w$.

\begin{theo}[Stability estimate for the Vlasov equation]\label{theo:stability-estimate-vlasov}
Assume that the kernels $K_\ell$ verify the assumptions \eqref{eq:hypothesis-kernels-bounded} and \eqref{eq:hypothesis-kernels-lipschitz}, and also that $K_\ell\in L^1(\mathbb{R}^{d(\ell+1)})$ for all $\ell\in \mathbb{N}$. Consider $w=(w_\ell)_{\ell\in \mathbb{N}}$ and $\bar w=(\bar w_\ell)_{\ell\in \mathbb{N}}$ satisfying the assumptions \eqref{eq:bounded-lipschitz-force-constants} and \eqref{eq:well-posedness-constant-Cp} for common constants $B_F, L_F, C_p>0$, and additionally suppose that
\begin{equation}\label{eq:stability-estimate-constant-D-infty}
\sum_{\ell=1}^\infty 4^\ell \Vert \hat{K}_\ell\Vert_{L^1}<\infty,\qquad D_\infty:=\sum_{\ell=1}^\infty 2^\ell (\Vert w_\ell\Vert_{L^\infty}+\Vert \bar w_\ell\Vert_{L^\infty})\Vert \hat{K}_\ell\Vert_{L^1}<\infty,
\end{equation}
For any initial data \corr{$\mu_0,\bar\mu_0\in \mathcal{P}_\nu(\mathbb{R}^d\times [0,1])$} with $p\in [1,\infty)$, let \corr{$\mu,\bar\mu\in C(\mathbb{R}_+,\mathcal{P}_\nu(\mathbb{R}^d\times [0,1])\mbox{-}d_{p,\nu})$} be the unique global-in-time distributional solutions to \eqref{eq:vlasov-equation-joint}-\eqref{eq:vlasov-equation-force-joint} issued at $\mu_0$ with given $w$ (respectively, $\bar\mu_0$ and $\bar w$) as in Theorem \ref{theo:well-posedness-vlasov}. Then, we have
$$
d_{p,\nu}(\mu_t,\bar\mu_t) \leq  e^{(C_p+L_F)t}  \left(d_{p,\nu}(\mu_{0},\bar\mu_{0}) +  \frac{D_\infty^{1/q}}{L_F}  d_{\square}(w,\bar w; (4^\ell\|\hat K_\ell\|_{L^1})_{\ell\in \mathbb{N}})^{1/p} \right),
$$
for every $t \geq 0$, where $\frac{1}{p}+\frac{1}{q}=1$, $\hat K_\ell$ denotes the Fourier transform of $K_\ell$, and $d_\square$ is the labeled cut distance in Definition \ref{defi:UR-hypergraphons}.
\end{theo}

We note that under the hypothesis \eqref{eq:stability-estimate-constant-D-infty}, the sequence  and therefore the labeled cut distance $d_{\square}(w,\bar w; (4^\ell\|\hat K_\ell\|_{L^1})_{\ell\in \mathbb{N}})$ in the stability estimate is well defined. Whilst above $w=(w_\ell)_{\ell\in \mathbb{N}}$ do not restrict to UR-hypergraphons only, we anticipate that in order for the stability estimate to be useful in the proof of the main Theorem \ref{theo:main} in Section \ref{sec:proof-main-result} we will have to restrict to UR-hypergraphons. More particularly, from the results in the hypergraph limit theories in Section \ref{subsec:UR-hypergraphons} we can only guarantee the decay of the cut distance in the class $\mathcal{H}_W$ of UR-hypergraphons as described in Proposition \ref{pro:compactness-UR-hypergraphons}.

\begin{rem}[On the Sobolev regularity assumption]\label{rem:stability-estimate-regularity-scaling}
We remark that the regularity-scaling conditions \eqref{eq:hypothesis-kernels-regularity-scaling} on the kernels $K_\ell$, and the scaling condition on the weights $w_\ell$
$$\sup_{\ell\in \mathbb{N}}\Vert w_\ell\Vert_{L^\infty}\leq W,$$ 
({\it i.e.}, $w\in \mathcal{H}_W$ by Definition \ref{defi:UR-hypergraphons}) ensure the above hypothesis \eqref{eq:stability-estimate-constant-D-infty} since we have
$$D_\infty\leq W\sum_{\ell=1}^\infty 4^\ell\Vert \hat{K}_\ell\Vert_{L^1}\lesssim  \frac{4^\ell\pi^{\frac{d\ell}{4}}}{\sqrt{\Gamma(\frac{d\ell}{2})}}\Vert K_\ell\Vert_{H^{\frac{d(\ell+1)}{2}+\varepsilon}}.$$
Specifically, note that for any $\ell\in \mathbb{N}$ and setting $k=\frac{d(\ell+1)}{2}+\varepsilon$ for some $\varepsilon>0$ arbitrarily small, the Cauchy-Schwarz inequality implies $\Vert \hat{K}_\ell\Vert_{L^1} \leq \alpha_\ell^{1/2}\beta_\ell^{1/2}$, where
\begin{align*}
\alpha_\ell&:=\int_{\mathbb{R}^{d(\ell+1)}}\frac{1}{(1+|z|^2+|\bz_\ell|^2)^k}\,dz\,d\bz_\ell,\\
\beta_\ell&:=\int_{\mathbb{R}^{d(\ell+1)}}(1+|z|^2+|\bz_\ell|^2)^k|\hat{K}_\ell(z,\bz_\ell)|^2\,dz\,d\bz_\ell.
\end{align*}
On the one hand, the first integral is finite because $2k>d(\ell+1)$ by our choice of $k$. In fact, it can be calculated explicitly in polar coordinated leading to
$$\alpha_\ell=\omega_{d(\ell+1)-1}\int_0^\infty \frac{r^{d(\ell+1)-1}}{(1+r^2)^k}\,dr=\frac{\omega_{d(\ell+1)-1}}{2}B\left(\frac{d(\ell+1)}{2},\frac{\varepsilon}{2}\right),$$
where $\omega_{d(\ell+1)-1}=\frac{2\pi^{d(\ell+1)/2}}{\Gamma(d(\ell+1)/2)}$ is the area of the sphere of $\mathbb{R}^{d(\ell+1)}$ and $B(a,b)=\frac{\Gamma(a)\Gamma(b)}{\Gamma(a+b)}=\int_0^\infty\frac{s^{a-1}}{(1+s)^{a+b}}\,ds$ is the Beta function. Therefore,
$$\alpha_\ell=\frac{\pi^{\frac{d(\ell+1)}{2}}\Gamma(\frac{\varepsilon}{2})}{\Gamma(\frac{d(\ell+1)}{2}+\varepsilon)}\lesssim \frac{\pi^{\frac{d\ell}{2}}}{\Gamma(\frac{d\ell}{2})}.$$
On the other hand, by definition of the fractional Sobolev through Bessel potentials we have
$$\beta_\ell^{1/2}=\Vert \hat{K}_\ell\Vert_{H^k}=\Vert \hat{K}_\ell\Vert_{H^{\frac{d(\ell+1)}{2}+\varepsilon}}.$$
\end{rem}

\begin{proof}[Proof of Theorem \ref{theo:stability-estimate-vlasov}]
We restrict our proof to the case $p>1$ but clear adaptations can be done to prove the case $p=1$. 
\medskip

$\diamond$ {\sc Step 1}: We obtain a first bound for the bounded-Lipschitz distance between $\mu_t^\xi$ and  $\bar \mu_t^\xi$.\\
By item (iii) of Proposition \ref{pro:distributional-solution-Vlasov-fibered} we know that for a.e. $\xi\in [0,1]$ the distributional solutions $\mu$ and $\bar\mu$ must satisfy
$$\mu_t^\xi=\mathcal{T}_t^\xi[w,\mu]_{\#}\mu_0^\xi,\qquad \bar\mu_t^\xi=\mathcal{T}_t^\xi[\bar w,\bar\mu]_{\#}\bar\mu_0^\xi,$$
for a.e. $t\in \mathbb{R}_+$, and in fact, the time-continuous representatives obtained in Theorem \ref{theo:well-posedness-vlasov} satisfy the above for all $t\in \mathbb{R}_+$. Therefore, we have 
\begin{equation}\label{eq:stability-estimate-vlasov-step-0}
\begin{aligned}
d_{\rm BL}(\mu_{t}^\xi,\bar\mu_{t}^\xi)&=\sup_{\Vert \phi\Vert_{\rm BL}\leq 1}\int_{\mathbb{R}^d}\phi(\mathcal{T}_{t}^\xi[w,\mu](x))\,d\mu_0^\xi(x)- \int_{\mathbb{R}^d}\phi(\mathcal{T}_{t}^\xi[w,\bar\mu](x))\,d\bar \mu_0^\xi(x)\\
&\leq \sup_{\Vert \phi\Vert_{\rm BL}\leq 1}\int_{\mathbb{R}^d}\left(\phi(\mathcal{T}_{t}^\xi[w,\mu](x)) - \phi(\mathcal{T}_{t}^\xi[\bar w,\bar\mu](x))\right)\,d\mu_0^\xi(x)\\
& \quad \quad +  \sup_{\Vert \phi\Vert_{\rm BL}\leq 1}\int_{\mathbb{R}^d}\phi(\mathcal{T}_{t}^\xi[\bar w,\bar \mu](x))\,d(\mu_0^\xi -\bar \mu_0^\xi)(x)\\
&\leq  \int_{\mathbb{R}^d}\vert X_w[\mu](t,x,\xi) - X_{\bar w}[\bar \mu](t,x,\xi)\vert\,d\mu_0^\xi(x)+  e^{L_Ft}  d_{\rm BL}(\mu_{0}^\xi,\bar\mu_{0}^\xi),
\end{aligned}
\end{equation}
where in the last line we have observed that since $\phi\in {\rm BL}_1(\mathbb{R}^d)$, for all $t\in \mathbb{R}_+$ and a.e. $\xi\in [0,1]$ the maps $x\in \mathbb{R}^d\mapsto \phi(\mathcal{T}_t^\xi[w,\mu](x))$ are bounded and Lipschitz-continuous with ${\rm BL}$ norm bounded by $e^{L_F t}$ thanks to item (iii) in Proposition \ref{pro:well-posed-characteristics}.

\medskip

$\diamond$ {\sc Step 2}: We control the difference $|X_w[\mu](t,x,\xi)-X_{\bar w}[\bar \mu](t,x,\xi)|$.\\
This step is similar to {\sc Step 2} in the proof of Theorem \ref{theo:well-posedness-vlasov}, where a similar estimate was obtained in \eqref{eq:well-posedness-vlasov-step-6} in the particular case that $w=\bar w$ and $\mu_0=\bar\mu_0$. Since this is the key step of the proof, which involves a certain continuity of the flow map with respect to the underlying $w$ and $\bar w$, we provide a proof. For all $t \in [0,T]$ and a.e. $\xi \in [0,1]$, we get
\begin{equation}\label{eq:stability-estimate-vlasov-step-1}
\begin{aligned}
& |X_w[\mu](t,x,\xi) - X_{\bar w}[\bar \mu](t,x,\xi)|\\
&\quad = \left\vert\int_0^t \left(F_w[\mu_s](X_w[\mu](s,x,\xi),\xi) - F_{\bar w}[\bar \mu_s](X_{\bar w}[\bar \mu](s,x,\xi),\xi)\right)\,ds\right\vert\\
&\quad \leq \int_0^t \left\vert F_w[\mu_s](X_w[\mu](s,x,\xi),\xi) - F_{w}[\mu_s](X_{\bar w}[\bar \mu](s,x,\xi),\xi)\right\vert\,ds\\
&\quad \quad   +  \int_0^t \left\vert F_{w}[ \mu_s](X_{\bar w}[\bar \mu](s,x,\xi),\xi)  - F_{\bar w}[\bar \mu_s](X_{\bar w}[\bar \mu](s,x,\xi),\xi) \right\vert\,ds\\
& \quad =: I_3 +I_4.
\end{aligned}
\end{equation}
Regarding the first term $I_3$, the Lipschitz continuity property of the force $F_w[\mu_s](x,\xi)$ with respect to $x$ in item (ii) of Proposition \ref{pro:properties-force-I} implies
\begin{equation}\label{eq:stability-estimate-vlasov-step-2}
I_3\leq L_F \int_0^t |X_w[\mu](s,x,\xi)-X_{\bar w}[\bar \mu](s,x,\xi)|\,ds.
\end{equation}
Regarding the second term $I_4$, we can rewrite it as $I_4 = I_{41} + I_{42}$, where 
\begin{align*}
I_{41}&:=\int_0^t \left\vert F_{w}[ \mu_s](X_{\bar w}[\bar \mu](s,x,\xi),\xi)  - F_{ w}[\bar \mu_s](X_{\bar w}[\bar \mu](s,x,\xi),\xi) \right\vert\,ds,\\
I_{42}&:=\int_0^t \left\vert F_{ w}[\bar \mu_s](X_{\bar w}[\bar \mu](s,x,\xi),\xi)  - F_{\bar w}[\bar \mu_s](X_{\bar w}[\bar \mu](s,x,\xi),\xi) \right\vert\,ds.
\end{align*}
From Proposition \ref{pro:properties-force-II}, we know that 
\begin{equation}\label{eq:stability-estimate-vlasov-step-3}
I_{41}\leq \int_0^t \left(\sum_{\ell=1}^\infty BL_\ell\sum_{k=1}^\ell \Vert w_\ell(\xi,\cdot)\Vert_{L^q_{\xi_k}L^1_{\hbxi_{\ell,k}}}\right)d_{p,\nu}(\mu_s,\bar\mu_s) ds.
\end{equation}
Besides, using the expression \eqref{eq:vlasov-equation-force} of the forces $F_w[\bar\mu_s]$ and $F_{\bar w}[\bar\mu_s]$ we get
\begin{equation}\label{eq:stability-estimate-vlasov-step-separation-variables}
\begin{aligned}
I_{42}& \leq  \int_0^t\sum_{\ell=1}^\infty \left\vert\int_{[0,1]^\ell} (w_\ell(\xi,\xi_1,\ldots,\xi_\ell)-\bar w_\ell(\xi,\xi_1,\ldots,\xi_\ell))\right.\\
&\left.\quad\times \left(\int_{\mathbb{R}^{d\ell}}K_\ell(X_{\bar w}[\bar \mu](s,x,\xi),x_1,\ldots,x_\ell)\,d\bar \mu_s^{\xi_1}(x_1)\,\cdots \,d\bar \mu_s^{\xi_\ell}(x_\ell)\right)d\xi_1\ldots\,d\xi_\ell\right\vert\,ds.
\end{aligned}
\end{equation}
Each of the integrals on $[0,1]^\ell$ is reminiscent of the action of the $\ell$-th order adjacency operator $T^{w_\ell-\bar w_\ell}$ over the function
$$(\xi_1,\ldots,\xi_\ell)\in [0,1]^\ell\longmapsto \int_{\mathbb{R}^{d\ell}}K_\ell(X_{\bar w}[\bar \mu](s,x,\xi),x_1,\ldots,x_\ell)\,d\bar \mu_s^{\xi_1}(x_1)\,\cdots \,d\bar \mu_s^{\xi_\ell}(x_\ell).$$
However, the later does not have separate variables and therefore we cannot readily invoke  the $\ell$-th order cut distance. To this end, inspired by \cite{BCN-24} (where the binary case $\ell=1$ on a one-dimensional periodic domain was considered) or also \cite{DM-22,OR-19}, we write $K_\ell$ as a superposition of functions with separate variables via the Fourier inversion formula
$$K_\ell(x,x_1,\ldots,x_\ell)=\int_{\mathbb{R}^{d(\ell+1)}} e^{2\pi i\,(x\cdot z+x_1\cdot z_1+\cdots +x_\ell\cdot z_\ell)}\hat{K}_\ell(z,z_1,\ldots,z_\ell)\,dz\,dz_1\ldots\,dz_\ell,$$
with $\hat{K}_\ell$ the Fourier transform. We note that $\hat{K}_\ell$ are all well defined because $K_\ell\in L^1(\mathbb{R}^{d(\ell+1)})$ by our assumptions, and also the inversion formula holds with an absolutely convergent integral because we are further assuming that $\hat{K}_\ell\in L^1(\mathbb{R}^{d(\ell+1)})$. Defining the function
$$
g(s,z,\xi):=\int_{\mathbb{R}^d}e^{2\pi i x\cdot z}\,d\bar\mu_s^{\xi}(x),\quad s\in \mathbb{R}_+,\,z\in \mathbb{R}^d,\,\xi\in [0,1],
$$
and since $\vert e^{2\pi i X_{\bar w}[\bar\mu](s,x,\xi)}\vert=1$, we can reformulate the term $I_{42}$ as follows
\begin{equation}\label{eq:stability-estimate-vlasov-step-4}
\begin{aligned}
I_{42}&\leq \int_0^t \sum_{\ell=1}^\infty\int_{\mathbb{R}^{d(\ell+1)}}\vert \hat{K}_\ell(z,\bz_\ell)\vert\,\left\vert\int_{[0,1]^\ell} (w_\ell(\xi,\bxi_\ell)-\bar w_\ell(\xi,\bxi_\ell))\prod_{k=1}^\ell g(s,z_k,\xi_k)\,d\bxi_\ell\right\vert\,dz\,d\bz_\ell\,ds\\
&=\int_0^t\sum_{\ell=1}^\infty\int_{\mathbb{R}^{d(\ell+1)}}\vert \hat{K}_\ell(z,\bz_\ell)\vert\,\left\vert \tilde{T}^{w_\ell-\bar w_\ell}\left[g(s,z_1,\cdot),\ldots,g(s,z_\ell,\cdot)\right](\xi)\right\vert\,dz\,d\bz_\ell\,ds.
\end{aligned}
\end{equation}
Above, and since $g$ is complex-valued, we have extended the definition of the adjacency operator \eqref{eq:adjacency-multilinear-operator} to complex valued-functions as follows, $\tilde{T}^{w_\ell-\bar w_\ell}: L^\infty([0,1],\mathbb{C})^\ell\longrightarrow L^1([0,1],\mathbb{C})$ such that
$$\tilde{T}^{w_\ell-\bar w_\ell}[g_1,\ldots,g_\ell]=\int_{[0,1]^\ell} (w_\ell(\xi,\bxi_\ell)-\bar w_\ell(\xi,\bxi_\ell))\,g_1(\xi_1)\cdots g_\ell(\xi_\ell)\,d\xi\,d\bxi_\ell,$$
for all $g_1,\ldots,g_\ell\in L^\infty([0,1],\mathbb{C})$. Thus, putting \eqref{eq:stability-estimate-vlasov-step-2}, \eqref{eq:stability-estimate-vlasov-step-3} and \eqref{eq:stability-estimate-vlasov-step-4} into \eqref{eq:stability-estimate-vlasov-step-1} we get 
\begin{equation}\label{eq:stability-estimate-vlasov-step-5}
\begin{aligned}
& |X_w[\mu](t,x,\xi) - X_{\bar w}[\bar \mu](t,x,\xi)|\\
& \leq L_F \int_0^t |X_w[\mu](s,x,\xi)-X_{\bar w}[\bar \mu](s,x,\xi)|\,ds \\
& \quad + \int_0^t \left(\sum_{\ell=1}^\infty BL_\ell\sum_{k=1}^\ell \Vert w_\ell(\xi,\cdot)\Vert_{L^q_{\xi_k}L^1_{\hbxi_{\ell,k}}}\right)d_{p,\nu}(\mu_s,\bar\mu_s)\,ds\\
& \quad + \int_0^t\sum_{\ell=1}^\infty\int_{\mathbb{R}^{d(\ell+1)}}\vert \hat{K}_\ell(z,\bz_\ell)\vert\,\left\vert \tilde{T}^{w_\ell-\bar w_\ell}\left[g(s,z_1,\cdot),\ldots,g(s,z_\ell,\cdot)\right](\xi)\right\vert\,dz\,d\bz_\ell\,ds.
\end{aligned}
\end{equation}

\medskip

$\diamond$ {\sc Step 3}: We deduce a bound on the integrated difference. \\
We integrate the above inequality \eqref{eq:stability-estimate-vlasov-step-5} with respect to $\mu_0^\xi(x)$. Thus, we get
\begin{align*}
& \int_{\mathbb{R}^d} |X_w[\mu](t,x,\xi) - X_{\bar w}[\bar \mu](t,x,\xi)| \,d\mu_0^\xi(x)\\
& \leq L_F \int_0^t \int_{\mathbb{R}^d} |X_w[\mu](s,x,\xi)-X_{\bar w}[\bar \mu](s,x,\xi)|  \,d\mu_0^\xi(x) \,ds\\
& \quad + \int_0^t \left(\sum_{\ell=1}^\infty BL_\ell\sum_{k=1}^\ell \Vert w_\ell(\xi,\cdot)\Vert_{L^q_{\xi_k}L^1_{\hbxi_{\ell,k}}}\right)d_{p,\nu}(\mu_s,\bar\mu_s)\,ds\\
& \quad + \int_0^t\sum_{\ell=1}^\infty\int_{\mathbb{R}^{d(\ell+1)}}\vert \hat{K}_\ell(z,\bz_\ell)\vert\,\left\vert \tilde{T}^{w_\ell-\bar w_\ell}\left[g(s,z_1,\cdot),\ldots,g(s,z_\ell,\cdot)\right](\xi)\right\vert\,dz\,d\bz_\ell\,ds.
\end{align*}
By applying Gr{\" o}nwall's lemma, we obtain 
\begin{equation}\label{eq:stability-estimate-vlasov-step-6}
\begin{aligned}
& \int_{\mathbb{R}^d} |X_w[\mu](t,x,\xi) - X_{\bar w}[\bar \mu](t,x,\xi)|\,d\mu_0^\xi(x)\\
& \leq \int_0^t  e^{L_F(t-s)} \Bigg\{ \left( \sum_{\ell=1}^\infty BL_\ell\sum_{k=1}^\ell \Vert w_\ell(\xi,\cdot)\Vert_{L^q_{\xi_k}L^1_{\hbxi_{\ell,k}}}\right)d_{p,\nu}(\mu_s,\bar\mu_s) \\
& \qquad\qquad+\sum_{\ell=1}^\infty\int_{\mathbb{R}^{d(\ell+1)}}\vert \hat{K}_\ell(z,\bz_\ell)\vert\,\left\vert \tilde{T}^{w_\ell-\bar w_\ell}\left[g(s,z_1,\cdot),\ldots,g(s,z_\ell,\cdot)\right](\xi)\right\vert\,dz\,d\bz_\ell\Bigg\}\,ds.
\end{aligned}
\end{equation}
\medskip

$\diamond$ {\sc Step 4}: We obtain a bound for the operator norm of $\tilde{T}^{w_\ell-\bar w_\ell}$.\\
We start by obtaining a bound for the operator norm of the real-valued operators $T^{w_\ell-\bar w_\ell}$. By the assumption \eqref{eq:stability-estimate-constant-D-infty} on $w$ and $\bar w$ we have that $T^{w_\ell-\bar w_\ell}$ is a bounded multilinear operator from $L^\infty([0,1])^\ell$ to $L^1([0,1])$, but also from $L^\infty([0,1])^\ell$ to $L^\infty([0,1])$. By interpolation, $T^{w_\ell-\bar w_\ell}$ is a bounded multilinear operator from $L^\infty([0,1])^\ell$ to any $L^p([0,1])$. More precisely, we have
$$
\| T^{w_\ell-\bar w_\ell}\left[\psi_1,\ldots,\psi_\ell\right]\|_{L_\xi^p}\leq   \| T^{w_\ell-\bar w_\ell}\left[\psi_1,\ldots,\psi_\ell\right]\|_{L_\xi^\infty}^{1/q}  \| T^{w_\ell-\bar w_\ell}\left[\psi_1,\ldots,\psi_\ell\right]\|_{L_\xi^1}^{1/p},
$$
for any $\psi_1,\ldots,\psi_\ell\in L^\infty([0,1])$, where $p$ and $q$ are conjugate exponents, i.e. $\frac{1}{p}+\frac{1}{q}=1$. Besides,
\begin{align*}
& \| T^{w_\ell - \bar w_\ell}\left[\psi_1,\ldots,\psi_\ell\right]\|_{L_\xi^\infty}\\
 & \quad = \esssup_{\xi \in [0,1]} \left|\int_{[0,1]^\ell}  (w_\ell(\xi,\xi_1,\ldots,\xi_\ell)-\bar w_\ell(\xi,\xi_1,\ldots,\xi_\ell)) \,\psi_1(\xi_1)\cdots\psi_\ell(\xi_\ell)\,d\xi_1 \dots \,d\xi_\ell\right|\\
  & \quad \leq (\|w_\ell\|_{L^\infty} + \|\bar w_\ell\|_{L^\infty})\Vert \psi_1\Vert_{L^\infty}\cdots \Vert \psi_\ell\Vert_{L^\infty},
\end{align*}
and also
\begin{align*}
& \| T^{w_\ell - \bar w_\ell}\left[\psi_1,\ldots,\psi_\ell\right]\|_{L_\xi^1}\\
 & \quad = \int_{[0,1]} \left|\int_{[0,1]^\ell}  (w_\ell(\xi,\xi_1,\ldots,\xi_\ell)-\bar w_\ell(\xi,\xi_1,\ldots,\xi_\ell)) \,\psi_1(\xi_1)\cdots\psi_\ell(\xi_\ell)\,d\xi_1 \dots \,d\xi_\ell\right|\,d \xi\\
  & \quad \leq \|   T^{w_\ell- \bar w_\ell}\|_{(L^\infty)^\ell\to L^1}\Vert \psi_1\Vert_{L^\infty}\cdots \Vert \psi_\ell\Vert_{L^\infty},
\end{align*}
by definition of the operator norm. Thus, we deduce that 
\begin{equation}\label{eq:stability-estimate-vlasov-step-7}
    \| T^{w_\ell - \bar w_\ell}\left[\psi_1,\ldots,\psi_\ell\right]\|_{L_\xi^p} \leq  (\|w_\ell\|_{L^\infty} + \|\bar  w_\ell\|_{L^\infty})^{1/q}  \|   T^{w_\ell- \bar w_\ell}\|_{(L^\infty)^\ell\to L^1}^{1/p}\,\Vert\psi_1\Vert_{L^\infty}\cdots \Vert \psi_\ell\Vert_{L^\infty},
\end{equation}
for every $\psi_1,\ldots,\psi_\ell\in L^\infty([0,1])$.

Now, we move to the bound of $\tilde T^{w_\ell-\tilde w_\ell}$. Given $g_1,\ldots, g_\ell\in L^\infty([0,1],\mathbb{C})$, we write $g_k=\varphi_k+i\psi_k$ for real and imaginary parts $\varphi_k,\psi_k\in L^\infty([0,1])$. Expanding the products $g_1(\xi_1)\cdots g_\ell(\xi_\ell)$ in $\tilde{T}^{w_\ell-\bar w_\ell}[g_1,\ldots,g_\ell]$ into real and imaginary parts through some combinatorics, taking  $L^p$ norm, using the triangly inequality and applying \eqref{eq:stability-estimate-vlasov-step-7} on each term we have
\begin{align*}
&\Vert\,|\tilde{T}^{w_\ell-\bar w_\ell}[g_1,\ldots,g_\ell]|\,\Vert_{L^p_\xi}\leq (\|w_\ell\|_{L^\infty} + \|\bar  w_\ell\|_{L^\infty})^{1/q}  \|   T_{w_\ell- \bar w_\ell}\|_{(L^\infty)^\ell\to L^1}^{1/p}\\
&\qquad\times\left(\sum_{k=0}^{\lfloor \ell/2\rfloor}\sum_{\substack{I\subset \llbracket1,\ell\rrbracket\\ \# I=2k}} \prod_{j\notin I} \Vert \varphi_j\Vert_{L^\infty}\prod_{j\in I} \Vert \psi_j\Vert_{L^\infty}+\sum_{k=0}^{\lfloor (\ell-1)/2\rfloor}\sum_{\substack{I\subset \llbracket1,\ell\rrbracket\\ \# I=2k+1}}  \prod_{j\notin I} \Vert \varphi_j\Vert_{L^\infty}\prod_{j\in I} \Vert \psi_j\Vert_{L^\infty}\right).
\end{align*}
In particular, if $\Vert g_k\Vert_{L^\infty}\leq 1$ then also $\Vert \varphi_k\Vert_{L^\infty}\leq 1$ and $\Vert \psi_k\Vert_{L^\infty}\leq 1$, which implies
\begin{align*}
\Vert\,|\tilde{T}^{w_\ell-\bar w_\ell}[g_1,\ldots,g_\ell]|\,\Vert_{L^p_\xi}&\leq (\|w_\ell\|_{L^\infty} + \|\bar  w_\ell\|_{L^\infty})^{1/q}  \|   T^{w_\ell- \bar w_\ell}\|_{(L^\infty)^\ell\to L^1}^{1/p}\\
&\qquad\qquad\times\left(\sum_{k=0}^{\lfloor \ell/2\rfloor} \binom{\ell}{2k}+\sum_{k=0}^{\lfloor (\ell-1)/2\rfloor} \binom{\ell}{2k+1}\right).
\end{align*}
Since the above sums of binomial coefficients equals $2^\ell$ then we have
\begin{equation}\label{eq:stability-estimate-vlasov-step-8}
\Vert\,|\tilde{T}^{w_\ell-\bar w_\ell}[g_1,\ldots,g_\ell]|\,\Vert_{L^p_\xi}\leq 2^\ell(\|w_\ell\|_{L^\infty} + \|\bar  w_\ell\|_{L^\infty})^{1/q}  \|   T^{w_\ell- \bar w_\ell}\|_{(L^\infty)^\ell\to L^1}^{1/p},
\end{equation}
for all $g_1,\ldots,g_\ell\in L^\infty([0,1],\mathbb{C})$ with $\Vert g_k\Vert_{L^\infty}\leq 1$.

\medskip
$\diamond$ {\sc Step 5}: Stability estimate.\\
Plugging \eqref{eq:stability-estimate-vlasov-step-6} into \eqref{eq:stability-estimate-vlasov-step-0} implies
\begin{align*}
 d_{\rm BL}(\mu_{t}^\xi,\bar\mu_{t}^\xi) &  \leq e^{L_Ft}  d_{\rm BL}(\mu_{0}^\xi,\bar\mu_{0}^\xi) \\
& + \int_0^t  e^{L_F(t-s)} \Bigg\{ \left( \sum_{\ell=1}^\infty BL_\ell\sum_{k=1}^\ell \Vert w_\ell(\xi,\cdot)\Vert_{L^q_{\xi_k}L^1_{\hbxi_{\ell,k}}}\right)d_{p,\nu}(\mu_s,\bar\mu_s) \\
& \qquad\qquad+\sum_{\ell=1}^\infty\int_{\mathbb{R}^{d(\ell+1)}}\vert \hat{K}_\ell(z,\bz_\ell)\vert\,\left\vert \tilde{T}^{w_\ell-\bar w_\ell}\left[g(s,z_1,\cdot),\ldots,g(s,z_\ell,\cdot)\right](\xi)\right\vert\,dz\,d\bz_\ell\Bigg\}\,ds.
\end{align*}
Taking $L^p$ norms with respect to $\xi$ and by Minkowski's integral inequality, we deduce
\begin{align*}
&d_{p,\nu}(\mu_t,\bar\mu_t)   \leq e^{L_Ft}  d_{p,\nu}(\mu_{0},\bar\mu_{0}) \\
& \qquad + \int_0^t  e^{L_F(t-s)} \Bigg\{  C_p\,d_{p,\nu}(\mu_s,\bar\mu_s) \\
& \qquad\quad\quad\quad +  \sum_{\ell=1}^\infty\int_{\mathbb{R}^{d(\ell+1)}}\vert \hat{K}_\ell(z,\bz_\ell)\vert\,\left\Vert\,\vert \tilde{T}^{w_\ell-\bar w_\ell}\left[g(s,z_1,\cdot),\ldots,g(s,z_\ell,\cdot)\right]\,\vert\right\Vert_{L^p_\xi}\,dz\,d\bz_\ell \Bigg\} ds \\
& \qquad \leq e^{L_Ft}  d_{p,\nu}(\mu_{0},\bar\mu_{0}) \\
& \qquad+ \int_0^t  e^{L_F(t-s)} \Bigg\{  C_p\,d_{p,\nu}(\mu_s,\bar\mu_s) +   \sum_{\ell=1}^\infty 2^\ell\|\hat K_\ell\|_{L^1}  \|T^{w_\ell- \bar w_\ell}\|_{(L^\infty)^\ell\to L^1}^{1/p}  (\|w_\ell\|_{L^\infty} + \|\bar w_\ell\|_{L^\infty})^{1/q}   \Bigg\}\, ds.
\end{align*}
where in the last inequality we have used the interpolation inequality \eqref{eq:stability-estimate-vlasov-step-8} with test functions $g_k(\xi)=g(s,z_k,\xi)$ for $k=1,\ldots,\ell$, which belong to $L^\infty([0,1],\mathbb{C})$ and have norm bounded by $1$.\\

We multiply both sides by $ e^{-L_Ft}$ and get
\begin{align*}
e^{-L_Ft} d_{p,\nu}(\mu_t,\bar\mu_t) &  \leq   d_{p,\nu}(\mu_{0},\bar\mu_{0}) \\
& \quad + \int_0^t   C_p  e^{-L_Fs}  d_{p,\nu}(\mu_s,\bar\mu_s)  ds\\
& \quad + \left(\int_0^t  e^{-L_Fs} ds \right) \sum_{\ell=1}^\infty 2^\ell \|\hat K_\ell\|_{L^1}  \|T^{w_\ell- \bar w_\ell}\|_{(L^\infty)^\ell\to L^1}^{1/p}  (\|w_\ell\|_{L^\infty} + \|\bar w_\ell\|_{L^\infty})^{1/q}\\
&  \leq   d_{p,\nu}(\mu_{0},\bar\mu_{0})  + \frac{1}{L_F}  \sum_{\ell=1}^\infty 2^\ell \|\hat K_\ell\|_{L^1}  \|T^{w_\ell- \bar w_\ell}\|_{(L^\infty)^\ell\to L^1}^{1/p}  (\|w_\ell\|_{L^\infty} + \|\bar w_\ell\|_{L^\infty})^{1/q}\\
& \quad + \int_0^t   C_p  e^{-L_Fs}  d_{p,\nu}(\mu_s,\bar\mu_s)  ds.
\end{align*}

We finally apply Gr{\" o}nwall's lemma and obtain 
\begin{align*}
&e^{-L_Ft} d_{p,\nu}(\mu_t,\bar\mu_t) \\
&\qquad \leq   \left(d_{p,\nu}(\mu_{0},\bar\mu_{0}) +  \frac{1}{L_F}  \sum_{\ell=1}^\infty 2^\ell \|\hat K_\ell\|_{L^1}  \|T^{w_\ell- \bar w_\ell}\|_{(L^\infty)^\ell\to L^1}^{1/p}  (\|w_\ell\|_{L^\infty} + \|\bar w_\ell\|_{L^\infty})^{1/q} \right)e^{C_pt}. 
\end{align*}
Multiplying the above by $e^{L_F t}$ and using H{\"o}lder's inequality on the sum over $\ell\in \mathbb{N}$ we deduce 
\begin{align*}
 & d_{p,\nu}(\mu_t,\bar\mu_t)  \leq  e^{(C_p+L_F)t}  \Bigg\{d_{p,\nu}(\mu_{0},\bar\mu_{0}) \\
 & \quad \quad \quad +  \frac{1}{L_F}  \left(\sum_{\ell=1}^\infty 2^\ell \|\hat K_\ell\|_{L^1}  (\|w_\ell\|_{L^\infty} + \|\bar w_\ell\|_{L^\infty})\right)^{1/q} \left(\sum_{\ell=1}^\infty 2^\ell \|\hat K_\ell\|_{L^1}   \|T^{w_\ell- \bar w_\ell}\|_{(L^\infty)^\ell\to L^1}\right)^{1/p}   \Bigg\},
\end{align*}
for all $t\geq 0$, which ends the proof by the definition of labeled cut distance $d_\square(w,\bar w;(4^\ell\Vert \hat K_\ell\Vert_{L^1})_{\ell \in \mathbb{N}})$ in Definition \ref{defi:UR-hypergraphons}, its relation to the multilinear operator norms in Proposition \ref{pro:cut-distance-operator-norm}, and by definition of the constant $D_\infty$.
\end{proof}

\begin{rem}[On the order-dependent regularity]\label{rem:stability-estimate-why-regularity-assumption}
We note that assumption \eqref{eq:stability-estimate-constant-D-infty} is imposing as it is order-dependent, and therefore the larger $\ell\in\mathbb{N}$, the stronger the decay we need on $\hat{K}_\ell$, or alternatively the smoother $K_\ell$ ({\it cf.} Remark \ref{rem:stability-estimate-regularity-scaling}). This is due to the high generality of the kernels $K_\ell(x,x_1,\ldots,x_\ell)$, which have no special structure. However, for special kernels like
$$K_\ell(x,x_1,\ldots,x_\ell)=G_\ell\left(x-\frac{x_1+\cdots +x_\ell}{\ell}\right),\quad (x,x_1,\ldots,x_\ell)\in \mathbb{R}^{d(\ell+1)},$$
we could have followed an alternative approach for the separation of variables argument in {\sc Step 2} of the proof of Theorem \ref{theo:stability-estimate-vlasov}. Specifically, as long as $G_\ell\in L^1$ and also $\hat{G}_\ell\in L^1$ ({\it e.g.} by assuming that $G_\ell\in H^{\frac{d}{2}+\varepsilon}$) the Fourier inversion formula of $G_\ell$ allows to write
\begin{align*}
&\int_{\mathbb{R}^{d\ell}}K_\ell(X_{\bar w}[\bar \mu](s,x,\xi),x_1,\ldots,x_\ell)\,d\bar\mu_s^{\xi_1}(x_1)\cdots\,d\bar\mu_s^{\xi_\ell}(x_\ell)\\
&\qquad =\int_{\mathbb{R}^d} e^{2\pi i X_{\bar w}[\bar \mu](s,x,\xi)\cdot z}\,g(s,z,\xi_1)\cdots g(s,z,\xi_\ell)\,\hat{G}_\ell(z)\,dz,
\end{align*}
and the above follows with $\tilde{T}_{w_\ell-\bar w_\ell}[g(s,z_1,\cdot),\ldots g(s,z_\ell,\cdot)]$ replaced by $\tilde{T}_{w_\ell-\bar w_\ell}[g(s,z,\cdot),\ldots g(s,z,\cdot)]$.
\end{rem}

\begin{rem}[On kernels in the Wiener algebra]\label{rem:stability-estimate-why-Fourier-transform}
The proof of Theorem \ref{theo:stability-estimate-vlasov} actually does not require that $\hat{K}_\ell\in L^1(\mathbb{R}^{d(\ell+1)})$, but rather that $K_\ell$ is the Fourier transform of a finite Borel measure. This condition still guarantees the good separation of variables which we exploited above. Specifically, it is enough if $K_\ell\in A(\mathbb{R}^{d(\ell+1)})$, where $A(\mathbb{R}^d)$ represents the Wiener algebra of absolutely convergent Fourier integrals of complex finite Borel measures, that is,
$$A(\mathbb{R}^{d}):=\left\{f:\,f(x)=\int_{\mathbb{R}^{d(\ell+1)}} e^{-2\pi i x\cdot y}\,d\mu(y)\mbox{ for all }x\in \mathbb{R}^d,\,\mu\in \mathcal{M}(\mathbb{R}^d)\right\},$$
with norm $\Vert f\Vert_{A}:=\Vert \mu\Vert_{\rm TV}$, see \cite{LST-12} for further details. In this case, \eqref{eq:stability-estimate-constant-D-infty} can be replaced by
$$
\sum_{\ell=1}^\infty 4^\ell \Vert K_\ell\Vert_{A}<\infty,\qquad \sum_{\ell=1}^\infty 2^\ell (\Vert w_\ell\Vert_{L^\infty}+\Vert \bar w_\ell\Vert_{L^\infty})\Vert K_\ell\Vert_{A}<\infty.
$$
This condition relaxes the integrability assumptions of $K_\ell$ and some of its derivatives as we observe for example in the following radial function 
$$f(x)=1-\frac{1}{(1+4\pi^2|x|^2)^{\alpha/2}},\quad x\in \mathbb{R}^d,$$
for any $0<\alpha<\frac{d}{2}$. Specifically, $f\in A(\mathbb{R}^d)$ with associated finite Borel measure 
$$\mu(x)=\delta_0(x)+G_\alpha(x)\,dx,$$
where $G_\alpha$ is the Bessel kernel
$$G_\alpha(x)=\frac{1}{(4\pi)^{\alpha/2}\Gamma(\frac{\alpha}{2})}\int_0^\infty \frac{e^{-\frac{\pi|x|^2}{r}-\frac{r}{4\pi}}}{r^{\frac{d-\alpha}{2}+1}}dr,\quad x\in \mathbb{R}^d,$$
see \cite[Chapter 5, Section 2, Proposition 2]{S-70}. However, $f\notin H^{\frac{d}{2}+\varepsilon}(\mathbb{R}^d)$ since it does not even decay at infinity. Indeed $f\notin L^2(\mathbb{R}^d)$ and more generally $f\notin \dot{H}^s(\mathbb{R}^d)$ unless $s>\frac{d}{2}-\alpha$.
\end{rem}

\begin{rem}[On the necessity of $p<\infty$]
We also note that  the case $p=\infty$ has been excluded in Theorem \ref{theo:stability-estimate-vlasov}. Whilst the same argument works, the dependence on the labeled cut distance $d_\square(w_\ell,\bar w_\ell;(4^\ell\Vert \hat{K}_\ell\Vert_{L^1})_{\ell\in \mathbb{N}})$ is lost (as it appears raised to the power $1/p$), meaning that Theorem \ref{theo:stability-estimate-vlasov} is no longer able to quantify stability with respect to the involved $w$ in the cut distance.

As discussed in Section \ref{subsubsubsec:mean-field-limits-graphs}, other analogous stability estimates studied thus far both for the case of binary \cite{KX-22}, and higher-order interactions \cite{KX-25} rely on $L^\infty$ estimates, which explains why stability of the Vlasov equation with respect to the graphon in the cut distance (a natural topology in graph limits theory) has not been much explored except probably in the recent paper \cite{BCN-24} and also \cite{JPS-25,JZ-23-arxiv}, but rather with respect to stronger distances.

Slightly changing how we control the factor \eqref{eq:stability-estimate-vlasov-step-separation-variables} in {\sc Step 2}, we may still produce an $L^\infty$ version of the stability estimate involving an $L^\infty_\xi L^1_{\bxi}$ norm on $w_\ell-\bar w_\ell$
$$
d_{\infty,\nu}(\mu_t,\bar\mu_t) \leq  e^{(C_\infty+L_F)t}  \left(d_{\infty,\nu}(\mu_{0},\bar\mu_{0}) +  \frac{1}{L_F}  \left\Vert\sum_{\ell=1}B_\ell\Vert w_\ell-\bar w_\ell\Vert_{L^1_{\bxi}}\right\Vert_{L^\infty_\xi} \right),
$$
for all $t\geq 0$, or a weaker uniform bounded-Lipschitz distance as in \cite{KX-22,KX-25}, when restricted to solutions continuous in $\xi\in [0,1]$. Nevertheless, in doing so we would no longer be in position to exploit the natural compactness property of the cut distance studied in \cite{RS-24}.
\end{rem}

\medskip

Following the same train of thoughts as in Corollary \ref{cor:well-posedness-continuum-limit-equation}, we note that Theorem \ref{theo:stability-estimate-vlasov} can also be applied to special solutions $\mu_t^\xi=\delta_{X(t,\xi)}$ and $\bar\mu_t^\xi=\delta_{\bar X(t,\xi)}$ to \eqref{eq:vlasov-equation}-\eqref{eq:vlasov-equation-force}, where $X$ and $\bar X$ solve the corresponding continuum-limit equation \eqref{eq:continuum-limit-equation}. 

\begin{cor}[Stability of the continuum-limit equation]\label{cor:stability-continuum-limit-equation}
Under the hypothesis in Theorem \ref{theo:stability-estimate-vlasov}, consider $X,\bar X\in C(\mathbb{R}_+,L^p([0,1],\mathbb{R}^d))$ solutions to the continuum-limit equation \eqref{eq:continuum-limit-equation} with given $w$ and $\bar w$. Then, we have
$$\Vert X(t,\cdot)-\bar X(t,\cdot)\Vert_{L^p}\leq e^{(C_p+L_F)t}\Bigg(\Vert X_0-\bar X_0\Vert_{L^p}+\frac{D_\infty^{1/q}}{L_F}d_\square(w,\bar w;(4^\ell\Vert \hat K_\ell\Vert_{L^1})_{\ell\in \mathbb{N}})^{1/p}\Bigg),$$
for all $t\in \mathbb{R}_+$.
\end{cor}

We refer to \cite{M-14} and also \cite{AP-21,AP-23-arxiv,GKX-23,PT-22-arxiv} for stability estimates in analogous continuum-limit equations restricted to binary interactions only. However, we emphasize that, to the best of our knowledge, Corollary \ref{cor:stability-continuum-limit-equation} is the first result that allows for a graph-limit based distance in place of the stronger $L^p$ based distances in previous literature.

\section{Mean-field limit}\label{sec:proof-main-result}
In this section we focus on the proof of the main result of the paper, namely, Theorem \ref{theo:main}, which we state below in a more rigorous formulation for clarity of the presentation.

\begin{theo}[Mean-field limit]\label{theo:main-rigorous-formulation}
Assume that the kernels $K_\ell$ and the weights $w_{ij_1\ldots j_\ell}^{\ell,N}$ satisfy Assumptions  \ref{hyp:main-kernels}, \ref{hyp:main-weights} and \ref{hyp:main-symmetry}. For any $(X_{1,0}^N,\ldots,X_{N,0}^N)$ with {\it i.i.d.} $X_{i,0}^N$ (but with law possibly depending on $N$) satisfying Assumption \ref{hyp:main-initial-data}, consider the unique solutions $(X_1^N,\ldots,X_N^N)$ to \eqref{eq:multi-agent-system}, and define the pair $(\mu^N,w^N)$ as in Definition \ref{defi:graphon-reformulation}, {\it i.e.},
\begin{align*}
\mu^{N,\xi}_t&:=\sum_{i=1}^N\mathds{1}_{I_i^N}(\xi)\delta_{X_i^N(t)},\quad \xi\in [0,1],\\
w_\ell^{N}(\xi,\xi_1,\ldots,\xi_\ell)&:=\sum_{i,j_1,\ldots,j_\ell=1}^N \mathds{1}_{I_i^N\times I_{j_1}^N\times\cdots\times I_{j_\ell}^N}(\xi,\xi_1,\ldots,\xi_\ell)\,N^\ell w^{\ell,N}_{ij_1\cdots j_\ell},\quad \xi,\xi_1,\cdots,\xi_\ell\in [0,1].
\end{align*}
Then, there are a subsequence $N_k\to \infty$, some measure-preserving maps $\Phi_k:[0,1]\longrightarrow [0,1]$ and a distributional solution \corr{$\mu\in C(\mathbb{R}_+,\mathcal{P}_\nu(\mathbb{R}^d\times [0,1])\mbox{-}d_{p,\nu})$} to \eqref{eq:vlasov-equation-joint}-\eqref{eq:vlasov-equation-force-joint} for some UR-hypergraphon $w=(w_\ell)_{\ell\in \mathbb{N}}\in \mathcal{H}_W$ such that
\begin{equation}\label{eq:main-result-rigorous-formulation-convergence}
\sup_{t\in [0,T]}\,d_{p,\nu}(\corr{\mathbb{E}\mu_t^{N_k,\Phi_k}},\mu_t)\to 0,\quad d_\square (w^{N_k,\Phi_k},w)\to 0
\end{equation}
for all $T\in \mathbb{R}_+$, where $\mu_t^{N_k,\Phi_k}$ and $w^{N_k,\Phi_k}=(w^{N_k,\Phi_k}_\ell)_{\ell\in \mathbb{N}}$ represent the rearrangements
$$\mu_t^{N_k,\Phi_k,\xi}=\mu_t^{N_k,\Phi_k(\xi)},\quad w^{N_k,\Phi_k}_\ell(\xi,\xi_1,\ldots,\xi_\ell)=w^{N_k}(\Phi_k(\xi),\Phi_k(\xi_1),\ldots,\Phi_k(\xi_\ell)).$$
\end{theo}

\begin{rem}[On the {\it i.i.d.} assumption of initial data]\label{rem:main-why-iid-initial-data}
We observe that whilst throughout the paper initial data $(X_{1,0}^N,\ldots,X_{N,0}^N)$ were only assumed independent, in the statement of our main Theorem \ref{theo:main-rigorous-formulation} we further impose that they are identically distributed according to the same law (possibly depending on $N$ but not $i$). Similar assumptions are considered in previous literature, see \cite{BCN-24}. As we shall show, this hypothesis is only a sufficient condition which helps proving compactness of initial data in {\sc Step 1} of the proof, as in that case initial data consist of measures with constant disintegrations, and then compactness in \corr{$(\mathcal{P}_\nu(\mathbb{R}^d\times [0,1]),d_{p,\nu})$} amounts to compactness in $(\mathcal{P}(\mathbb{R}^d),d_{\rm BL})$. However, it could be possible to remove the {\it i.i.d.} property and also prove compactness of non-constant initial data, in a compatible way with the compactness of the UR-hypergraphons $w^N$ associated with the finite weights $w_{ij_1\cdots j_\ell}^{\ell,N}$. We leave this question for future work.
\end{rem}

\corr{
\begin{rem}[On a constructive convergence result]
Theorem \ref{theo:main-rigorous-formulation} presents a mean-field convergence result by compactness, which implies that the limiting initial data and the limiting UR-hypergraphon are guaranteed to exist under mild scaling conditions on the initial data and initial weights, but they are not given explicitly. As an alternative, one could provide a constructive result, starting with a limit hypergraphon and discretizing it, as in Proposition \ref{prop:conv_hypergraph_to_graphon1} or Proposition \ref{prop:conv_hypergraph_to_graphon2}, and prove the mean-field limit by leveraging the same key stability estimate of Theorem \ref{theo:stability-estimate-vlasov}. If both the convergence of the initial data and that of the hypergraphs to the limiting objetcs had known rates, then the convergence of $\mu^N_t$ would also become quantitative.
\end{rem}
}

\begin{proof}[Proof of Theorem \ref{theo:main-rigorous-formulation}]
Let $(\bar X_1^N,\ldots,\bar X_N^N)$ be the unique solution to \eqref{eq:multi-agent-system-mckean-sde} in Lemma \ref{lem:multi-agent-system-mckean-sde-well-posedness}, $(\bar \lambda^{N,i})_{1\leq i\leq N}$ their associated laws \eqref{eq:laws-of-bar-Xi}, and $(\bar \mu^{N,\xi})_{\xi\in [0,1]}$ the graphon reformulation in Definition \ref{defi:graphon-reformulation}, {\it i.e.}, 
$$\bar\mu^{N,\xi}_t:=\sum_{i=1}^N\mathds{1}_{I_i^N}(\xi)\bar\lambda^{N,i}_t,\quad \xi\in [0,1].$$
The proof is divided into four steps. First, we define the limiting $\mu$ and $w$ solving the Vlasov equation \eqref{eq:vlasov-equation-joint}-\eqref{eq:vlasov-equation-force-joint}. Second, we study the fluctuation of fibered measures $\corr{\mathbb{E}\mu^N}$ associated with the original multi-agent system \eqref{eq:multi-agent-system} with respect to the fibered measures $\bar\mu^N$ associated with the intermediate multi-agent system \eqref{eq:multi-agent-system-mckean-sde}. Third, we study the fluctuations of $\bar\mu^N$ with respect to the obtained limiting solution $\mu$. Finally, we compare $\corr{\mathbb{E}\mu^N}$ and $\mu$.

\medskip

$\diamond$ {\sc Step 1}: Definition of the limiting $w$ and $\mu$.\\
Regarding the definition of the limiting $w$, we note that by assumptions \eqref{eq:hypothesis-weights-uniform-bound} and \eqref{eq:hypothesis-weights-full-symmetry} we have that $w^N=(w^N_\ell)_{\ell\in \mathbb{N}}\in \mathcal{H}_W$ with uniform bound $W>0$. Therefore, by \cite[Lemma 33]{RS-24} there exists $w=(w_\ell)_{\ell\in \mathbb{N}}\in \mathcal{H}_W$ with the same $W$ such that $\delta_\square(w^{N_k},w)\to 0$ for a suitable subsequence $N_k$. Since $\delta_\square$ represents the unlabeled cut distance, the above means that there exists $\Phi_k:[0,1]\longrightarrow [0,1]$ measure-preserving maps such that 
\begin{equation}\label{eq:main-rigorous-formulation-compactness-UR-hypergraphon}
d_\square(w^{N_k,\Phi_k},w)\to 0,
\end{equation}
where $d_\square$ is the labeled cut distance and $w^{N_k,\Phi_k}=(w^{N_k,\Phi_k}_\ell)_{\ell\in \mathbb{N}}$ is the rearranged UR-hypergraphon
$$w^{N_k,\Phi_k}_\ell(\xi,\xi_1,\ldots,\xi_\ell)=w^{N_k}(\Phi_k(\xi),\Phi_k(\xi_1),\ldots,\Phi_k(\xi_\ell)).$$ 
For practical reasons we similarly define the rearranged $\mu^{N_k,\Phi_k}$ and $\bar\mu^{N_k,\Phi_k}$ as follows:
$$\mu_t^{N_k,\Phi_k,\xi}=\mu_t^{N_k,\Phi_k(\xi)},\quad \bar\mu_t^{N_k,\Phi_k,\xi}=\bar\mu_t^{N_k,\Phi_k(\xi)}.$$

Regarding the definition of the limiting $\mu$, we first prescribe its initial datum $\mu_0$. Since $(X_{i,0}^N)_{1\leq i\leq N}$ are distributed identically according to some $\lambda^N\in \mathcal{P}(\mathbb{R}^d)$, then $\bar\mu^N_0$ has constant disintegrations, and for this reason so does the rearranged $\bar\mu^{N_k,\Phi_k}_0$ namely, 
$$\bar\mu^{N_k,\Phi_k,\xi}_0=\lambda^{N_k},\quad \xi\in [0,1],$$
for all $k\in \mathbb{N}$. We also note that by the control \eqref{eq:hypothesis-initial-data} of the first order moments in Assumption \ref{hyp:main-initial-data} we have $\sup_{k\in \mathbb{N}}\int_{\mathbb{R}^d}|x|\,d\lambda^{N_k}(x)<\infty$, and therefore $(\lambda^{N_k})_{k\in \mathbb{N}}$ is uniformly tight. Hence, by Prokhorov's theorem there is some subsequence, which we still denote by the same letter for simplicity, and there is $\lambda\in \mathcal{P}(\mathbb{R}^d)$ such that $d_{\rm BL}(\lambda^{N_k},\lambda)\to 0$. We define the initial datum
$$\mu_0^\xi:=\lambda,\quad \xi\in [0,1],$$
which is also constant in $\xi\in [0,1]$. Note that this implies
\begin{equation}\label{eq:main-rigorous-formulation-compactness-initial-datum}
d_{p,\nu}(\bar\mu^{N_k,\Phi_k},\mu_0)\to 0.
\end{equation}
We finally define \corr{$\mu \in C(\mathbb{R}_+,\mathcal{P}_\nu(\mathbb{R}^d\times [0,1])\mbox{-}d_{p,\nu})$} to be the distributional solution to \eqref{eq:vlasov-equation-joint}-\eqref{eq:vlasov-equation-force-joint} with given $w$ and initial datum $\mu_0$ as in Theorem \ref{theo:well-posedness-vlasov}.

\medskip

\corr{
$\diamond$ {\sc Step 2}: Comparing $\mathbb{E}\mu^N$ and $\bar\mu^N$.\\
Given $\xi\in [0,1]$, let us fix $i\in\llbracket1,N\rrbracket$ such that $\xi\in I_i^N$. This implies that $\mu^{N,\xi}_t=\delta_{X_i^N(t)}$ and $\bar\mu^{N,\xi}_t=\bar\lambda_t^{N,i}=\mathbb{E}\delta_{\bar X_i^N(t)}$. Then, 
\begin{align*}
d_{\rm BL}(\mathbb{E}\mu^{N,\xi}_t,\bar\mu^{N,\xi}_t) & = \sup_{\Vert \phi\Vert_{\rm BL}\leq 1} \int_{\R^d} \phi(x) \,d(\mathbb{E}\mu^{N,\xi}_t(x)-\bar\mu^{N,\xi}_t(x))\\
& = \sup_{\Vert \phi\Vert_{\rm BL}\leq 1} \int_{\R^d} \phi(x) \,d(\mathbb{E}\delta_{X_i^N(t)}(x)-\bar\lambda_t^{N,i}(x))\\
& = \sup_{\Vert \phi\Vert_{\rm BL}\leq 1} \int_{\R^d} \phi(x) \,d\mathbb{E}\left(\delta_{X_i^N(t)}-\delta_{\bar X_i^N(t)}\right)(x)\\
& = \sup_{\Vert \phi\Vert_{\rm BL}\leq 1} \mathbb{E}(\phi(X_i^N(t)) - \phi(\bar X_i^N(t)))\\
& \leq \mathbb{E}|X_i^N(t) - \bar X_i^N(t)|.
\end{align*}
Thus, taking $L^p$ norms with respect to $\xi$ we obtain
\begin{equation}\label{eq:main-rigorous-formulation-propagation-independence}
\,d_{p,\nu}(\mathbb{E}\mu_t^N,\bar \mu_t^N)= \left(\sum_{i=1}^N\int_{I_i^N}d_{\rm BL}(\mathbb{E}\mu^{N,\xi}_t,\bar\mu^{N,\xi}_t)^p d\xi\right)^{1/p}\leq\left(\frac{1}{N}\sum_{i=1}^N(\mathbb{E}|X_i^N(t)-\bar X_i^N(t)|)^p\right)^{1/p}.
\end{equation}
}
\medskip

$\diamond$ {\sc Step 3}: Comparing $\bar\mu^{N_k,\Phi_k}$ and $\mu$.\\
First, since $\bar\mu^{N_k}$ is a distributional solution to \eqref{eq:vlasov-equation-joint}-\eqref{eq:vlasov-equation-force-joint} given $w^{N_k}$ by virtue of Lemma \ref{lem:graphon-reformulation-solves-Vlasov}, and since $\Phi_k$ are measure-preserving maps, then  $\bar\mu^{N_k,\Phi_k}$ is also a distributional solution to \eqref{eq:vlasov-equation-joint}-\eqref{eq:vlasov-equation-force-joint} given $w^{N_k,\Phi_k}$. Therefore, we can use the stability estimate in Theorem \ref{theo:stability-estimate-vlasov} to compare it to the distributional solution $\mu$ with given $w$ defined in {\sc Step 1}. Specifically, we obtain
\begin{equation}\label{eq:main-rigorous-formulation-stability-estimate}
\begin{aligned}
&\sup_{t\in [0,T]}d_{p,\nu}(\bar\mu_t^{N_k,\Phi_k},\mu_t)\\
&\qquad\leq  e^{(C_p+L_F)T}  \left(d_{p,\nu}(\bar\mu_{0}^{N_k,\Phi_k},\mu_{0}) +  \frac{D_\infty^{1/q}}{L_F} d_\square(w^{N_k,\Phi_k},w; (2^\ell\|\hat K_\ell\|_{L^1})_{\ell\in \mathbb{N}})^{1/p} \right),
\end{aligned}
\end{equation}
for all $T\in \mathbb{R}_+$.

\medskip

$\diamond$ {\sc Step 4}: Comparing $\corr{\mathbb{E}\mu^{N_k,\Phi_k}}$ and $\mu$.\\
Using the triangle inequality and taking expectations and supremum over $t\in [0,T]$ we have
\corr{\begin{align*}
&\sup_{t\in [0,T]} d_{p,\nu}(\mathbb{E}\mu^{N_k,\Phi_k}_t,\mu_t)\\
&\qquad\leq \sup_{t\in [0,T]} d_{p,\nu}(\mathbb{E}\mu^{N_k,\Phi_k}_t,\bar \mu^{N_k,\Phi_k}_t)+\sup_{t\in [0,T]} d_{p,\nu}(\bar \mu^{N_k,\Phi_k},\mu_t)\\
&\qquad=\sup_{t\in [0,T]}d_{p,\nu}(\mathbb{E}\mu^{N_k}_t,\bar \mu^{N_k}_t)+\sup_{t\in [0,T]} d_{p,\nu}(\bar \mu^{N_k,\Phi_k},\mu_t)\\
&\qquad\leq \sup_{t\in [0,T]}\left(\frac{1}{N}\sum_{i=1}^N(\mathbb{E}|X_i^N(t)-\bar X_i^N(t)|)^p\right)^{1/p}\\
&\qquad\qquad + e^{(C_p+L_F)T}  \left(d_{p,\nu}(\bar\mu_{0}^{N_k,\Phi_k},\mu_{0}) +  \frac{D_\infty^{1/q}}{L_F} d_\square(w^{N_k,\Phi_k},w; (2^\ell\|\hat K_\ell\|_{L^1})_{\ell\in \mathbb{N}})^{1/p} \right),
    \end{align*}
    }
where in the second step we have used that $L^p$ norms are invariant under measure-preserving rearrangements, and in the last step we have used the control \eqref{eq:main-rigorous-formulation-propagation-independence} from {\sc Step 2} and \eqref{eq:main-rigorous-formulation-stability-estimate} from {\sc Step 3}. Using the convergence of initial data in \eqref{eq:main-rigorous-formulation-compactness-initial-datum} and the convergence of UR-hypergraphons in \eqref{eq:main-rigorous-formulation-compactness-UR-hypergraphon} which we have by construction of $\mu_0$ and $w$, along with Lemma \ref{lem:multi-agent-system-mckean-sde-error-estimate} and more particularly the decay in Remark \ref{rem:propagation-independence}, we end the proof.
\end{proof}

The above mean-field limit in Theorem \ref{theo:main-rigorous-formulation} is precise enough to retain the full structure in the limit as $N\to \infty$, and not only averaged obversables like in \cite{JPS-25} for extended graphons, and \cite{BCN-24} for dense random graphs. However, by virtue of Proposition \ref{pro:fibered-probability-measures-Lpdbl-distance-relation-to-others} it is shown that
$$d_{\rm BL}(\corr{\mathbb{E}\pi_{x\#}\mu_t^{N_k,\Phi_k}},\pi_{x\#}\mu_t)\leq \,d_{p,\nu}(\corr{\mathbb{E}\mu_t^{N_k,\Phi_k}},\mu_t).$$
Since $\pi_{x\#}\mu_t^{N_k,\Phi_k}=\frac{1}{N}\sum_{i=1}^N\delta_{X_i^{N_k}(t)}$ reduces to the standard empirical measures, then Theorem \ref{theo:main-rigorous-formulation} has the following straightforward consequence.

\begin{cor}[Convergence of the empirical measures]\label{cor:main-limit-empirical-measures}
Under the assumptions in Theorem \ref{theo:main-rigorous-formulation}, the standard empirical measures $\hat{\mu}^{N_k}_t:=\frac{1}{N}\sum_{i=1}^N\delta_{X_i^{N_k}(t)}$ satisfy
$$\sup_{t\in [0,T]}\,d_{\rm BL}(\corr{\mathbb{E}\hat{\mu}_t^{N_k}},\hat{\mu}_t)\to 0,$$
for all $T\in \mathbb{R}_+$, and $\hat{\mu}_t=\pi_{x\#}\mu_t\in C(\mathbb{R}_+,\mathcal{P}(\mathbb{R}^d)\mbox{-narrow})$ solves
\begin{equation}\label{eq:vlasov-equation-limit-empirical-measures}
\partial_t\hat{\mu}_t+\divop_x\left(\int_0^1 F_w[\mu_t](\cdot,\xi)\,\mu_t^\xi\,d\xi\right)=0,\quad t\geq 0,\,x\in \mathbb{R}^d.
\end{equation}
\end{cor}

\begin{rem}[The exchangeable case]
The equation \eqref{eq:vlasov-equation-limit-empirical-measures} for the limit  $\hat{\mu}$ of empirical measures is generally not closed as it depends on the full structured limit $\mu$. However, for the all-to-all UR-hypergraphon $w_\ell\equiv 1$, the system becomes exchangeable and $\hat{\mu}$ solves a closed system
$$
\begin{cases}
\displaystyle\partial_t\hat{\mu}_t+\divop_x(\hat{F}[\hat{\mu}_t]\hat{\mu}_t)=0,\quad t\geq 0,\,x\in \mathbb{R}^d,\\
\displaystyle\hat{F}[\hat{\mu}_t](x):=\sum_{\ell=1}^\infty\int_{\mathbb{R}^{d\ell}}K_\ell(x,x_1,\ldots,x_\ell)\,d\hat{\mu}_t(x_1)\ldots\,d\hat{\mu}_t(x_\ell),
\end{cases}
$$
which corresponds to the standard Vlasov equation for exchangeable systems.
\end{rem}

For the general non-exchangeable case, the above observable $\hat{\mu}_t=\pi_{x\#}\mu_t$ is only the first level of a larger hierarchy of observables indexed by directed hypertrees. Namely, the unsigned finite measure in the divergence term in \eqref{eq:vlasov-equation-limit-empirical-measures} actually has a nice structure, and its evolution involes higher-order observables that are easy to characterize. Collecting all the family of involved observables we have the following result, which can be seen as a higher-order extension of the hierarchy of observables indexed by trees formulated in the binary case in \cite{JPS-25,JZ-23-arxiv}.

\begin{rem}[A hierarchy indexed by hypertrees]\label{rem:main-hierarchy-observables}
Under the assumptions in Theorem \ref{theo:main-rigorous-formulation}, let us define the family of observables associated with the limit $(w,\mu)$
\begin{equation}\label{eq:hierarchy-hypertrees-definition}
\begin{aligned}
&\tau(H,w,\mu_t)(x_1,\ldots,x_{\# V(H)})\\
&\qquad:=\int_{[0,1]^{\# V(H)}} \prod_{\ell\in \mathbb{N}}\prod_{(i;j_1,\ldots,j_\ell)\in E(H)}w_\ell(\xi_i,\xi_{j_1},\ldots,\xi_{j_\ell})\,\prod_{i\in V(H)}\mu_t^{\xi_i}(x_i)\prod_{i\in V(H)}\,d\xi_i,
\end{aligned}
\end{equation}
for $H$ any directed hypertree. Then, $\tau(H,w,\mu)$ solves the hierarchy of coupled equations
\begin{multline}\label{eq:hierarchy-hypertrees-equations}
\partial_t\tau(H,w,\mu_t)(x_1,\ldots,x_{\# V(H)})\\
+\sum_{i=1}^{\# V(H)}\sum_{\ell=1}^\infty \divop_{x_i}\left(\int_{\mathbb{R}^{d\ell}}K_\ell(x_i,y_1,\ldots,y_\ell)\,\tau (H_i^\ell,w,\mu_t)(x_1,\ldots,x_{\# V(H)},dy_1,\ldots,dy_\ell)\right)=0,
\end{multline}
in distributional sense, where $H_i^\ell$ consist in the new directed hypertree constructed from $H$ by adding a directed hyperedge of cardinality $\ell+1$ to the node $i\in V(H)$ (see Figure \ref{fig:hypertrees}).
\end{rem}

\begin{figure}[t]
\begin{center}
\includegraphics[width=0.48\textwidth]{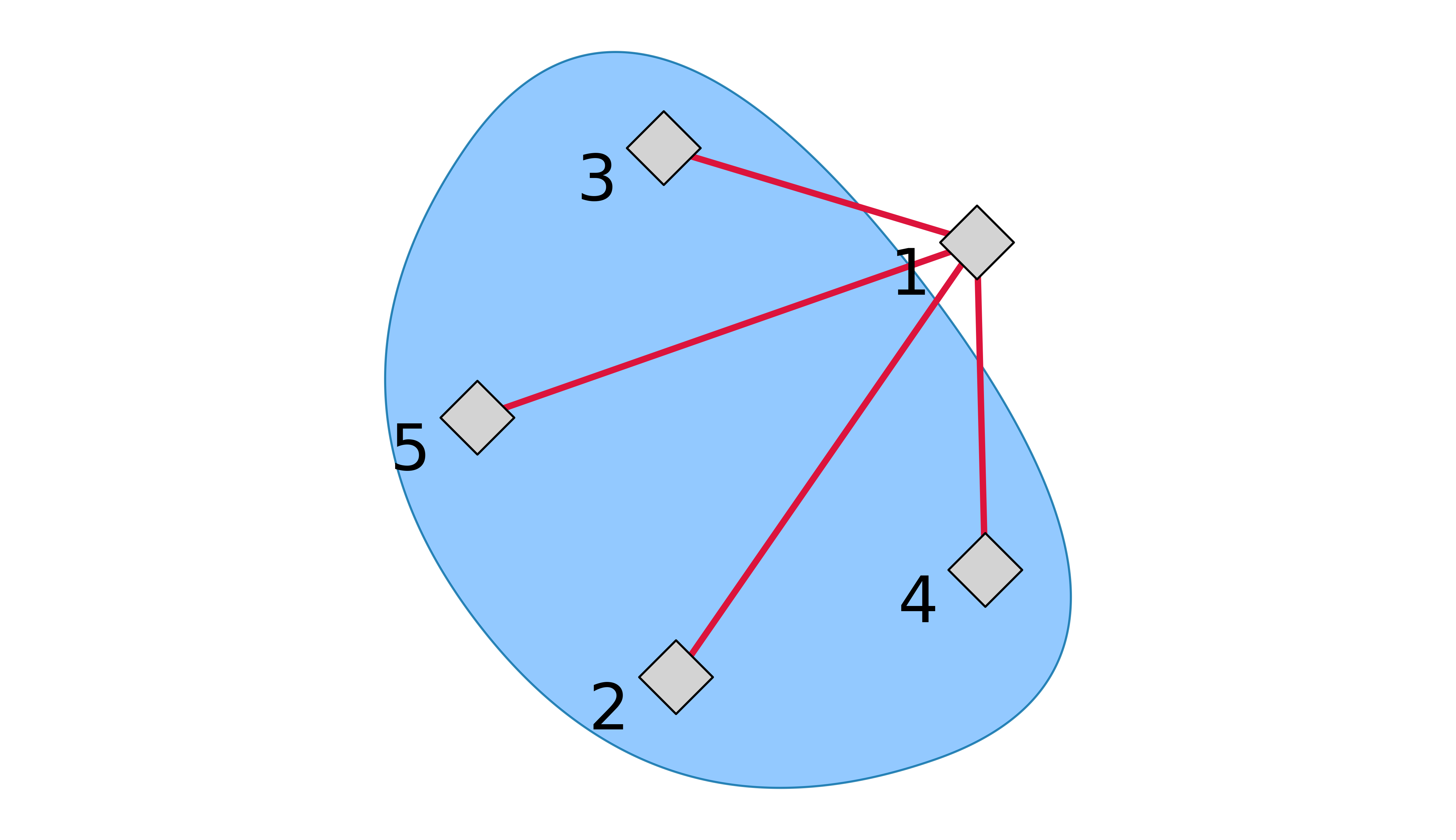}
\includegraphics[width=0.48\textwidth]{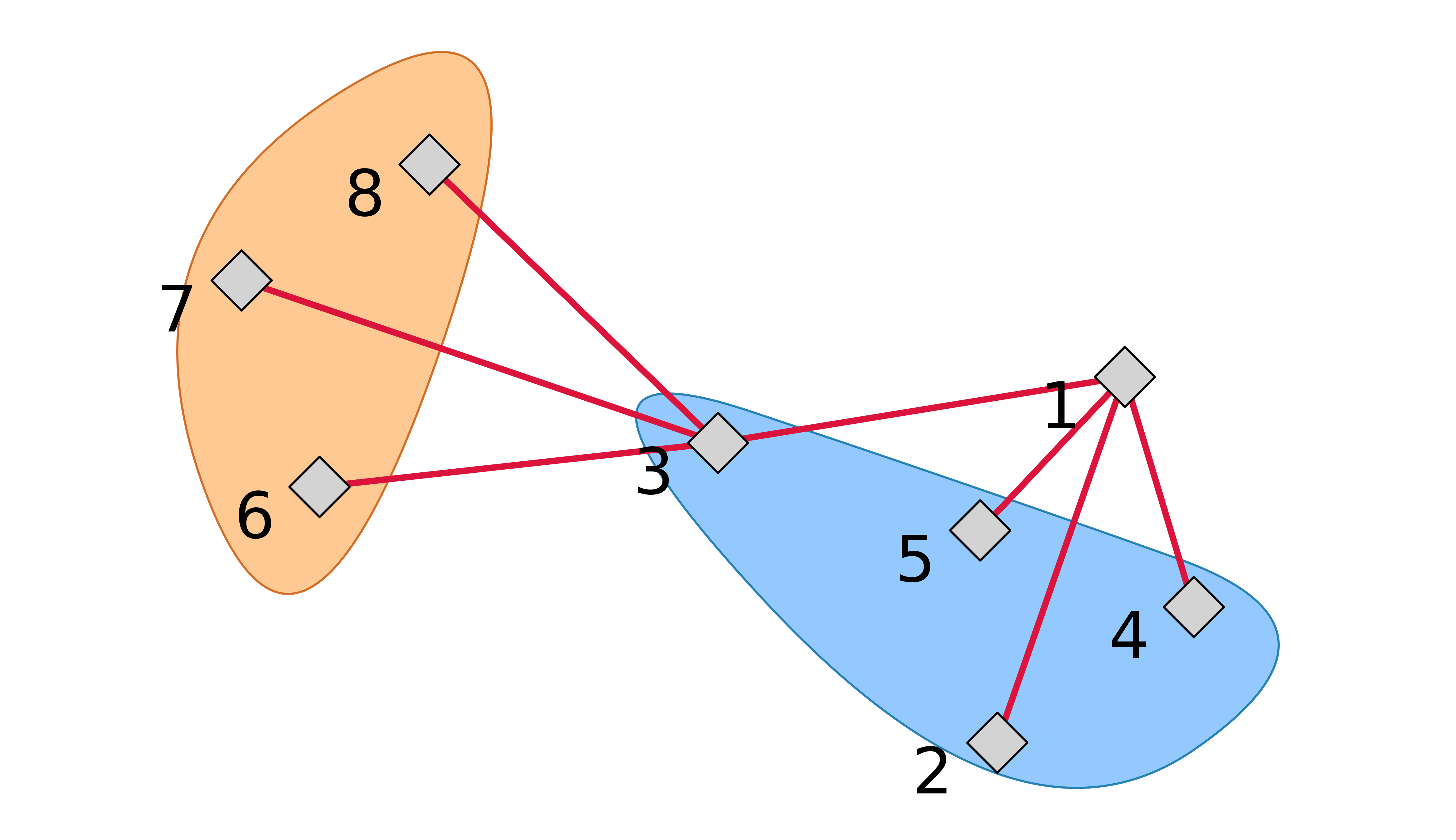}
\caption{Directed hypertree $H$ with nodes $V(H)=\{1,2,3,4,5\}$ and a single directed hyperedge $e=\{1;2,3,4,5\}$, and associated directed hypertree $H_3^3$.}
\label{fig:hypertrees}
\end{center}
\end{figure}

In \cite{JPS-25}, the space of pairs $(w,\mu)$ was endowed with the topology of the convergence of the observables \eqref{eq:hierarchy-hypertrees-definition}. In the present paper, pairs $(w,\mu)$ are endowed with the product topology of the cut-distance on the set of UR-hypergraphons $\mathcal{H}_W$ and the $d_{p,\nu}$ distance on the space of probability measures \corr{$\mathcal{P}_\nu(\mathbb{R}^d\times [0,1])$}. At this stage, it is not clear which of them is finer. In particular, an interesting open question is to characterize whether $\tau(H,w^{N_k},\mu^{N_k})\to\tau(H,w,\mu)$ is implied by the convergence \eqref{eq:main-result-rigorous-formulation-convergence}, or conversely. We leave this question to future work.

\section{Numerical simulations}\label{sec:numerics}

To illustrate the convergence result of Theorem \ref{theo:main}, we provide some numerical simulations based on two examples of hypergraphs presented in Section \ref{subsec:examples}. 
For computational reasons, we set dimension $d=1$, and we only consider hypergraphs of bounded rank $r=3$.

\subsection{Hypergraph for a homogeneous group of rank $3$}

Consider the particle system \eqref{eq:multi-agent-system} posed on the hypergraph for a homogeneous group of rank $3$, that is,
\begin{equation}\label{eq:hypergraph_homogeneous_sim}
w^{2,N}_{ij_1j_2} = \begin{cases}
\displaystyle  \frac{1}{N^2}, & \text{if } \quad \max_{k_1,k_2\in \{i,j_1,j_2\}} |k_1-k_2| \leq \theta N,\\
0, & \text{otherwise},
\end{cases}
\end{equation}
for some $\theta\in (0,1)$, and $w^{\ell,N}_{i j_1\cdots j_\ell} = 0$ for all $\ell\neq 2$ and $i, j_1,\ldots, j_\ell\in\llbracket 1,N\rrbracket$. This hypergraph was introduced in Equation \eqref{eq:hypergraph_homogeneous} in Section \ref{subsec:examples}, and is represented in Figure \ref{fig:Hypergraph-nearestneighbor} (right).

We also consider the linear interaction kernel given by 
\begin{equation}
\begin{array}{cccl}
K_2(x,x_1,x_2)=\frac{x_1+x_2}{2}-x.
\end{array}
\label{eq:K2}
\end{equation} 

A slight simplification of the proof of Proposition \ref{prop:convergence-discontinuous} allows to prove that the above sequence of hypergraphs \eqref{eq:hypergraph_homogeneous_sim} converges as $N$ tends to infinity to the limit UR-hypergraphon (with actually bounded rank) given by $w_\ell\equiv 0$ for all $\ell\neq 2$, and 
\begin{equation}\label{eq:hypergraphon_homogeneous_sim}
w_{2}(\xi_0,\xi_1,\xi_2) = \begin{cases}
\displaystyle  1, & \text{if } \quad \max_{i,j\in \{0,1,2\}} |\xi_{i}-\xi_{j}| \leq \theta,\\
0, & \text{otherwise.}
\end{cases}
\end{equation}
In all that follows, the parameter $\theta$ is fixed to the numerical value $\theta=0.1$. Arguing as our main Theorem \ref{theo:main-rigorous-formulation}, we then have that the solution $(X_i^N)_{i\in\{1,\ldots,N\}}$ to \eqref{eq:multi-agent-system}-\eqref{eq:hypergraph_homogeneous_sim}-\eqref{eq:K2} converges as $N$ tends to infinity to the solution $\mu$ to the Vlasov equation \eqref{eq:vlasov-equation}-\eqref{eq:hypergraphon_homogeneous_sim}-\eqref{eq:K2}.

In Figure \ref{fig:mut}, the solution $\mu$ to \eqref{eq:vlasov-equation}-\eqref{eq:hypergraphon_homogeneous_sim}-\eqref{eq:K2} with uniform initial condition $\mu_0 = d\xi_{\lfloor [0,1]}\, dx_{\lfloor [0,1]}$ is represented for different time steps.
The solution was computed using a finite difference scheme and numerical integration.
Notice that due to the attractive nature of the interaction kernel $K_2$ in \eqref{eq:K2}, we observe a concentration phenomenon. Moreover, the convergence speed seems faster for the central labels $\xi\in [\theta, 1-\theta]$ and slower for edge labels $\xi$ close to $0$ or $1$. 
This differentiated convergence speed can be explained by the hypergraphon $w_2$ in \eqref{eq:hypergraphon_homogeneous_sim}, which promotes interaction for central labels, and demotes interactions for edge labels. 
Indeed, an agent with label $\xi\in [\theta, 1-\theta]$ belongs to all hyperedges $(\xi,\xi_1,\xi_2)$ such that $\xi_1, \xi_2\in [\xi-\theta,\xi+\theta]$.
On the other hand, an agent with label $\xi=0$ belongs to hyperedges $(\xi,\xi_1,\xi_2)$ such that $\xi_1, \xi_2\in [0,\theta]$, which reduces its chances of interaction.

\begin{figure}[h!]
\includegraphics[width = 0.24\textwidth]{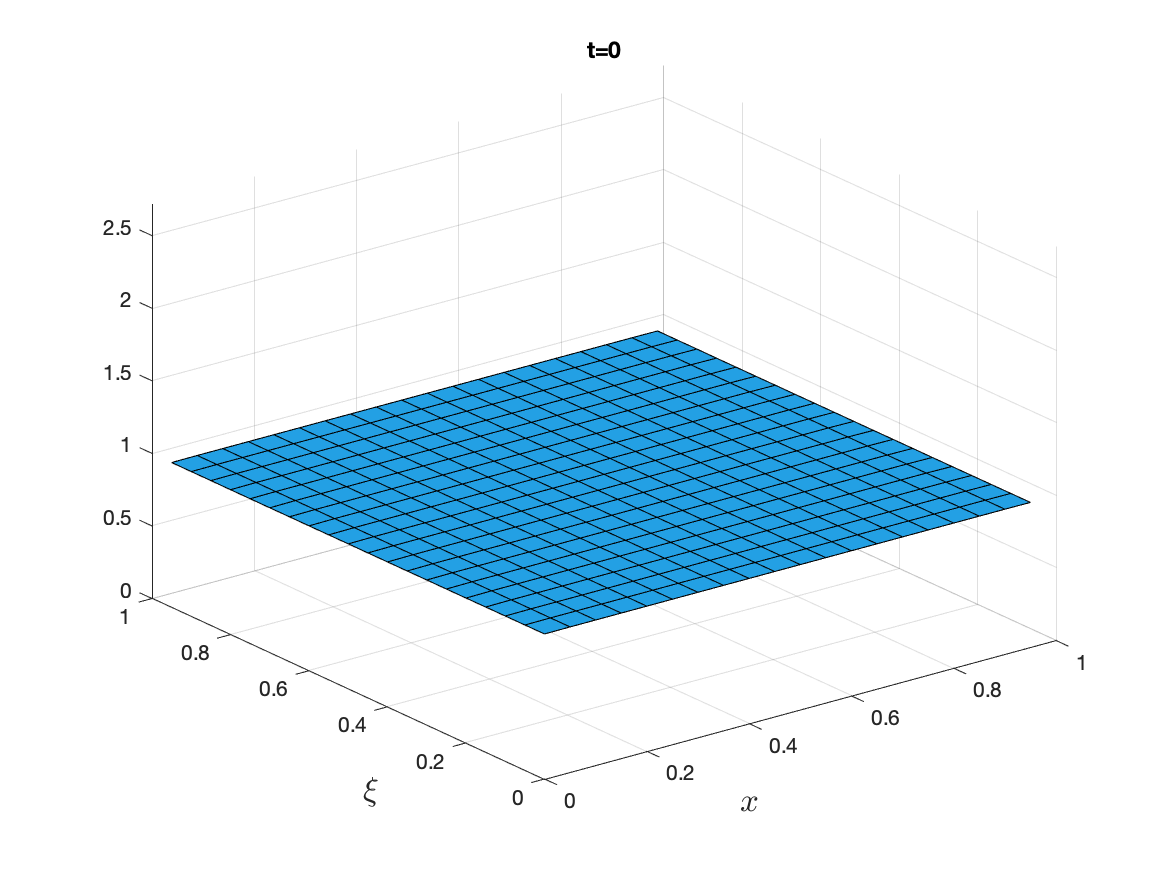}
\includegraphics[width = 0.24\textwidth]{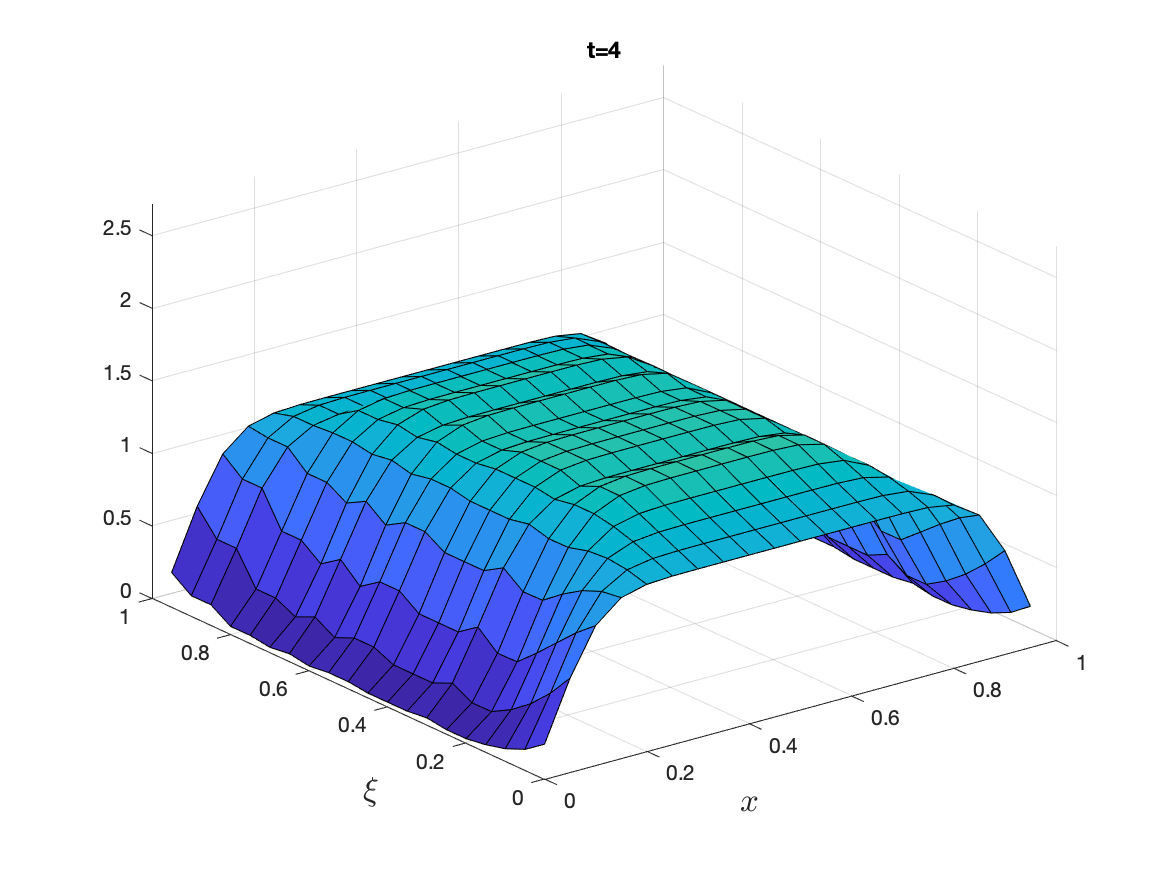}
\includegraphics[width = 0.24\textwidth]{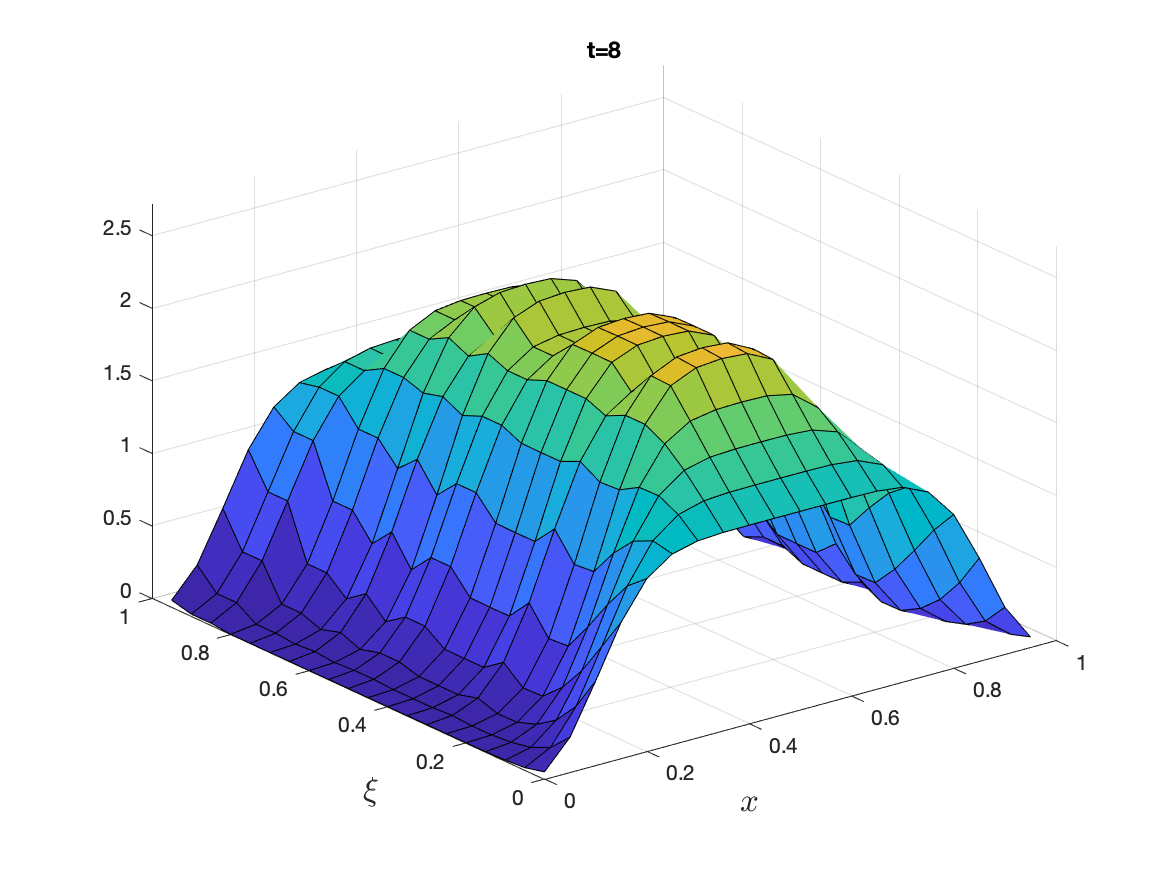}
\includegraphics[width = 0.24\textwidth]{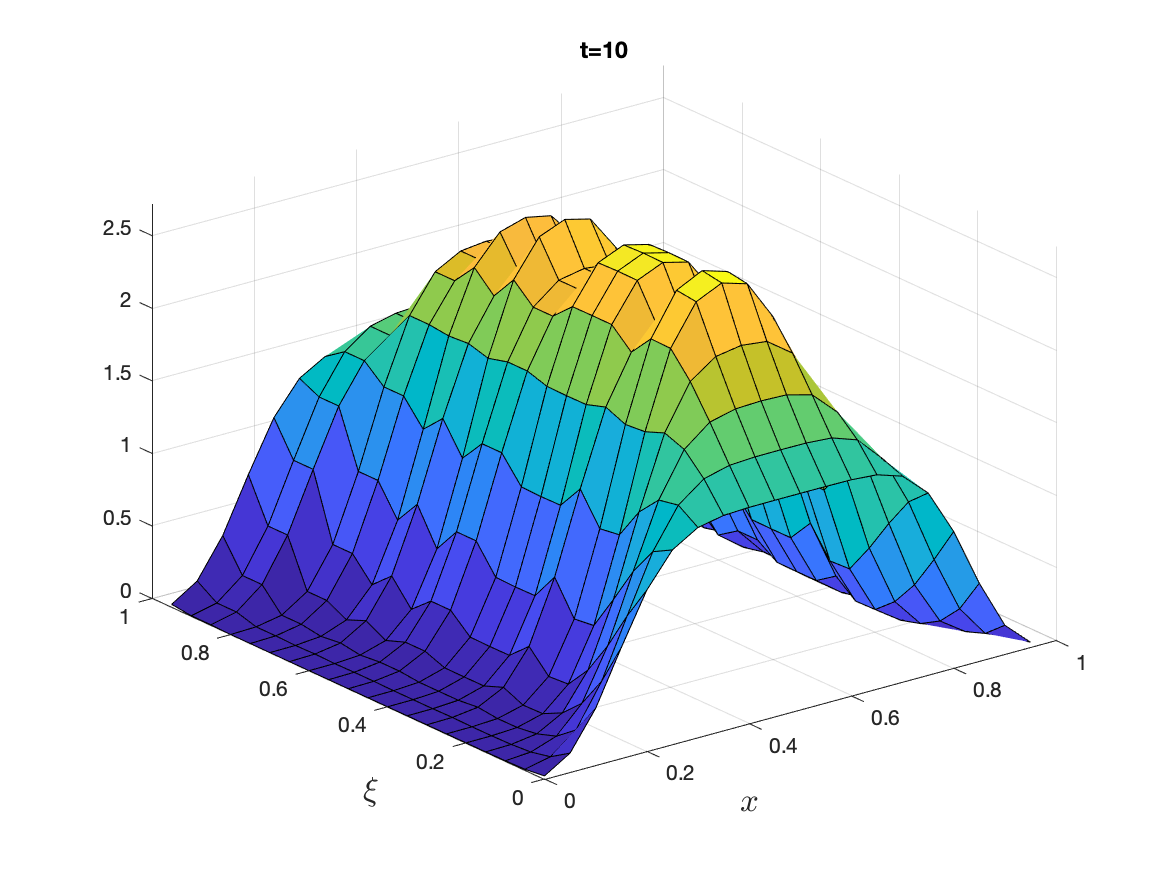}
\caption{Evolution of $\mu_t$, solution to the limit equation \eqref{eq:vlasov-equation}-\eqref{eq:hypergraphon_homogeneous_sim}-\eqref{eq:K2}, at times $t=0$, $t=4$, $t=8$ and $t=10$.}
\label{fig:mut}
\end{figure}

To compare the solutions of the discrete and continuum systems, in Figure \ref{fig:mut_xt} we superimpose such a solution $\mu_t$ to \eqref{eq:vlasov-equation}-\eqref{eq:hypergraphon_homogeneous_sim}-\eqref{eq:K2} with uniform initial condition on $[0,1]\times [0,1]$ (represented by a color gradient) with the solution $(X_i^N(t))_{i\in \{1,\ldots,N\}}$ to \eqref{eq:multi-agent-system}-\eqref{eq:hypergraph_homogeneous_sim}-\eqref{eq:K2} with initial condition drawn randomly from the uniform distribution on $[0,1]\times[0,1]$.
For comparison, in Figure \ref{fig:xt} we also display the binned density computed by summing the number of agents in each square of dimension $0.1\times 0.1$.

\begin{figure}[h!]
\includegraphics[width = 0.24\textwidth]{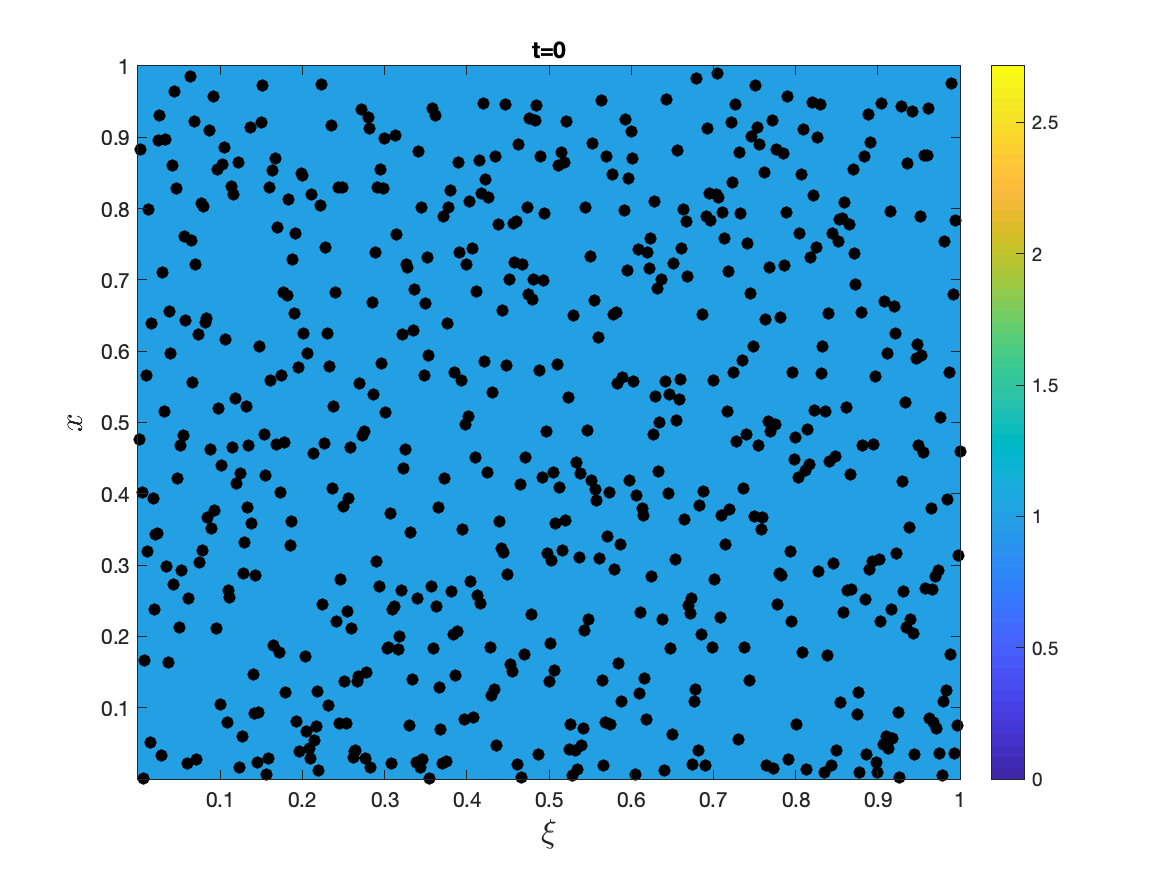}
\includegraphics[width = 0.24\textwidth]{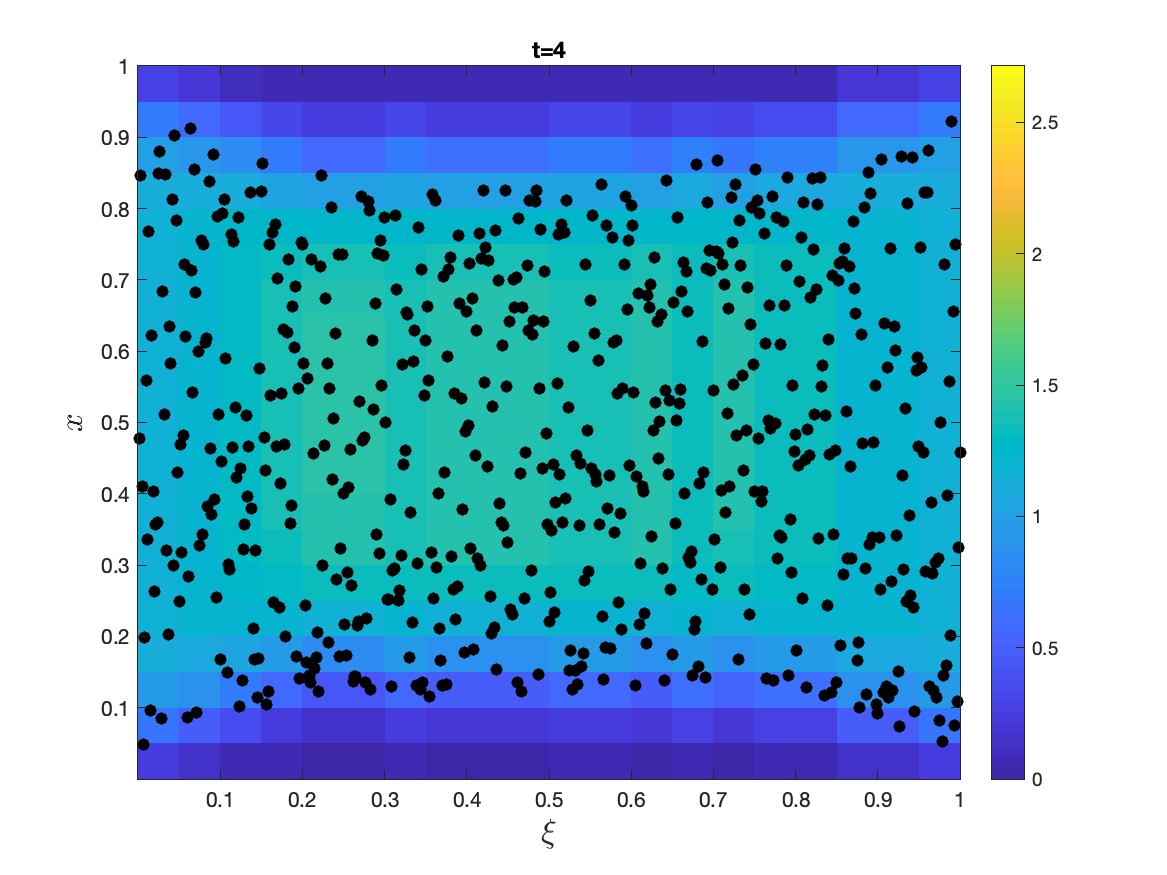}
\includegraphics[width = 0.24\textwidth]{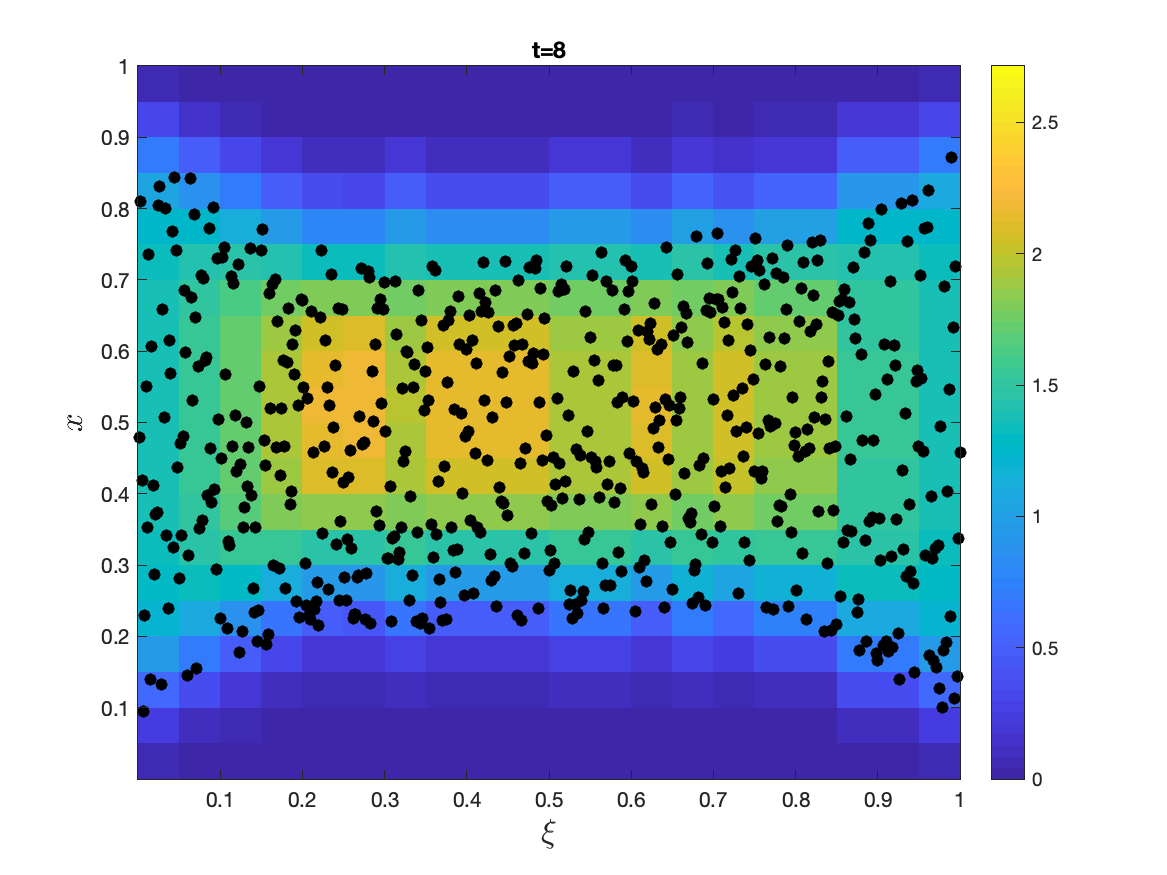}
\includegraphics[width = 0.24\textwidth]{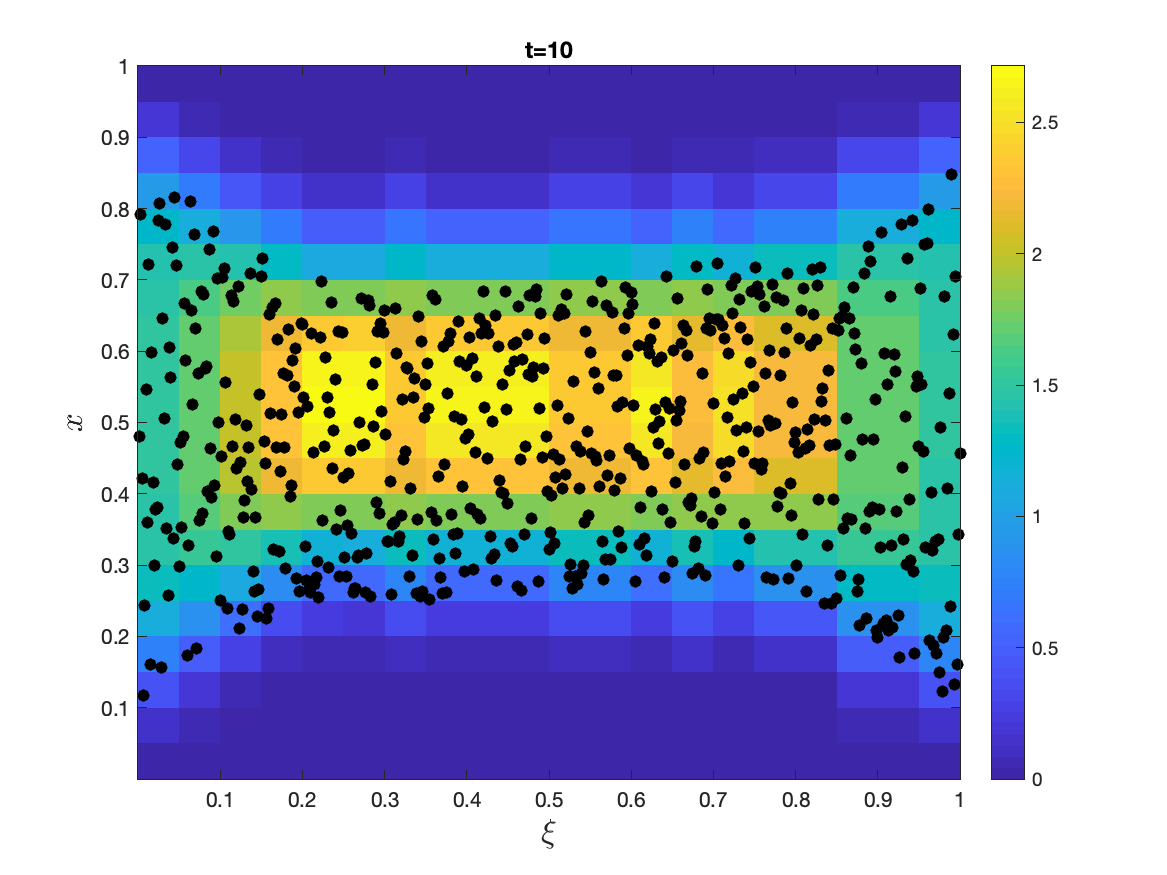}
\caption{Evolution of $\mu_t$, solution to the limit equation \eqref{eq:vlasov-equation}-\eqref{eq:hypergraphon_homogeneous_sim}-\eqref{eq:K2} (color gradient) at times $t=0$, $t=4$, $t=8$ and $t=10$, superimposed with the solution $X_i^N(t)$ of the microscopic system \eqref{eq:multi-agent-system}-\eqref{eq:hypergraph_homogeneous_sim}-\eqref{eq:K2} for $N=600$ (black dots).}
\label{fig:mut_xt}
\end{figure}

\begin{figure}[h!]
\includegraphics[width = 0.24\textwidth]{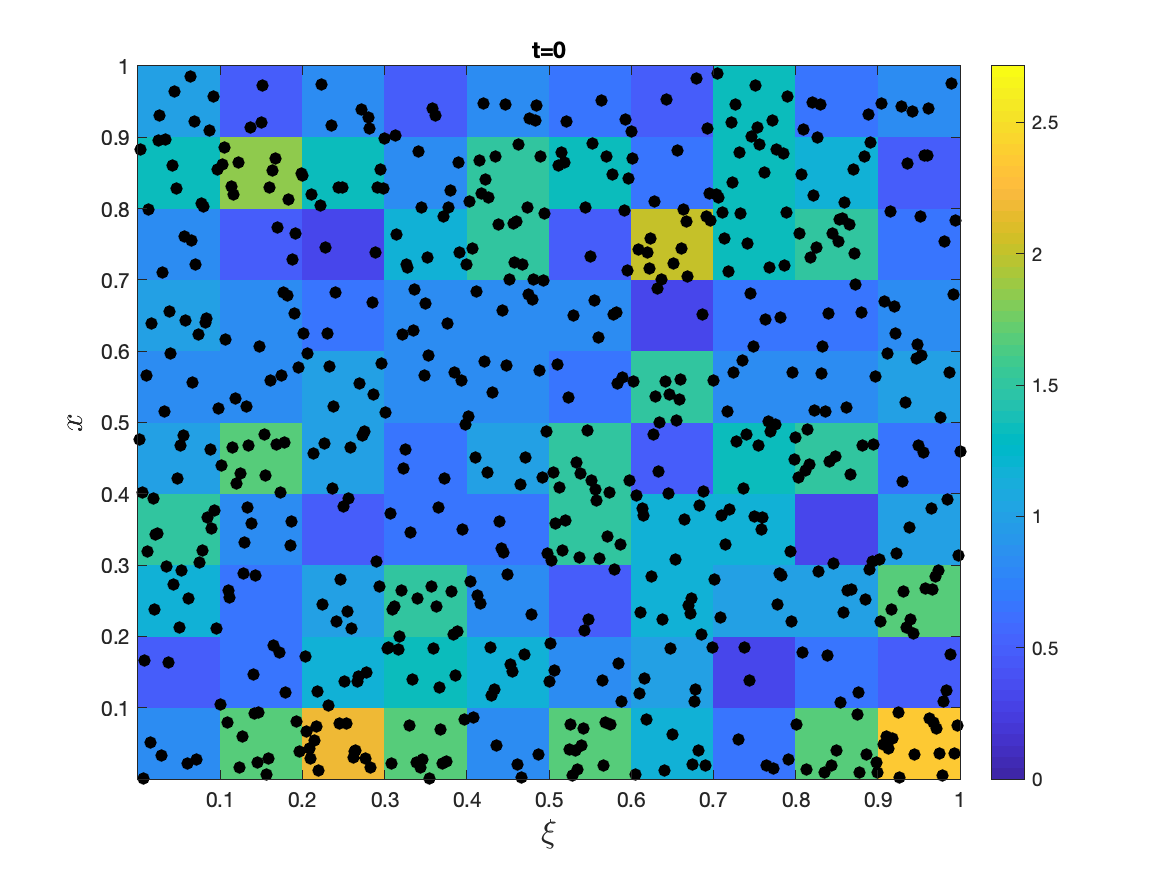}
\includegraphics[width = 0.24\textwidth]{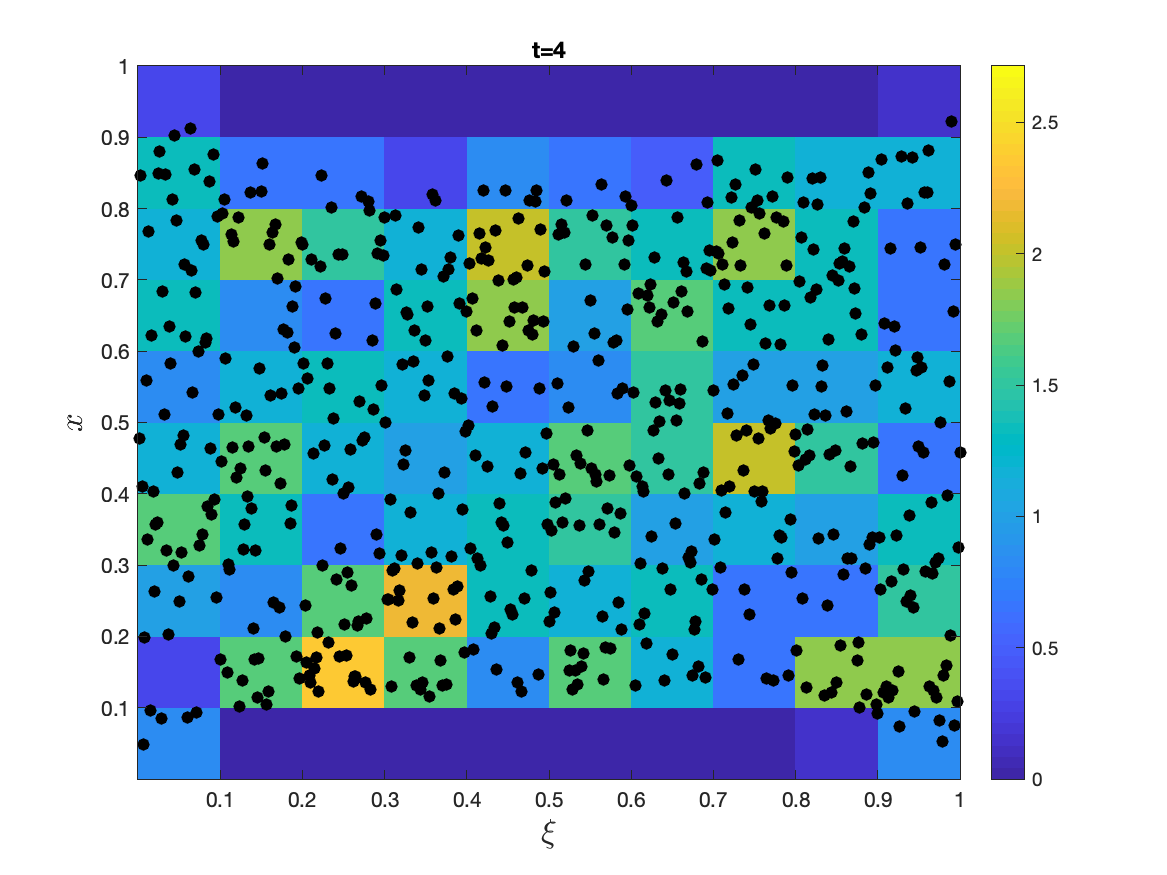}
\includegraphics[width = 0.24\textwidth]{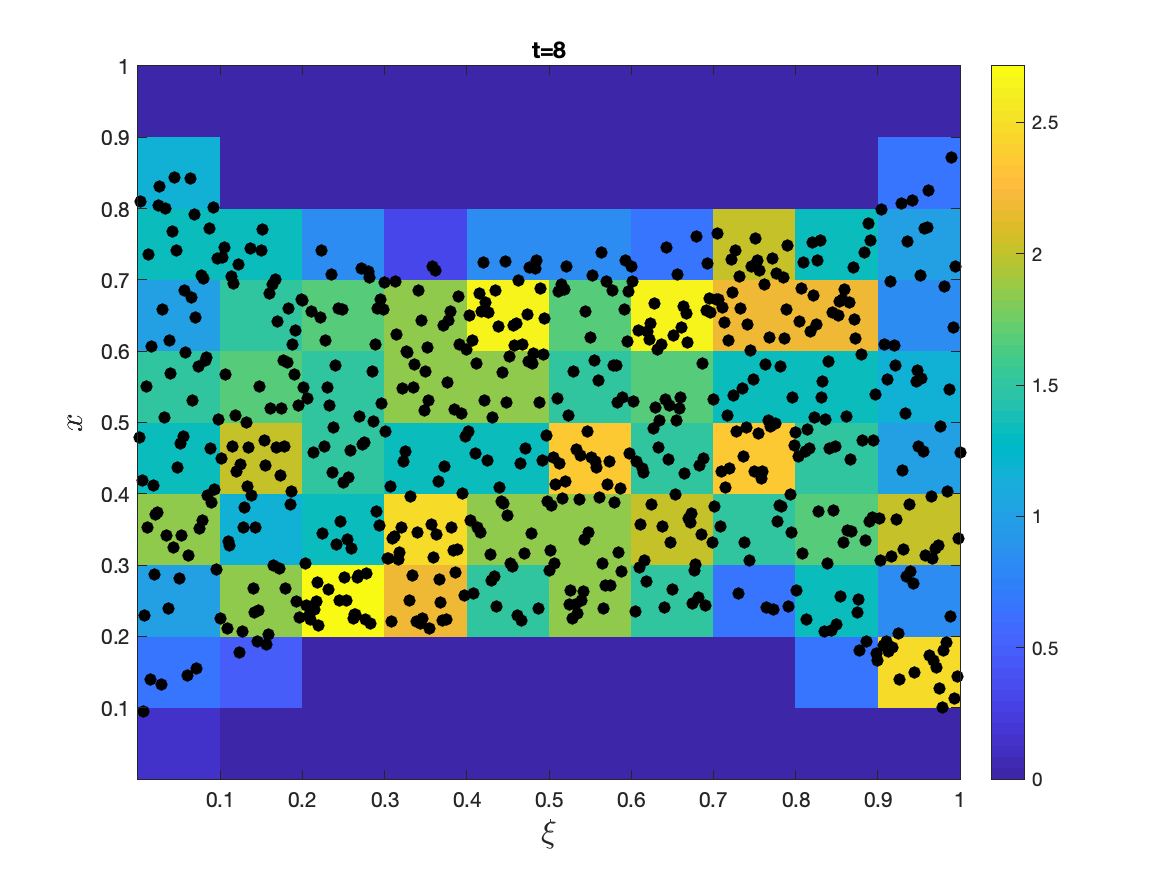}
\includegraphics[width = 0.24\textwidth]{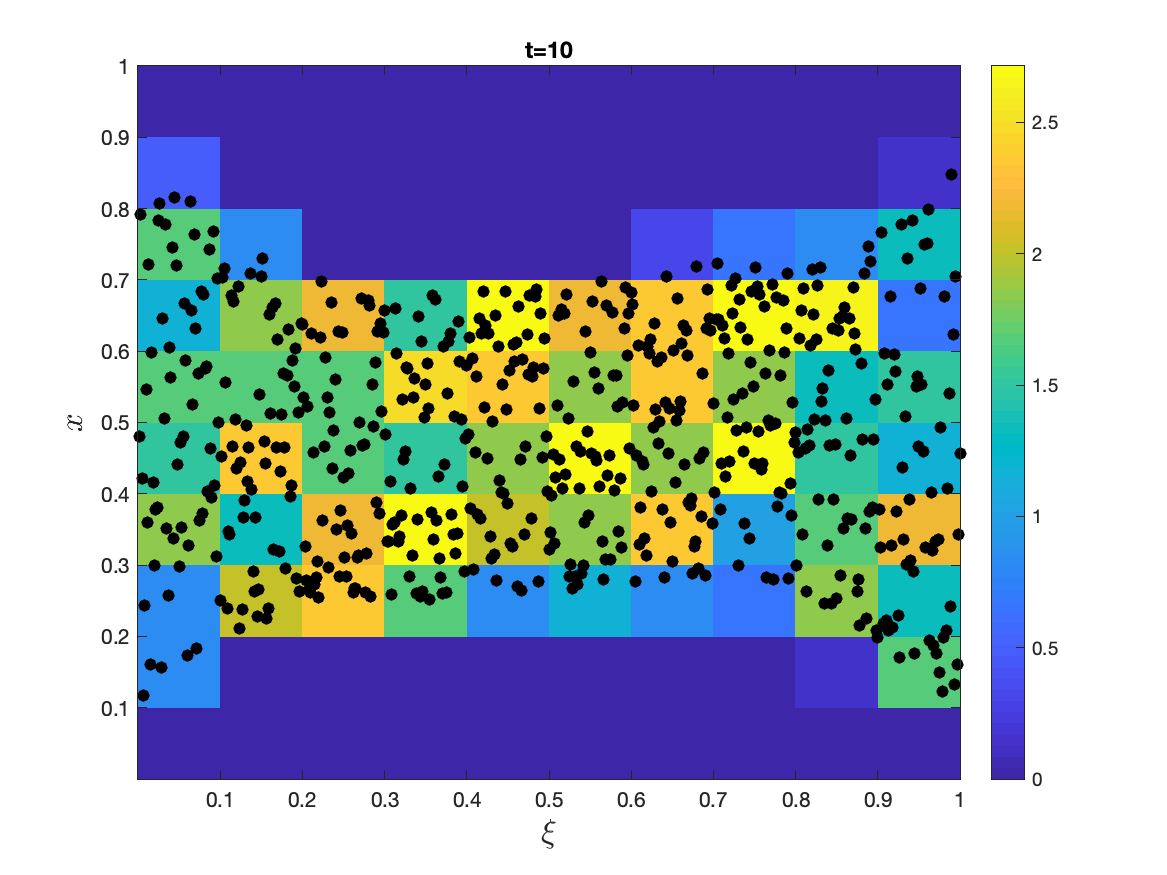}
\caption{Evolution of the solution $X_i^N(t)$ of the microscopic system \eqref{eq:multi-agent-system}-\eqref{eq:hypergraph_homogeneous_sim}-\eqref{eq:K2} for $N=600$ (black dots) at times $t=0$, $t=4$, $t=8$ and $t=10$, together with the binned density obtained by counting the number of agents in each square of dimension $0.1\times 0.1$ (color gradient).}
\label{fig:xt}
\end{figure}

In Figure \ref{fig:xt10}, we illustrate the convergence of the empirical measure to the solution of the limit equation by providing a comparison of the binned density of $X_i^N$ at the final time $t=10$ for different numbers of agents: 
$N=100$, $N=200$, $N=400$ and $N=600$.  
\begin{figure}[h!]
\includegraphics[width = 0.24\textwidth]{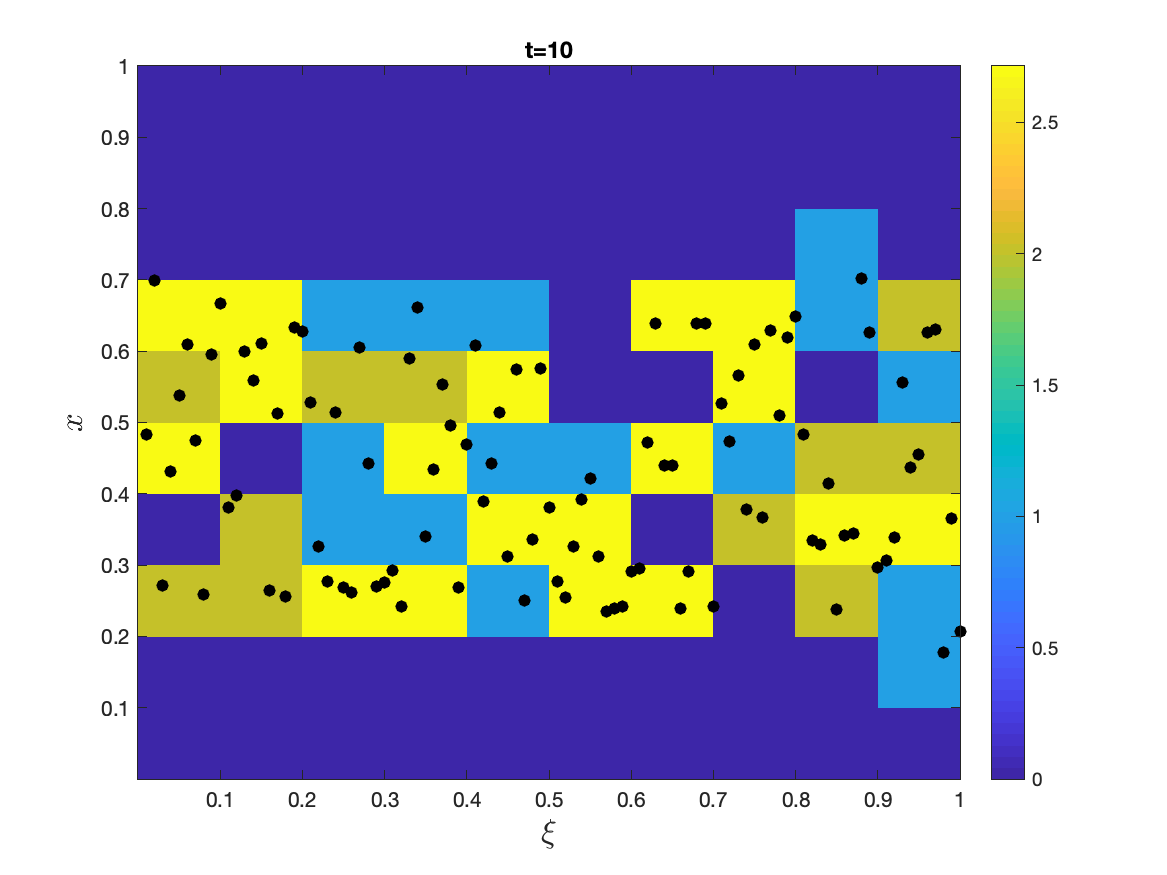}
\includegraphics[width = 0.24\textwidth]{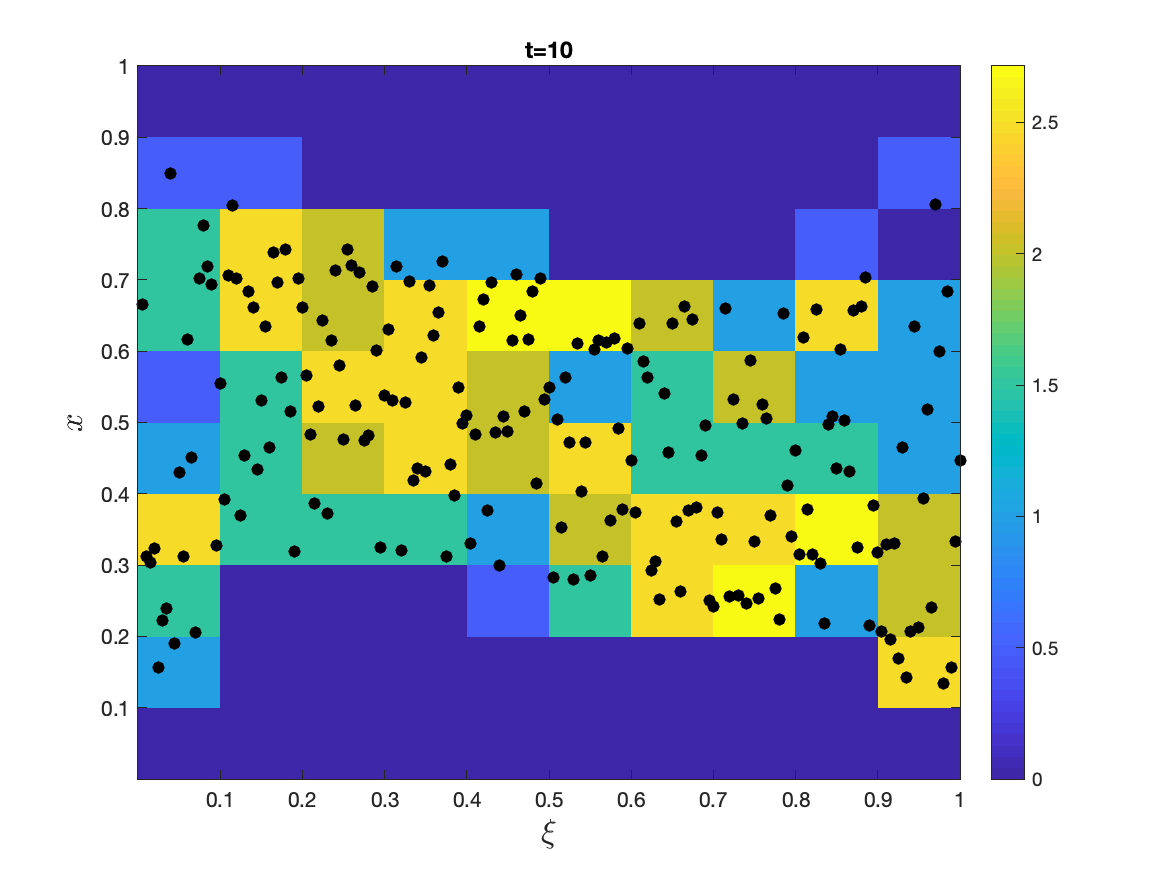}
\includegraphics[width = 0.24\textwidth]{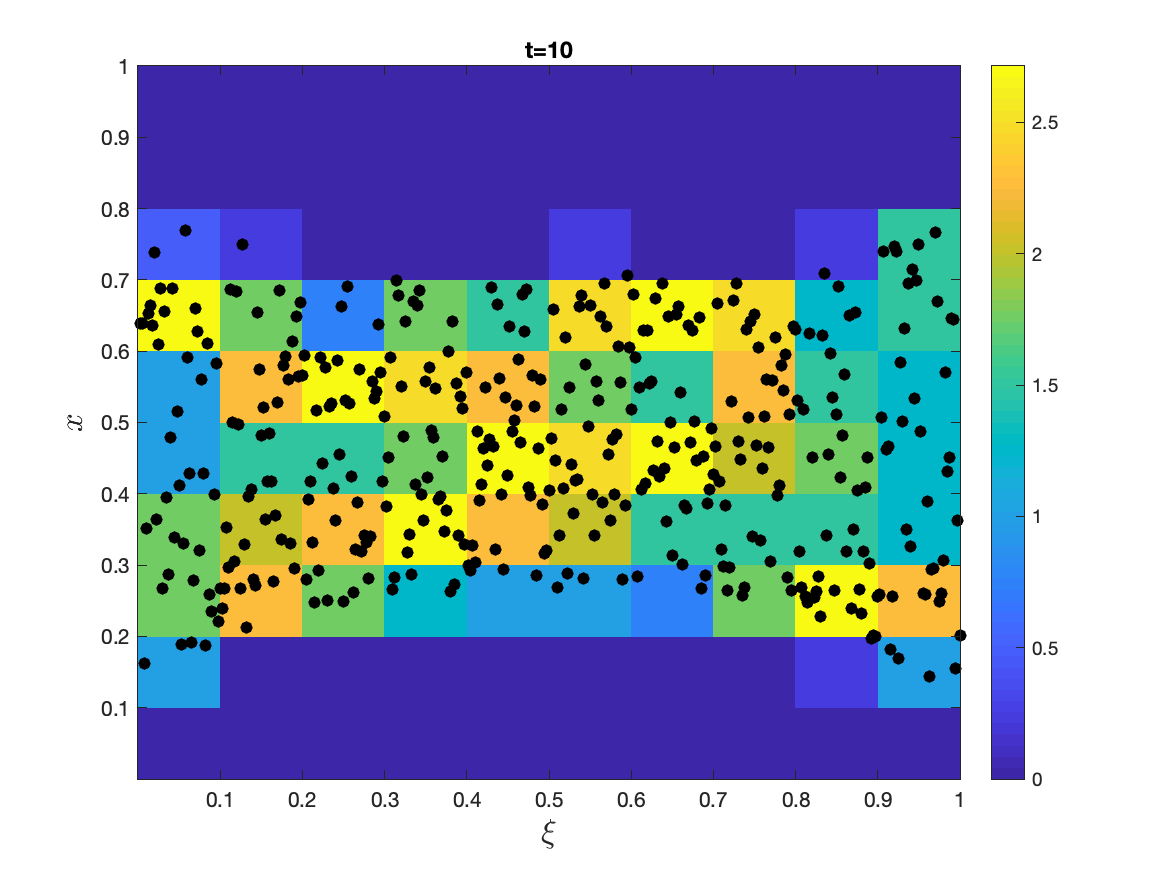}
\includegraphics[width = 0.24\textwidth]{Hypergraph3_mu0unif_x_N600_Nx20_Nxi20_Model3_Typegraphon2_a_linear_11}
\caption{Solution $X_i^N(T)$ of the microscopic system \eqref{eq:multi-agent-system}-\eqref{eq:hypergraph_homogeneous_sim}-\eqref{eq:K2} (black dots) at final time $T=10$, together with the binned density obtained by counting the number of agents in each square of dimension $0.1\times 0.1$ (color gradient) for $N=100$, $N=200$, $N=400$ and $N=600$ (left to right).}
\label{fig:xt10}
\end{figure}

\subsection{Hypergraph for a balanced group of rank 3}

We now consider the particle system \eqref{eq:multi-agent-system} posed on the hypergraph for a balanced group of rank $3$, that is,
\begin{equation}\label{eq:hypergraph_balanced_sim}
w^{2,N}_{j_0 j_1 j_2} = f\left(\frac{1}{N}\left|\frac{1}{3}\sum_{k=0}^2 j_k-\frac{N+1}{2}\right|\right).
\end{equation}
and $w^{\ell,N}_{i j_1\cdots j_\ell} = 0$ for all $\ell\neq 2$ and $i, j_1,\cdots, j_\ell\in\llbracket 1,N\rrbracket$, where $f: [0,\frac{1}{2}]\rightarrow\R_+$ is a continuous decreasing function.
This hypergraph was introduced in the example in Equation \eqref{eq:hypergraph_balanced} in Section \ref{subsec:examples}, and is represented in Figure \ref{fig:Hypergraph-balanced} (right).

Since $f$ is continuous, it can be shown by Proposition \ref{prop:conv_hypergraph_to_graphon2} (see also Remark \ref{rem:conv-balanced}) that the above sequence of hypergraphs \eqref{eq:hypergraph_balanced_sim} converges as $N$ tends to infinity to the limit UR-hypergraphon (with actually bounded rank) given by $w_\ell\equiv 0$ for all $\ell\neq 2$, and 
\begin{equation}\label{eq:hypergraphon_balanced_sim}
w_2(\xi_0,\xi_1,\xi_2) = f\left(\left|\frac{\xi_0+\xi_1+\xi_2}{3}-\frac{1}{2}\right|\right).
\end{equation}
As in the previous section, we consider the linear interaction kernel $K_2$ given by \eqref{eq:K2}.
The function $f$ is taken to be the continuous function 
$f: x\mapsto 4(x-\frac12)^2$.
Then, arguing as in Theorem \ref{theo:main-rigorous-formulation} we can show that the particle system \eqref{eq:multi-agent-system}-\eqref{eq:hypergraph_balanced_sim}-\eqref{eq:K2} converges to the solution to the Vlasov equation \eqref{eq:vlasov-equation}-\eqref{eq:hypergraphon_balanced_sim}-\eqref{eq:K2}, as a direct application of Proposition \ref{prop:conv_hypergraph_to_graphon2} . 

In Figure \ref{fig:mut-balanced}, the solution $\mu$ to \eqref{eq:vlasov-equation}-\eqref{eq:hypergraphon_balanced_sim}-\eqref{eq:K2} with uniform initial condition $\mu_0 = d\xi_{| [0,1]}\, dx_{| [0,1]}$ is represented for different time steps.

\begin{figure}[h!]
\includegraphics[width = 0.24\textwidth]{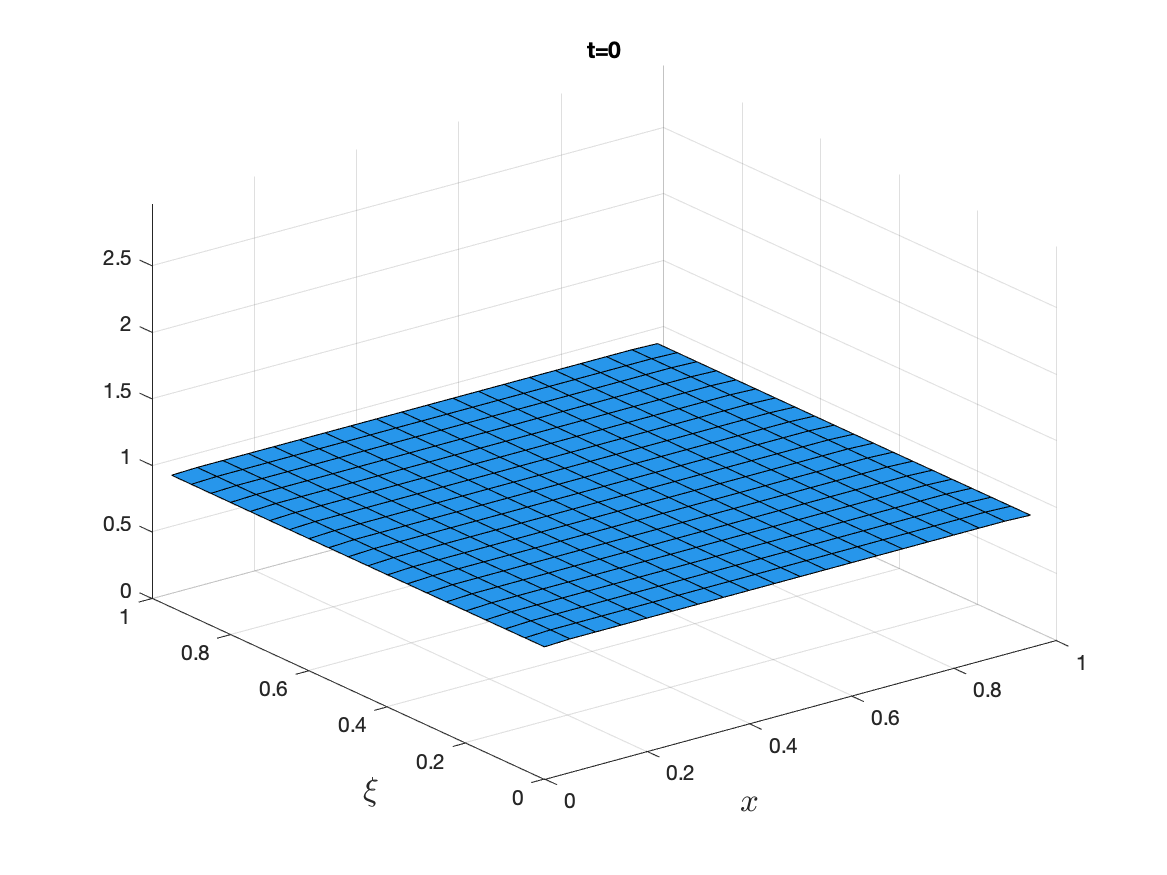}
\includegraphics[width = 0.24\textwidth]{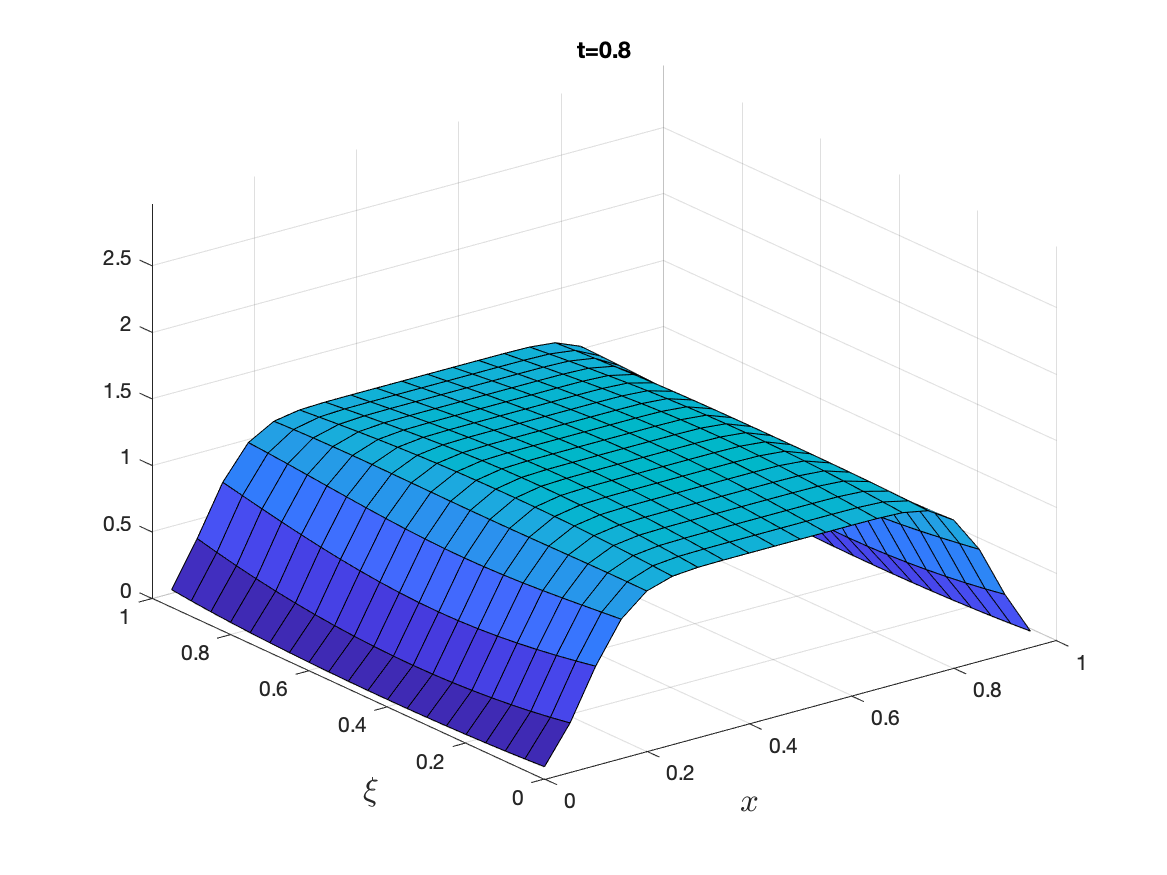}
\includegraphics[width = 0.24\textwidth]{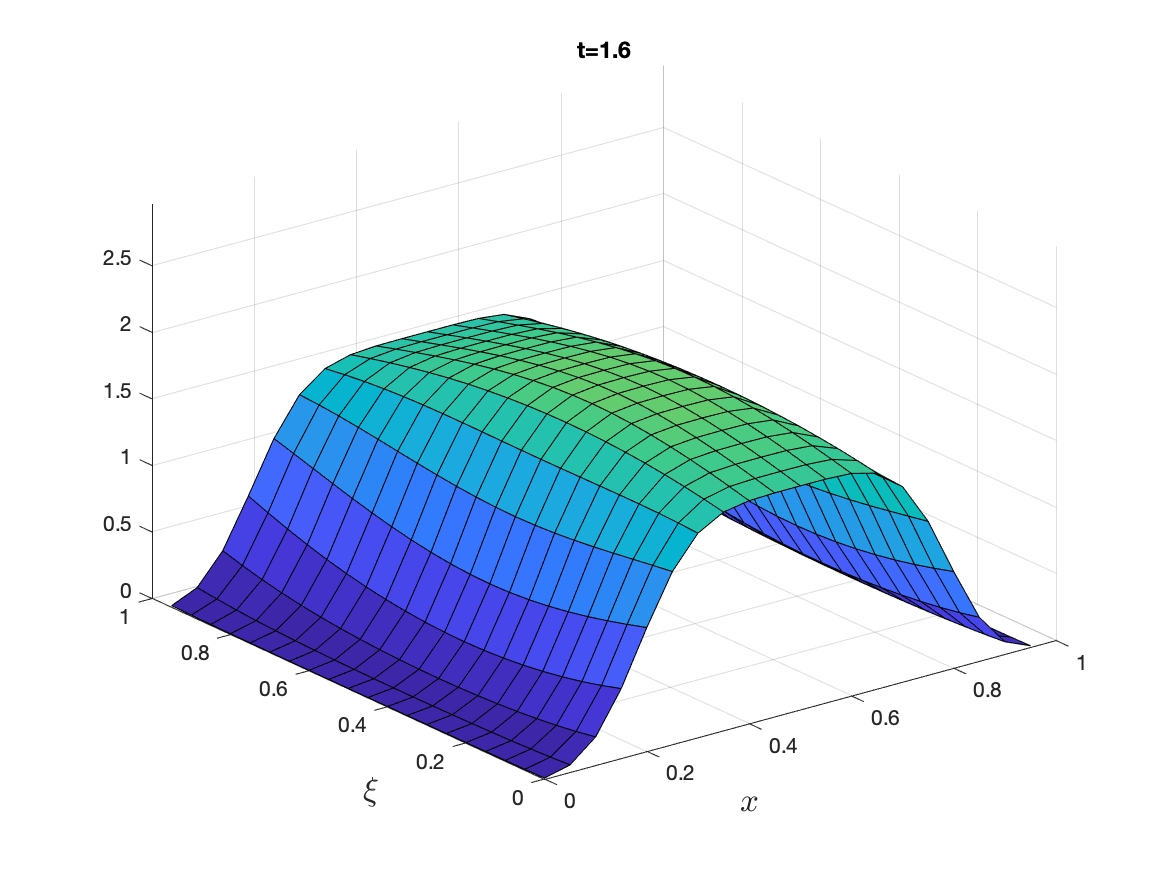}
\includegraphics[width = 0.24\textwidth]{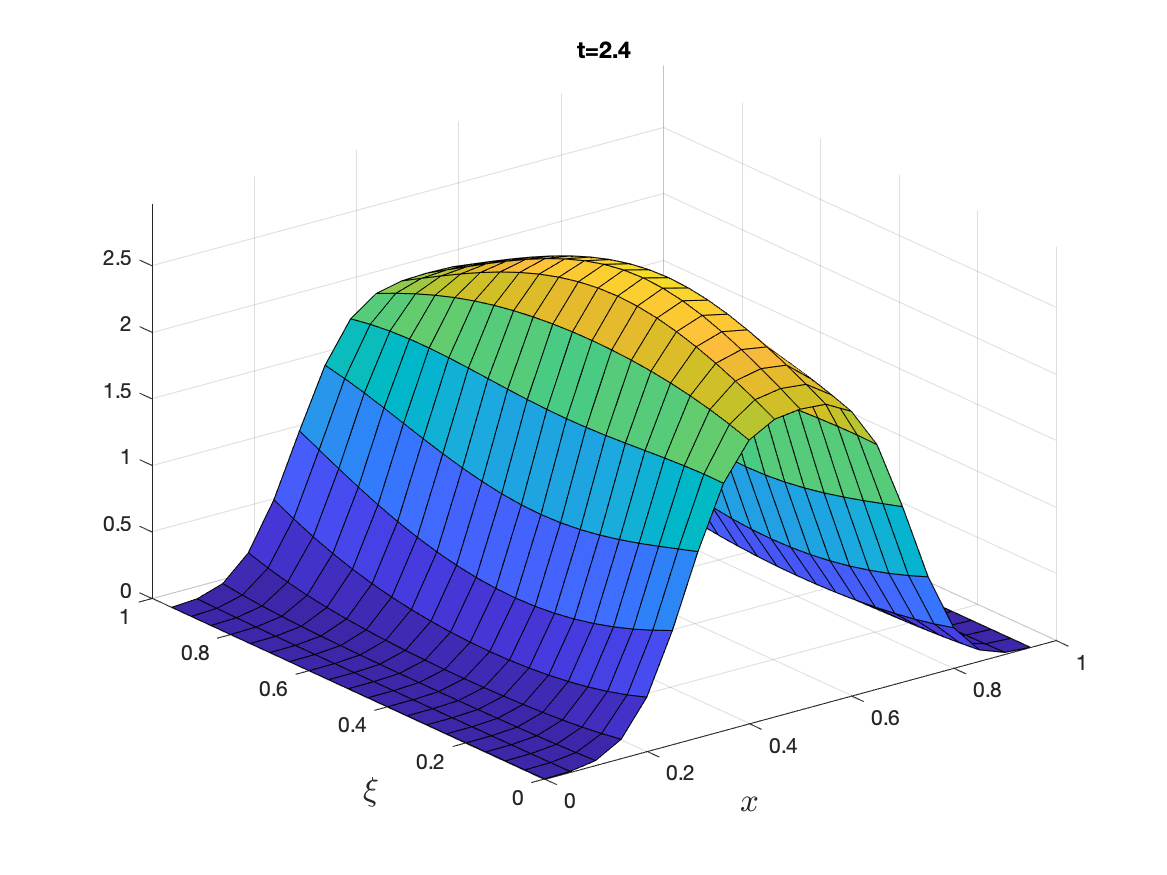}
\caption{Evolution of $\mu_t$, solution to the limit equation \eqref{eq:vlasov-equation}-\eqref{eq:hypergraphon_balanced_sim}-\eqref{eq:K2}, at times $t=0$, $t=0.8$, $t=1.6$ and $t=2.4$.}
\label{fig:mut-balanced}
\end{figure}

To compare with the solutions to the discrete and continuum systems, in Figure \ref{fig:mut_xt-balanced} we superimpose the solution $\mu_t$ to \eqref{eq:vlasov-equation}-\eqref{eq:hypergraphon_balanced_sim}-\eqref{eq:K2} with uniform initial condition on $[0,1]\times [0,1]$ (represented by a color gradient) with the solution $(X^N_i(t))_{i\in \{1,\ldots,N\}}$ solution to \eqref{eq:multi-agent-system}-\eqref{eq:hypergraph_balanced_sim}-\eqref{eq:K2} with an initial condition drawn randomly from the uniform distribution on $[0,1]\times[0,1]$.
For comparison, in Figure \ref{fig:xt-balanced} we also display the binned density computed by counting the number of agents in each square of dimension $0.1\times 0.1$.

\begin{figure}[h!]
\includegraphics[width = 0.24\textwidth]{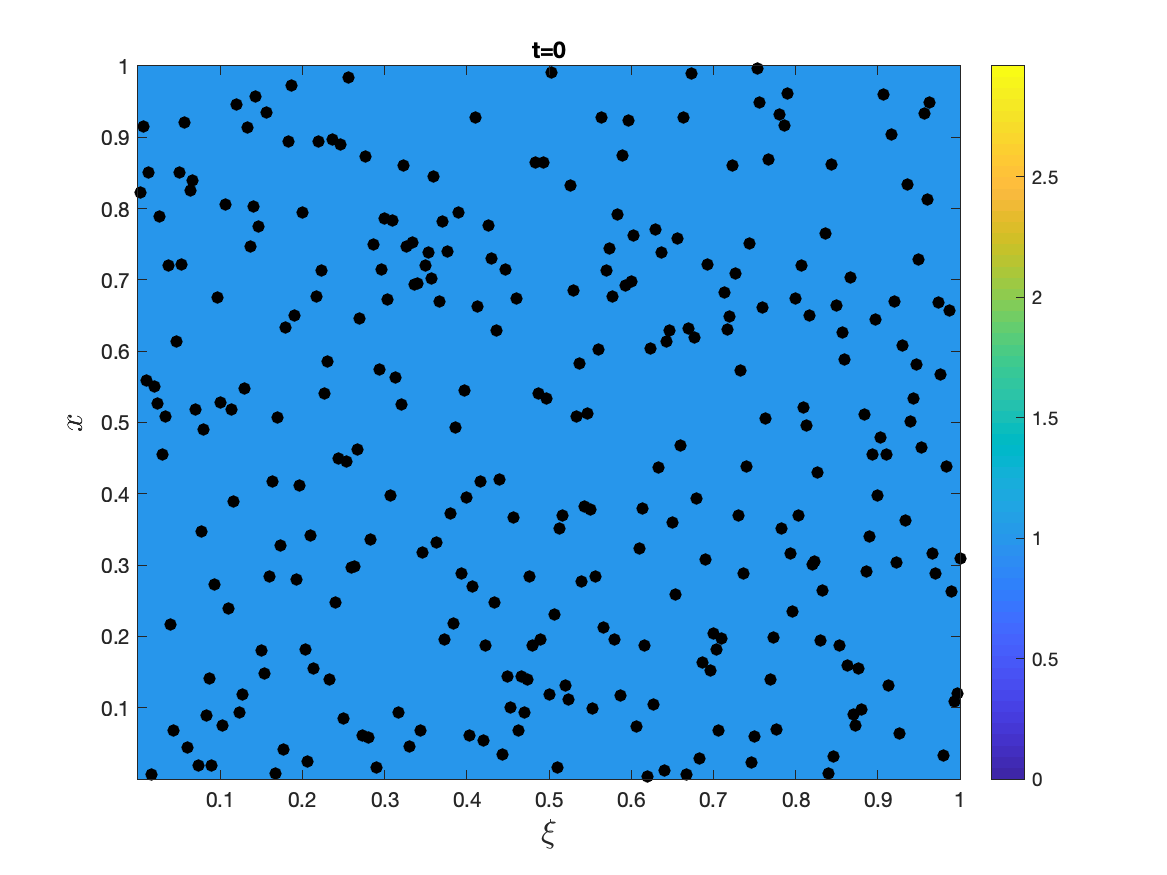}
\includegraphics[width = 0.24\textwidth]{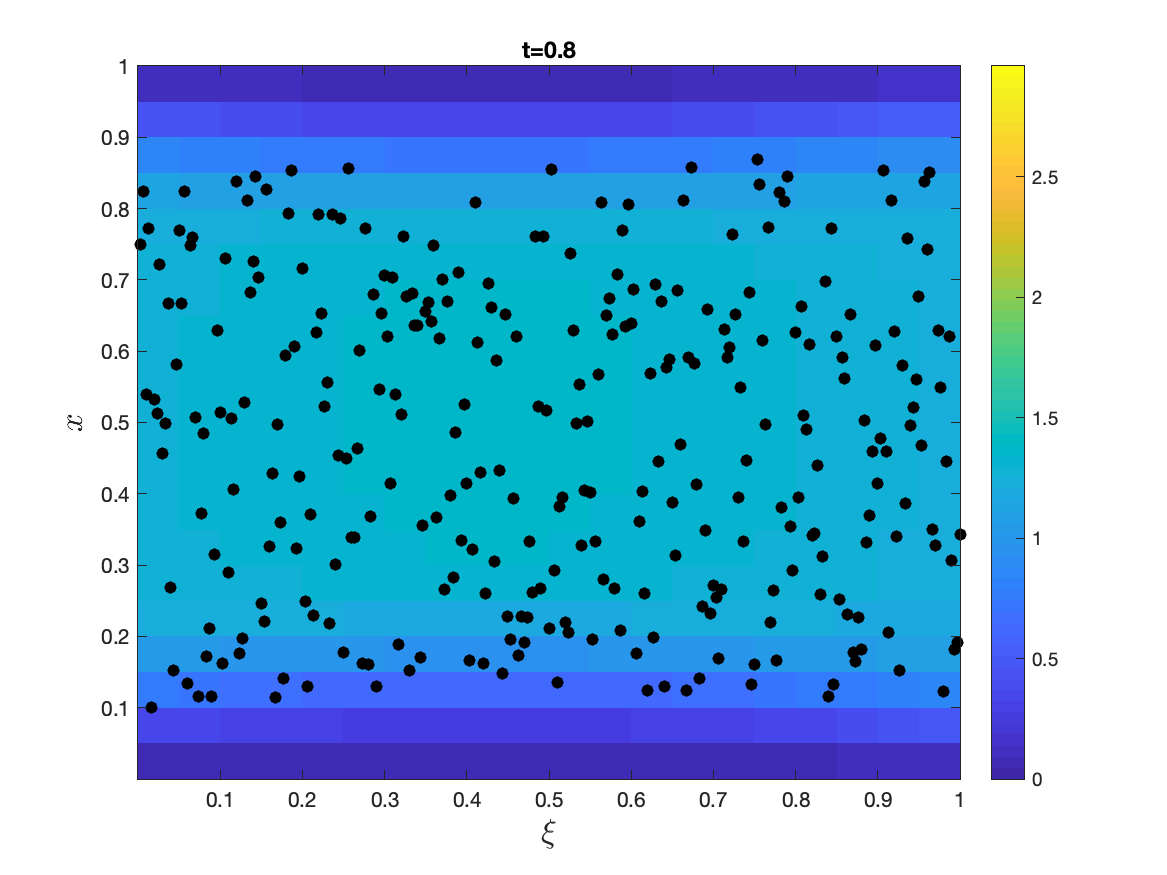}
\includegraphics[width = 0.24\textwidth]{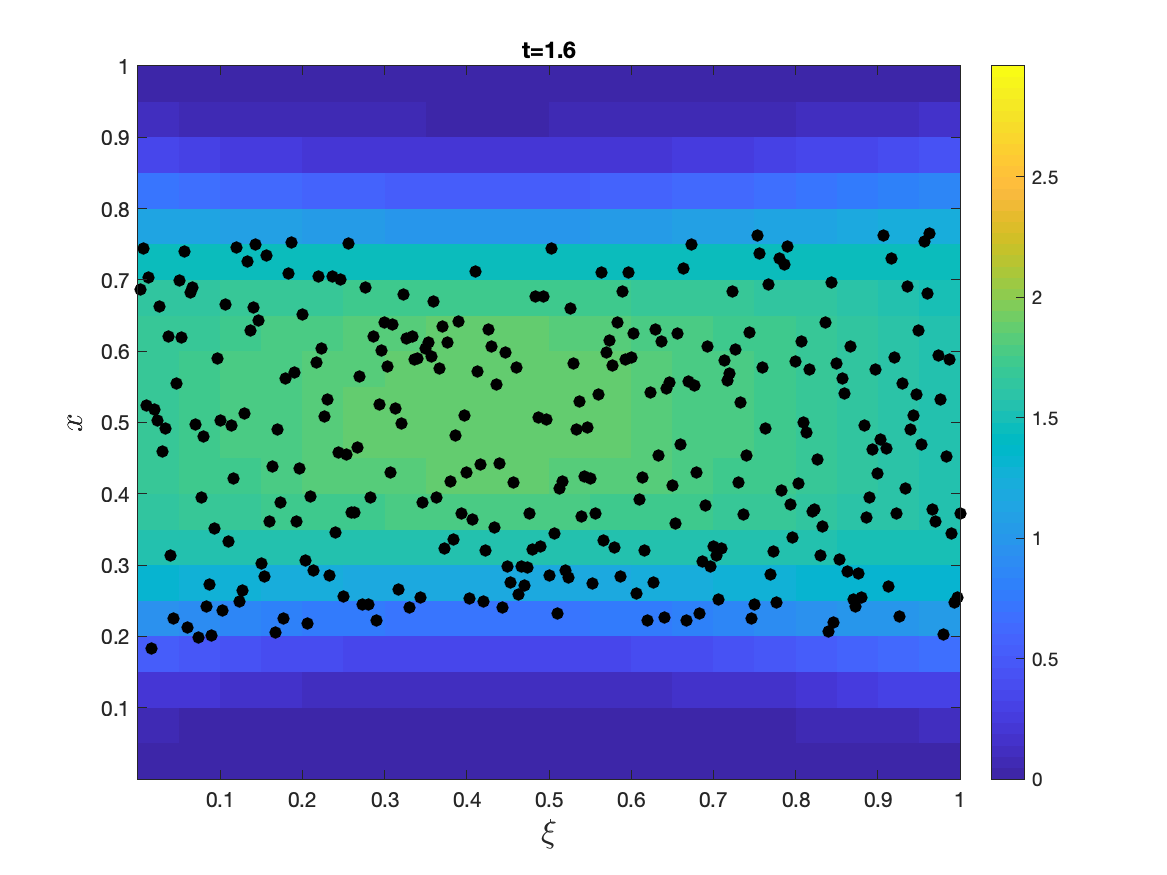}
\includegraphics[width = 0.24\textwidth]{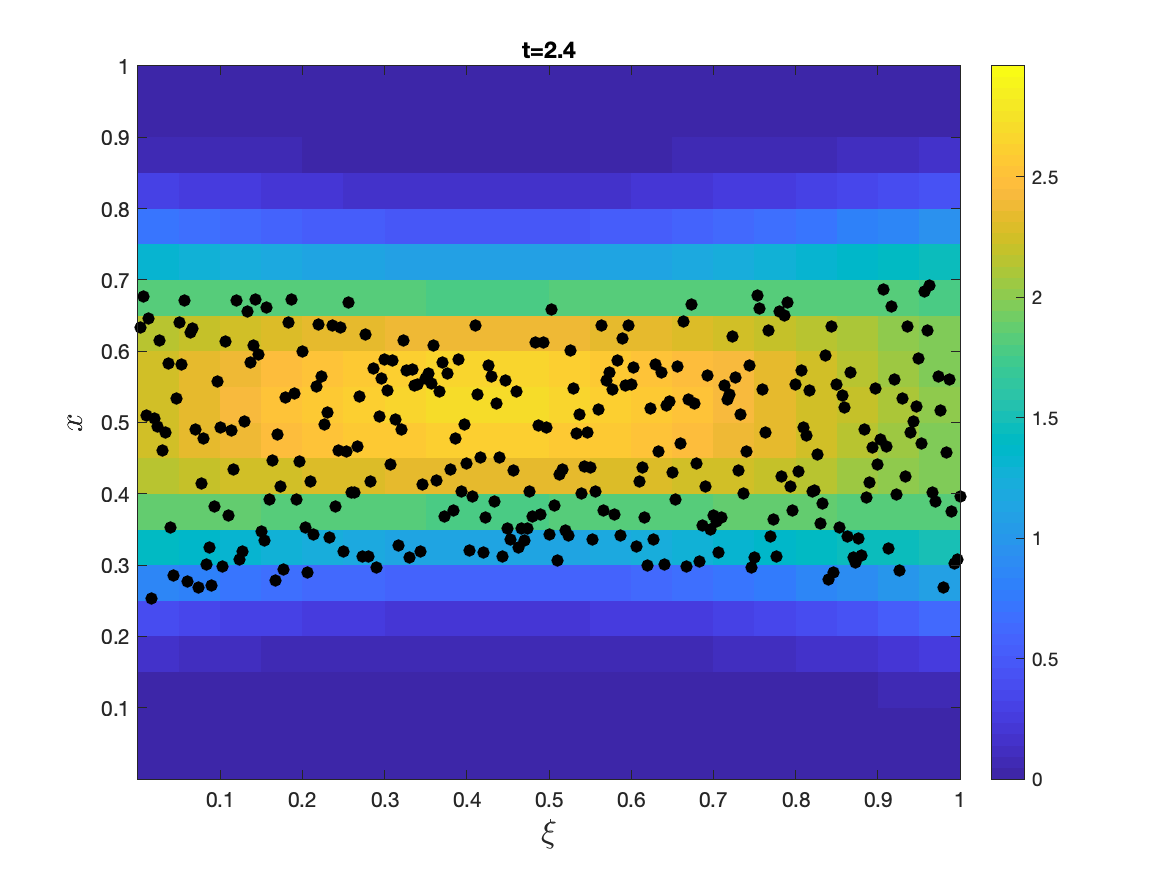}
\caption{Evolution of $\mu_t$, solution to the limit equation \eqref{eq:vlasov-equation}-\eqref{eq:hypergraphon_balanced_sim}-\eqref{eq:K2} (color gradient) at times $t=0$, $t=0.8$, $t=1.6$ and $t=2.4$, superimposed with the solution $X_i^N(t)$ of the microscopic system \eqref{eq:multi-agent-system}-\eqref{eq:hypergraph_balanced_sim}-\eqref{eq:K2} for $N=300$ (black dots).}
\label{fig:mut_xt-balanced}
\end{figure}

\begin{figure}[h!]
\includegraphics[width = 0.24\textwidth]{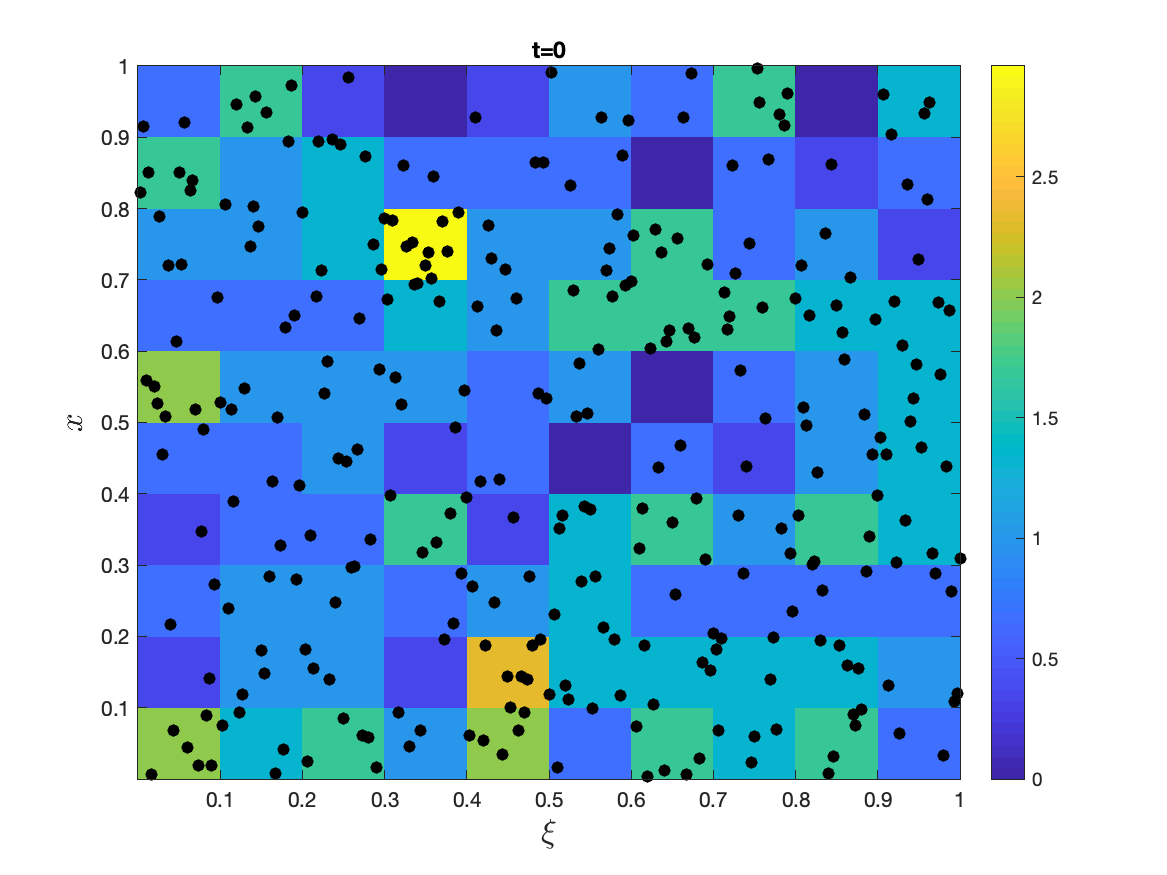}
\includegraphics[width = 0.24\textwidth]{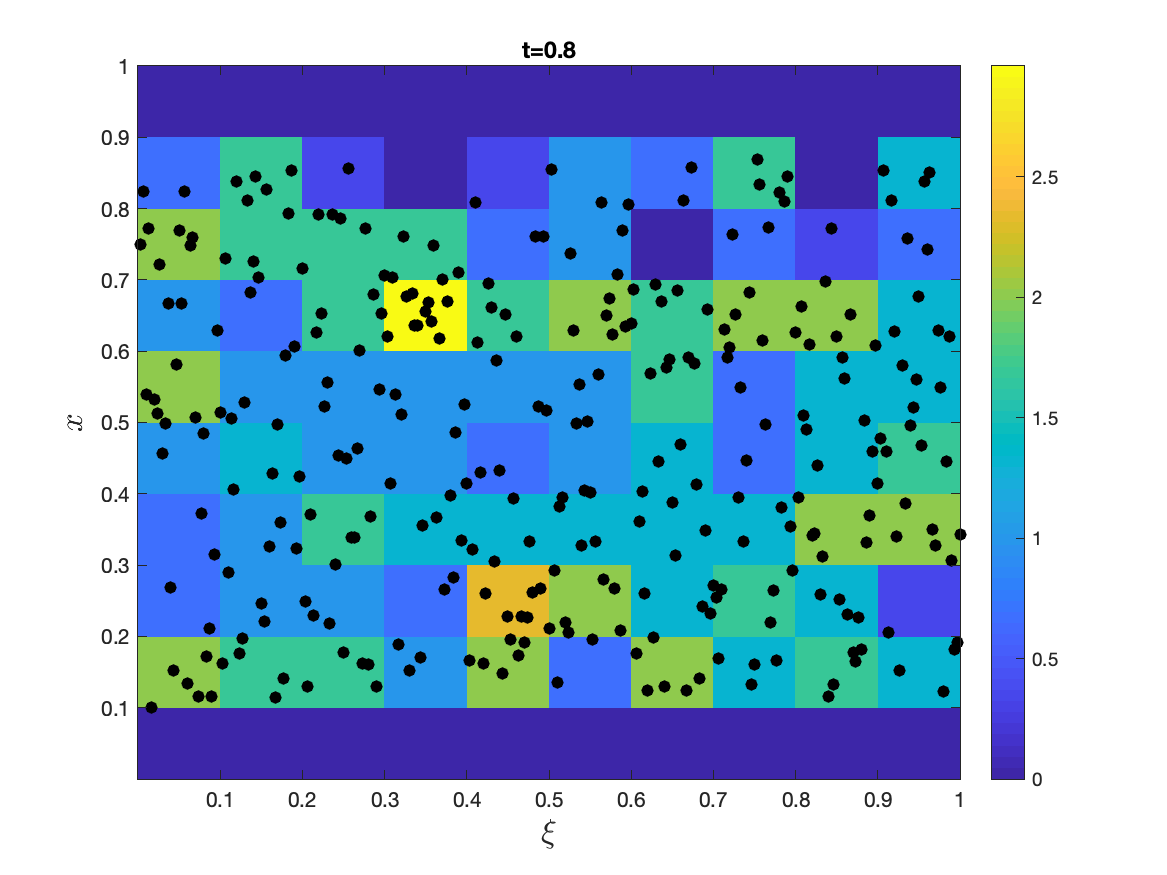}
\includegraphics[width = 0.24\textwidth]{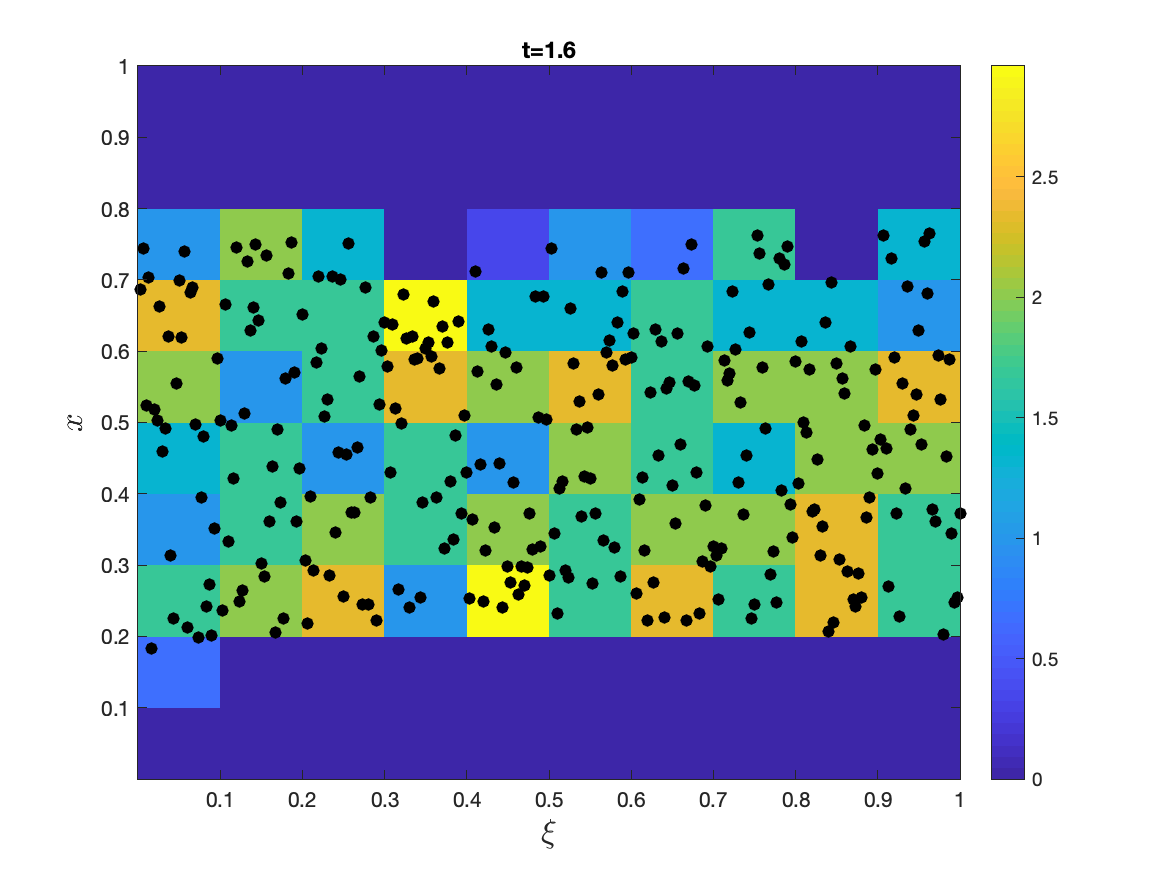}
\includegraphics[width = 0.24\textwidth]{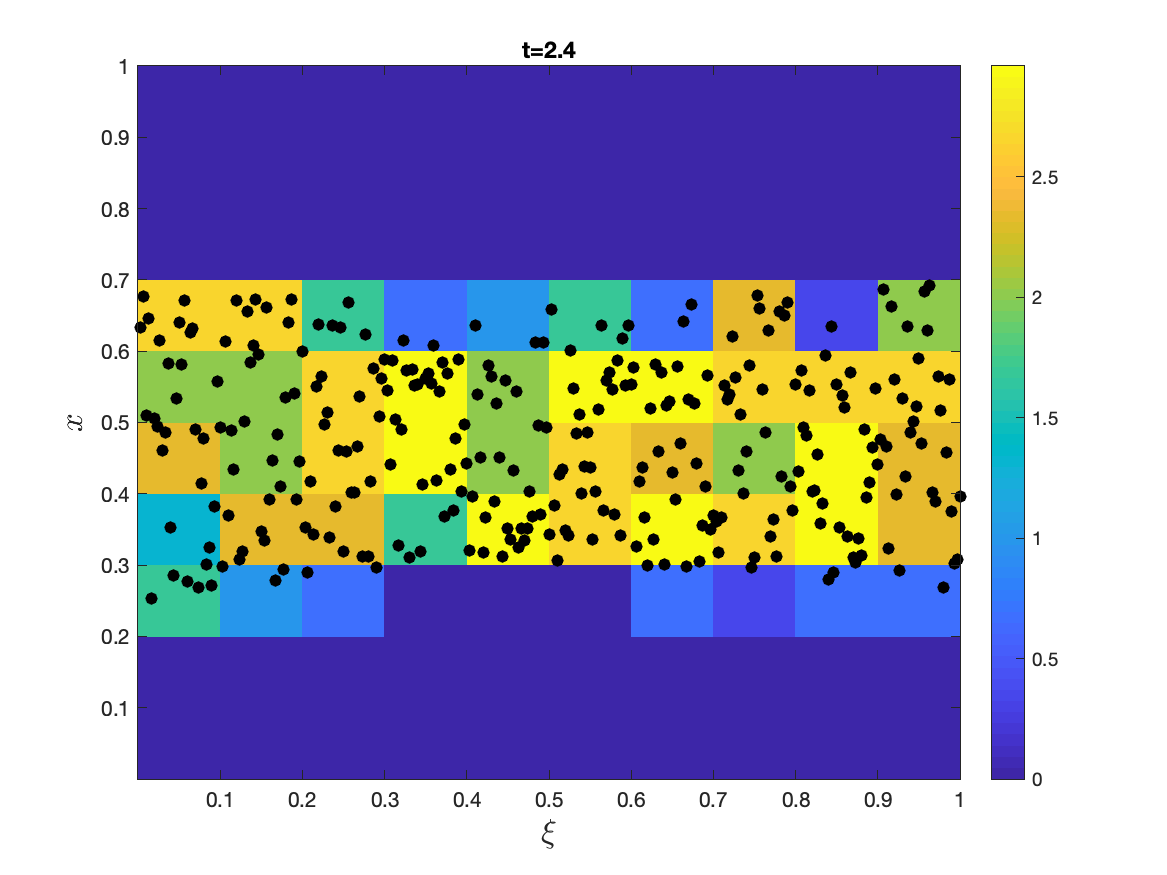}
\caption{Evolution of the solution $X_i^N(t)$ of the microscopic system \eqref{eq:multi-agent-system}-\eqref{eq:hypergraph_balanced_sim}-\eqref{eq:K2} for $N=300$ (black dots) at times $t=0$, $t=0.8$, $t=1.6$ and $t=2.4$, together with the binned density obtained by counting the number of agents in each square of dimension $0.1\times 0.1$ (color gradient).}
\label{fig:xt-balanced}
\end{figure}

\section{Conclusions and perspectives}\label{sec:conclusions}

In this paper, we derived rigorously the convergence of an interacting particle system with higher-order interactions toward a Vlasov-type equation on an unbounded-rank hypergraphon. The convergence was established taking advantage of a natural topology arising from the graph theory, {\it i.e.} the cut-distance, and its corresponding compactness properties. The innovation of our work lies both in the derivation of a model which allows for unbounded rank interactions, and in the exploitation of a natural topology coming from graph theory. There are however a certain number of directions that we would like to pursue in order to enrich this first result. We leave these questions to future work. 

Firstly, we observe that whilst throughout the paper initial data were only assumed independent, in the statement of our main Theorem \ref{theo:main-rigorous-formulation} we further impose that they are identically distributed according to the same law (possibly depending on $N$ but not $i$). This hypothesis is only a sufficient condition which helps to prove compactness of initial data. However, it could be possible to remove the {\it i.i.d.} property and also prove compactness of non-constant $\bar\mu^N_0$ associated to independent (but not {\it i.i.d.}) initial data $X_{1,0}^N,\ldots,X_{N,0}^N$, in a compatible way with the compactness of the UR-hypergraphons $w^N$ associated to the finite weights $w_{ij_1\cdots j_\ell}^{\ell,N}$.

Secondly, but closely related to the above question, in this paper pairs $(w,\mu)$ are endowed with the product topology of the cut-distance on the set of UR-hypergraphons $\mathcal{H}_W$ and the $d_{p,\nu}$ distance on the space of probability measures \corr{$\mathcal{P}_\nu(\mathbb{R}^d\times [0,1])$}, while in \cite{JPS-25} the space of pairs $(w,\mu)$ was endowed with the topology of the convergence of observables analogous to $\tau(H,w,\mu)$ in \eqref{eq:hierarchy-hypertrees-definition}. An interesting open question is to characterize whether the convergence $\tau(H,w^N,\mu^N)\to\tau(H,w,\mu)$ for all directed hypertree $H$ is implied by the convergence $\delta_\square(w^N,w)\to 0$ and $d_{p,\nu}(\mu^N,\mu)\to 0$ in the above-mentioned product topology, or conversely. 

Lastly, models for multi-agent dynamics on hypergraphs have recently attracted the attention of the Physics' community, with a focus on exploring how equilibrium states and speed of convergence are influenced by the presence of higher-order interactions (whether it be consensus in opinion dynamics, synchronization in Kuramoto-type oscillator dynamics, or spread of epidemics in SIR-type models). 
In that aim, the mean-field formulation provides a main advantage, as its formalism usually simplifies the study of the multi-agent system's
equilibrium states. Hence, provided that the long-time behavior of the macroscopic model approximates well that of the microscopic system, one could infer
from it important information concerning the collective dynamics and the phase transitions of the model.

\appendix

\section{Convergence of hypergraphs for homogeneous groups}\label{app:proof_convergence_hypergraphs_homogeneous_groups}

\begin{pro}\label{prop:convergence-discontinuous}
Consider the sequence of hypergraphs $(H_N)_{N\in\N}$ for homogeneous groups defined for some radius $\theta\in (0,1)$ by the hyperedge weights $(w^{\ell,N}_{ij_1\cdots j_\ell})_{i,j_1,\ldots,j_\ell\in \llbracket 1,N\rrbracket}$ given by
$$
w^{\ell,N}_{ij_1\cdots j_\ell} = \begin{cases}
\displaystyle  \frac{1}{N^\ell},& \text{if } \quad \max_{k_1,k_2\in \{i,j_1,\cdots, j_\ell\}} |k_1-k_2| \leq \theta N,\\
0,& \text{otherwise},
\end{cases}
$$
for each $N\in\N$ and $\ell\in \llbracket 1,N-1\rrbracket$.
Then, $(H_N)_{N\in\N}$ converges in the labeled cut-distance $d_\square$ (and hence also in the unlabeled cut-distance $\delta_\square$) to the UR-hypergraph $w=(w_\ell)_{\ell\in\N}$ given by
$$
w_{\ell}(\xi_0,\ldots,\xi_\ell) = \begin{cases}
\displaystyle  1, & \text{if } \quad \max_{k_1,k_2\in \{1,\ldots,\ell\}} |\xi_{k_1}-\xi_{k_2}| \leq \theta,\\
0,& \text{otherwise},
\end{cases}
$$
for all $\ell\in \mathbb{N}$.
\end{pro}

\begin{proof}
We begin by noticing that the sequence of hypergraphs $H_N$ can be obtained from the UR-hypergraphon $w=(w_\ell)_{\ell\in \mathbb{N}}$ by evaluating the hypergraphon $w_\ell$ on the grid points as in Proposition \ref{prop:conv_hypergraph_to_graphon2}. However, the hypergraphon $w$ does not satisfy the continuity assumption in Proposition \ref{prop:conv_hypergraph_to_graphon2}. We will see that, in this case, convergence of $(H_N)_{N\in\N}$ still holds.

Consider the piecewise-constant UR-hypergraphon $w^{H_N}$ defined from the hypergraph weights as in Equation \eqref{eq:hypergraph_to_hypergraphon}, {\it i.e.}
\[
\begin{split}
w_\ell^{H_N} (\xi_0,\bxi_\ell) &= N^\ell w^{\ell,N}_{i j_1\cdots j_\ell} = N^\ell \frac{1}{N^\ell} w_\ell\left(\frac{i-1}{N},\ldots,\frac{j_\ell-1}{N}\right)
 = N^\ell \frac{1}{N^\ell} w_\ell\left(\frac{\lfloor N\xi_0\rfloor}{N},\ldots,\frac{\lfloor N\xi_\ell \rfloor}{N}\right) \\
& = \begin{cases}
\displaystyle  1, & \text{if } \quad \max_{k_1,k_2\in \llbracket 0,\ell\rrbracket} \left|\frac{\lfloor N\xi_{k_1} \rfloor}{N}-\frac{\lfloor N\xi_{k_2} \rfloor}{N}\right| \leq \theta,\\
0, & \text{otherwise},
\end{cases}
\end{split}
\]
for all $(i,j_1,\ldots,j_\ell)\in \llbracket 1,N\rrbracket^{\ell+1}$, for all $(\xi_0, \xi_1,\ldots,\xi_\ell)\in I_i^N\times I_{j_1}^N\times\cdots\times I_{j_\ell}^N$. 

We aim to show that $\lim_{N\to\infty} \|w_\ell-w_\ell^{H_N}\|_{L^1(I^{\ell+1})} = 0$ for all $\ell\in\N$. To this end we shall argue differently for $I^{\ell+1}\setminus U$ and on $U$, where $U$ is the subset defined by 
$$U := \left\{(\xi_0,\bxi_\ell)\in I^{\ell+1}:\, \theta -\frac{1}{N}\leq \max_{(i,j)\in\llbracket 0,N\rrbracket^2} |\xi_i-\xi_j| \leq \theta+\frac{1}{N}\right\}.$$

We begin to show that for all $(\xi_0,\bxi_\ell)\in I^{\ell+1}\setminus U$, it holds $w_\ell(\xi_0,\bxi_\ell) = w_\ell^{H_N}(\xi_0,\bxi_\ell)$.

Firstly, consider $(\xi_0,\bxi_\ell)\in I^{\ell+1}\setminus U$, and suppose that  $\max_{(i,j)\in\llbracket 0,N\rrbracket^2} |\xi_i-\xi_j| < \theta-\frac{1}{N}$. Then, $w_\ell(\xi_0,\bxi_\ell) =1$.
Moreover, for all $(i,j)\in\llbracket 0,N\rrbracket^2$, it holds
\[
\left| \frac{\lfloor N\xi_i\rfloor}{N} - \frac{\lfloor N\xi_j\rfloor}{N}\right| = \left|\xi_i+\frac{\{N\xi_i\}}{N}-\xi_j-\frac{\{N\xi_j\}}{N}\right| \leq |\xi_i-\xi_j| + \frac{1}{N}\left| \{\xi_iN\}-\{\xi_jN\}\right| < \theta -\frac{1}{N}+\frac{1}{N} = \theta,
\]
where for $z\in\R_+$, we denoted $\{z\} := z-\lfloor z \rfloor \in [0,1)$.
Thus from the definition of $w^{H_N}_\ell$ above, it holds $w^{H_N}_\ell(\xi_0,\bxi_\ell) =1=w_\ell(\xi_0,\bxi_\ell)$.

Secondly, if $(\xi_0,\bxi_\ell)\in I^{\ell+1}\setminus U$, and suppose that there exists ${(i,j)\in\llbracket 0,N\rrbracket^2}$ such that $|\xi_i-\xi_j| > r+\frac{1}{N}$, then $w_\ell(\xi_0,\bxi_\ell) =0$.
Moreover, 
\[
\left| \frac{\lfloor N\xi_i\rfloor}{N} - \frac{\lfloor N\xi_j\rfloor}{N}\right| \geq \left| |\xi_i-\xi_j| - \frac{1}{N}\left| \{\xi_iN\}-\{\xi_jN\}\right|\right| > \theta +\frac{1}{N}-\frac{1}{N} = \theta,
\]
so that $w^{H_N}_\ell(\xi_0,\bxi_\ell) = 0 = w_\ell(\xi_0,\bxi_\ell)$.

Then, altogether implies
\[
\begin{split}
\|w_\ell-w_\ell^{H_N}\|_{L^1(I^{\ell+1})} & = \int_U |w_\ell(\xi_0,\bxi_\ell)-w_\ell^{H_N}(\xi_0,\bxi_\ell)|\,d\xi_0\,d\bxi_\ell
\leq \int_U d\xi_0\,d\bxi_\ell  \\
&\leq\begin{pmatrix}
\ell+1\\
2
\end{pmatrix}
 \int_{I^{\ell-1}}\int_{\theta-\frac{1}{N}\leq |\xi_0-\xi_1| \leq \theta+\frac{1}{N}} d\xi_0\, d\xi_1\cdots d\xi_\ell \leq \frac{4}{N}
\begin{pmatrix}
\ell+1\\
2
\end{pmatrix}
 = 2\frac{\ell(\ell+1)}{N}.
\end{split}
\]
Hence, $\lim_{N\to\infty} \|w_\ell-w_\ell^{H_N}\|_{L^1(I^{\ell+1})} = 0$ for each $\ell\in\N$, and concluding with the dominated convergence theorem as in Proposition \ref{prop:conv_hypergraph_to_graphon1}, we deduce that for any summable sequence $(\alpha_\ell)_{\ell\in\N}$, it holds $\lim_{N\to\infty} d_\square(w_\ell,w_\ell^{H_N};(\alpha_\ell)_{\ell\in\N}) = 0$.
\end{proof}

\bibliographystyle{amsplain} 
\bibliography{APP_2024_mean_field_UR_hypergraphons}

\end{document}